\documentclass[10pt,reqno]{amsart} % I find amsat easier to read on a computer, feel free to switch it back - B
\usepackage[margin=1in]{geometry} 

\usepackage{graphicx}
\usepackage[nice]{nicefrac}
\usepackage{amssymb}
\usepackage{epstopdf}
\usepackage{amsmath,amsthm,amscd,amssymb} %, mathrsfs} % I don't have the fonts for mathrsfs on my laptop - B
\usepackage{verbatim}
\allowdisplaybreaks
\usepackage{latexsym}
\usepackage[colorlinks,citecolor=red,pagebackref,hypertexnames=false]{hyperref}

\usepackage{enumerate}
\usepackage{tikz}
\usetikzlibrary{decorations.pathreplacing}

\numberwithin{equation}{section}

\theoremstyle{plain}
\newtheorem{theorem}{Theorem}[section]
\newtheorem{lemma}[theorem]{Lemma}

\newtheorem{proposition}[theorem]{Proposition}

\theoremstyle{definition}
\newtheorem{definition}[theorem]{Definition}

\newtheorem{problem}[theorem]{Problem}

\newtheorem{obs}{Observation}

\theoremstyle{remark}
\newtheorem{remark}[theorem]{Remark}

\newtheorem{notation}[theorem]{Notation}
\newtheorem{claim}[theorem]{Claim}

\newcommand{\vphi}{\varphi}
\newcommand*{\rom}[1]{\expandafter\@slowromancap\romannumeral #1@}
\newcommand{\f}{\widehat}

\title[\parbox{14cm}{\centering{}} \quad]{Five-Linear Singular Integral Estimates of Brascamp-Lieb Type}

\author{Camil Muscalu}
\address{Department of Mathematics, Cornell University, Ithaca, NY }
\email{camil@math.cornell.edu}

\author{Yujia Zhai}
\address{Laboratoire de Math\'ematiques, Universit\'e de Nantes, Nantes}
\email{yujia.zhai@univ-nantes.fr}

\begin{document}
\maketitle

\begin{abstract} 
We prove the full range of estimates for a five-linear singular integral of Brascamp-Lieb type. The study is methodology-oriented with the goal to develop a sufficiently general technique to estimate singular integral variants of Brascamp-Lieb inequalities that do not obey H\"older scaling. The invented methodology constructs localized analysis on the entire space from local information on its subspaces of lower dimensions and combines such tensor-type arguments with the generic localized analysis. A direct consequence of the boundedness of the five-linear singular integral is a Leibniz rule which captures nonlinear interactions of waves from transversal directions. %We prove a Leibniz rule which corresponds to the boundedness of one case of bi-parameter paraproducts with flag singularities. What is more interesting than the result is the machinery invented - a robust stopping-time argument which incorporates information on subspaces to derive estimates on the entire space. As expected to be an efficient treatment of multi-parameter objects, the machinery indeed gives rise to a general class of Leibniz rules. 
%We prove full range of estimates for bi-parameter flag paraproducts, including end-point and mixed-norm estimates, on restricted function spaces. The machinery invented to treat the multi-parameter objects is a robust stopping-time argument which incorporates information on subspaces to derive estimates on the entire space.
\end{abstract}

\section{Introduction} \label{section_intro}
% subject to study
% connection with previous ones - applications 
% machinery - robust treatment for multi-parameter objects, - estimates involving L^\infty norms(nonlinear interaction between waves coming transversally) 
\subsection{Background and Motivation} \label{section_intro_background}
Brascamp-Lieb inequalities refer to inequalities of the form
\begin{align} \label{classical_bl}
\displaystyle \int_{\mathbb{R}^n} \big|\prod_{j=1}^{m}F_j(L_j(x))\big| dx  \leq \text{BL}(\textbf{L,p})\prod_{j=1}^{m}\left(\int_{\mathbb{R}^{k_j}}|F_j|^{p_j}\right)^{\frac{1}{p_j}},
\end{align}
%\begin{align} \label{classical_bl}
%\displaystyle \big|\int_{\mathbb{R}^n} \prod_{j=1}^{m}F_j(L_j(x)) dx \big| \leq \text{BL}(\textbf{L,p})\prod_{j=1}^{m}\|F_j\|_{L^{p_j}(\mathbb{R}^{k_j})},
%\end{align}
where $\text{BL}(\textbf{L,p})$ represents the Brascamp-Lieb constant depending on $\textbf{L} := (L_j)_{j=1}^m$ and $\textbf{p} := (p_j)_{j=1}^m$. For each $1 \leq j \leq m$, $L_j: R^{n} \rightarrow R^{k_j}$ is a linear surjection and $p_j \geq 0$. One equivalent formulation of (\ref{classical_bl}) is 
\begin{align} \label{classical_bl_exp}
\displaystyle \bigg(\int_{\mathbb{R}^n} \big|\prod_{j=1}^{m}F_j(L_j(x))\big|^r dx\bigg)^{\frac{1}{r}} \leq \text{BL}(\textbf{L},r\textbf{p})\prod_{j=1}^{m}\left(\int_{\mathbb{R}^{k_j}}|F_j|^{rp_j}\right)^{\frac{1}{rp_j}},
\end{align}
%\begin{align} \label{classcial_bl_exp}
%\displaystyle \bigg(\int_{\mathbb{R}^n} \big|\prod_{j=1}^{m}F_j(L_j(x))\big|^r dx\bigg)^{\frac{1}{r}} \leq \text{BL}(\textbf{L},r\textbf{p})\prod_{j=1}^{m}\|F_j\|_{L^{rp_j}(\mathbb{R}^{k_j})},
%\end{align}
for any $r > 0$. Brascamp-Lieb inequalities have been well-developed in \cite{bl}, \cite{bcct}, \cite{bbcf}, \cite{bbbf}, \cite{chv}. Examples of Brascamp-Lieb inequalities consist of H\"older's inequality and the Loomis-Whitney inequality.

We adopt the informal definition that an $n$-linear singular integral operator $T$ is \textit{of Brascamp-Lieb type} if $T(F_1,\cdots, F_n)$ is reduced to the integrand on the left hand side of a classical Brascamp-Lieb inequality when the kernels are replaced by Dirac distributions $\delta_0$. As a consequence of the definition, a singular integral operator of Brascamp-Lieb type obeys the same scaling property as the corresponding classical Brascamp-Lieb inequality, which will be called the basic inequality. We will also refer to the estimate for the singular integral operator of Brascamp-Lieb type as the \textit{singular integral estimate of Brascamp-Lieb type}. %\footnote{Basic inequalities refer to the inequalities obtained for Dirac kernels.} 
For the readers familiar with the recent expository work of Durcik and Thiele in \cite{dt2}, this is similar to the generic estimate for the singular Brascamp-Lieb form described in (2.3) from \cite{dt2}.

Singular integral estimates with H\"older scaling have been studied extensively, including boundedness of single-parameter paraproducts \cite{cm} and multi-parameter paraproducts \cite{cptt}, \cite{cptt_2}, single-parameter flag paraproducts \cite{c_flag}, bilinear Hilbert transform \cite{lt}, multilinear operators of arbitrary rank \cite{mtt2002}, etc. 
%Our study is originally motivated by private conversations with Jonathan Bennett about singular integral variant of Brascamp-Lieb inequalities. 
But it is of course natural to ask if there are similar singular integral estimates corresponding to Brascamp-Lieb inequalities that do not satisfy H\"older scaling. 
%This question  was  asked  to us  byJonathan  Bennett  during  a  conference  in  Matsumoto,  Japan,  in  February  2016.   
So far, to the best of our knowledge, the only research article in the literature where the term "singular Brascamp-Lieb" has been used is the recent work by Durcik and Thiele \cite{dt}. However, we would like to emphasize that the basic inequalities  corresponding to the "cubic singular expressions" considered in \cite{dt} are still H\"older's inequality, and the term "singular Brascamp-Lieb" was used to underline that the necessary and sufficient boundedness condition (1.6) of \cite{dt} is of the same flavor as the one for classical Brascamp-Lieb inequalities stated as (8) in \cite{bcct}.
%A more formal definition with full generality is stated in (2.3) of \cite{dt2}. %which to describe a generic framework. %which summarizes singular Brascamp-Lieb estimates of both H\"older and non-H\"older types.

Techniques to tackle multilinear singular integral operators with H\"older scaling \cite{cm}, \cite{cptt}, \cite{cptt_2}, \cite{c_flag}, \cite{lt}, \cite{mtt2002} usually involve localizations on phase space subsets of the full-dimension. In contrast, the understanding of singular integral operators of Brascamp-Lieb type with input functions defined on subspaces of the domain of the output function  
%estimates corresponding to Brascamp-Lieb inequalities with $k_j < n$ for some $k_j$ in (\ref{classical_bl_exp}) 
(and thus with non-H\"older scaling) is in its primitive stage. The ultimate goal would be to develop a general methodology to treat a large class of singular Brascamp-Lieb estimates with non-H\"older scaling. It is natural to believe that such an approach would need to extract and integrate local information on subspaces of lower dimensions. Also due to its multilinear structure, localizations on the entire space could be necessary as well and a hybrid of both localized analyses would be demanded.

The subject of our study in this present paper is one of the simplest multilinear operators, whose complete understanding cannot be reduced to earlier results\footnote{Many cases of arbitrary complexity follow from the mixed-norm estimates for vector-valued inequalities in the paper by Benea and the first author \cite{bm2}.} %Many examples of singular integral estimates of Brascamp-Lieb type, including the boundedness of paraproducts and the bilinear Hilbert transform of Loomis-Whitney type, can be derived from works of Benea and Muscalu 
and which requires such a new type of analysis. More precisely, it is the five-linear operator defined by
%The subject of our study is the simplest nontrivial multi-parameter multilinear operator whose boundedness remains unknown.It is natural to realize that in order to serve as a good candidate for the establishment of localization on subspaces, such operator outputs a function of exactly two variables and inputs some functions of one variable. In particular, it is a five-linear singular integral operator defined as
%with a symbol $m$ 
%$$
%m(\xi) = a((\xi_1,\eta_1),(\xi_2,\eta_2)) b((\xi_1,\eta_1),(\xi_2,\eta_2),(\xi_3,\eta_3)) 
%$$
%where 
\begin{align} \label{bi_flag_int}
&  T_{K_1K_2}(f_1, f_2, g_1, g_2, h)(x,y)  \nonumber \\
:= &  p.v. \displaystyle \int_{\mathbb{R}^{10}} K_1\big((t_1^1, t_1^2),(t_2^1,t_2^2))K_2((s_1^1,s_1^2), (s_2^1,s_2^2), (s_3^1,s_3^2)\big) \cdot \nonumber \\
&\quad \quad \quad f_1(x-t_1^1-s_1^1)f_2(x-t_2^1-s_2^1)g_1(y-t_1^2-s_1^2)g_2(y-t_2^2-s_2^2)h(x-s_3^1,y-s_3^2) d\vec{t_1} d\vec{t_2} d \vec{s_1} d\vec{s_2} d\vec{s_3},
\end{align}
where $\vec{t_i} = (t_i^1, t_i^2)$, $\vec{s_j} = (s_j^1,s_j^2)$ for $i = 1, 2$ and $j = 1,2,3$. In (\ref{bi_flag_int}), $K_1$ and $K_2$ are Calder\'on-Zygmund kernels that satisfy
\begin{align*}
& |\nabla K_1(\vec{t_1}, \vec{t_2})| \lesssim \frac{1}{|(t_1^1,t_2^1)|^{3}}\frac{1}{|(t_1^2,t_2^2)|^{3}}, \nonumber \\
&  |\nabla K_2(\vec{s_1}, \vec{s_2}, \vec{s_3})| \lesssim \frac{1}{|(s_1^1,s_2^1,s_3^1)|^{4}}\frac{1}{|(s_1^2,s_2^2, s_3^2)|^{4}} .
\end{align*}
%There are several perspectives for the multilinear operator of our study. 
%Our goal is to prove H\"older-type estimates for (\ref{bi_flag_int}).
As one can see, the operator $T_{K_1K_2}$ takes two functions depending on the $x$ variable ($f_1$ and $f_2$), two functions depending on the $y$ variable ($g_1$ and $g_2$) and one depending on both $x$ and $y$ (namely $h$) into another function of $x$ and $y$. 
Our goal is to prove that $T_{K_1K_2}$ satisfies the mapping property
$$
L^{p_1}(\mathbb{R}) \times L^{q_1}(\mathbb{R}) \times L^{p_2}(\mathbb{R}) \times L^{q_2}(\mathbb{R}) \times L^{s}(\mathbb{R}^2) \rightarrow L^{r}(\mathbb{R}^2)
 $$
for $1 < p_1, p_2, q_1, q_2, s \leq \infty$, $r >0$, $(p_1,q_1), (p_2,q_2) \neq (\infty, \infty)$ with
\begin{equation} \label{bl_exp}
\frac{1}{p_1} + \frac{1}{q_1} + \frac{1}{s} = \frac{1}{p_2} + \frac{1}{q_2} + \frac{1}{s} = \frac{1}{r}.
\end{equation}
To verify that the boundedness of $T_{K_1K_2}$ qualifies to be a singular integral estimate of Brascamp-Lieb type, one can remove the singularities by setting
\begin{align*}
& K_1(\vec{t_1}, \vec{t_2}) = \delta_{\textbf{0}}(\vec{t_1}, \vec{t_2}), \nonumber \\
& K_2(\vec{s_1}, \vec{s_2}, \vec{s_3}) = \delta_{\textbf{0}}(\vec{s_1}, \vec{s_2}, \vec{s_3}),
\end{align*}
and express its boundedness explicitly as 
%(\ref{bi_flag_int}) to (\ref{classical_bl})
\begin{align} \label{flag_bl}
\|f_1(x) f_2(x) g_1(y) g_2(y) h(x,y)\|_{r} \lesssim \|f_1\|_{L^{p_1}(\mathbb{R})}\|f_2\|_{L^{q_1}(\mathbb{R})}\|g_1\|_{L^{p_2}(\mathbb{R})} \|g_4\|_{L^{q_2}(\mathbb{R})}\|h\|_{L^{s}(\mathbb{R}^2)}.
\end{align}
%One can view (\ref{bi_flag_int}) as a five-linear operator which outputs functions of two variables. Meanwhile, some input functions depend on only one variable. The boudedness property of such operator is analogous to Brascamp-Lieb inequalities.  
The above inequality follows from H\"older's inequality and the Loomis-Whitney inequality, which, in this simple two dimensional case, is the same as Fubini's theorem. Clearly, it is an inequality of the same type as (\ref{classical_bl_exp}), with a different homogeneity than H\"older. Moreover, this reduction shows that (\ref{bl_exp}) is indeed a necessary condition for the boundedness exponents of (\ref{flag_bl}) and thus of  (\ref{bi_flag_int}). %Another remark is that the operator $T_{K_1K_2}$ has also a bi-parameter flavor, since the singularity in (\ref{bi_flag_int}) involves the product of two kernels. %is $p_i, q_i, s \geq 1$ and $r > 0$ with 
%$$
%\frac{1}{p_1} + \frac{1}{p_2} + \frac{1}{s} = \frac{1}{q_1} + \frac{1}{q_2} + \frac{1}{s} = \frac{1}{r}.
%$$ 

%In contrast to the classical Brascamp-Lieb inequalities, the understanding of its singular variants is far beyond satisfaction.
%For the readers who may echo the above discussion with the singular Brascamp-Lieb estimates by Durcik and Thiele \cite{dt}, the particular class of singular Brascamp-Lieb forms studied in \cite{dt}, namely those of certain cubical type, is fundamentally different from (\ref{bi_flag_int}). More precisely, when the kernel in the Brascamp-Lieb form in \cite{dt} is reduced to a Dirac distribution, then the inequality is still H\"older's inequality where all functions in the form are of the full dimension. The recent survey \cite{dt2} describes a very general framework in which such singular Brascamp-Lieb estimates fit naturally. Another perspective to observe the difference is that the singular integral in \cite{dt} involves a single Calder\'on-Zygmund kernel while (\ref{bi_flag_int}) concerns the product of two kernels. 
%Another remark is that the operator $T_{K_1K_2}$ is a bi-parameter singular integral of Brascamp-Lieb type, which differs from the singular Brascamp-Lieb integral defined in (1.7) of \cite{dt} in the sense that (1.7) of \cite{dt} involves a single Calder\'on-Zygmund kernel while (\ref{bi_flag_int}) concerns the product of two kernels. 

\subsection{Connection with Other Multilinear Objects} \label{section_intro_connection}
The connection with other well-established multilinear operators that we will describe next justifies that $T_{K_1K_2}$ defined in (\ref{bi_flag_int}) is a reasonably simple and interesting operator to study, with the hope of inventing a general method that can handle a large class of singular integral estimates of Brascamp-Lieb type with non-H\"older scaling. 
%Application of the machinery invented gives rise to the full range of boundedness for (\ref{bi_flag_mult}):
%$$
%L_x^{p_1}(L_y^{p_2}) \times L_x^{q_1}(L_y^{q_2}) \times L^s(\mathbb{R}^2) \rightarrow L^{r}(\mathbb{R}^2)
%$$
%for $1 < p_1, p_2, q_1, q_2, s, r \leq \infty$, $(p_1,p_2), (q_1,q_2) \neq (\infty,\infty)$ and $
%\frac{1}{p_1} + \frac{1}{p_2} + \frac{1}{s} = \frac{1}{q_1} + \frac{1}{q_2} + \frac{1}{s}= \frac{1}{r}.
%$
%In the cases when $(p_1, q_1) = (\infty,\infty)$ or $(p_2, q_2) = (\infty,\infty)$, one can plug in constant functions as functions in $L^{\infty}$ to conclude that boundedness fails in those cases. The boundedness of bi-parameter flag paraproduct on generic function spaces still remains open.

Let $\mathcal{M}(\mathbb{R}^d)$ denote the set of all bounded symbols $m \in L^{\infty}(\mathbb{R}^d)$ smooth away from the origin and satisfying the Marcinkiewicz-H\"ormander-Mihlin condition
%For the sake of convenience, $m_i \in \mathcal{M}(\mathbb{R}^n)$ denotes Coifman-Meyer symbols throughout the following discussion. 
\begin{equation*}
\left|\partial^{\alpha} m(\xi) \right| \lesssim \frac{1}{|\xi|^{|\alpha|}}
\end{equation*}
for any $\xi \in \mathbb{R}^d \setminus \{0\}$ and sufficiently many multi-indices $\alpha$.
The simplest singular integral operator which corresponds to the two-dimensional Loomis-Whitney inequality would be
\begin{equation} \label{tensor_ht}
T_{m_1m_2}(f, g)(x,y) :=  \int_{\mathbb{R}^2} m_1(\xi)m_2(\eta) \f{f}(\xi) \f{g}(\eta) e^{2 \pi i x \xi} e^{2\pi i y\eta}d\xi d\eta,
\end{equation}
%where $m_1, m_2 \in \mathcal{M}(\mathbb{R})$ are a Coifman-Meyer symbols,
where $m_1, m_2 \in \mathcal{M}(\mathbb{R})$. (\ref{tensor_ht}) is a tensor product of Hilbert transforms whose boundedness are well-known. The bilinear variant of (\ref{tensor_ht}) can be expressed as

\begin{align} \label{tensor_para}
&T_{m_1m_2}(f_1,f_2, g_1, g_2)(x,y) \nonumber \\
:=  & \int_{\mathbb{R}^4} m_1(\xi_1,\xi_2) m_2(\eta_1,\eta_2) \f{f_1}(\xi_1) \f{f_2}(\xi_2)\f{g_1}(\eta_1)  \f{g_2}(\eta_2)e^{2 \pi i x(\xi_1+\xi_2)}e^{2 \pi i y(\eta_1+\eta_2)} d\xi_1 d\xi_2 d\eta_1 d\eta_2,
\end{align}
%where $m_1, m_2 \in \mathcal{M}(\mathbb{R}^2)$ are Coifman-Meyer symbols. 
where $m_1, m_2 \in \mathcal{M}(\mathbb{R}^2)$. It can be separated as a tensor product of single-parameter paraproducts whose boundedness are proved by Coifman-Meyer's theorem \cite{cm}. To avoid trivial tensor products of single-parameter operators, one then completes (\ref{tensor_para}) by adding a generic function of two variables thus obtaining
\begin{align} \label{bi_pp}
&T_{b}(f_1, f_2, g_1, g_2, h)(x,y)  \nonumber \\
:=& \int_{\mathbb{R}^6} b((\xi_1,\eta_1),(\xi_2,\eta_2),(\xi_3,\eta_3))  \f{f_1 \otimes g_1}(\xi_1, \eta_1) \f{f_2 \otimes g_2}(\xi_2, \eta_2) \f{h}(\xi_3,\eta_3) \nonumber \\
& \quad \quad \cdot e^{2 \pi i x(\xi_1+\xi_2+ \xi_3)}e^{2 \pi i y(\eta_1+\eta_2+ \eta_3)} d\xi_1 d\xi_2 d\xi_3 d\eta_1 d\eta_2 d\eta_3,
\end{align}
where %for any $\vec{\alpha}, \vec{\beta} \geq 0$,
\begin{align*}
& \left|\partial^{\alpha}_{(\xi_1,\xi_2,\xi_3)} \partial^{\beta}_{(\eta_1,\eta_2, \eta_3)} b \right| \lesssim \frac{1}{|(\xi_1,\xi_2,\xi_3)|^{|\alpha|}|(\eta_1,\eta_2,\eta_3)|^{|\beta|}}
\end{align*}
for sufficiently many multi-indices $\alpha$ and $\beta$. Such a multilinear operator is indeed a bi-parameter paraproduct whose theory has been developed by Muscalu, Pipher, Tao and Thiele \cite{cptt}. It also appeared naturally in nonlinear PDEs, such as Kadomtsev-Petviashvili equations studied by Kenig \cite{k}. To reach beyond bi-parameter paraproducts, one then replaces the singularity in each subspace by a flag singularity. In one dimension, the corresponding trilinear operator takes the form
\begin{equation} \label{flag}
T_{m_1m_2}(f_1,f_2,f_3)(x) := \int_{\mathbb{R}^3} m_1(\xi_1,\xi_2)m_2(\xi_1,\xi_2,\xi_3) \f{f_1}(\xi_1) \f{f_2}(\xi_2) \f{f_3}(\xi_3) e^{2 \pi i x(\xi_1+\xi_2+\xi_3)} d\xi_1 d\xi_2 d\xi_3,
\end{equation}
where $m_1 \in \mathcal{M}(\mathbb{R}^2)$ and $m_2 \in \mathcal{M}(\mathbb{R}^3)$. The operator (\ref{flag}) was studied by Muscalu \cite{c_flag} using time-frequency analysis which applies not only to the operator itself, but also to all of its adjoints. Miyachi and Tomita \cite{mt} extended the $L^p$-boundedness for $p>1$ established in \cite{c_flag} to all Hardy spaces $H^p$ with $p > 0$. The single-parameter flag paraproduct and its adjoints are closely related to various nonlinear partial differential equations, including nonlinear Schr\"odinger equations and water wave equations as discovered by Germain, Masmoudi and Shatah \cite{gms}. Its bi-parameter variant is indeed related to the subject of our study and is equivalent to (\ref{bi_flag_int}):
\begin{align} \label{bi_flag_mult}
T_{ab}(f_1, f_2, g_1, g_2, h)(x,y) := \int_{\mathbb{R}^6} & a((\xi_1,\eta_1),(\xi_2,\eta_2)) b((\xi_1,\eta_1),(\xi_2,\eta_2),(\xi_3,\eta_3)) \f{f_1 \otimes g_1}(\xi_1, \eta_1) \f{f_2 \otimes g_2}(\xi_2, \eta_2) \nonumber \\
& \cdot \f{h}(\xi_3,\eta_3) e^{2\pi i x(\xi_1+\xi_2+\xi_3)}e^{2\pi i y(\eta_1+\eta_2+\eta_3)} d \xi_1 d\xi_2 d\xi_3 d\eta_1 d\eta_2 d\eta_3,
\end{align}
where 
\begin{align*}
& \left|\partial^{\alpha_1}_{(\xi_1,\xi_2)} \partial^{\beta_1}_{(\eta_1,\eta_2)} a\right| \lesssim \frac{1}{|(\xi_1,\xi_2)|^{|\alpha_1|}|(\eta_1,\eta_2)|^{|\beta_1|}}, \nonumber \\
& \left|\partial^{\alpha_2}_{(\xi_1,\xi_2,\xi_3)} \partial^{\beta_2}_{(\eta_1,\eta_2, \eta_3)} b \right| \lesssim \frac{1}{|(\xi_1,\xi_2,\xi_3)|^{|\alpha_2|}|(\eta_1,\eta_2,\eta_3)|^{|\beta_2|}},
\end{align*}
for sufficiently many multi-indices $\alpha_1, \beta_1, \alpha_2$ and $\beta_2$. The equivalence can be derived with 
\begin{align*}
& a = \f{K_1}, \nonumber \\
& b = \f{K_2}.
\end{align*}

The general bi-parameter trilinear flag paraproduct is defined on larger function spaces where the tensor products are replaced by general functions in the plane.\footnote{Its boundedness is at present an open question, raised by the first author of the article on several occasions.}
% In particular,
%\begin{align*}
%& f_1 \otimes g_1 \rightarrow F_1^{xy} \nonumber \\
%&  f_2 \otimes g_2 \rightarrow F_2^{xy}.
%\end{align*}
From this perspective, $T_{ab}$ or equivalently $T_{K_1K_2}$ defined in (\ref{bi_flag_mult}) and (\ref{bi_flag_int}) respectively can be viewed as a trilinear operator with the desired mapping property 
$$
T_{ab}: L^{p_1}_{x}(L^{p_2}_y) \times L^{q_1}_{x}(L^{q_2}_y) \times L^{s}(\mathbb{R}^2) \rightarrow L^{r}(\mathbb{R}^2)
$$ 
for $ 1 < p_1, p_2, q_1, q_2, s \leq \infty$, $r > 0$, $(p_1, q_1), (p_2,q_2) \neq (\infty, \infty)$ and $\frac{1}{p_1} + \frac{1}{q_1} + \frac{1}{s} = \frac{1}{p_2} + \frac{1}{q_2} + \frac{1}{s} = \frac{1}{r}$, where the first two function spaces are restricted to be tensor product spaces. The condition that $(p_1, q_1), (p_1,q_2) \neq (\infty, \infty)$ is inherited from single-parameter flag paraproducts and can be verified by the unboundedness of the operator when $f_1, f_2 \in L^{\infty}(\mathbb{R})$ are constant functions. Lu, Pipher and Zhang \cite{lpz} showed that the general bi-parameter flag paraproduct can be reduced to an operator given by a symbol with better singularity using an argument inspired by Miyachi and Tomita \cite{mt}. The boundedness of the reduced multiplier operator still remains open. The reduction allows an alternative proof of $L^p$-boundedness for (\ref{bi_flag_mult}) as long as $p \neq \infty$. However, we emphasize again that we will not take this point of view now, and instead, we treat our operator $T_{ab}$ as a five-linear operator.

\subsection{Methodology} \label{section_intro_method}
%While the object of study has implications and connections with other multilinear operators, the machinery implemented in the proof has interests of its own. It involves various stopping-time decompositions:
As one may notice from the last section, the five-linear operator $T_{ab}$ ( or $T_{K_1K_2}$) contains the features of the bi-parameter paraproduct defined in (\ref{bi_pp}) and the single-parameter flag paraproduct defined in (\ref{flag}), which hints that the methodology would embrace localized analyses of both operators. Nonetheless, it is by no means a simple concatenation of two existing arguments. The methodology includes
\begin{enumerate}
\item
\textbf{tensor-type stopping-time decomposition} which refers to an algorithm that first implements a one-dimensional stopping-time decomposition for each variable and then combines information for different variables to obtain estimates for operators involving several variables;

\item
\textbf{general two-dimensional level sets stopping-time decomposition} which refers to an algorithm that partitions the collection of dyadic rectangles such that the dyadic rectangles in each sub-collection intersect with a certain level set non-trivially;
\end{enumerate}
and the main novelty lies in 
\begin{enumerate}[(i)]
\item 
the construction of two-dimensional stopping-time decompositions from stopping-time decompositions on one-dimensional subspaces;
\item
the hybrid of tensor-type and general two-dimensional level sets stopping-time decompositions in a meaningful fashion.
\end{enumerate}

The methodology outlined above is considered to be robust in the sense that it captures all local behaviors of the operator. The robustness may also be verified by the entire range of estimates obtained. After closer inspection of the technique, it would not be surprising that the technique gives estimates involving $L^{\infty}$ norms. %even mixed-norm estimates involving $L^{\infty}$-norms. 
In particular, the tensor-type stopping-time decompositions process information on each subspaces independently. As a consequence, when some function defined on some subspace lies in $L^{\infty}$, one simply ``forgets" about that function and glues the information from subspaces in an intelligent way specified later.

\subsection{Structure} \label{section_intro_structure}
The paper is organized as follows: main theorems are stated in Section \ref{section_result} followed by preliminary definitions and theorems introduced in Section \ref{section_prelim}. Section \ref{section_discrete_model} describes the reduced discrete model operators and estimates one needs to obtain for the model operators while the reduction procedure is postponed to Appendix II. Section \ref{section_size_energy} gives the definition and estimates for the building blocks in the argument - sizes and energies. Sections \ref{section_thm_haar_fixed}, \ref{section_thm_haar}, \ref{section_thm_inf_fixed_haar} and \ref{section_thm_inf_haar} focus on estimates for the model operators in the Haar case. All four sections start with a specification of the stopping-time decompositions used. %Section 6,7, 8 and 9 are devoted to estimates for different model operators, except for the ones involving $L^{\infty}$ which are discussed separately in Section 10, 11 and 12. 
Section \ref{section_fourier} extends all the estimates in the Haar setting to the general Fourier case.

It is also important to notice that Section \ref{section_thm_haar_fixed} develops an argument for one of the simpler model operators with emphasis on the key geometric feature implied by a stopping-time decomposition, that is the sparsity condition. Section \ref{section_thm_haar} focuses on a more complicated model which requires not only the sparsity condition, but also a Fubini-type argument which is discussed in details.  Sections \ref{section_thm_inf_fixed_haar} and \ref{section_thm_inf_haar} are devoted to estimates involving $L^{\infty}$ norms and the arguments for those cases are similar to the ones in Section \ref{section_thm_haar_fixed}, in the sense that the sparsity condition is sufficient to obtain the results.

\subsection{Acknowledgements.}
We thank Jonathan Bennett for the inspiring conversation we had in Matsumoto, Japan, in
February 2016, that triggered our interest in considering and understanding singular integral generalizations of Brascamp-Lieb inequalities, and, in particular, the study of the present paper. We also thank Guozhen Lu, Jill Pipher and Lu Zhang for discussions about their recent work in \cite{lpz}. Finally, we thank Polona Durcik and Christoph Thiele for the recent conversation which clarified the similarities and differences between the results in \cite{dt} and those in our paper and \cite{bm2}.

The first author was partially supported by a Grant from the Simons Foundation. The second author was partially supported by the ERC Project FAnFArE no. 637510.

\section{Main Results} \label{section_result}
%\subsection{Five-Linear Singular Integral Estimates}
We state the main results in Theorem \ref{main_theorem} and \ref{main_thm_inf}. Theorem \ref{main_theorem} proves the boundedness when $p_i, q_i$ are strictly between $1$ and infinity whereas Theorem \ref{main_thm_inf} deals with the case when $p_i = \infty$ or $q_j = \infty$ for some $i\neq j$.  %Since the arguments for estimates in the latter case is different from the former case, we will develop the arguments separately. 

\begin{theorem} \label{main_theorem}
Suppose $a \in L^{\infty}(\mathbb{R}^4)$, $b\in L^{\infty}(\mathbb{R}^6)$, where $a$ and $b$ are smooth away from $\{(\xi_1,\xi_2) = 0 \} \cup \{(\eta_1,\eta_2) = 0 \}$ and $\{(\xi_1, \xi_2,\xi_3) = 0 \} \cup \{(\eta_1,\eta_2,\eta_3) = 0\}$ respectively and satisfy the following Marcinkiewicz conditions:
\begin{align*}
& |\partial^{\alpha_1}_{\xi_1} \partial^{\alpha_2}_{\eta_1} \partial^{\beta_1}_{\xi_2} \partial^{\beta_2}_{\eta_2} a(\xi_1,\eta_1, \xi_2,\eta_2)| \lesssim  \frac{1}{|(\xi_1,\xi_2)|^{\alpha_1 + \beta_1}} \frac{1}{|(\eta_1,\eta_2)|^{\alpha_2+\beta_2}}, \nonumber \\
& |\partial^{\bar{\alpha_1}}_{\xi_1} \partial^{\bar{\alpha_2}}_{\eta_1} \partial^{\bar{\beta_1}}_{\xi_2} \partial^{\bar{\beta_2}}_{\eta_2}\partial^{\bar{\gamma_1}}_{\xi_3} \partial^{\bar{\gamma_2}}_{\eta_3}b(\xi_1,\eta_1, \xi_2,\eta_2, \xi_3, \eta_3)| \lesssim  \frac{1}{|(\xi_1,\xi_2, \xi_3)|^{\bar{\alpha_1} + \bar{\beta_1}+\bar{\gamma_1}}} \frac{1}{|(\eta_1,\eta_2, \eta_3)|^{\bar{\alpha_2}+\bar{\beta_2}+ \bar{\gamma_2}}}
\end{align*}
for sufficiently many multi-indices $\alpha_1,\alpha_2,\beta_1,\beta_2, \bar{\alpha_1}, \bar{\alpha_2},\bar{\beta_1},\bar{\beta_2}, \bar{\gamma_1}, \bar{\gamma_2} \geq 0$. For $f_1, f_2, g_1,g_2 \in \mathcal{S}(\mathbb{R})$ and $h \in \mathcal{S}(\mathbb{R}^2)$ where $\mathcal{S}(\mathbb{R})$ and $\mathcal{S}(\mathbb{R}^2)$ denote the Schwartz spaces, define
\begin{align}  \label{bi_flag}
\displaystyle T_{ab}(f_1, f_2, g_1 ,g_2,h)(x,y) := \int_{\mathbb{R}^6} & a(\xi_1,\eta_1,\xi_2,\eta_2) b(\xi_1,\eta_1,\xi_2,\eta_2,\xi_3,\eta_3) \nonumber \\
& \hat{f_1}(\xi_1)\hat{f_2}(\xi_2)\hat{g_1}(\eta_1)\hat{g_2}(\eta_2)\hat{h}(\xi_3,\eta_3) \nonumber \\
& e^{2\pi i x(\xi_1+\xi_2+\xi_3)}e^{2\pi i y(\eta_1+\eta_2+\eta_3)} d \xi_1 d\xi_2 d\xi_3 d\eta_1 d\eta_2 d\eta_3.
\end{align}
Then for $1< p_1, p_2,  q_1, q_2 < \infty, 1 < s \leq \infty$, $r > 0$, $\frac{1}{p_1} + \frac{1}{q_1} + \frac{1}{s} =\frac{1}{p_2} + \frac{1}{q_2} + \frac{1}{s} =   \frac{1}{r} $, $T_{ab}$ satisfies the following mapping property
$$
T_{ab}: L^{p_1}(\mathbb{R}) \times L^{q_1}(\mathbb{R}) \times L^{p_2}(\mathbb{R}) \times L^{q_2}(\mathbb{R}) \times L^{s}(\mathbb{R}^2) \rightarrow L^{r}(\mathbb{R}^2).
$$
\end{theorem}

\begin{theorem} \label{main_thm_inf}
Let $T_{ab}$ be defined as (\ref{bi_flag}). 
Then for $1< p < \infty$, $1 < s \leq \infty$, $r >0$, $\frac{1}{p} + \frac{1}{s} = \frac{1}{r}$, $T_{ab}$ satisfies the following mapping property
\begin{align*}
T_{ab}: & L^{p}(\mathbb{R})  \times L^{\infty}(\mathbb{R}) \times L^{p}(\mathbb{R}) \times L^{\infty}(\mathbb{R}) \times L^{s}(\mathbb{R}^2) \rightarrow L^{r}(\mathbb{R}^2) \nonumber 
%&  L^{\infty} \times L^{p} \times L^{s} \rightarrow L^{r} \nonumber \\
%& L_x^{p}(L_y^{\infty}) \times L_x^{\infty}(L^p_y) \times L^{s} \rightarrow L^{r} \nonumber \\
%& L_x^{\infty}(L^p_y) \times L_x^{p}(L_y^{\infty}) \times L^{s} \rightarrow L^{r} \nonumber \\
\end{align*}
where $p_1 = p_2 = p$ as imposed by (\ref{bl_exp}).
\end{theorem}
\begin{remark}
%\begin{enumerate}
%\item
The cases $(i)  q_1 = q_2 < \infty$ and $p_1= p_2= \infty$ $(ii) p_1 = q_2 < \infty$ and $p_2 = q_1 = \infty$ $(iii) q_1 = p_2 < \infty$ and $p_1 = q_2 = \infty$ follows from the same argument by symmetry.
%\end{enumerate}
\end{remark}

\subsection{Restricted Weak-Type Estimates} \label{section_result_rw}
For the Banach estimates when $r > 1$, H\"older's inequality involving maximal function operator, square function operator and hybrid operators (Definition \ref{def_hybrid}) is sufficient. The argument resembles the Banach estimates for the single-parameter flag paraproduct. The quasi-Banach estimates when $r < 1$ is trickier and requires a careful treatment. In this case, we use multilinear interpolations and reduce the desired estimates specified in Theorem \ref{main_theorem} and Theorem \ref{main_thm_inf}  to the following restricted weak-type estimates for the associated multilinear form\footnote{Multilinear form, denoted by $\Lambda$, associated to an n-linear operator $T(f_1, \ldots, f_n)$ is defined as $\Lambda(f_1, \ldots, f_n, f_{n+1}) := \langle T(f_1,\ldots, f_n), f_{n+1}\rangle $.}.
\begin{theorem}\label{thm_weak}
%Let $1< p,q,s \leq \infty$ and $r>0$ be such that $\frac{1}{p} + \frac{1}{q} + \frac{1}{s} = \frac{1}{r}$. Then a trilinear operator $T$ is said to be of the \textit{restricted weak type} $(p,q,s,r)$ if and only if
Let $T_{ab}$ denote the operator defined in (\ref{bi_flag}). Suppose that $1< p_1, p_2,  q_1, q_2 < \infty, 1 < s <2$, $0 < r <1$, $\frac{1}{p_1} + \frac{1}{q_1} + \frac{1}{s} =\frac{1}{p_2} + \frac{1}{q_2} + \frac{1}{s} =   \frac{1}{r}$. Then for every measurable set $F_1, F_2, G_1, G_2  \subseteq \mathbb{R}, E \subset \mathbb{R}^2$ of positive and finite measure and every measurable function $|f_1(x)| \leq \chi_{F_1}(x)$, $|f_2(x)| \leq \chi_{F_2}(x)$, $|g_1(y)| \leq \chi_{G_1}(y)$, $|g_2(y)| \leq \chi_{G_2}(y)$, $h \in L^{s}(\mathbb{R}^2)$, there exists $E' \subseteq E$ with $|E'| > |E|/2$ such that the multilinear form associated to $T_{ab}$ satisfies
\begin{equation} \label{thm_weak_explicit}
|\Lambda(f_1, f_2, g_1, g_2, h,\chi_{E'}) | \lesssim |F_1|^{\frac{1}{p_1}} |G_1|^{\frac{1}{p_2}} |F_2|^{\frac{1}{q_1}} |G_2|^{\frac{1}{q_2}} \|h\|_{L^{s}(\mathbb{R}^2)}|E|^{\frac{1}{r'}},
\end{equation}
where $r'$ represents the conjugate exponent of $r$.
\end{theorem}
\vskip.15in
\begin{theorem} \label{thm_weak_inf}
Let $T_{ab}$ denote the operator defined in (\ref{bi_flag}). Suppose that $1< p < \infty$, $1 < s < 2$, $0 < r < 1$, $\frac{1}{p} + \frac{1}{s} = \frac{1}{r}$. Then for every measurable set 
%$F_1, G_1\subseteq \mathbb{R}$, 
$E \subset \mathbb{R}^2$ of positive and finite measure and every measurable function 
%$|f_1(x)| \leq \chi_{F_1}(x)$, $|g_1(y)| \leq \chi_{G_1}(y)$, 
$f_1,g_1 \in L^p(\mathbb{R})$, $f_2, g_2 \in L^{\infty}(\mathbb{R})$, $h \in L^{s}(\mathbb{R}^2)$, there exists $E' \subseteq E$ with $|E'| > |E|/2$ such that the multilinear form associated to $T_{ab}$ satisfies
\begin{equation} \label{thm_weak_inf_explicit}
|\Lambda(f_1, f_2, g_1, g_2, h,\chi_{E'}) | \lesssim \|f_1\|_{L^p(\mathbb{R})} \|f_2\|_{L^{\infty}(\mathbb{R})}  \|g_1\|_{L^p(\mathbb{R})} \|g_2\|_{L^{\infty}(\mathbb{R})} \|h\|_{L^s(\mathbb{R}^2)}|E|^{\frac{1}{r'}},
\end{equation}
where $r'$ represents the conjugate exponent of $r$.
\end{theorem}

\begin{remark}
%\begin{enumerate}
%\item
%The case $s = \infty$ can be obtained by interpolation of multilinear forms described in Lemma 9.6 of \cite{cw}. 
%In particular, let $T_{ab}^*$ be the adjoint operator which satisfies
%$$
%\langle T_{ab}(f_1,f_2, g_1, g_2, h),l \rangle = \langle T_{ab}^*(f_1,f_2, g_1, g_2, l), h \rangle
%$$
%Then due to the symmetry between $T_{ab}$ and $T_{ab}^*$, one can derive the same mapping property for $T_{ab}^*$, which combined with the boundedness of $T_{ab}$ gives rise to the range where $s = \infty$.
%\item
%Theorem \ref{thm_weak_inf} corresponds to the case $p_1 = p_2 = p$ and $q_1 = q_2 = \infty$. The other cases can be treated with the same argument due to the symmetry. 
%\item
Theorem \ref{thm_weak} and \ref{thm_weak_inf} hint the necessity of localization and the major subset $E'$ of $E$ is constructed based on the philosophy to localize the operator where it is well-behaved.
%\end{enumerate}
\end{remark}
The reduction of Theorem \ref{main_theorem} and \ref{main_thm_inf} to Theorem \ref{thm_weak} and \ref{thm_weak_inf} respectively will be postponed to Appendix I. In brief, it depends on the interpolation of multilinear forms described in Lemma 9.6 of \cite{cw} and a tensor product version of Marcinkiewicz interpolation theorem.

\vskip .15in

\subsection{Application - Leibniz Rule} \label{section_result_lr}
A direct corollary of Theorem \ref{main_theorem} is a Leibniz rule which captures the nonlinear interaction of waves coming from transversal directions. In general, Leibniz rules refer to inequalities involving norms of derivatives. The derivatives are defined in terms of Fourier transforms. More precisely, for $\alpha \geq 0$ and $f \in \mathcal{S}(\mathbb{R}^d)$ a Schwartz function in $\mathbb{R}^d$, define the homogeneous derivative of $f$ as
\begin{equation*}
D^{\alpha}f := \mathcal{F}^{-1}\left(|\xi|^{\alpha}\f{f}(\xi)\right).
\end{equation*}
Leibniz rules are closely related to boundedness of multilinear operators discussed in Section \ref{section_intro_connection}. For example, the boundedness of one-parameter paraproducts give rise to a Leibniz rule by Kato and Ponce \cite{kp}. For $f, g \in \mathcal{S}(\mathbb{R}^d)$ and $\alpha > 0$ sufficiently large,
\begin{equation} \label{lb_para}
 \| D^{\alpha} (fg)\|_r \lesssim \|D^{\alpha} f \|_{p_1} \|g \|_{q_1} +  \| f \|_{p_2} \|D^{\alpha}g \|_{q_2}
\end{equation}
with $1 < p_i, q_i < \infty, \frac{1}{p_i}+ \frac{1}{q_i} = \frac{1}{r}, i= 1,2.$ %Here $D^{\alpha} F := \mathcal{F}^{-1}(|(\xi,\eta)|^{\alpha}\hat{F}(\xi,\eta))$, where $\xi, \eta \in \mathbb{R}$.
The inequality in (\ref{lb_para}) generalizes the trivial and well-known Leibniz rule when $\alpha = 1$ and states that the derivative for a product of two functions can be dominated by the terms which involve the highest order derivative hitting on one of the functions. The reduction of (\ref{lb_para}) to the boundedness of one-parameter paraproducts is routine (see Section 2 in \cite{cw} for details) and can be applied to other Leibniz rules with their corresponding multilinear operators, including the boundedness of our operator $T_{ab}$ and its Leibniz rule stated in Theorem \ref{lb_main} below.  The Leibniz rule stated in Theorem \ref{lb_main} deals with partial derivatives, where the partial derivative of $f \in \mathcal{S}(\mathbb{R}^d)$ is defined, for $(\alpha_1,\ldots, \alpha_d)$ with $\alpha_1, \ldots, \alpha_d \geq 0$, as
\begin{equation*}
D_1^{\alpha_1}\cdots D_d^{\alpha_d}f := \mathcal{F}^{-1}\left(|\xi_1|^{\alpha_1} \cdots |\xi_d|^{\alpha_d}\f{f}(\xi_1,\ldots, \xi_d)\right).
\end{equation*}
For the statement of the Leibniz rule, we will adopt the notation $f^x$ for the function $f$ depending on the variable $x$.
\begin{theorem} \label{lb_main}
Suppose $f_1, f_2 \in \mathcal{S}(\mathbb{R})$, $ g_1, g_2 \in \mathcal{S}(\mathbb{R})$ and $h \in \mathcal{S}(\mathbb{R}^2).$ Then for $\beta_1, \beta_2, \alpha_1, \alpha_2 > 0$ sufficiently large and $1 < p^j_1, p^j_2, q^j_1, q^j_2, s^j \leq \infty$, $r >0$, $(p^j_1, q^j_1), (p^j_2, q^j_2) \neq (\infty, \infty)$, $\frac{1}{p^j_1} + \frac{1}{q^j_1} + \frac{1}{s^j}= \frac{1}{p^j_2} + \frac{1}{q^j_2} + \frac{1}{s^j}= \frac{1}{r} $ for each $j = 1, \ldots, 16 $,
\begin{align*}
& \|D_1^{\beta_1} D_2^{\beta_2}(D_1^{\alpha_1}D_2^{\alpha_2}(f_1^x f_2^x g_1^y  g_2^y) h^{x,y})\|_{L^r(\mathbb{R}^2)}  \nonumber \\
 \lesssim & \ \ \text{sum of \ \ }16 \text{\ \ terms of the forms: \ \ }
\nonumber \\ &  \|D_1^{\alpha_1+\beta_1}f_1\|_{L^{p^1_1}(\mathbb{R})} \|f_2\|_{L^{q^1_1}(\mathbb{R})}  \|D_2^{\alpha_2 + \beta_2}g_1\|_{L^{p^1_2}(\mathbb{R})} \|g_2\|_{L^{q^1_2}(\mathbb{R})} \|h\|_{L^{s^1}(\mathbb{R}^2)} + 
\nonumber \\ & \|f_1\|_{L^{p^2_1}(\mathbb{R})} \|D_1^{\alpha+\beta_1}f_2\|_{L^{q^2_1}(\mathbb{R})} \|D_2^{\alpha_2 + \beta_2}g_1\|_{L^{p^2_2}(\mathbb{R})} \|g_2\|_{L^{q^2_2}(\mathbb{R})} \|h\|_{L^{s^2}(\mathbb{R}^2)} + 
\nonumber \\ & \|D_1^{\alpha+\beta_1}f_1\|_{L^{p^3_1}(\mathbb{R})} \|f_2\|_{L^{q^3_1}(\mathbb{R})} \|D_2^{\alpha_2}g_1\|_{L^{p^3_2}(\mathbb{R})} \|g_2\|_{L^{q^3_2}(\mathbb{R})} \|D_2^{\beta_2}h\|_{L^{s^3}(\mathbb{R}^2)} + \ldots
\end{align*}
\end{theorem}
\begin{remark}
The reasoning for the number ``16'' is that 
\begin{enumerate} [(i)]
\item
for $\alpha_1$, there are $2$ possible distributions of highest order derivatives thus yielding 2 terms;
\item
for $\alpha_2$, there are $2$ terms for the same reason in (i);
%\item
%thus there are $4$ terms for the inner derivative;
\item
for $\beta_1$, it can hit $h$ or some function which comes from the dominant terms of $D^{\alpha_1}(f_1 f_2)$ and which have two choices as illustrated in (i), thus generating $2 \times 2 = 4$ terms; 
\item 
for $\beta_2$, there would be $4$ terms for the same reason in (iii).
\end{enumerate}
By summarizing (i)-(iv), one has the count $4 \times 4 = 16$.
\end{remark}
\begin{remark}
As commented in the beginning of this section, $f_1$ and $f_2$ in Theorem \ref{lb_main} can be viewed as waves coming from one direction while $g_1$ and $g_2$ are waves from the orthogonal direction.  The presence of $h$, as a generic wave in the plane, makes the interaction nontrivial.  
\end{remark}

\section{Preliminaries} \label{section_prelim}
\subsection{Terminology} \label{section_prelim_term}
We will first introduce the notation which will be useful throughout the paper. 
\begin{definition} \label{bump}
Suppose $I \subseteq \mathbb{R}$ is an interval. Then we say a smooth function $\phi$ is a \textit{bump function adapted to $I$} if 
$$
|\phi^{(l)}(x)| \leq C_l C_M \frac{1}{|I|^l} \frac{1}{\big(1+\frac{|x-x_I|}{|I|}\big)^M}
$$
for sufficiently many derivatives $l$, where $x_I$ denotes the center of the interval $I$ and $M$ is a large positive number.
Moreover, suppose $\mathcal{I}$ is a collection of dyadic intervals and if for any $I \in \mathcal{I}$, $\phi_I$ is an $L^2$-normalized bump function adapted to $I$, then $(\phi_I)_{I \in \mathcal{I}}$ is denoted by \textit{a family of $L^2$-normalized adapted bump functions}. 
\end{definition}

\begin{definition} \label{def_lacunary}
The family of $L^2$-normalized adapted bump functions $(\phi_I)_{I \in \mathcal{I}}$ is \textit{lacunary} if and only if for every $ I \in \mathcal{I}$, 
$$
\text{supp}\ \ \f{\phi_I} \subseteq [-4|I|^{-1}, -\frac{1}{4}|I|^{-1}] \cup [\frac{1}{4}|I|^{-1}, 4|I|^{-1}].
$$
A family of $L^2$-normalized bump functions $(\phi_I)_{I \in \mathcal{I}}$ is \textit{non-lacunary} if and only if for every $ I \in \mathcal{I}$, 
$$
\text{supp}\ \ \f{\phi_I} \subseteq [-\frac{1}{4}|I|^{-1}, \frac{1}{4}|I|^{-1}].
$$
We usually denote bump functions in lacunary family by $(\psi_I)_I$  and those in non-lacunary family by $(\vphi_I)_I$.
%The Heisenburg boxes for $\vphi_I$ and $\psi_I$ are displayed as follows:

\end{definition}

We will now define cutoff functions correspond to adapted bump functions given in Definition \ref{bump}, which can be viewed as simplified variants of bump functions.  

\begin{definition} \label{bump_walsh}
Define 
$$
\psi^H(x) :=  
\begin{cases}
1 \ \ \text{for}\ \  x \in [0,\frac{1}{2})\\
-1 \ \ \text{for}\ \  x \in [\frac{1}{2},1).\\
\end{cases}
$$
Let $I := [n2^{k},(n+1)\cdot2^k)$ denote a dyadic interval. Then the Haar wavelet on $I$ is defined as
$$
\psi^H_I(x) := 2^{-\frac{k}{2}}\psi^H(2^{-k}x-n).
$$
The $L^2$-normalized indicator function on $I$ is expressed as
$$
\vphi^H_I(x) := |I|^{-\frac{1}{2}}\chi_{I}(x).
$$
We adopt the notation $\phi^{H}_I$ for both $\psi^H_I$ and $\vphi_I^H$ and differentiate them by referring $\psi_I^H$ as the \textit{lacunary} case of the cutoff function $\phi_I^H$ and $\vphi_I^H$ as the \textit{non-lacunary} case. We thus have an analogous and discontinuous variant of adapted bump functions. In particular, Haar wavelets correspond to lacunary (or equivalently having Fourier support away from the origin) family of bump functions and $L^2$-normalized indicator functions correspond to non-lacunary family of bump functions.  
\end{definition} 
%\begin{notation}
%$L^2$-normalized adapted bump functions refer to both the function $\phi_I$ defined in Definition \ref{bump} and the function $\phi_I^H$ in Definition \ref{bump_walsh} whereas the latter will be specified as 
%\end{notation}

We shall remark that the boundedness of the multilinear form described in Theorem \ref{thm_weak} and \ref{thm_weak_inf} can be reduced to the estimates of discrete model operators which are defined in terms of bump functions of the form specified in Definition \ref{bump}. The precise statements are included in Theorem \ref{thm_weak_mod} and \ref{thm_weak_inf_mod} and the proof is discussed in Appendix II. However, we will first study the simplified model operators with the general bump functions replaced by Haar wavelets and indicator functions defined in Definition \ref{bump_walsh}. The arguments for the simplified models capture the main challenges while avoiding some technical aspects. We will leave the generalization and the treatment of the technical details to Section \ref{section_fourier}. The simplified models would be denoted as Haar models and we will highlight the occasions when the Haar models are considered. 
 %We will use $\psi$ and $\vphi$ for general lacunary and non-lacunary bump functions respectively and $\psi^{H}$  and Haar wavelets and $\vphi$ for non-lacunary bump functions and indicator functions. 

\
\begin{comment}
Suppose $\mathcal{I}$ is a collection of dyadic intervals. 
Then a family of $L^2$-normalized bump functions $(\phi_I)_{I \in \mathcal{I}}$ is \textit{non-lacunary} if and only if for every $ I \in \mathcal{I}$, 
$$
\text{supp}\ \ \f{\phi_I} \subseteq [-4|I|^{-1}, 4|I|^{-1}]
$$
A family of $L^2$-normalized bump functions $(\phi_I)_{I \in \mathcal{I}}$ is \textit{lacunary} if and only if for every $ I \in \mathcal{I}$, 
$$
\text{supp}\ \ \f{\phi_I} \subseteq [-4|I|^{-1}, \frac{1}{4}|I|^{-1}] \cup [\frac{1}{4}|I|^{-1}, 4|I|^{-1}]
$$
We usually denote bump functions in non-lacunary family by $(\vphi_I)_I$ and those in lacunary family by $(\psi_I)_I$.
\end{comment}

\subsection{Useful Operators - Definitions and Theorems} \label{section_prelim_useful_op}
We also give explicit definitions for the Hardy-Littlewood maximal function, the discretized Littlewood-Paley square function and the hybrid square-and-maximal functions that will appear naturally in the argument.
\begin{definition}
The \textit{Hardy-Littlewood maximal operator} $M$ is defined as
$$
Mf(\vec{x}) = \sup_{\vec{x} \in B} \frac{1}{|B|} \int_{B}|f(\vec{u})|d\vec{u}
$$
where the supremum is taken over all open balls $B \subseteq \mathbb{R}^d$ containing $\vec{x}$.
\end{definition}

%\begin{definition}
%The \textit{Littlewood-Paley square function operator} $S$ is defined as 
%$$
%Sf(x) = \sum_{k}(f*\psi_k)(x)
%$$
%where $(\psi_k)_k$ is a family of $L^2$-normalized functions such that 
%$\f{\psi_k} \subseteq [2^{-k-1}, 2^{-k+1}]$
%\end{definition}

\begin{definition}
Suppose $\mathcal{I}$ is a finite family of dyadic intervals and $(\psi_I)_I$ a lacunary family of $L^2$-normalized bump functions. The \textit{discretized Littlewood-Paley square function operator} $S$ is defined as
 $$
 Sf(x) = \bigg(\sum_{I \in \mathcal{I}}\frac{|\langle f, \psi_I\rangle|^2 }{|I|}\chi_{I}(x)\bigg)^{\frac{1}{2}}.
 $$

\end{definition}

\begin{definition} \label{def_hybrid}
Suppose $\mathcal{R}$ is a finite collection of dyadic rectangles. Let $(\phi_R)_{R \in \mathcal{R}}$ denote the family of $L^2$-normalized bump functions with $\phi_R = \phi_I \otimes \phi_J$ where $R= I \times J$. Let $(\phi^H_R)_{R \in \mathcal{R}}$ denote the family of cutoff functions with $\phi_R^H = \phi^H_I \otimes \phi^H_J$ where $R= I \times J$.
\begin{enumerate}
\item
the \textit{double square function operator} $SS$ is defined as
$$
\displaystyle SSh(x,y) = \bigg(\sum_{I \times J } \frac{|\langle h, \psi_{I} \otimes \psi_J \rangle|^2 }{|I||J|}
\chi_{I \times J} (x,y)\bigg)^{\frac{1}{2}}
$$
and the \textit{Haar double square function operator} $(SS)^H$ is defined as
$$
\displaystyle (SS)^Hh(x,y) = \bigg(\sum_{I \times J } \frac{|\langle h, \psi^H_{I} \otimes \psi^H_J \rangle|^2 }{|I||J|}
\chi_{I \times J} (x,y)\bigg)^{\frac{1}{2}};
$$

\item 
the \textit{hybrid maximal-square operator} $MS$ is defined as
$$
MSh(x,y) = \sup_{I}\frac{1}{|I|^{\frac{1}{2}}} \bigg(\sum_{J} \frac{|\langle h, \vphi_I \otimes \psi_J \rangle|^2}{|J|} \chi_{J}(y)\bigg)^{\frac{1}{2}}\chi_I(x)
$$
and the \textit{Haar hybrid maximal-square operator} $(MS)^H$ is defined as
$$
(MS)^Hh(x,y) = \sup_{I}\frac{1}{|I|^{\frac{1}{2}}} \bigg(\sum_{J} \frac{|\langle h, \vphi^H_I \otimes \psi^H_J \rangle|^2}{|J|} \chi_{J}(y)\bigg)^{\frac{1}{2}}\chi_I(x);
$$

\item 
the \textit{hybrid square-maximal operator} $SM$ is defined as
$$ \displaystyle
SMh(x,y) = \bigg(\sum_{I} \frac{\big(\sup_{J}\frac{|\langle h,\psi_I \otimes \vphi_J \rangle|}{|J|}\chi_J(y) \big)}{|I|}\chi_{I}(x)\bigg)^{\frac{1}{2}}
$$
and the \textit{Haar hybrid square-maximal operator} $(SM)^H$ is defined as
$$ \displaystyle
(SM)^Hh(x,y) = \bigg(\sum_{I} \frac{\big(\sup_{J}\frac{|\langle h,\psi^H_I \otimes \vphi^H_J \rangle|}{|J|}\chi_J(y) \big)}{|I|}\chi_{I}(x)\bigg)^{\frac{1}{2}};
$$

\item
the \textit{double maximal function} $MM$ is defined as
$$
MM h(x,y) = \sup_{(x,y) \in R} \frac{1}{|R|}\int_{R}|h(s,t)| ds dt,
$$
where the supremum is taken over all dyadic rectangles in $\mathcal{R}$ containing $(x,y)$.
\end{enumerate}
\end{definition}

The following theorem about the operators defined above is used frequently in the arguments. The proof of the theorem and other contexts where the hybrid operators appear can be found in \cite{cw}, \cite{cf} and \cite{fs}.
\begin{theorem} \label{maximal-square}
\noindent
\begin{enumerate} %Let the Hardy-Littlewood maximal operator $M$, the discretized Littlewood-Paley square function operator and the hybrid square-and-maximal function operators 
\item
$M$ is bounded in $L^{p}(\mathbb{R}^{d})$ for $1< p \leq \infty$ and $M: L^{1} \longrightarrow L^{1,\infty}$.
\item
$S$ is bounded in $L^{p}(\mathbb{R})$ for $1< p < \infty$.
\item
The operators $SS, (SS)^H, MS, (MS)^H, SM, (SM)^H, MM$ are bounded in $L^{p}(\mathbb{R}^2)$ for $1 < p < \infty$.
\end{enumerate}
\end{theorem}
\vskip.25in

\section{Discrete Model Operators} \label{section_discrete_model}

In this section, we will introduce the discrete model operators whose boundedness implies the estimates specified in Theorem \ref{thm_weak} and Theorem \ref{thm_weak_inf}. The reduction procedure follows from a routine treatment which has been discussed in \cite{cw}. The details will be enclosed in Appendix II for the sake of completeness. The model operators are usually more desirable because they are more ``localizable''. We will first introduce the definition of some ``pararproduct''-type operators - all of which will be useful in the proof of the main theorems.
%two of them will appear in the definition of discrete model operators and the others are closely related will be used in the proof of the main theorems. 

\begin{definition} \label{B_definition}
Let $\mathcal{Q}$ denote a finite collection of dyadic intervals. Suppose that  $(\phi^i_Q)_{Q \in \mathcal{Q}}$ for $i = 1, 2, 3$ are families of $L^2$-normalized adapted bump functions (Definition \ref{bump}) such that at least two families are lacunary (Definition \ref{def_lacunary}). We define the bilinear operator $B_{\mathcal{Q}}$ by
\begin{align} \label{B_global_definition}
& B_{\mathcal{Q}}(v_1,v_2) := \sum_{Q \in \mathcal{Q}}\frac{1}{|Q|^{\frac{1}{2}}}\langle v_1, \phi_Q^1 \rangle \langle v_2, \phi_Q^2 \rangle \phi_Q^3.
\end{align}
Fix a dyadic interval $P$ and a non-negative number $\#$. We define localized bilinear operators $B_{\mathcal{Q}, P}$ and $B^{\#}_{\mathcal{Q}, P}$ by
\begin{align}
B_{\mathcal{Q}, P}(v_1,v_2) :=&  \sum_{\substack{Q \in \mathcal{Q} \\ |Q| \geq |P|}}\frac{1}{|Q|^{\frac{1}{2}}}\langle v_1, \phi_Q^1 \rangle \langle v_2, \phi_Q^2 \rangle \phi_Q^3, \label{B_local0_haar} \\
B_{\mathcal{Q}, P}^{\#}(v_1,v_2) :=&  \sum_{\substack{Q \in \mathcal{Q} \\ |Q| \sim 2^{\#}|P|}}\frac{1}{|Q|^{\frac{1}{2}}}\langle v_1, \phi_Q^1 \rangle \langle v_2, \phi_Q^2 \rangle \phi_Q^3. \label{B_fixed_fourier}
\end{align}
\end{definition}
We define the analogous localized bilinear operators in the Haar model as follows.
\begin{definition} \label{B_definition_haar}
Let $\mathcal{Q}$ denote a finite collection of dyadic intervals. Suppose that  $(\phi^i_Q)_{Q \in \mathcal{Q}}$ for $i = 1, 2$ are families of $L^2$-normalized adapted bump functions and $(\phi^{3,H}_Q)_{Q \in \mathcal{Q}}$ is a family of $L^2$-normalized cutoff functions (Definition \ref{bump_walsh}) such that at least two of the three families are lacunary. Fix a dyadic interval $P$ and a non-negative number $\#$. We define the bilinear operators 
\begin{align}% \label{B_local_defintion}
B_{\mathcal{Q}}^H(v_1,v_2) :=&  \sum_{\substack{Q \in \mathcal{Q} \\}}\frac{1}{|Q|^{\frac{1}{2}}}\langle v_1, \phi_Q^1 \rangle \langle v_2, \phi_Q^2 \rangle \phi_Q^{3,H}, \label{B_global_haar}\\
B_{\mathcal{Q}, P}^H(v_1,v_2) :=&  \sum_{\substack{Q \in \mathcal{Q} \\ |Q| \geq |P|}}\frac{1}{|Q|^{\frac{1}{2}}}\langle v_1, \phi_Q^1 \rangle \langle v_2, \phi_Q^2 \rangle \phi_Q^{3,H}, \label{B_local_definition_haar}  \\
B_{\mathcal{Q}, P}^{\#,H}(v_1,v_2) :=&  \sum_{\substack{Q \in \mathcal{Q} \\ |Q| \sim 2^{\#}|P|}}\frac{1}{|Q|^{\frac{1}{2}}}\langle v_1, \phi_Q^1 \rangle \langle v_2, \phi_Q^2 \rangle \phi_Q^{3,H}. \label{B_local_definition_haar_fix_scale} 
\end{align}

\end{definition}

The discrete model operators are defined as follows.

\begin{definition} \label{discrete_model_op}
Suppose $\mathcal{I}, \mathcal{J}, \mathcal{K}$, $\mathcal{L}$ are finite collections of dyadic intervals.  Suppose $\displaystyle(\phi^i_I)_{I\in \mathcal{I}}$, $ (\phi^j_J)_{J \in \mathcal{J}}$, %$(\phi^k_K)_{K \in \mathcal{K}}$, $(\phi^{l}_L)_{L \in \mathcal{L}}$, 
$i, j, 
%k, l 
= 1, 2, 3$ are families of $L^2$-normalized adapted bump functions. %adapted to $I, J, 
%K, L$ respectively. 
We further assume that at least two families of $(\phi^i_I)_{I\in \mathcal{I}}$ for  $i = 1, 2, 3$ are lacunary. Same conditions are assumed for families $ (\phi^j_J)_{J \in \mathcal{J}}$ for $j = 1, 2, 3$.
%, $(\phi^k_K)_{K \in \mathcal{K}} $ and $(\phi^l_L)_{L \in \mathcal{L}} $
In some models, we specify the lacunary and non-lacunary families by explicitly denoting the functions in the lacunary family as $\psi$ and those in the non-lacunary family as $\vphi$. Let $\#_1, \#_2$ denote some positive integers. Define \newline
\begin{enumerate}
\item
$$ \Pi_{\text{flag}^0 \otimes \text{paraproduct}}(f_1, f_2, g_1, g_2, h)(x,y) := \displaystyle \sum_{I \times J \in \mathcal{I} \times \mathcal{J}} \frac{1}{|I|^{\frac{1}{2}} |J|} \langle B_{\mathcal{K},I}(f_1,f_2),\vphi_I^1 \rangle \langle g_1,\phi^1_J \rangle \langle g_2, \phi^2_J \rangle \langle h, \psi_I^{2} \otimes \phi_{J}^2 \rangle \psi_I^{3} \otimes \phi_{J}^3(x,y);$$ 
%where 
%$$B_I(f_1,f_2) := \displaystyle \sum_{K \in \mathcal{K}:|K| \geq |I|} \frac{1}{|K|^{\frac{1}{2}}}\langle f_1, \phi_K^1 \rangle \langle f_2, \phi_K^2 \rangle \phi_K^3.$$
\item
$$ \Pi_{\text{flag}^{\#_1} \otimes \text{paraproduct}}(f_1,  f_2, g_1, g_2, h)(x,y) := \displaystyle \sum_{I \times J \in \mathcal{I} \times \mathcal{J}} \frac{1}{|I|^{\frac{1}{2}} |J|} \langle B^{\#_1}_{\mathcal{K},I}(f_1,f_2),\vphi_I^1 \rangle \langle g_1,\phi^1_J \rangle \langle g_2, \phi^2_J \rangle \langle h, \psi_I^{2} \otimes \phi_{J}^2 \rangle \psi_I^{3}\otimes \phi_{J}^3(x,y) ;$$ 
%where 
%$$B^{\#_1}_I(f_1,f_2) := \displaystyle \sum_{K \in \mathcal{K}:|K| \sim 2^{\#_1} |I|} \frac{1}{|K|^{\frac{1}{2}}}\langle f_1, \phi_K^1 \rangle \langle f_2, \phi_K^2 \rangle \phi_K^3.$$

\item
$$ \Pi_{\text{flag}^0 \otimes \text{flag}^0}(f_1,  f_2, g_1, g_2, h)(x,y) := \displaystyle \sum_{I \times J \in \mathcal{I} \times \mathcal{J}} \frac{1}{|I|^{\frac{1}{2}} |J|^{\frac{1}{2}}} \langle B_{\mathcal{K},I}(f_1,f_2),\vphi_I^1 \rangle \langle B_{\mathcal{L},J}(g_1, g_2), \vphi_J^1 \rangle  \langle h, \psi_I^{2} \otimes \psi_J^{2} \rangle \psi_I^{3} \otimes \psi_J^{3}(x,y);$$ 
%\begin{align*}
%B_I(f_1,f_2) := & \displaystyle \sum_{K \in \mathcal{K}:|K| \geq |I|} \frac{1}{|K|^{\frac{1}{2}}}\langle f_1, \phi_K^1 \rangle \langle f_2, \phi_K^2 \rangle \phi_K^3, \\
%\tilde{B}_J(g_1,g_2) :=& \displaystyle \sum_{L \in \mathcal{L}:|L| \geq |J|} \frac{1}{|L|^{\frac{1}{2}}}\langle g_1, \phi_L^1 \rangle \langle g_2, \phi_L^2 \rangle \phi_L^3.
%\end{align*}

\item
$$ \Pi_{\text{flag}^0 \otimes \text{flag}^{\#_2}}(f_1, f_2, g_1, g_2, h)(x,y) := \displaystyle \sum_{I \times J \in \mathcal{I} \times \mathcal{J}} \frac{1}{|I|^{\frac{1}{2}} |J|^{\frac{1}{2}}} \langle B_{\mathcal{K},I}(f_1,f_2),\vphi_I^1 \rangle \langle B_{\mathcal{L},J}^{\#_2}(g_1, g_2), \vphi_J^1 \rangle  \langle h, \psi_I^{2} \otimes \psi_J^{2} \rangle \psi_I^{3} \otimes \psi_J^{3}(x,y);$$ 
%where 
%\begin{align*}
%B_I(f_1,f_2)(x) :=&  \displaystyle \sum_{K \in \mathcal{K}:|K| \geq |I|} \frac{1}{|K|^{\frac{1}{2}}}\langle f_1, \phi_K^1 \rangle \langle f_2, \phi_K^2 \rangle \phi_K^3(x), \\
%\tilde{B}_J^{\#_2}(g_1,g_2)(y) :=& \displaystyle \sum_{L \in \mathcal{L}:|L| \sim 2^{\#_2}|J|} \frac{1}{|L|^{\frac{1}{2}}}\langle g_1, \phi_L^1 \rangle \langle g_2, \phi_L^2 \rangle \phi_L^3(y).
%\end{align*}

\item
$$\Pi_{\text{flag}^{\#_1}\otimes \text{flag}^{\#_2}}(f_1,  f_2, g_1, g_2, h)(x,y) := \displaystyle \sum_{I \times J \in \mathcal{I} \times \mathcal{J}} \frac{1}{|I|^{\frac{1}{2}} |J|^{\frac{1}{2}}} \langle B^{\#_1}_{\mathcal{K},I}(f_1,f_2),\vphi_I^1 \rangle \langle B^{\#_2}_{\mathcal{L},J}(g_1, g_2), \vphi_J^1 \rangle  \langle h, \psi_I^{2} \otimes \psi_J^{2} \rangle \psi_I^{3} \otimes \psi_J^{3} (x,y).$$
%where 
%\begin{align*}
%B^{\#_1}_I(f_1,f_2):=& \displaystyle \sum_{K \in \mathcal{K}:|K| \sim 2^{\#_1} |I|} \frac{1}{|K|^{\frac{1}{2}}}\langle f_1, \phi_K^1 \rangle \langle f_2, \phi_K^2 \rangle \phi_K^3, \\
%\tilde{B}_J^{\#_2}(g_1,g_2) :=& \displaystyle \sum_{L \in \mathcal{L}:|L| \sim 2^{\#_2} |J|} \frac{1}{|L|^{\frac{1}{2}}}\langle g_1, \phi_L^1 \rangle \langle g_2, \phi_L^2 \rangle \phi_L^3.
%\end{align*}

\end{enumerate}

\end{definition}
 %which permits the use of our key tool - the stopping-time decomposition based on some variant of averages of functions on rectangles.
The mapping properties of the discrete model operators are stated as follows.
%\subsection{Models for $T_{ab}$}
\begin{theorem} \label{thm_weak_mod}
Let  $\Pi_{\text{flag}^0 \otimes \text{paraproduct}}$, $ \Pi_{\text{flag}^{\#_1} \otimes \text{paraproduct}}$, $\Pi_{\text{flag}^0 \otimes \text{flag}^0}$, $\Pi_{\text{flag}^0 \otimes \text{flag}^{\#_2}}$ and $\Pi_{\text{flag}^{\#_1}\otimes \text{flag}^{\#_2}}$ be multilinear operators specified in Definition \ref{discrete_model_op}. Then all of them satisfy the mapping property stated in Theorem \ref{thm_weak}, where the constants are independent of $\#_1,\#_2$ and the cardinalities of the collections $\mathcal{I}, \mathcal{J}, \mathcal{K}$ and $\mathcal{L}$. %!!!!!!!!!!!!!!!!!!!!!!!check

%$$L^{p} \times L^q \times L^s \rightarrow L^r$$
%for $ 1 < p, q, s< \infty$, $\displaystyle \frac{1}{p} + \frac{1}{q} + \frac{1}{s} = \frac{1}{r}$.
\end{theorem}

%\begin{theorem}
%Let $\Pi$ defined as above. Then $L^{\infty}$  %!!!!!!!!!!!!!!!!
%\end{theorem}

\begin{theorem} \label{thm_weak_inf_mod}
Let  $\Pi_{\text{flag}^0 \otimes \text{paraproduct}}$, $ \Pi_{\text{flag}^{\#_1} \otimes \text{paraproduct}}$, $\Pi_{\text{flag}^0 \otimes \text{flag}^0}$, $\Pi_{\text{flag}^0 \otimes \text{flag}^{\#_2}}$ and $\Pi_{\text{flag}^{\#_1}\otimes \text{flag}^{\#_2}}$ be multilinear operators specified in Definition \ref{discrete_model_op}. Then all of them satisfy the mapping property stated in Theorem \ref{thm_weak_inf}, where the constants are independent of $\#_1,\#_2$ and the cardinalities of the collections $\mathcal{I}, \mathcal{J}, \mathcal{K}$ and $\mathcal{L}$. %!!!!!!!!!!!!!!!!!!!!!!!check
\end{theorem}

The following sections are devoted to the proofs of Theorem \ref{thm_weak_mod} and \ref{thm_weak_inf_mod} which would imply Theorem \ref{thm_weak} and \ref{thm_weak_inf}. We will mainly focus on discrete model operators defined in $(3)$ (Section \ref{section_thm_haar}) and $(5)$ (Section \ref{section_thm_haar_fixed}), whose arguments consist of all the essential tools that are needed for other discrete models. 

\section{Sizes and Energies} \label{section_size_energy}
The notion of sizes and energies appear first in \cite{mtt} and \cite{mtt2}. Since they will play important roles in the main arguments, the explicit definitions of sizes and energies are introduced and some useful properties are highlighted in this section.
\begin{definition} \label{def_size_energy}
Let $\mathcal{I}$ be a finite collection of dyadic intervals. Let $(\psi_I)_{I \in \mathcal{I}}$ denote a lacunary family of $L^2$-normalized bump functions and $(\vphi_I)_{I \in \mathcal{I}}$ a non-lacunary family of $L^2$-normalized bump functions. Define
\begin{enumerate}[(1)]
\item
$$\text{size}_{\mathcal{I}}((\langle f, \vphi_I \rangle)_{I \in \mathcal{I}}) := \sup_{I \in \mathcal{I}} \frac{|\langle f, \vphi_I\rangle|}{|I|^{\frac{1}{2}}};$$
%and 
%$$\text{size}_{\mathcal{I}}((\langle f, \vphi_I^H \rangle)_{I \in \mathcal{I}}) := \sup_{I \in \mathcal{I}} \frac{|\langle f, \vphi^H_I\rangle|}{|I|^{\frac{1}{2}}};$$
\item
$$\text{size}_{\mathcal{I}}((\langle f, \psi_I \rangle)_{I \in \mathcal{I}}) := \sup_{I_0 \in \mathcal{I}} \frac{1}{|I_0|}\left\Vert \bigg(\sum_{\substack{I \subseteq I_0 \\ I \in \mathcal{I}}} \frac{|\langle f, \psi_I \rangle|^2}{|I|} \chi_{I}\bigg)^{\frac{1}{2}}\right\Vert_{1,\infty};$$
\item
$$\text{energy}^{1,\infty} _{\mathcal{I}}((\langle f, \vphi_I \rangle)_{I \in \mathcal{I}}) := \sup_{n \in \mathbb{Z}} 2^{n} \sup_{\mathbb{D}_n} \sum_{I \in \mathbb{D}_n} |I|$$
where $\mathbb{D}_n$ ranges over all collections of disjoint dyadic intervals in $\mathcal{I}$ satisfying 
$$
\frac{|\langle f,\vphi_I \rangle|}{|I|^{\frac{1}{2}}} > 2^n;
$$
\item
$$\text{energy}^{1,\infty} _{\mathcal{I}}((\langle f, \psi_I \rangle)_{I \in \mathcal{I}}) := \sup_{n \in \mathbb{Z}} 2^{n} \sup_{\mathbb{D}_n} \sum_{I \in \mathbb{D}_n} |I|$$
where $\mathbb{D}_n$ ranges over all collections of disjoint dyadic intervals in $\mathcal{I}$ satisfying 
$$
 \frac{1}{|I|}\left\Vert \bigg(\sum_{\substack{\tilde{I} \subseteq I \\ \tilde{I} \in \mathcal{I}}} \frac{|\langle f, \psi_{\tilde{I}} \rangle|^2}{|\tilde{I}|} \chi_{\tilde{I}}\bigg)^{\frac{1}{2}}\right\Vert_{1,\infty} > 2^{n};
$$
\item
For $t>1$, define
$$\text{energy}^{t}_{\mathcal{I}}((\langle f, \vphi_I \rangle)_{I \in \mathcal{I}}) := \left(\sum_{n \in \mathbb{Z}}2^{tn}\sup_{\mathbb{D}_n}\sum_{I \in \mathbb{D}_n}|I| \right)^{\frac{1}{t}}$$
where
$\mathbb{D}_n$ ranges over all collections of disjoint dyadic intervals in $\mathcal{I}$ satisfying 
$$
\frac{|\langle f,\vphi_I \rangle|}{|I|^{\frac{1}{2}}} > 2^n.
$$
\end{enumerate}
\end{definition}
The sizes and energies involving cutoff functions in the Haar model can be defined analogously. 
\begin{definition} \label{def_size_energy_haar}
Let $\mathcal{I}$ be a finite collection of dyadic intervals. Let $(\psi^H_I)_{I \in \mathcal{I}}$ denote a family of Haar wavelets and $(\vphi^H_I)_{I \in \mathcal{I}}$ a family of $L^2$-normalized indicator functions.
Define
$$
\text{size}_{\mathcal{I}}((\langle f, \vphi^H_I \rangle)_{I \in \mathcal{I}}), \ \ \text{size}_{\mathcal{I}}((\langle f, \psi^H_I \rangle)_{I \in \mathcal{I}})
$$
and 
$$
\text{energy}^{1,\infty} _{\mathcal{I}}((\langle f, \psi^H_I \rangle)_{I \in \mathcal{I}}), \ \  \text{energy}^{1,\infty} _{\mathcal{I}}((\langle f, \vphi^H_I \rangle)_{I \in \mathcal{I}}), \ \ \text{energy}^{t}_{\mathcal{I}}((\langle f, \vphi^H_I \rangle)_{I \in \mathcal{I}})\ \ \text{for} \ \ t > 1 
$$
by substituting $\vphi_I$ with $\vphi^H_I$ and  $\psi_I$ with $\psi^H_I$ in Definition \ref{def_size_energy} respectively.
\end{definition}

\subsection{Useful Facts about Sizes and Energies} \label{section_size_energy_fact}
The following propositions describe facts about sizes and energies which will be heavily employed later on. Propositions \ref{JN} and \ref{size} are routine and the proofs can be found in Section 2 of \cite{cw}. Proposition \ref{energy_classical} consists of two parts - the first part is discussed in \cite{cw} while the second part is less standard and the proof will be included in Section \ref{Proof_prop_energy_classical}.

Proposition \ref{B_en_global} describes a ``global'' energy estimate for the bilinear operators $B_{\mathcal{Q}}$ defined in (\ref{B_global_definition}) of Definition \ref{B_definition}  and the proof follows from the boundedness of paraproducts (\cite{cm}, \cite{cw}). Propositions \ref{size_cor} and \ref{B_en} highlight the ``local'' size and energy estimates for the analogous operators defined under the Haar model assumption, namely $B_{\mathcal{Q},P}^H$ %$\tilde{B}_J^H$
and $B_{\mathcal{Q},P}^{\#, H}$
% , $\tilde{B}_J^{\#_2,H}$ 
 defined in (\ref{B_local_definition_haar}) and  (\ref{B_local_definition_haar_fix_scale}) respectively. More precisely, Propositions \ref{size_cor} and \ref{B_en} take into the consideration that the operators $B_{\mathcal{Q},P}^H$ %$\tilde{B}_J^H$
and $B_{\mathcal{Q},P}^{\#, H}$ %defined in (\ref{B_haar_def}) and (\ref{B_size_haar}) 
are localized to intersect certain level sets which carry crucial information for the size and energy estimates. The emphasis on the Haar model assumption keeps track of the arguments we need to modify for the general Fourier case. %The arguments for Proposition \ref{size_cor} and Proposition \ref{B_en} request localizations and more careful treatments that will be discussed in subsequent sections.
%The proof of Lemma \ref{B_size} will be developed in Section 5.2. Proposition \ref{B_en_global} and \ref{B_en} follow from Lemma \ref{B_global_norm} and Lemma \ref{B_loc_norm} respectively, whose proofs will be included in Section 5.4.

%As one would notice from the energy estimates, localizations play an important role in generating different results. The last section in the section will be devoted to the mechanism and conditions one needs to achieve desired localizations.
%and \ref{B_en} highlight the localizations one need to take advantage of in the size and energy estimates for the operator $B$ defined in the discrete model operators specified in Section 4. 

\begin{proposition}[John-Nirenberg] \label{JN}
Let $\mathcal{I}$ be a finite collection of dyadic intervals. For any sequence $(a_I)_{I \in \mathcal{I}}$ and $r > 0$, define the BMO-norm for the sequence as
$$
\|(a_I)_I\|_{\text{BMO}(r)} := \sup_{I_0 \in \mathcal{I}}\frac{1}{|I_0|^{\frac{1}{r}}} \left\Vert \left(\sum_{I \subseteq I_0} \frac{|a_I|^2}{|I|}\chi_{I}(x)\right)^{\frac{1}{2}}\right\Vert_r.
$$
Then for any $0 < p < q < \infty$, 
$$
\|(a_I)_I\|_{\text{BMO}(p)} \simeq \|(a_I)_I \|_{\text{BMO}(q)}.
$$
\end{proposition}

\begin{proposition} \label{size}
Suppose $f \in L^1(\mathbb{R})$. Then
$$\text{size}_{\mathcal{I}}\big((\langle f,\phi_I \rangle)_{I \in \mathcal{I}}\big), \text{size}_{\mathcal{I}}\big((\langle f,\phi^H_I \rangle)_{I \in \mathcal{I}} \big) \lesssim \sup_{I \in \mathcal{I}}\int_{\mathbb{R}}|f|\tilde{\chi}_I^M dx$$
for $M > 0$ and the implicit constant depends on $M$.  $\tilde{\chi}_I$ is an $L^{\infty}$-normalized bump function adapted to $I$.
\end{proposition} 

\begin{proposition} \label{energy_classical}
\noindent 
\begin{enumerate}
\item Suppose $f \in L^1(\mathbb{R})$. Then
$$
\text{energy}^{1,\infty}_{\mathcal{I}}((\langle f, \phi_I \rangle))_{I \in \mathcal{I}}, \text{energy}^{1,\infty}_{\mathcal{I}}((\langle f, \phi^H_I \rangle))_{I \in \mathcal{I}} \lesssim \|f\|_1.
$$
\item Suppose $f \in L^t(\mathbb{R})$ for $t >1$. Then
$$
\text{energy}^t_{\mathcal{I}}((\langle f, \vphi_I \rangle))_{I\in \mathcal{I}}, \text{energy}^t_{\mathcal{I}}((\langle f, \vphi^H_I \rangle))_{I \in \mathcal{I}}\lesssim \|f\|_t.
$$
\end{enumerate}
\end{proposition}

\begin{proposition}[Global Energy] \label{B_en_global}
Suppose that $V_1, V_2 \subseteq \mathbb{R}$ are sets of finite measure and $|v_i| \leq \chi_{V_i}$ for $i = 1,2$. 
Let $\mathcal{P}$ 
%and $\mathcal{Q}$ 
%and $\mathcal{L}$ are
denote a finite collection of dyadic intervals. 
Let $B_{\mathcal{Q}}$ and $B^H_{\mathcal{Q}}$ denote the bilinear operators defined in (\ref{B_global_definition}) of Definition \ref{B_definition} and (\ref{B_global_haar}) of Definition \ref{B_definition_haar} respectively.
%Assume that $(\phi^i_Q)_{Q \in \mathcal{Q}}$ for $i = 1, 2, 3$ are families of $L^2$-normalized adapted bump functions such that at least two families are lacunary. Let $B_{\mathcal{Q}}$ denote the bilinear operator defined in (\ref{B_global_definition}) of Definition \ref{B_definition}.
%Define 
%\begin{align} \label{B_global_def}
%& B(f_1,f_2)(x):= \sum_{K \in \mathcal{K}} \frac{1}{|K|^{\frac{1}{2}}}\langle f_1, \phi_K^1\rangle \langle f_2, \phi_K^2 \rangle \phi_K^3 (x), \nonumber \\
%& \tilde{B}(g_1,g_2)(y):= \sum_{L \in \mathcal{L}} \frac{1}{|L|^{\frac{1}{2}}}\langle g_1, \phi_L^1\rangle \langle g_2, \phi_L^2 \rangle \phi_L^3 (y).
%\end{align}
\begin{enumerate}
\item 
Then for any $0 < \rho <1$, one has
\begin{align*}
& \text{energy}^{1,\infty}_{\mathcal{P}}((\langle B_{\mathcal{Q}}(v_1,v_2), \vphi_P \rangle)_{P \in \mathcal{P}}), \text{energy}^{1,\infty}_{\mathcal{P}}((\langle B^H_{\mathcal{Q}}(v_1,v_2), \vphi_P \rangle)_{P \in \mathcal{P}}) \lesssim |V_1|^{\rho}|V_2|^{1-\rho}. 
%& \text{energy}^{1,\infty}_{\mathcal{J}}((\langle \tilde{B}(g_1,g_2), \vphi_J \rangle)_{J \in \mathcal{J}}) \lesssim |G_1|^{\rho'}|G_2|^{1-\rho'}. \nonumber \\
\end{align*}
\item Suppose that $t >1$. Then for any $0 \leq \theta_1,  \theta_2 <1$, with $\theta_1 + \theta_2 = \frac{1}{t}$, one has
\begin{align*}
& \text{energy}^t_{\mathcal{P}}((\langle B_{\mathcal{Q}}(v_1,v_2), \vphi_P \rangle)_{P \in \mathcal{P}}), \text{energy}^t_{\mathcal{P}}((\langle B^H_{\mathcal{Q}}(v_1,v_2), \vphi_P \rangle)_{P \in \mathcal{P}}) \lesssim |V_1|^{\theta_1}|V_2|^{\theta_2}.
%& \text{energy}^s_{\mathcal{J}}((\langle \tilde{B}(g_1,g_2), \vphi_J \rangle)_{J \in \mathcal{J}}) \lesssim |G_1|^{\zeta_1}|G_2|^{\zeta_2}. \nonumber \\
\end{align*}
\end{enumerate}
\end{proposition}

It is not difficult to observe that Proposition \ref{B_en_global} follows immediately from Proposition \ref{energy_classical} and the following lemma.

\begin{lemma}\label{B_global_norm}
%Suppose that $t,s \geq 1$. Then for any $0 \leq \theta_1,\theta_1',\theta_2,\zeta_2 <1$ with $\theta_1 + \theta_2 = \frac{1}{t}$ and $\theta_1' + \zeta_2 = \frac{1}{t'}$, one has
%\begin{align*}
%& \|B(f_1,f_2)\|_t \lesssim |F_1|^{\theta_1}|F_2|^{\theta_2}\nonumber \\
%&  \|\tilde{B}(g_1,g_2)\|_{t'} \lesssim |G_1|^{\theta_1'}|G_2|^{\zeta_2}
%\end{align*}
%Suppose that $\mathcal{Q}$ 
%and $\mathcal{L}$ are
%is a finite collection of dyadic intervals. 
%\begin{enumerate}
%\item
%Assume $(\phi^i_Q)_{Q \in \mathcal{Q}}$ for $i = 1, 2, 3$ are families of $L^2$-normalized adapted bump functions such that at least two families are lacunary.
 Let $B_{\mathcal{Q}}$ and $B^H_{\mathcal{Q}}$ denote the bilinear operators defined in (\ref{B_global_definition}) and (\ref{B_global_haar}) respectively.
 %Suppose that $1 < p,q \leq \infty$, $0 < t < \infty$ and $\frac{1}{t} = \frac{1}{p}+ \frac{1}{q} $. 
Then for any $v_1 \in L^{p}$, $v_2 \in L^{q}$ with $1 < p,q \leq \infty$, $0 < t < \infty$ and $\frac{1}{t} = \frac{1}{p}+ \frac{1}{q} $,
\begin{align*}
& \|B_{\mathcal{Q}}(v_1,v_2)\|_{t},  \|B^H_{\mathcal{Q}}(v_1,v_2)\|_{t} \lesssim \|v_1\|_{L^{p}} \|v_2\|_{L^{q}}. \nonumber 
%& \|\tilde{B}(g_1,g_2)\|_{s} \lesssim \|g_1\|_{L^{p_2}} \|g_2\|_{L^{q_2}}.
\end{align*}
%\item
%Replacing the bump function $\phi_Q^3$ by the cutoff function $\phi_Q^{3,H}$ such that $(\phi^i_Q)_{Q \in \mathcal{Q}}$, $i = 1,2$ and $(\phi_Q^{3,H})_{Q \in \mathcal{Q}}$ satisfying the conditions in Definition \ref{B_definition_haar}. Let $B^H_{\mathcal{Q}}$ be the bilinear operator defined in (\ref{B_global_haar}). Then
%\begin{equation*}
% \|B^H_{\mathcal{Q}}(v_1,v_2)\|_{t} \lesssim \|v_1\|_{L^{p}} \|v_2\|_{L^{q}}.
%\end{equation*}
% \end{enumerate}
\end{lemma}
By identifying that $B_{\mathcal{Q}}$ (\ref{B_global_definition}) and $B^H_{\mathcal{Q}}$ (\ref{B_global_haar}) are one-parameter paraproducts, Lemma \ref{B_global_norm} is a restatement of Coifman-Meyer's theorem on the boundedness of paraproducts \cite{cm}. 

We will now turn our attention to the local size estimate for $(\langle B_{\mathcal{Q}, P
}^{\#,H}(v_1, v_2), \vphi_{P} \rangle)_{P \in \mathcal{P}}$ 
%and $(\langle \tilde{B}_J^{\#_2,H}, \vphi_J \rangle)_J$ 
and the local energy estimate for $(\langle B_{\mathcal{Q}, P}^H(v_1, v_2), \vphi_P \rangle )_{P \in \mathcal{P}}$. 
In the Haar model, for any fixed dyadic intervals $Q$ and $P$ with $|Q|\geq |P|$, the only non-degenerate case $\langle \phi_Q^{3,H}, \vphi_{P}^H \rangle \neq 0$ is that $Q \supseteq P$. Such observation provides natural localizations for the sequence $(\langle B_{\mathcal{Q}, P}^{\#,H}(v_1, v_2), \vphi^H_P \rangle)_{P \in \mathcal{P}}$ and thus for the sequences $(\langle v_1, \phi_Q^1 \rangle)_{Q \in \mathcal{Q}}$ and $(\langle v_2, \phi_Q^2 \rangle)_{Q \in \mathcal{Q}}$ as explicitly stated in the following lemma. 

\begin{lemma}[Localization of Sizes in the Haar Model] \label{B_size}
%Suppose that $V_1, V_2 \subseteq \mathbb{R}$ are sets of finite measure and $|v_i| \leq \chi_{V_i}$ for $i = 1,2$. 
Let $S$ denote a measurable subset of $\mathbb{R}$ and $\mathcal{P}'$ a finite collection of dyadic intervals 
%and $S'$ a measurable subset of  $\mathbb{R}_{y}$
such that for any $P \in \mathcal{P}'$,  $P \cap S \neq \emptyset$.
% and $J \cap S' \neq \emptyset$ for any $J \in \mathcal{J}'$ 
%Suppose that $\mathcal{Q}$ 
%and $\mathcal{L}$ are
%is a finite collection of dyadic intervals and $(\phi^i_Q)_{Q \in \mathcal{Q}}$ for $i = 1, 2$ 
%are families of $L^2$-normalized adapted bump functions 
%and $(\phi^{3,H}_Q)_{Q \in \mathcal{Q}}$ are families of cutoff functions satisfying the conditions in Definition \ref{B_definition_haar}.
%is a family of $L^2$-normalized bump functions in the Haar model such that at least two of the three families are lacunary.
%$(\phi^i_Q)_{Q \in \mathcal{Q}}$ for $i = 1, 2, 3$ are families of $L^2$-normalized adapted bump functions such that at least two families are lacunary. 
For each $P \in \mathcal{P}'$, let $B_{\mathcal{Q},P}^{\#, H}$ denote the bilinear operator defined in (\ref{B_local_definition_haar_fix_scale}) of Definition \ref{B_definition}.  
Then for any measurable functions $v_1$ and $v_2$,
$$
\text{size}_{\mathcal{P'}}((\langle B_{\mathcal{Q},P}^{\#,H}(v_1, v_2), \vphi^H_P \rangle)_{P \in \mathcal{P}'}) \lesssim \sup_{\substack{Q \in \mathcal{Q} \\ Q \cap S \neq \emptyset}}\frac{|\langle v_1, \phi_Q^1 \rangle|}{|Q|^{\frac{1}{2}}} \sup_{\substack {Q \in \mathcal{Q} \\ Q \cap S \neq \emptyset}}\frac{|\langle v_2, \phi_Q^2 \rangle|}{|Q|^{\frac{1}{2}}}.
$$
%$$
%\text{size}_{\mathcal{J}'}((\langle \tilde{B}_J^{\#_2,H}, \vphi^H_J \rangle)_{J \in \mathcal{J}'}) \lesssim \sup_{L \cap S' \neq \emptyset}\frac{|\langle g_1, \phi_L^1 \rangle|}{|L|^{\frac{1}{2}}} \sup_{L \cap S' \neq \emptyset}\frac{|\langle g_2, \phi_L^2 \rangle|}{|L|^{\frac{1}{2}}}.
%$$
\end{lemma}

The localization generates more quantitative and useful estimates for the sizes involving $B_{\mathcal{Q}, P}^{\#,H}$ when $S$ is a level set of the Hardy-Littlewood maximal functions $Mv_1$ and $Mv_2$ as elaborated in the following proposition. 

\begin{remark}
One notational comment is that $C$, $C_1, C_2$ and $C_3$ used throughout the paper denote some sufficiently large constants greater than 1.
\end{remark}

\begin{proposition} [Local Size Estimates in the Haar Model]\label{size_cor}
Suppose that $V_1, V_2 \subseteq \mathbb{R}$ are sets of finite measure and $|v_i| \leq \chi_{V_i}$ for $i = 1,2$. Let $\tilde{n}, \tilde{m}$ denote some integers and $\ \mathcal{U}_{\tilde{n},\tilde{m}}:=\{z: Mv_1(z) \leq C2^{\tilde{n}} |V_1|\} \cap \{ z: Mv_2(z) \leq C 2^{\tilde{m}} |V_2|\}$. Further assume that $ \mathcal{P}' $ is a finite collection of dyadic intervals such that $P \cap \mathcal{U}_{\tilde{n},\tilde{m}} \neq \emptyset$ for any $P \in \mathcal{P}'$. For each $P \in \mathcal{P}'$, $B_{\mathcal{Q},P}^{\#, H}$ denote the bilinear operator defined in (\ref{B_local_definition_haar_fix_scale}).
%Suppose that $V_1, V_2 \subseteq \mathbb{R}$ are sets of finite measure and $|v_i| \leq \chi_{V_i}$ for $i= 1,2$. Let $\tilde{n}, \tilde{m}$ denote some integers and $\ \mathcal{U}_{\tilde{n},\tilde{m}}:=\{z: Mv_1(z) \leq C_12^{\tilde{n}} |V_1|\} \cap \{ z: Mv_2(z) \leq C_1 2^{\tilde{m}} |V_2|\}$. Further assume that $ \mathcal{P}' $ is a finite collection of dyadic intervals such that $P \cap \mathcal{U}_{\tilde{n},\tilde{m}} \neq \emptyset$ for any $P \in \mathcal{P}'$. Suppose that $\mathcal{Q}$ is a finite collection of dyadic intervals and  $(\phi^i_Q)_{Q \in \mathcal{Q}}$ for $i = 1, 2, 3$ are families of $L^2$-normalized adapted bump functions such that at least two families are lacunary. Let $B_{\mathcal{Q},P}^{\#, H}$ denote the bilinear operator defined in (\ref{B_local_definition_haar}) of Definition \ref{B_definition}. 
Then
$$
\text{size}_{\mathcal{P}'}((\langle B_{\mathcal{Q}, P}^{\#,H}(v_1, v_2), \vphi^H_P \rangle)_{P \in \mathcal{P}'}) \lesssim (C_1 2^{\tilde{n}}|V_1|)^{\alpha} (C_1 2^{\tilde{m}}|V_2|)^{\beta},$$
%$$
%\text{size}_{\mathcal{J}'}((\langle \tilde{B}_J^{\#_2,H}, \vphi^H_J \rangle)_{J \in \mathcal{J}'}) \lesssim (C_2 2^{n_2}|G_1|)^{\alpha_2} (C_2 2^{m_2}|G_2|)^{\beta_2},
%$$
for any $ 0 \leq \alpha, \beta \leq 1$.
\end{proposition}
The proof of the proposition follows directly from Lemma \ref{B_size} and the trivial estimates
$$
\sup_{Q \cap \mathcal{U}_{\tilde{n},\tilde{m}} \neq \emptyset}\frac{|\langle v_1, \phi_Q^1 \rangle|}{|Q|^{\frac{1}{2}}} \lesssim \min(C 2^{\tilde{n}}|V_1|,1) \leq (C 2^{\tilde{n}}|V_1|)^{\alpha}
$$
and
$$
\sup_{Q \cap \mathcal{U}_{\tilde{n},\tilde{m}} \neq \emptyset}\frac{|\langle v_2, \phi_Q^2 \rangle|}{|Q|^{\frac{1}{2}}} \lesssim \min(C 2^{\tilde{m}}|V_2|,1) \leq (C 2^{\tilde{m}}|V_2|)^{\beta}
$$
for any $0 \leq \alpha,\beta \leq 1$.
%$$
%\sup_{L \cap \mathcal{U}'_{n_2,m_2} \neq \emptyset}\frac{|\langle g_1, \phi_L^1 \rangle|}{|L|^{\frac{1}{2}}} \lesssim \min(C_2 2^{n_2}|G_1|,1),
%$$
%$$
%\sup_{L \cap \mathcal{U}'_{n_2,m_2} \neq \emptyset}\frac{|\langle g_2, \phi_L^2 \rangle|}{|L|^{\frac{1}{2}}} \lesssim \min(C_2 2^{m_2}|G_2|,1).
%$$

We will also explore the local energy estimates which are ``stronger'' than the global energy estimates. Heuristically, in the case when $v_1 \in L^{p}$ and $ v_2 \in L^{q}$ with $|v_1| \leq \chi_{V_1}$ and $|v_2| \leq \chi_{V_2}$ for $p,q>1$ and close to $1$, the global energy estimates would not yield the desired boundedness exponents for $|V_1|$ and $|V_2| $ whereas one could take advantages of the local energy estimates to obtain the result. %This remark would become more evident in the discussion in subsequent sections. 
In the Haar model, a perfect localization can be achieved for energy estimates involving bilinear operators $B^H_{\mathcal{Q},P}$ specified in (\ref{B_local_definition_haar}) of Definition \ref{B_definition}. In particular,  the corresponding energy estimates can be compared to the energy estimates for %$(\langle B^{n_1,m_1}_0, \vphi_I \rangle )_{I \in \mathcal{I}'}$ and $(\langle \tilde{B}^{n_2,m_2}_0, \vphi_J \rangle )_{J \in \mathcal{J}'}$ where $B^{\tilde{n},\tilde{m}}_0$ 
%and $\tilde{B}^{n_2,m_2}_0$
%is a 
localized operators defined as follows.

\begin{comment}
\begin{definition} 
Suppose that $\phi_{K}^{i}, \phi_{L}^{j}$ for $i, j = 1,2$ are $L^2$-normalized bump functions adapted to $I$ and $J$ respectively. Further assume that $\phi_K^{3,H}$ and  $\phi_L^{3,H} $ are Haar wavelets or $L^2$-normalized indicator functions on $K$ and $L$. One will consider Haar wavelets to be a lacunary family and $L^2$-normalized indicator functions to be a non-lacunary family. Suppose that at least two families of $(\phi_{K}^{1})_K, (\phi_{K}^2)_K$ and $(\phi_{K}^{3,H})_K$ are lacunary and that at least two families of $(\phi_{L}^{1})_L, (\phi_{L}^2)_L$ and $(\phi_{L}^{3,H})_L$ are lacunary. Let
\begin{align}
& B_{I}^H(f_1, f_2)(x) :=  \sum_{K: |K| \geq |I|} \frac{1}{|K|^{\frac{1}{2}}}\langle f_1, \phi_K^1\rangle \langle f_2, \phi_K^2 \rangle \phi_K^{3,H} (x) \nonumber \\
& \tilde{B}_{J}^H(g_1, g_2)(y) :=  \sum_{L: |L| \geq |J|} \frac{1}{|L|^{\frac{1}{2}}}\langle g_1, \phi_L^1\rangle \langle g_2, \phi_L^2 \rangle \phi_L^{3,H} (y).
\end{align}
%as stated in Lemma \ref{localization_haar}. 
\end{definition}
\end{comment}

\begin{definition}\label{B^0_def}
Let $\mathcal{U}_{\tilde{n},\tilde{m}}$ be defined as the level set described in Proposition \ref{size_cor}. Let $\mathcal{Q}$ denote a finite collection of dyadic intervals.
\begin{enumerate}
\item
Suppose that  $(\phi^i_Q)_{Q \in \mathcal{Q}}$ for $i = 1, 2$ are families of $L^2$-normalized adapted bump functions and $(\phi^{3,H}_Q)_{Q \in \mathcal{Q}}$ is a family of $L^2$-normalized cutoff functions %such that at least two of the three families are lacunary.
%Suppose that $(\phi^i_Q)_{Q \in \mathcal{Q}}$ for $i = 1, 2, 3$ are families of $L^2$-normalized adapted bump functions 
such that the family $(\phi_Q^{3,H})_{Q \in \mathcal{Q}}$ is lacunary while at least one of the families $(\phi^1_Q)_{Q \in \mathcal{Q}}$ and $(\phi^2_Q)_{Q \in \mathcal{Q}}$ is lacunary.
%such that at least two families are lacunary. 
 %such that at least two families are lacunary. 
%And suppose that $\mathcal{P}'$ is a finite collection of dyadic intervals such that $P \cap \mathcal{U}_{\tilde{n},\tilde{m}} \neq \emptyset$ for any $P \in \mathcal{P}'$. 
Define the bilinear operator
 \begin{align}
 %\ \ \text{if}\ \ \phi_K^{3,H} \ \ \text{is}\ \ L^2 \text{-normalized indicator func.} \\ 
%\displaystyle \sum_{K: K \cap \mathcal{U}_{n_1,m_1} \neq \emptyset} \frac{1}{|K|^{\frac{1}{2}}}\langle f_1, \vphi_K^1\rangle \langle f_2, \psi_K^2 \rangle \psi_K^{3,H} (x) \quad \ \ \ \text{if}\ \ \phi_K^{3,H} \ \ \text{is Haar wavelet},  \\ 
%\end{cases}
 B^{\tilde{n},\tilde{m},0}_{\mathcal{Q},\text{lac}}(v_1,v_2):= &\displaystyle \sum_{\substack{Q \in \mathcal{Q} \\ Q \cap \mathcal{U}_{\tilde{n},\tilde{m}} \neq \emptyset}} \frac{1}{|Q|^{\frac{1}{2}}}\langle v_1, \phi_Q^1\rangle \langle v_2, \phi_Q^2 \rangle \psi_Q^{3,H}. \label{B^0_lac}  %\quad \ \ \ \text{if}\ \ \phi_K^{3,H} \ \ \text{is Haar wavelet}, 
 \end{align}
 \item
Suppose that  $(\phi^i_Q)_{Q \in \mathcal{Q}}$ for $i = 1, 2$ are families of $L^2$-normalized adapted bump functions and $(\phi^{3,H}_Q)_{Q \in \mathcal{Q}}$ is a family of $L^2$-normalized cutoff functions such that the family $(\phi^{3,H}_Q)_{Q \in \mathcal{Q}}$ is nonlacunary while the families $(\phi^i_Q)_{Q \in \mathcal{Q}}$ for $i \neq 3$ are both lacunary. Define the bilinear operator
 \begin{align}
B^{\tilde{n},\tilde{m},0}_{\mathcal{Q},\text{nonlac}}(v_1,v_2) := &
%\begin{cases} \displaystyle
 \sum_{\substack{ Q \in \mathcal{Q}  \\Q: Q \cap \mathcal{U}_{\tilde{n},\tilde{m}} \neq \emptyset}} \frac{1}{|Q|^{\frac{1}{2}}}|\langle v_1, \psi_Q^1\rangle| |\langle v_2, \psi_Q^2 \rangle| |\vphi_Q^{3,H} |. \label{B^0_nonlac}
\end{align}
\end{enumerate}
%$$
%\tilde{B}^{n_2,m_2}_0(g_1,g_2)(y):= 
%\begin{cases} \displaystyle
%\sum_{L: L \cap \mathcal{U}'_{n_2,m_2} \neq \emptyset} \frac{1}{|L|^{\frac{1}{2}}}|\langle g_1, \psi_L^1\rangle| |\langle g_2, \psi_L^2 \rangle| |\vphi_L^{3,H} (y)| \ \ \text{if}\ \ \phi_L^{3,H}  \ \ \text{is}\ \ L^2 \text{-normalized indicator func.}   \\ 
%\displaystyle \sum_{L: L \cap \mathcal{U}'_{n_2,m_2} \neq \emptyset} \frac{1}{|L|^{\frac{1}{2}}}\langle g_1, \vphi_L^1\rangle \langle g_2, \psi_L^2 \rangle \psi_L^{3,H} (y) \quad \ \ \ \text{if}\ \ \phi_L^{3,H} \ \ \text{is Haar wavelet}.  \\ 
%\end{cases}
%$$
\end{definition}

\begin{remark}
We would like to emphasize that $B^{\tilde{n},\tilde{m},0}_{\mathcal{Q},\text{lac}}$ and $B^{\tilde{n},\tilde{m},0}_{\mathcal{Q},\text{nonac}}$ are localized to intersect level sets $\mathcal{U}_{\tilde{n},\tilde{m}}$ nontrivially. It is not difficult to imagine that the energy estimates for $(\langle B^{\tilde{n},\tilde{m},0}_{\mathcal{Q},\text{lac}}(v_1,v_2), \vphi_P \rangle )_{P \in \mathcal{P}'}$ and $(\langle B^{\tilde{n},\tilde{m},0}_{\mathcal{Q},\text{nonlac}}(v_1,v_2), \vphi_P \rangle )_{P \in \mathcal{P}'}$ would be better than the ``global'' energy estimates
% ($i.e. \ \ \text{energy}_{\mathcal{P}}(\langle B_{\mathcal{Q}}(v_1, v_2), \vphi_P\rangle_{P \in \mathcal{P}} )$ 
stated in Proposition \ref{B_en_global} since one can now employ the information about intersections with level sets to control
$
\frac{|\langle v_1, \phi_Q^1 \rangle|}{|Q|^{\frac{1}{2}}} \text{ and } \frac{|\langle v_2, \phi_Q^2 \rangle|}{|Q|^{\frac{1}{2}}}. %\frac{|\langle g_1, \phi_L^1 \rangle|}{|L|^{\frac{1}{2}}}, \frac{|\langle g_2, \phi_L^2 \rangle|}{|L|^{\frac{1}{2}}}.
$
The energy estimates for $(\langle B^{H}_{\mathcal{Q},P}(v_1,v_2), \vphi_P \rangle )_{P \in \mathcal{P}'}$ can indeed be reduced to the energy estimates for $(\langle B^{\tilde{n},\tilde{m},0}_{\mathcal{Q}, \text{lac}}, \vphi_P \rangle)_{P \in \mathcal{P}'}$ or  $(\langle B^{\tilde{n},\tilde{m},0}_{\mathcal{Q}, \text{nonlac}}, \vphi_P \rangle)_{P \in \mathcal{P}'}$  %and $(\langle \tilde{B}^{n_2,m_2}_0, \vphi_J \rangle )_{J \in \mathcal{J}'}$ 
as stated in Lemma \ref{localization_haar}.
\end{remark}
 
\begin{lemma}[Localization of Energies in the Haar Model] \label{localization_haar}
Suppose that $\mathcal{P}'$ is a finite collection of dyadic intervals such that $P \cap \mathcal{U}_{\tilde{n},\tilde{m}} \neq \emptyset$ for any $P \in \mathcal{P}'$. For each $P \in \mathcal{P}'$, let $B^{H}_{\mathcal{Q},P}$ denote the bilinear operators defined in (\ref{B_local_definition_haar}).
%Suppose $\mathcal{Q}$ is a finite collection of dyadic intervals and $(\phi^i_Q)_{Q \in \mathcal{Q}}$ for $i = 1, 2$ and $(\phi^{3,H}_Q)_{Q \in \mathcal{Q}}$ are families of cutoff functions satisfying the conditions in Definition \ref{B_definition_haar}.
%$(\phi^i_Q)_{Q \in \mathcal{Q}}$ for $i = 1, 2, 3$ are families of $L^2$-normalized adapted bump functions such that at least two families are lacunary. 
Let $B^{\tilde{n},\tilde{m},0}_{\mathcal{Q}, \text{lac}}$ and $B^{\tilde{n},\tilde{m},0}_{\mathcal{Q}, \text{nonlac}}$ denote the bilinear operators defined in Definition \ref{B^0_def} where the bump functions in $ B^{\tilde{n},\tilde{m},0}_{\mathcal{Q}, \text{lac}}$ or $B^{\tilde{n},\tilde{m},0}_{\mathcal{Q}, \text{nonlac}}$ are the same as the ones in $B^H_{\mathcal{Q},P}$.
%\begin{align*}
%& B^{n_1,m_1}_0(f_1,f_2)(x):=\sum_{K: K \cap S \neq \emptyset} \frac{1}{|K|^{\frac{1}{2}}}\langle f_1, \phi_K^1\rangle \langle f_2, \phi_K^2 \rangle \phi_K^{3,H} (x) \nonumber \\
%& \tilde{B}^{n_2,m_2}_0(g_1,g_2)(y):=\sum_{L: L \cap S' \neq \emptyset} \frac{1}{|L|^{\frac{1}{2}}}\langle g_1, \phi_L^1\rangle \langle g_2, \phi_L^2 \rangle \phi_L^{3,H} (y)
%\end{align*}
Then for $t > 1$ and any measurable functions $v_1$ and $v_2$,
\begin{align}
& \text{energy}^t_{\mathcal{P}'}((\langle B^H_{\mathcal{Q},P}(v_1,v_2), \vphi^H_P \rangle)_{P \in \mathcal{P}'})\nonumber \\
\leq &
\begin{cases}
\text{energy}^t_{\mathcal{P}'}((\langle B^{\tilde{n},\tilde{m},0}_{\mathcal{Q}, \text{lac}}(v_1,v_2), \vphi^H_P \rangle)_{P \in \mathcal{P}'}) \ \ \ \ \ \text{if} \ \ (\phi^{3,H}_Q)_{Q \in \mathcal{Q}} \text{ is a lacunary family}  \\
\\
\text{energy}^t_{\mathcal{P}'}((\langle B^{\tilde{n},\tilde{m},0}_{\mathcal{Q}, \text{nonlac}}(v_1,v_2), \vphi^H_P \rangle)_{P \in \mathcal{P}'}) \ \ \text{if} \ \ (\phi^{3,H}_Q)_{Q \in \mathcal{Q}} \text{ is a non-lacunary family},
%& \text{energy}^t_{\mathcal{J}'}((\langle \tilde{B}^H_J, \vphi^H_J \rangle)_{J \in \mathcal{J}'}) \leq \text{energy}^t_{\mathcal{J}'}((\langle \tilde{B}^{n_2,m_2}_0, \vphi_J^H \rangle)_{J \in \mathcal{J}'});
\end{cases}
\end{align}
and
\begin{align}
& \text{energy}^{1,\infty}_{\mathcal{P}'}((\langle B^H_{\mathcal{Q},P}(v_1,v_2), \vphi^H_P \rangle)_{P \in \mathcal{P}'}) \nonumber \\
\leq &
\begin{cases}
\text{energy}^{1,\infty}_{\mathcal{P}'}((\langle B^{\tilde{n},\tilde{m},0}_{\mathcal{Q}, \text{lac}}(v_1,v_2), \vphi^H_P \rangle)_{P \in \mathcal{P}'}) \ \ \ \ \ \text{if} \ \ (\phi^{3,H}_Q)_{Q \in \mathcal{Q}} \text{ is a lacunary family}  \\
\\
\text{energy}^{1,\infty}_{\mathcal{P}'}((\langle B^{\tilde{n},\tilde{m},0}_{\mathcal{Q}, \text{nonlac}}(v_1,v_2), \vphi^H_P \rangle)_{P \in \mathcal{P}'}) \ \text{if} \ \ (\phi^{3,H}_Q)_{Q \in \mathcal{Q}} \text{ is a non-lacunary family}.
%& \text{energy}^t_{\mathcal{J}'}((\langle \tilde{B}^H_J, \vphi^H_J \rangle)_{J \in \mathcal{J}'}) \leq \text{energy}^t_{\mathcal{J}'}((\langle \tilde{B}^{n_2,m_2}_0, \vphi_J^H \rangle)_{J \in \mathcal{J}'});
\end{cases}
\end{align}
%\begin{align}
%& \text{energy}^{1,\infty}_{\mathcal{P}'}((\langle B^H_{\mathcal{Q},P}, \vphi^H_P \rangle)_{P \in \mathcal{P}'}) \leq \text{energy}^{1,\infty}_{\mathcal{I}'}((\langle B^{n_1,m_1}_0, \vphi^H_I \rangle)_{I \in \mathcal{I}'}).
%& \text{energy}^{1,\infty}_{\mathcal{J}'}((\langle \tilde{B}^H_J, \vphi^H_J \rangle)_{J \in \mathcal{J}'}) \leq \text{energy}^{1,\infty}_{\mathcal{J}'}((\langle \tilde{B}^{n_2,m_2}_0, \vphi_J^H \rangle)_{J \in \mathcal{J}'}).
%\end{align}
\end{lemma}

The following local energy estimates will play a crucial role in the proof of our main theorem. 
\begin{proposition}[Local Energy Estimates in the Haar Model] \label{B_en}
Suppose that $V_1, V_2 \subseteq \mathbb{R}$ are sets of finite measure and $|v_i| \leq \chi_{V_i}$ for $i= 1,2$. %Let $\mathcal{I}''$ denote a subcollection such that $I \cap \mathcal{U}_{n_1+1,m_1}$
%Also suppose for any $I \in \mathcal{I}'$, $I \cap S \neq \emptyset$ and for any $J \in \mathcal{J}'$, $J \cap S' \neq \emptysxfet$.  
Let $\mathcal{P}'$ denote the collection of dyadic intervals satisfying the condition described in Lemma \ref{localization_haar} and $B_{\mathcal{Q},P}^{ H}$ the bilinear operator defined in (\ref{B_local_definition_haar}) for each $P \in \mathcal{P}'$.
%Let $\mathcal{P}'$ denote finite collections of dyadic intervals such that $P \cap \mathcal{U}_{\tilde{n},\tilde{m}} \neq \emptyset$ for any $P \in \mathcal{P}'$.  Suppose $\mathcal{Q}$ is a finite collection of dyadic intervals and $(\phi^i_Q)_{Q \in \mathcal{Q}}$ for $i = 1, 2, 3$ are families of $L^2$-normalized adapted bump functions such that at least two families are lacunary. Let $B_{\mathcal{Q},P}^{ H}$ denote the bilinear operator defined in (\ref{B_local_defintion_haar}) of Definition \ref{B_definition}.  
Further assume that $\frac{1}{p} + \frac{1}{q} > 1$. 
\begin{enumerate}[(i)]
\item
Suppose that $t >1$. Then for any $0 \leq \theta_1,\theta_2 <1$ with $\theta_1 + \theta_2 = \frac{1}{t}$, one has 
%for any $1 \leq t < \infty$, $1< p' \leq \infty $ and $\frac{1}{p}+ \frac{1}{p'}$ = 1,
\begin{align} \label{B_en_t}
%\sum_{l}2^{lp}\sup_{\mathbb{D}}\sum_{\substack{I \in \mathbb{D} \\ \frac{|\langle B_I, \vphi_I \rangle|}{|I|^{\frac{1}{2}}} > 2^{l} \\ I \in \mathcal{I}'}}|I|
& \text{energy}^{t} _{\mathcal{P}'}((\langle B^H_{\mathcal{Q},P}(v_1,v_2), \vphi_P^H \rangle)_{P \in \mathcal{P}'}) \lesssim  C_1^{\frac{1}{p}+ \frac{1}{q} - \theta_1 - \theta_2}2^{\tilde{n}(\frac{1}{p} - \theta_1)}2^{\tilde{m}(\frac{1}{q} - \theta_2)}|V_1|^{\frac{1}{p}}|V_2|^{\frac{1}{q}}.
%& \text{energy}^{t} _{\mathcal{J}'}((\langle \tilde{B}^H_J, \vphi_J^H \rangle)_{J \in \mathcal{J}'}) \lesssim C_2^{\frac{1}{p_2}+ \frac{1}{q_2} - \zeta_1 - \zeta_2}2^{n_2(\frac{1}{p_2} - \zeta_1)}2^{m_2(\frac{1}{q_2} - \zeta_2)}|G_1|^{\frac{1}{p_2}}|G_2|^{\frac{1}{q_2}}.
%\text{size}_{K \in \mathcal{K}: K \cap S \neq \emptyset}((\langle f_1, \vphi_K \rangle)_{K})^{1-\theta_1}\text{size}_{K \in \mathcal{K}:K \cap S \neq \emptyset}((\langle f_2, \psi_K \rangle)_{K})^{1-\theta_2} \nonumber \\
%& \cdot |F_1|^{\theta_1}|F_2|^{\theta_2}|S_2|^{\theta_3-\frac{1}{p'}}
%\text{energy}^{2} _{\mathcal{J}'}((\langle \tilde{B_J}, \vphi_J \rangle)_{J \in \mathcal{J}'}) \lesssim & \text{size}_{L \in \mathcal{L}: L \cap S' \neq \emptyset}((\langle g_1, \vphi_K \rangle)_{L})^{1-\theta_1}\text{size}_{L \in \mathcal{L}:L \cap S' \neq \emptyset}((\langle g_2, \psi_L \rangle)_{L})^{1-\theta_2} \nonumber \\
%& \cdot |G_1|^{\theta_1}|G_2|^{\theta_2}|S'_2|^{\theta_3-\frac{1}{p'}}
\end{align}

\item
For any $0 < \theta_1,\theta_2 <1$ with $\theta_1 + \theta_2 = 1$,
%\item
%Suppose that $\mathcal{I}'' \subseteq \mathcal{I}'$ is a finite sub-collection of dyadic intervals such that for any $I \in \mathcal{I}''$,
% \begin{align} \label{I_intersect_level}
%|I \cap \{Mf_1 >C_12^{n_1-1}|F_1| \}| > \frac{1}{10}|I|, \nonumber\\
%|I \cap \{Mf_2 >C_12^{m_1-1}|F_2| \}| > \frac{1}{10}|I|. 
%\end{align}
%Similarly, let $\mathcal{J}'' \subseteq \mathcal{J}'$ represent a finite sub-collection of dyadic intervals such that for any $J \in \mathcal{J}''$,
% \begin{align}
%|J \cap \{Mg_1 >C_22^{n_2-1}|G_1| \}| > \frac{1}{10}|J|, \nonumber\\
%|J \cap \{Mg_2 >C_22^{m_2-1}|G_2| \}| > \frac{1}{10}|J|. 
%\end{align}
\begin{align} \label{B_en_1}
&\text{energy}^{1,\infty}_{\mathcal{P}'}((\langle B^H_P(v_1,v_2), \vphi_P^H \rangle)_{P \in \mathcal{P}'}) \lesssim  C_1^{\frac{1}{p}+ \frac{1}{q} - \theta_1 - \theta_2} 2^{\tilde{n}(\frac{1}{p} - \theta_1)} 2^{\tilde{m}(\frac{1}{q} - \theta_2)} |V_1|^{\frac{1}{p}} |V_2|^{\frac{1}{q}}.
%& \text{energy}^{1,\infty}_{\mathcal{J}''}((\langle \tilde{B}^H_J, \vphi_J^H \rangle)_{J \in \mathcal{J}''}) \lesssim C_2^{\frac{1}{p_2}+ \frac{1}{q_2} - \zeta_1 - \zeta_2} 2^{n_2(\frac{1}{p_2} - \zeta_1)} 2^{m_2(\frac{1}{q_2} - \zeta_2)} |G_1|^{\frac{1}{p_2}} |G_2|^{\frac{1}{q_2}}.
%\text{size}_{K \in \mathcal{K}: K \cap S \neq \emptyset}((\langle f_1, \vphi_K \rangle)_{K})^{1-\theta_1}\text{size}_{K \in \mathcal{K}:K \cap S \neq \emptyset}((\langle f_2, \psi_K \rangle)_{K})^{1-\theta_2} \nonumber \\
%& \cdot |F_1|^{\theta_1}|F_2|^{\theta_2}
%\text{energy}_{\mathcal{I}'}((\langle B_I, \vphi_I \rangle)_{I \in \mathcal{I}'}) \lesssim & \text{size}_{K \in \mathcal{K}: K \cap S \neq \emptyset}((\langle f_1, \vphi_K \rangle)_{K})^{1-\theta_1}\text{size}_{K \in \mathcal{K}:K \cap S \neq \emptyset}((\langle f_2, \psi_K \rangle)_{K})^{1-\theta_2} \nonumber \\
%& \cdot |F_1|^{\theta_1}|F_2|^{\theta_2}
\end{align}
%\item
%\begin{align*}
%\text{energy}^{p,\infty} _{\mathcal{I}'}((\langle B_I, \vphi_I \rangle)_{I \in \mathcal{I}'}) \lesssim & \text{size}_{K \in \mathcal{K}: K \cap S \neq \emptyset}((\langle f_1, \vphi_K \rangle)_{K})^{1-\theta_1}\text{size}_{K \in \mathcal{K}:K \cap S \neq \emptyset}((\langle f_2, \psi_K \rangle)_{K})^{1-\theta_2} \nonumber \\
%& \cdot |F_1|^{\theta_1}|F_2|^{\theta_2}|S_2|^{\theta_3-\frac{1}{p'}}
%\end{align*}

%\begin{align*}
%\text{energy}^{p,\infty} _{\mathcal{J}'}((\langle \tilde{B_J}, \vphi_J \rangle)_{J \in \mathcal{J}'}) \lesssim & \text{size}_{L \in \mathcal{L}: L \cap S' \neq \emptyset}((\langle g_1, \vphi_K \rangle)_{L})^{1-\theta_1}\text{size}_{L \in \mathcal{L}:L \cap S' \neq \emptyset}((\langle g_2, \psi_L \rangle)_{L})^{1-\theta_2} \nonumber \\
%& \cdot |G_1|^{\theta_1}|G_2|^{\theta_2}|S'_2|^{\theta_3-\frac{1}{p'}}
%\end{align*}

\end{enumerate}
\end{proposition}

\begin{remark}
The condition that 
\begin{equation} \label{diff_exp}
\frac{1}{p} + \frac{1}{q} > 1
 \end{equation}
is required in the proof the proposition. Moreover, the energy estimates in Proposition \ref{B_en} are useful for the proof of the main theorems in the range of exponents specified as (\ref{diff_exp}). A simpler argument without the use of Proposition \ref{B_en} can be applied for the other case
$$
\frac{1}{p} + \frac{1}{q}  \leq 1.
$$
\end{remark}

\begin{remark}\label{loc_easy_haar}
%begin{enumerate}
%\item
%Proposition \ref{B_en} (i) will be proved directly. For Proposition \ref{B_en} (ii),
%\item
Thanks to the localization specified in Lemma \ref{localization_haar}, it suffices to prove that 
$$\text{energy}^t_{\mathcal{P}'}((\langle B^{\tilde{n},\tilde{m},0}_{\mathcal{Q},\text{lac}}(v_1,v_2), \vphi_P^H \rangle)_{P \in \mathcal{P}'}), \ \ \ \text{energy}^t_{\mathcal{P}'}((\langle B^{\tilde{n},\tilde{m},0}_{\mathcal{Q},\text{nonlac}}(v_1,v_2), \vphi_P^H \rangle)_{P \in \mathcal{P}'}),$$ 
and
$$\text{energy}^{1,\infty}_{\mathcal{P}'}((\langle B^{\tilde{n},\tilde{m},0}_{\mathcal{Q},\text{lac}}(v_1,v_2), \vphi_P^H \rangle)_{P \in \mathcal{P}'}), \ \ \ \text{energy}^{1,\infty}_{\mathcal{P}'}((\langle B^{\tilde{n},\tilde{m},0}_{\mathcal{Q},\text{nonlac}}(v_1,v_2), \vphi_P^H \rangle)_{P \in \mathcal{P}'}),$$ 
%$$\text{energy}^{1,\infty}_{\mathcal{I}'}((\langle B^{n_1,m_1}_0, \vphi_I^H \rangle)_{I \in \mathcal{I}'}), \ \ \text{energy}^{1,\infty}_{\mathcal{J}'}((\langle \tilde{B}^{n_2,m_2}_0, \vphi_J^H \rangle)_{J \in \mathcal{J}'}) $$ 
satisfy the same estimates on the right hand side of the inequalities in Proposition \ref{B_en}. Equivalently,
\begin{enumerate}[(i')]
\item
for any $0 \leq \theta_1,\theta_2 <1$ with $\theta_1 + \theta_2 = \frac{1}{t}$,
\begin{align} \label{prop_en_equiv}
%\sum_{l}2^{lp}\sup_{\mathbb{D}}\sum_{\substack{I \in \mathbb{D} \\ \frac{|\langle B_I, \vphi_I \rangle|}{|I|^{\frac{1}{2}}} > 2^{l} \\ I \in \mathcal{I}'}}|I|
& \text{energy}^{t} _{\mathcal{P}'}((\langle B^{\tilde{n},\tilde{m},0}_{\mathcal{Q},\text{lac}}(v_1,v_2), \vphi_P^H \rangle)_{P \in \mathcal{P}'}),  \text{energy}^{t} _{\mathcal{P}'}((\langle B^{\tilde{n},\tilde{m},0}_{\mathcal{Q},\text{nonlac}}(v_1,v_2), \vphi_P^H \rangle)_{P \in \mathcal{P}'}) \nonumber\\
\lesssim &  C^{\frac{1}{p}+ \frac{1}{q} - \theta_1 - \theta_2}2^{\tilde{n}(\frac{1}{p} - \theta_1)}2^{\tilde{m}(\frac{1}{q} - \theta_2)}|V_1|^{\frac{1}{p}}|V_2|^{\frac{1}{q}};
%\text{size}_{K \in \mathcal{K}: K \cap S \neq \emptyset}((\langle f_1, \vphi_K \rangle)_{K})^{1-\theta_1}\text{size}_{K \in \mathcal{K}:K \cap S \neq \emptyset}((\langle f_2, \psi_K \rangle)_{K})^{1-\theta_2} \nonumber \\
%& \cdot |F_1|^{\theta_1}|F_2|^{\theta_2}|S_2|^{\theta_3-\frac{1}{p'}}
%\text{energy}^{2} _{\mathcal{J}'}((\langle \tilde{B_J}, \vphi_J \rangle)_{J \in \mathcal{J}'}) \lesssim & \text{size}_{L \in \mathcal{L}: L \cap S' \neq \emptyset}((\langle g_1, \vphi_K \rangle)_{L})^{1-\theta_1}\text{size}_{L \in \mathcal{L}:L \cap S' \neq \emptyset}((\langle g_2, \psi_L \rangle)_{L})^{1-\theta_2} \nonumber \\
%& \cdot |G_1|^{\theta_1}|G_2|^{\theta_2}|S'_2|^{\theta_3-\frac{1}{p'}}
\end{align}
\item
for any $0 < \theta_1,\theta_2 <1$ with $\theta_1 + \theta_2 = 1$,
\begin{align}\label{prop_en_1_equiv}
& \text{energy}^{1,\infty} _{\mathcal{P}'}((\langle B^{\tilde{n},\tilde{m},0}_{\mathcal{Q},\text{lac}}(v_1,v_2), \vphi_P^H \rangle)_{P \in \mathcal{P}'}),  \text{energy}^{1,\infty} _{\mathcal{P}'}((\langle B^{\tilde{n},\tilde{m},0}_{\mathcal{Q},\text{nonlac}}(v_1,v_2), \vphi_P^H \rangle)_{P \in \mathcal{P}'}) \nonumber \\
\lesssim  & C^{\frac{1}{p}+ \frac{1}{q} - \theta_1 - \theta_2}2^{\tilde{n}(\frac{1}{p} - \theta_1)}2^{\tilde{m}(\frac{1}{q} - \theta_2)}|V_1|^{\frac{1}{p}}|V_2|^{\frac{1}{q}}.%& \text{energy}^{1,\infty}_{\mathcal{J}''}((\langle \tilde{B}^{n_2,m_2}_0, \vphi_J^H \rangle)_{J \in \mathcal{J}''}) \lesssim C_2^{\frac{1}{p_2}+ \frac{1}{q_2} - \zeta_1 - \zeta_2} 2^{n_2(\frac{1}{p_2} - \zeta_1)} 2^{m_2(\frac{1}{q_2} - \zeta_2)} |G_1|^{\frac{1}{p_2}} |G_2|^{\frac{1}{q_2}}.
\end{align}
\end{enumerate}
\end{remark}

%Proposition \ref{energy_classical} and Lemma \ref{B_loc_norm} below indeed yield the desired estimates.  Therefore, Proposition \ref{B_en} follows as commented in Remark \ref{loc_easy_haar}.
%The estimate (\ref{prop_en_1_equiv}) will be proved directly, which implies Proposition \ref{B_en} (ii). 
Due to Proposition \ref{energy_classical}, the proof of (\ref{prop_en_equiv}) and (\ref{prop_en_1_equiv}) and thus of Proposition \ref{B_en}  can be reduced to verifying Lemma \ref{B_loc_norm}.

\begin{lemma} %[Haar Model] 
\label{B_loc_norm}
Suppose that $V_1, V_2 \subseteq \mathbb{R}$ are sets of finite measure and $|v_i| \leq \chi_{V_i}$ for $i= 1,2$. %Suppose $\mathcal{Q}$ is a finite collection of dyadic intervals and $(\phi^i_Q)_{Q \in \mathcal{Q}}$ for $i = 1, 2$ and $(\phi^{3,H}_Q)_{Q \in \mathcal{Q}}$ are families of cutoff functions satisfying the conditions in Definition \ref{B^0_def}.
Let $B^{\tilde{n},\tilde{m},0}_{\mathcal{Q},\text{lac}}$ and $B^{\tilde{n},\tilde{m},0}_{\mathcal{Q},\text{nonlac}}$ denote the bilinear operators defined in (\ref{B^0_lac}) and (\ref{B^0_nonlac}) of Definition \ref{B^0_def}. Then for $t \geq 1$, 
\begin{align} \label{B^0_norm}
& \|B^{\tilde{n},\tilde{m},0}_{\mathcal{Q},\text{lac}}(v_1,v_2)\|_t, \|B^{\tilde{n},\tilde{m},0}_{\mathcal{Q},\text{nonlac}}(v_1,v_2)\|_t  \lesssim C_1^{\frac{1}{p}+ \frac{1}{q} - \theta_1 - \theta_2} 2^{\tilde{n}(\frac{1}{p} - \theta_1)} 2^{\tilde{m}(\frac{1}{q} - \theta_2)} |V_1|^{\frac{1}{p}} |V_2|^{\frac{1}{q}}. 
%& \|\tilde{B}_0^{n_2,m_2}(g_1,g_2)\|_{s} \lesssim C_2^{\frac{1}{p_2}+ \frac{1}{q_2} - \zeta_1 - \zeta_2} 2^{n_2(\frac{1}{p_2} - \zeta_1)} 2^{m_2(\frac{1}{q_2} - \zeta_2)} |G_1|^{\frac{1}{p_2}} |G_2|^{\frac{1}{q_2}}. \nonumber \\
\end{align}
where $0 \leq \theta_1,\theta_2 <1$ and $\theta_1 + \theta_2 = \frac{1}{t}$.
\end{lemma}
\vskip .15in
\subsection{Proof of Proposition \ref{energy_classical} $(2)$} \label{Proof_prop_energy_classical}
One observes that for each $n$, there exists a disjoint collection of intervals, denoted by $\mathbb{D}^{0}_n$ such that \begin{equation} 
\text{energy}^{t} _{\mathcal{I}'}((\langle f, \vphi_I \rangle)_{I \in \mathcal{I}'}) = \bigg(\sum_{n}2^{tn} \sum_{\substack{I \in \mathbb{D}_n^{0}\\ I \in \mathcal{I'}}}|I|\bigg)^{\frac{1}{t}}\label{energy_p}
\end{equation}\label{B_energy}
%Fix $n \in \mathbb{Z}$, the localization argument in the proof of (i) can be applied here as well and one can rewrite
where for any $I \in \mathbb{D}^0_n$,
\begin{equation} \label{energy_interval}
\frac{|\langle f, \vphi_I \rangle|}{|I|^{\frac{1}{2}}} > 2^{n}.
\end{equation}
%By the same reasoning in (i), 
Meanwhile for any $x \in I$,
\begin{equation*}
Mf(x) \geq \frac{|\langle f, \vphi_I \rangle|}{|I|^{\frac{1}{2}}},
\end{equation*}
which implies that 
$$I \subseteq \{Mf(x) > 2^{n}\}.$$
for any $I \in \mathcal{I}'$ satisfying (\ref{energy_interval}). Then by the disjointness of $\mathbb{D}^0_n$, one can estimate the energy as follows:
$$ 
\text{energy}^{t} _{\mathcal{I}'}((\langle f, \vphi_I \rangle)_{I \in \mathcal{I}'} \leq \big(\sum_{n}2^{tn }  |\{Mf(x) > 2^{n}\}|\big)^{\frac{1}{t}} \lesssim \|Mf\|_{t}.
$$
One can then apply the fact that the mapping property of maximal operator $M: L^{t} \rightarrow L^{t}$ for $t >1$ and derive
$$
\|Mf\|_{t} \lesssim \|f\|_{t}.
$$
\begin{remark}
To prove the same estimate for 
$$
\text{energy}^{t} _{\mathcal{I}'}((\langle f, \vphi^H_I \rangle)_{I \in \mathcal{I}'}), 
$$
one can replace $\vphi_I$ by $\vphi^H_I$ and the proof above still goes through. 
\end{remark}
%\end{enumerate}

\subsection{Proof of Lemma \ref{B_size}} \label{section_proof_B_size}

%\begin{proof}[Proof of Lemma \ref{B_size}]
%Without loss of generality, we will prove the first size estimate and the second follows from the same argument. 
One recalls the definition of 
$$
\text{size}_{\mathcal{P'}}((\langle B_{\mathcal{Q}, P}^{\#,H}(v_1,v_2), \vphi^H_P \rangle)_{P \in \mathcal{P}'} = \frac{|\langle B^{\#,H}_{\mathcal{Q}, P_0}(v_1,v_2),\vphi_{P_0}^H \rangle|}{|P_0|^{\frac{1}{2}}}
$$
for some $P_0 \in \mathcal{P}'$ with the property that $P_0 \cap S \neq \emptyset$ by the assumption. Then
\begin{align} \label{B^no_local_size}
\frac{|\langle B^{\#,H}_{\mathcal{Q},P_0}(v_1,v_2),\vphi_{P_0}^H \rangle|}{|P_0|^{\frac{1}{2}}} \leq & \frac{1}{|P_0|}\sum_{Q:|Q|\sim 2^{\#}|P_0|}\frac{1}{|Q|^{\frac{1}{2}}}|\langle v_1, \phi_Q^1 \rangle| |\langle v_2, \phi_Q^2 \rangle| |\langle |P_0|^{\frac{1}{2}}\vphi^H_{P_0},\phi_Q^{3,H} \rangle| \nonumber \\
= & \frac{1}{|P_0|}\sum_{Q:|Q|\sim 2^{\#}|P_0|}\frac{|\langle v_1, \phi_Q^1 \rangle|}{|Q|^{\frac{1}{2}}} \frac{|\langle v_2, \phi_Q^2 \rangle|}{|Q|^{\frac{1}{2}}} |\langle |P_0|^{\frac{1}{2}}\vphi_{P_0}^H, |Q|^{\frac{1}{2}}\phi_Q^{3,H}  \rangle|.
\end{align}
%In the Haar model, $|K|>|I|$ implies that $K \supseteq I$ and $K \cap S \neq \emptyset$. 
Since $ \vphi_{P_0}^H$ and $\phi_Q^{3,H}$ are compactly supported on $P_0$ and $Q$ respectively with $|P_0| \leq |Q|$, one has 
$$
\langle |P_0|^{\frac{1}{2}}\vphi_{P_0}^H, |Q|^{\frac{1}{2}}\phi_Q^{3,H}  \rangle \neq 0 
$$
if and only if 
$$
P_0 \subseteq Q.
$$
By the hypothesis that $P_0 \cap S \neq \emptyset$, one derives that $Q \cap S\neq \emptyset$ and
\begin{align*}
%\frac{|\langle B^{\#}_{P_0}(v_1,v_2),\vphi_{P_0}^1 \rangle|}{|P_0|^{\frac{1}{2}}}
(\ref{B^no_local_size}) \leq &\frac{1}{|P_0|} \sup_{Q \cap S \neq \emptyset}\frac{|\langle v_1, \phi_Q^1 \rangle|}{|Q|^{\frac{1}{2}}} \sup_{Q \cap S \neq \emptyset}\frac{|\langle v_2, \phi_Q^2 \rangle|}{|Q|^{\frac{1}{2}}}\sum_{Q:|Q|\sim 2^{\#_1}|P_0|}|\langle |P_0|^{\frac{1}{2}}\vphi^H_{P_0}, |Q|^{\frac{1}{2}}\phi_Q^{3,H}  \rangle| \nonumber \\
\lesssim & \frac{1}{|P_0|} \sup_{Q \cap S \neq \emptyset}\frac{|\langle v_1, \phi_Q^1 \rangle|}{|Q|^{\frac{1}{2}}} \sup_{Q \cap S\neq \emptyset}\frac{|\langle v_2, \phi_Q^2 \rangle|}{|Q|^{\frac{1}{2}}} \cdot |P_0|,
\end{align*}
where the last inequality holds trivially given that $|P_0|^{\frac{1}{2}}\vphi^H_{P_0}$ is indeed an indicator function of $P_0$ and $|Q|^{\frac{1}{2}}\phi_Q^{3,H}$ is majorized by the indicator function of $Q$. %In the general case when $\vphi_{P_0}$ is a bump function adapted to $P_0$, the inequality follows from the fact that fix an interval $P_0$, $\{K: |K| \sim 2^{\#_1}|I|\}$ is a disjoint collection of intervals. 
This completes the proof of the proposition.
%\end{proof}
\subsection{Proof of Lemma \ref{localization_haar}} \label{section_proof_localization_haar}
%Suppose that for any $I \in \mathcal{I}'$, $I \cap \mathcal{U}_{n_1,m_1} \neq \emptyset$.
%By definition of energy, given that $\mathcal{I}$ is a finite collection of intervals, there exists $n \in \mathbb{Z}$ and a disjoint collection of dyadic intervals $\mathbb{D}^0_{n}$ such that
%\begin{equation} 
%\text{energy} _{\mathcal{I}'}((\langle B^H_I, \vphi^H_I \rangle)_{I \in \mathcal{I}'}) := 2^{n} \sum_{\substack{I \in \mathbb{D}^0_{n}\\ I \in \mathcal{I'}}}|I| \label{energy}
%\end{equation}\label{B_energy}
%where 
%\begin{equation} \label{st_interval}
%\frac{|\langle B^H_I, \vphi^H_I \rangle|}{|I|^{\frac{1}{2}}} > 2^{n}.
%\end{equation}
%\newline 
%\noindent
According to the definition of energy (Definition \ref{def_size_energy_haar}), it suffices to prove that for any $P \in \mathcal{P}'$, %$n \in \mathbb{Z}$ and any disjoint collection of dyadic intervals $\mathbb{D}^0_{n}$ such that for any $I \in \mathbb{D}^0_{n}$,  
%\begin{equation} \label{st_interval}
%\frac{|\langle B^H_I, \vphi^H_I \rangle|}{|I|^{\frac{1}{2}}} > 2^{n},
%\end{equation}\tilde{m}
\begin{equation} \label{goal_localization_en}
|\langle B^H_{\mathcal{Q},P}(v_1,v_2), \vphi^H_P \rangle| \leq
\begin{cases}
 |\langle B^{\tilde{n},\tilde{m},0}_{\mathcal{Q},\text{lac}}(v_1,v_2), \vphi^H_{P} \rangle|  \ \ \ \ \ \ \ \text{ if } (\phi_Q^{3,H})_{Q \in \mathcal{Q}} \text{ is a lacunary family}\\
 \\
 |\langle B^{\tilde{n},\tilde{m},0}_{\mathcal{Q},\text{nonac}}(v_1,v_2), \vphi^H_{P} \rangle|  \ \ \ \ \text{ if } (\phi_Q^{3,H})_{Q \in \mathcal{Q}} \text{ is a nonlacunary family}, 
 \end{cases}
\end{equation}
where the bump functions in $ B^{\tilde{n},\tilde{m},0}_{\mathcal{Q}, \text{lac}}$ or $B^{\tilde{n},\tilde{m},0}_{\mathcal{Q}, \text{nonlac}}$ are the same as the ones in $B^H_{\mathcal{Q},P}$.
\vskip .15in
\noindent
\textbf{Case I. $(\phi^3_K)_K$ is lacunary. }
One recalls that in the Haar model,
$$
\langle B^H_{\mathcal{Q},P}(v_1,v_2), \vphi_P^H \rangle := \frac{1}{|P|^{\frac{1}{2}}} \sum_{\substack{Q \in \mathcal{Q} \\ |Q| \geq |P|}} \frac{1}{|Q|^{\frac{1}{2}}} \langle v_1, \phi_Q^1\rangle \langle v_2, \phi_Q^2 \rangle \langle \vphi^H_P,\psi_Q^{3,H} \rangle
$$
where $\vphi^H_P$ is an $L^2$-normalized indicator function of $P$ and $\psi_Q^{3,H}$ is a Haar wavelet on $Q$ with $P$ and $Q$ being dyadic intervals. It is not difficult to observe that 
\begin{equation} \label{haar_biest_cond}
\langle \vphi^H_P,\psi_Q^{3,H} \rangle \neq 0  \iff
Q \supseteq P.
\end{equation}
Given $P \cap \mathcal{U}_{\tilde{n},\tilde{m}} \neq \emptyset$, one can deduce that $Q \cap \mathcal{U}_{\tilde{n},\tilde{m}} \neq \emptyset$. As a consequence,
\begin{align} \label{haar_biest}
\langle B^H_{\mathcal{Q}, P}(v_1,v_2), \vphi^H_P \rangle =& \sum_{\substack{Q \in \mathcal{Q} \\ Q \cap \mathcal{U}_{\tilde{n},\tilde{m}} \neq \emptyset \\ |Q| \geq |P|}} \frac{1}{|Q|^{\frac{1}{2}}} \langle v_1, \phi_Q^1\rangle \langle v_2, \phi_Q^2 \rangle \langle \vphi^H_P,\psi_Q^{3,H} \rangle \nonumber \\
= &\sum_{\substack{Q \in \mathcal{Q} \\ Q \cap \mathcal{U}_{\tilde{n},\tilde{m}} \neq \emptyset}} \frac{1}{|Q|^{\frac{1}{2}}} \langle v_1, \phi_Q^1\rangle \langle v_2, \phi_Q^2 \rangle \langle \vphi^H_P,\psi_Q^{3,H} \rangle.
\end{align}
%Let 
%$$
%B^{\tilde{n},\tilde{m}}_0(v_1, v_2)(x) :=  \sum_{\substack{K \in \mathcal{K} \\ K \cap \mathcal{U}_{\tilde{n},\tilde{m}} \neq \emptyset}} \frac{1}{|K|^{\frac{1}{2}}} \langle v_1, \phi_K^1\rangle \langle v_2, \phi_K^2 \rangle \psi_K^{3,H}(x).
%$$
By the definition of $B^{\tilde{n},\tilde{m},0}_{\mathcal{Q},\text{lac}}$ (\ref{B^0_lac}) with the choice the bump functions to be the same as the ones in $B^H_{\mathcal{Q},P}$, one can conclude that
\begin{equation*} 
\langle B^H_P(v_1,v_2), \vphi^H_P \rangle = \langle B^{\tilde{n},\tilde{m},0}_{\mathcal{Q},\text{lac}}(v_1,v_2), \vphi^H_P \rangle,
\end{equation*}
which is the desired estimate in Case I highlighted in (\ref{goal_localization_en}).
\begin{remark}[Biest trick]\label{biest_trick_rmk}
In the Haar model, equation (\ref{haar_biest}) trivially holds due to (\ref{haar_biest_cond}). Such technique of replacing the operator defined in terms of $P$ (namely $B_P^H$) by another operator independent of $P$ (namely $B^{\tilde{n},\tilde{m}}_0$) is called \textbf{biest trick} which allows neat energy estimates for
\begin{align*}
& \text{energy}_{\mathcal{P}'}^t((\langle B^{\tilde{n},\tilde{m},0}_{\mathcal{Q},\text{lac}}(v_1,v_2), \vphi_P^H \rangle)_{P \in \mathcal{P}'}) \text{ and } \text{energy}^{1,\infty}_{\mathcal{P}'}((\langle B^{\tilde{n},\tilde{m},0}_{\mathcal{Q},\text{lac}}(v_1,v_2), \vphi_P^H \rangle)_{P \in \mathcal{P}'})
\end{align*} 
and yields local energy estimates (\ref{B_en_t}) and (\ref{B_en_1}) described in Proposition \ref{B_en}.
\end{remark}
\noindent
\textbf{Case II:  $(\phi^3_Q)_Q$ is non-lacunary. }
Since $\vphi^{3,H}_Q$ and $\vphi_P^H$ are $L^2$-normalized indicator functions of $Q$ and $P$ respectively, $|Q| \geq |P|$ implies that $Q \supseteq P$. As a result, $Q \cap \mathcal{U}_{\tilde{n},\tilde{m}} \neq \emptyset$ given $P \cap \mathcal{U}_{\tilde{n},\tilde{m}} \neq \emptyset$. Then
%$$
%\frac{|\langle B_P^H, \vphi_P^H \rangle|}{|P|^{\frac{1}{2}}} = \frac{1}{|P|^{\frac{1}{2}}} \bigg| \sum_{\substack{Q \in \mathcal{Q} \\ Q \supseteq P}} \frac{1}{|Q|^{\frac{1}{2}}} \langle v_1, \phi_Q^1\rangle \langle v_2, \phi_Q^2 \rangle \langle \vphi_P^H,\vphi_Q^3 \rangle \bigg|
%$$
%where $ \tilde{\chi}_P := \frac{\vphi_P}{|P|^{\frac{1}{2}}}$ is a $L^{\infty}$-normalized bump function.
%By the assumption that $P \cap S \neq \emptyset$ for any $P \in \mathcal{P}'$, one can derive that $Q \cap S \neq \emptyset$ given $Q \supseteq P$. Therefore, one can rewrite 
\begin{align*}
\frac{|\langle B_{\mathcal{Q},P}^H(v_1,v_2), \vphi_P^H \rangle|}{|P|^{\frac{1}{2}}} = & \frac{1}{|P|^{\frac{1}{2}}} \bigg|\sum_{\substack{Q \in \mathcal{Q} \\ Q \supseteq P \\ Q \cap \mathcal{U}_{\tilde{n},\tilde{m}} \neq \emptyset}} \frac{1}{|Q|^{\frac{1}{2}}} \langle v_1, \psi_Q^1\rangle \langle v_2, \psi_Q^2 \rangle \langle \vphi^{H}_P,\vphi_Q^{3,H} \rangle \bigg| \nonumber \\
\leq & \frac{1}{|P|^{\frac{1}{2}}} \sum_{\substack{Q \in \mathcal{Q} \\ Q \supseteq P \\ Q \cap \mathcal{U}_{\tilde{n},\tilde{m}} \neq \emptyset}} \frac{1}{|Q|^{\frac{1}{2}}} |\langle v_1, \psi_Q^1\rangle| |\langle v_2, \psi_Q^2 \rangle| \langle \vphi^{H}_P,|\vphi_Q^{3,H}| \rangle,
\end{align*}
where the last inequality follows from the fact that $\vphi_P^H$ is an indicator function and thus non-negative. One can drop the condition $Q \supseteq P$ in the sum and bound the above expression by
$$
\frac{|\langle B_{\mathcal{Q}, P}^H(v_1,v_2), \vphi_P^H \rangle|}{|P|^{\frac{1}{2}}} \leq \frac{1}{|P|} \sum_{\substack{Q \in \mathcal{Q}\\ Q \cap \mathcal{U}_{\tilde{n},\tilde{m}} \neq \emptyset}} \frac{1}{|Q|^{\frac{1}{2}}} |\langle v_1, \psi_Q^1\rangle| |\langle v_2, \psi_Q^2 \rangle| \langle  \vphi_P^H,|\vphi_Q^{3,H}| \rangle.
$$
%One can again apply the localization using the assumption that $\displaystyle \bigcup_{P \in \mathcal{P}'} P \subseteq S_2$ to deduce that
%$$
%\frac{|\langle B_P, \vphi_P \rangle|}{|P|^{\frac{1}{2}}} \leq \frac{1}{|P|} \sum_{\substack{Q \in \mathcal{Q}\\ Q \cap S \neq \emptyset}} \frac{1}{|Q|^{\frac{1}{2}}} |\langle v_1, \phi_Q^1\rangle| |\langle v_2, \phi_Q^2 \rangle| \langle  |\tilde{\chi}_P|\chi_{S_2},|\vphi_Q^3| \rangle 
%$$
One recalls the definition of $B_{\mathcal{Q},\text{nonlac}}^{\tilde{n},\tilde{m},0}$ (\ref{B^0_nonlac}) with the choice the bump functions to be the same as the ones in $B^H_{\mathcal{Q},P}$ and deduces 
%$$
%B^{\tilde{n},\tilde{m}}_0(x) := \displaystyle \sum_{\substack{Q \in \mathcal{Q}\\ Q \cap \mathcal{U}_{\tilde{n},\tilde{m}} \neq \emptyset}} \frac{1}{|Q|^{\frac{1}{2}}} |\langle v_1, \phi_Q^1\rangle| |\langle v_2, \phi_Q^2 \rangle| |\vphi_Q^3|(x).$$
%The discussion above yields that
$$
\frac{|\langle B_{\mathcal{Q}, P}^H(v_1,v_2), \vphi_P^H \rangle|}{|P|^{\frac{1}{2}}} \leq \frac{|\langle B_{\mathcal{Q},\text{nonlac}}^{\tilde{n},\tilde{m},0}(v_1,v_2), \vphi_P^H \rangle|}{|P|^{\frac{1}{2}}}
$$
which agrees with the estimate described in (\ref{goal_localization_en}). 
%yields
%\begin{align*}
%\text{energy}^{1,\infty}_{\mathcal{P}'}(\langle B_P^H, \vphi_P^H \rangle_{P \in \mathcal{P'}}) \leq & \text{energy}^{1,\infty}_{\mathcal{P}'}(\langle B^{\tilde{n},\tilde{m}}_0, \vphi_P^H \rangle_{P \in \mathcal{P'}}) \nonumber\\
%\text{energy}^t_{\mathcal{P}'}(\langle B_P^H, \vphi_P^H \rangle_{P \in \mathcal{P'}}) \leq & \text{energy}^t_{\mathcal{P}'}(\langle B^{\tilde{n},\tilde{m}}_0, \vphi_P^H \rangle_{P \in \mathcal{P'}})
%\end{align*}
%for $t > 1$. %The same reasoning can be applied to $\text{energy}^{1,\infty}_{\mathcal{J}'}(\langle \tilde{B}_J^H, \vphi_J^H \rangle_{J \in \mathcal{J'}})$ and $\text{energy}^s_{\mathcal{J}'}((\langle \tilde{B}^H_J, \vphi^H_J \rangle)_{J \in \mathcal{J}'})$ for $s > 1$. %prove that for $s > 1$,
%$$
%\text{energy}^s_{\mathcal{J}'}((\langle \tilde{B}^H_J, \vphi^H_J \rangle)_{J \in \mathcal{J}'}) \leq \text{energy}^t_{\mathcal{J}'}((\langle \tilde{B}^{n_2,m_2}_0, \vphi_J^H \rangle)_{J \in \mathcal{J}'}).
%$$
This completes the proof of the lemma.
\begin{remark}
$B^{\tilde{n},\tilde{m},0}_{\mathcal{Q},\text{lac}}$ and $B^{\tilde{n},\tilde{m},0}_{\mathcal{Q},\text{nonlac}}$ are perfectly localized in the sense that the dyadic intervals (that matter) intersect with $\mathcal{U}_{\tilde{n},\tilde{m}}$ nontrivially given that $P \cap \mathcal{U}_{\tilde{n},\tilde{m}} \neq \emptyset$. As will be seen from the proof of Lemma \ref{B_loc_norm}, such localization is essential in deriving desired estimates. In the general Fourier case, more efforts are needed to create similar localizations as will be discussed in Section \ref{section_fourier}.
\end{remark}

\subsection{Proof of Lemma \ref{B_loc_norm}} \label{section_proof_B_loc_norm}
The estimates described in Lemma \ref{B_loc_norm} can be obtained by a very similar argument for proving the boundedness of one-parameter paraproducts discussed in Section 2 of \cite{cw}. We would include the customized proof here since the argument depends on a one-dimensional stopping-time decomposition which is also an important ingredient for our tensor-type stopping-time decompositions that will be introduced in later sections.
\subsubsection{One-dimensional stopping-time decomposition - maximal intervals} \label{section_size_energy_one_dim_st_maximal}
For a sequence $$(a_Q)_{Q \in \mathcal{Q}} := (v, \vphi_Q)_{Q \in \mathcal{Q}} \ \ \text{or} \ \ (v,\vphi^H_Q)_{Q \in \mathcal{Q}},$$
we perform the the following stopping-time decomposition. 
Given finiteness of the collection of dyadic intervals $\mathcal{Q}$, there exists some $K_1 \in \mathbb{Z}$ such that 
$$\frac{|a_Q|}{|Q|^{\frac{1}{2}}} \leq C_1 2^{K_1} \text{energy}_{\mathcal{Q}}((a_Q)_{Q \in \mathcal{Q}}).$$ We can pick the largest interval $Q_{\text{max}}$ such that 
$$\frac{|a_{Q_{\text{max}}} |}{|Q_{\text{max}}|^{\frac{1}{2}}} > C_1 2^{K_1-1}\text{energy}_{\mathcal{Q}}((a_Q)_{Q \in \mathcal{Q}}).$$
Then we define a tree
$$U:= \{Q \in \mathcal{Q}: Q \subseteq Q_{\text{max}}\},$$
and let 
$$Q_U := Q_{\text{max}},$$ 
usually called \textit{tree-top}. 
Now we look at $\mathcal{Q} \setminus U$ and repeat the above step to choose maximal intervals and collect their subintervals in their corresponding sets. Since $\mathcal{Q}$ is finite, the process will eventually end. We then collect all $U$'s in a set $\mathbb{U}_{K_1-1}$. Next we repeat the above algorithm to $\displaystyle \mathcal{Q} \setminus \bigcup_{U \in \mathbb{U}_{K_1-1}} U$. We thus obtain a decomposition $$\displaystyle \mathcal{Q} = \bigcup_{k}\bigcup_{U \in \mathbb{U}_{k}}U.$$ If, otherwise, the sequence is formed in terms of bump or cutoff functions in a lacunary family, namely
$$
(a_Q)_{Q \in \mathcal{Q}} :=  (v, \psi_Q)_{Q \in \mathcal{Q}} \ \ \text{or} \ \ (v,\psi^H_Q)_{Q \in \mathcal{Q}},
$$
then the same procedure can be performed to 
$$
\frac{1}{|Q|} \left\Vert \bigg(\sum_{\substack{Q' \in \mathcal{Q} \\ Q' \subseteq Q}}\frac{|a_{Q'}|^2 }{|Q'|}\chi_{Q'}\bigg)^{\frac{1}{2}}\right\Vert_{1,\infty}.
$$
%One simple observation is that the above procedure can be applied to general sequences indexed by dyadic intervals. 
\vskip .15in
%One can thus apply the same algorithm to $\mathcal{J} := \{J: I \times J \in \mathcal{R}\}$. We denote the decomposition as $\displaystyle \mathcal{J} = \bigcup_{l_2}\bigcup_{S \in \mathbb{S}_{l_2}}S$ with respect to the sequence $\big(\frac{|\langle B_{J}(g_1,g_2), \vphi^1_J \rangle|}{|J|^{\frac{1}{2}}}\big)_{J \in \mathcal{J}}$, where $S$ represents for a tree with the tree-top.
The next proposition summarizes the information from the stopping-time decomposition and the details of the proof are included in Section 2 of \cite{cw}.
\begin{proposition}\label{st_prop}
Suppose $\displaystyle \mathcal{Q} = \bigcup_{k}\bigcup_{U \in \mathbb{U}_{k}}U$ is a decomposition obtained from the stopping-time algorithm specified above, then for any $k \in \mathbb{Z}$, one has
\begin{equation*}
\displaystyle 2^{k-1}\text{energy}_{\mathcal{Q}}((a_Q)_{Q \in \mathcal{Q}}) \leq \text{size}_{\bigcup_{U \in \mathbb{U}_k}U}((a_Q)_{Q \in \mathcal{Q}}) \leq \min(2^{k}\text{energy}_{\mathcal{Q}}((a_Q)_{Q \in \mathcal{Q}}),\text{size}_{\mathcal{Q}}((a_Q)_{Q \in \mathcal{Q}})).
\end{equation*}
In addition,
$$
\sum_{U \in \mathbb{U}_k} |Q_{U}| \lesssim 2^{-k}.
$$
\end{proposition}

The next lemma follows from the stopping-time decomposition, Proposition \ref{st_prop} and Proposition \ref{JN}, whose proof is discussed carefully in Section 2.9 of \cite{cw}. It plays an important role in proving Lemma \ref{B_loc_norm} as can be seen in Section \ref{section_pf_B_loc_norm}.
Suppose that $\mathcal{Q}$ is a finite collection of dyadic intervals and we would like to estimate
\begin{equation} \label{1-parameter-paraproduct}
\sum_{Q \in \mathcal{Q}}\frac{1}{|Q|^{\frac{1}{2}}} a_Q^1 a_Q^2 a_Q^3
\end{equation}
where for $1 \leq i \leq 3$,
\begin{align*}
a_Q^i := \langle v_i,\phi_Q^i \rangle \ \ \text{or} \ \   \langle v_i,\phi_Q^{i,H} \rangle 
\end{align*}
and at least two of the three families of $L^2$-normalized bump or cutoff functions are lacunary. 
\begin{lemma} \label{s-e}
 The trilinear form (\ref{1-parameter-paraproduct}) can be estimated by
$$
\bigg|\sum_{Q \in \mathcal{Q}}\frac{1}{|Q|^{\frac{1}{2}}}a_Q^1 a_Q^2 a_Q^3 \bigg|  \lesssim \prod_{i=1}^3 \text{size}_{\mathcal{Q}} \big((a_Q^i)_{Q \in \mathcal{Q}} \big)^{1-\theta_i}\text{energy}^{1,\infty}_{\mathcal{Q}}\big((a_Q^i)_{Q \in \mathcal{Q}} \big)^{\theta_i},
$$
for any $0 \leq \theta_1, \theta_2, \theta_3 < 1$ and $\theta_1 + \theta_2 + \theta_3 = 1$. The implicit constant depends on $\theta_1,\theta_2$ and $\theta_3$. 
\end{lemma}

\subsubsection{Proof of Lemma \ref{B_loc_norm}} \label{section_pf_B_loc_norm}
We will focus on proving (\ref{B^0_norm}) for $\|B_{\mathcal{Q}, \text{lac}}^{\tilde{n},\tilde{m},0}(v_1,v_2)\|_t$ for $t \geq 1$ and the estimate for $\|B_{\mathcal{Q}, \text{nonlac}}^{\tilde{n},\tilde{m},0}(v_1,v_2)\|_t$ follows from the exactly same argument.
\begin{enumerate} 
\item
\textbf{Estimate of $ \|B_{\mathcal{Q}, \text{lac}}^{\tilde{n},\tilde{m},0}(v_1,v_2)\|_1$.} 
%With the abuse of notations, $\|B_0^{n_1,m_1}\|_1$ represents for different functions in Case $I$ and $II$. Nevertheless, they can be estimated by the same argument. 
For any $\eta \in L^{\infty}$ one has
\begin{align*}
|\langle B_{\mathcal{Q}.\text{lac}}^{\tilde{n},\tilde{m},0}(v_1,v_2),\eta \rangle |\leq & \sum_{\substack{Q \in \mathcal{Q} \\ Q \cap \mathcal{U}_{\tilde{n},\tilde{m}} \neq \emptyset}} \frac{1}{|Q|^{\frac{1}{2}}} |\langle v_1, \phi_Q^1\rangle| |\langle v_2, \phi_Q^2 \rangle| |\langle \eta, \psi_Q^{3,H} \rangle |.
\end{align*}
%where 
%\begin{equation*}
%\phi^{3}_{Q}:=\begin{cases}
%\psi^{3,H}_Q \quad \quad \ \ \text{in Case}\ \ I\\
%|\vphi^{3,H}_{Q}| \quad \quad \text{in Case} \ \ II.
%\end{cases}
%\end{equation*}
Let $\mathcal{Q}'$ denote the sub-collection
$$\mathcal{Q'}:= \{Q \in \mathcal{Q}: Q \cap \mathcal{U}_{\tilde{n},\tilde{m}} \neq \emptyset \}.$$
Then, one can apply Lemma \ref{s-e} to obtain
\begin{align*}
& |\langle B_{\mathcal{Q},\text{lac}}^{\tilde{n},\tilde{m},0}(v_1.v_2), \eta \rangle | \nonumber \\
\lesssim & \text{\ \ size}_{\mathcal{Q}'} ((\langle v_1, \phi^1_Q \rangle)_{Q \in \mathcal{Q}'})^{1-\theta_1}\text{size}_{\mathcal{Q}'}((\langle v_2, \phi^2_Q \rangle)_{Q \in \mathcal{Q}'})^{1-\theta_2} \text{size}_{\mathcal{Q}'}((\langle \eta, \psi^{3,H}_Q \rangle)_{Q \in \mathcal{Q}'})^{1-\theta_3} \nonumber \\
& \text{\ \ energy} _{\mathcal{Q}'}((\langle v_1, \phi^1_Q\rangle)_{Q \in \mathcal{Q}'})^{\theta_1}\text{energy} _{\mathcal{Q}'}((\langle v_2, \phi^2_Q\rangle)_{Q \in \mathcal{Q}'})^{\theta_2} \text{energy} _{\mathcal{Q}'}((\langle \eta, \psi^{3,H}_Q\rangle)_{Q \in \mathcal{Q}'})^{\theta_3},
\end{align*}
for any $0 \leq \theta_1,\theta_2, \theta_3 <1$ with $\theta_1 + \theta_2 + \theta_3 = 1$. By Proposition \ref{size} and the fact that $Q \cap \mathcal{U}_{\tilde{n},\tilde{m}} \neq \emptyset$ for any $Q \in \mathcal{Q}'$, one deduces that
\begin{align} 
& \text{size}_{\mathcal{Q}'}((\langle v_1, \phi^1_Q \rangle)_{Q \in \mathcal{Q}'}) \lesssim \min(2^{\tilde{n}}|V_1|,1) \leq  (2^{\tilde{n}}|V_1|)^{\alpha}, \label{f_size1} \\
& \text{size}_{\mathcal{Q}'}((\langle v_2, \phi^2_Q \rangle)_{Q \in \mathcal{Q}'}) \lesssim \min( 2^{\tilde{m}}|V_2|,1) \leq ( 2^{\tilde{m}}|V_2|)^{\beta}. \label{f_size}
\end{align}
for any $0 \leq \alpha, \beta \leq 1$.
One also recalls that $\eta \in L^{\infty}$, which gives 
\begin{equation} \label{inf_size}
\text{size}_{Q \in \mathcal{Q}}((\langle \eta, \psi^{3,H}_Q \rangle)_{Q \in \mathcal{Q}'}) \lesssim 1.
\end{equation}
\begin{comment}
\begin{align*}
& |\langle B_0^{\tilde{n},\tilde{m}}, \eta \rangle | \nonumber \\
\lesssim & \text{\ \ size}_{Q \in \mathcal{Q}: Q \cap \mathcal{U}_{\tilde{n},\tilde{m}} \neq \emptyset}((\langle v_1, \phi^1_Q \rangle)_{Q})^{1-\theta_1}\text{size}_{Q \in \mathcal{Q}:Q \cap \mathcal{U}_{\tilde{n},\tilde{m}} \neq \emptyset}((\langle v_2, \phi^2_Q \rangle)_{Q})^{1-\theta_2} \text{size}_{Q \in \mathcal{Q}}((\langle \eta, \phi^3_Q \rangle)_{Q})^{1-\theta_3} \nonumber \\
& \text{\ \ energy} _{\mathcal{Q}}((\langle v_1, \phi^1_Q\rangle)_{Q})^{\theta_1}\text{energy} _{\mathcal{Q}}((\langle v_2, \phi^2_Q\rangle)_{Q})^{\theta_2} \text{energy} _{\mathcal{Q}}((\langle \eta, \phi^3_Q\rangle)_{Q})^{\theta_3},
\end{align*}
for any $0 \leq \theta_1,\theta_2, \theta_3 <1$ with $\theta_1 + \theta_2 + \theta_3 = 1$. By applying Proposition \ref{size}, one deduces that
$$
\text{size}_{Q \in \mathcal{Q}}((\langle \eta, \phi^3_Q \rangle)_{Q}) \lesssim 1.
$$
 and using the fact that $\eta \in L^{\infty}$, one has
$$
\text{size}_{Q \in \mathcal{Q}}((\langle \eta, \phi^3_Q \rangle)_{Q}) \lesssim 1.
$$
\end{comment}
%Moreover, Proposition \ref{energy_classical} implies that
%\begin{align}
%\text{energy}_{\mathcal{Q}'}((\langle v_1, \phi^1_Q\rangle)_{Q}) \leq \text{energy}_{\mathcal{Q}}((\langle v_1, \phi^1_Q\rangle)_{Q}) \leq 
%\end{align}
By choosing $\theta_3 = 0$ and combining the estimates (\ref{f_size1}), (\ref{f_size}) and (\ref{inf_size}) with the energy estimates described in Proposition \ref{energy_classical}, one obtains 
\begin{align}
|\langle B_{\mathcal{Q},\text{lac}}^{\tilde{n},\tilde{m},0}(v_1,v_2), \eta \rangle |\lesssim & (C_12^{\tilde{n}}|V_1|)^{\alpha(1-\theta_1)} (C_1 2^{\tilde{m}}|V_2|)^{\beta(1-\theta_2)}|V_1|^{\theta_1}|V_2|^{\theta_2} \|\eta\|_{L^{\infty}} \nonumber \\
= & C_1^{\alpha(1-\theta_1)+ \beta(1-\theta_2)}2^{\tilde{n}\alpha(1-\theta_1)}2^{\tilde{m}\beta(1-\theta_2)}|V_1|^{\alpha(1-\theta_1)+\theta_1}|V_2|^{\beta(1-\theta_2)+\theta_2} \|\eta\|_{L^{\infty}},
\end{align}
where $0 < \theta_1,\theta_2 < 1$ with $\theta_1 + \theta_2 = 1$ and $ 0 \leq \alpha, \beta \leq 1$.
Therefore, one can conclude that
\begin{align*}
& \|B_{\mathcal{Q},\text{lac}}^{\tilde{n},\tilde{m},0}(v_1,v_2)\|_1 \lesssim C_1^{\alpha(1-\theta_1)+ \beta(1-\theta_2)}2^{\tilde{n}\alpha(1-\theta_1)}2^{\tilde{m}\beta(1-\theta_2)}|V_1|^{\alpha(1-\theta_1)+\theta_1}|V_2|^{\beta(1-\theta_2)+\theta_2}.
%\text{\ \ size}_{K \in \mathcal{K}: K \cap S \neq \emptyset}((\langle v_1, \phi^1_K \rangle)_{K})^{1-\theta_1}\text{size}_{K \in \mathcal{K}:K \cap S \neq \emptyset}((\langle v_2, \phi^2_K \rangle)_{K})^{1-\theta_2} \nonumber \\
%& \cdot |V_1|^{\theta_1}|V_2|^{\theta_2}|S_2|^{\theta_3}
\end{align*}
By choosing $\alpha(1-\theta_1)+\theta_1 = \frac{1}{p}$ and $\beta(1-\theta_2)+\theta_2 = \frac{1}{q}$ which is possible given $\frac{1}{p} + \frac{1}{q} > 1$, one obtains the desired result.
%for some $0 \leq \theta_1,\theta_2, \theta_3 <1$ with $\theta_1 + \theta_2 + \theta_3 = 1$.
\vskip 0.25 in
\item %One first observes that 
%$$
%2^{n} \sum_{\substack{I \in \mathbb{D}_{n}\\ I \in \mathcal{I'}}}|I| \leq \big(\sum_{n}2^{2n} \sum_{\substack{I \in \mathbb{D}_{n}\\ I \in \mathcal{I'}}}|I|\big)^{\frac{1}{2}}$$
%for any collection of intervals $\mathbb{D}_n$ satisyfying
%\begin{enumerate}
%\item $\mathbb{D}_n$ is a disjoint collection of intervals
%\item For any $I \in \mathbb{D}_n$,
%$$
%\frac{|\langle B_I, \vphi_I \rangle|}{|I|^{\frac{1}{2}}} > 2^{n}
%$$
%\end{enumerate}

%By the definition of energies, one then has
%$$
%\text{energy}^{p,\infty} _{\mathcal{I}'}((\langle B_I, \vphi_I \rangle)_{I \in \mathcal{I}'}) \leq \text{energy}^{p} _{\mathcal{I}'}((\langle B_I, \vphi_I \rangle)_{I \in \mathcal{I}'})
%$$
%and it suffices to show that $\text{energy}^{p} _{\mathcal{I}'}((\langle B_I, \vphi_I \rangle)_{I \in \mathcal{I}'})$ satisfies the desired estimate.
\noindent %to reorder!!!!!!!!
\textbf{Estimate of $\| B_{\mathcal{Q}, \text{lac}}^{\tilde{n},\tilde{m},0}(v_1,v_2)\|_{t}$ for $t >1$.}
We will first prove restricted weak-type estimates for $B_{\mathcal{Q},\text{lac}}^{\tilde{n},\tilde{m},0}$ specified in Claim \ref{en_weak_p} and then the strong-type estimates in Claim \ref{en_strong_p} follow from the standard interpolation technique.
\begin{claim} \label{en_weak_p}
$\| B_{\mathcal{Q}, \text{lac}}^{\tilde{n},\tilde{m},0}(v_1,v_2)\|_{\tilde{t},\infty} \lesssim C_1^{\frac{1}{p} + \frac{1}{q}-\theta_1 -\theta_2}2^{\tilde{n}(\frac{1}{p}-\theta_1)}2^{\tilde{m}(\frac{1}{q}-\theta_2)}|V_1|^{\frac{1}{p}}|V_2|^{\frac{1}{q}},$
\newline
where $0 \leq \theta_1,\theta_2 < 1$ with $\theta_1 + \theta_2 = \frac{1}{\tilde{t}}$ and $\tilde{t} \in (t-\delta, t+ \delta)$ for some $\delta > 0 $ sufficiently small.
\end{claim}

\begin{claim} \label{en_strong_p}
$\| B_{\mathcal{Q},\text{lac}}^{\tilde{n},\tilde{m},0}(v_1,v_2)\|_{\tilde{t}} \lesssim C_1^{\frac{1}{p} + \frac{1}{q}-\theta_1-\theta_2} 2^{\tilde{n}(\frac{1}{p}-\theta_1)}2^{\tilde{m}(\frac{1}{q}-\theta_2)}|V_1|^{\frac{1}{p}}|V_2|^{\frac{1}{q}},$
\newline
where $0 \leq \theta_1,\theta_2 < 1$ with $\theta_1 + \theta_2 = \frac{1}{t}$.
\end{claim}

\begin{proof}[Proof of Claim \ref{en_weak_p}]
It suffices to apply the dualization and prove that for any $\chi_S \in L^{\tilde{t}'}$, 
$$
|\langle B_{\mathcal{Q},\text{lac}}^{\tilde{n},\tilde{m},0}(v_1,v_2), \chi_S \rangle| \lesssim 2^{\tilde{n}(\frac{1}{p}-\theta_1)}2^{\tilde{m}(\frac{1}{q}-\theta_2)}|V_1|^{\frac{1}{p}}|V_2|^{\frac{1}{q}}|S|^{\frac{1}{\tilde{t}'}}
$$
where $0 \leq \theta_1,\theta_2 < 1$ with $\theta_1 + \theta_2 = \frac{1}{\tilde{t}}$.

The multilinear form can be estimated using a similar argument developed for $\| B_{\mathcal{Q}, \text{lac}}^{\tilde{n},\tilde{m},0}(v_1,v_2)\|_1$. In particular, %let 
%$$
%\mathcal{Q}' := \{Q \in \mathcal{Q}: Q \cap \mathcal{U}_{\tilde{n},\tilde{m}} \neq \emptyset\}. 
%$$
%Then
\begin{align} \label{linear_form_p}
& |\langle B_{\mathcal{Q},\text{lac}}^{\tilde{n},\tilde{m},0}(v_1,v_2), \chi_S \rangle| \nonumber \\
 \lesssim & \text{\ \ size}_{\mathcal{Q}'}((\langle v_1, \phi^1_Q \rangle)_{Q\in \mathcal{Q}'})^{1-\theta_1}\text{size}_{\mathcal{Q}'}((\langle v_2, \phi^2_Q \rangle)_{Q \in \mathcal{Q}'})^{1-\theta_2} \text{size}_{\mathcal{Q}'}((\langle \chi_S, \psi^{3,H}_Q \rangle)_{Q \in \mathcal{Q}'})^{1-\theta_3} \nonumber \\
& \text{\ \ energy} _{\mathcal{Q}'}((\langle v_1, \phi^1_Q\rangle)_{Q \in \mathcal{Q}'})^{\theta_1}\text{energy} _{\mathcal{Q}'}((\langle v_2, \phi^2_Q\rangle)_{Q \in \mathcal{Q}'})^{\theta_2} \text{energy} _{\mathcal{Q}'}((\langle \chi_S , \psi^{3,H}_Q\rangle)_{Q \in \mathcal{Q}'})^{\theta_3},
\end{align}
for any $0 \leq \theta_1,\theta_2, \theta_3 <1$ with $\theta_1 + \theta_2 + \theta_3 = 1$. The size and energy estimates involving $v_1, v_2$, namely (\ref{f_size}) and Proposition \ref{energy_classical}, are still valid. %Here $\phi^3_Q$ are defined differently in Case $I$ and $II$. However, 
One also applies the same straightforward estimates that
\begin{align}
& \text{size}_{\mathcal{Q}'}((\langle \chi_S , \psi^{3,H}_Q \rangle)_{Q \in \mathcal{Q}'}) \lesssim 1, \label{set_size_en1}\\
& \text{energy} _{\mathcal{Q}'}((\langle \chi_S , \psi^{3,H}_Q\rangle)_{Q \in \mathcal{Q}'}) \lesssim |S|. \label{set_size_en}
\end{align}
By plugging in the above estimates (\ref{set_size_en1}), (\ref{set_size_en}), (\ref{f_size1}) and (\ref{f_size}) into (\ref{linear_form_p}), one has
\begin{align*}
|\langle B_{\mathcal{Q},\text{lac}}^{\tilde{n},\tilde{m},0}(v_1,v_2), \chi_S \rangle| \lesssim & (C_1 2^{\tilde{n}}|V_1|)^{\alpha(1-\theta_1)} (C_1 2^{\tilde{m}}|V_2|)^{\beta(1-\theta_2)}|V_1|^{\theta_1}|V_2|^{\theta_2}|S|^{\theta_3} \nonumber \\
= & C_1^{\alpha(1-\theta_1) + \beta(1-\theta_2)}2^{\tilde{n}\alpha(1-\theta_1)}2^{\tilde{m}\beta(1-\theta_2)}|V_1|^{\alpha(1-\theta_1) + \theta_1} |V_2|^{\beta(1-\theta_2)+ \theta_2} |S|^{\theta_3},
%\text{\ \ size}_{Q \in \mathcal{Q}: Q \cap S \neq \emptyset}((\langle v_1, \phi^1_Q \rangle)_{Q})^{1-\theta_1}\text{size}_{Q \in \mathcal{Q}:Q \cap S \neq \emptyset}((\langle v_2, \phi^2_Q \rangle)_{Q})^{1-\theta_2}  \nonumber \\
%& \cdot \text{\ \ energy} _{\mathcal{Q}}((\langle v_1, \phi^1_Q\rangle)_{Q})^{\theta_1}\text{energy} _{\mathcal{Q}}((\langle v_2, \phi^2_Q\rangle)_{Q})^{\theta_2} |S|^{\kappa\theta_3}|S_2|^{(1-\kappa)\theta_3}
\end{align*}
for any $0 \leq \alpha, \beta \leq 1$. Let $\theta_3  = \frac{1}{\tilde{t}'}$, then $\theta_1 + \theta_2 = \frac{1}{\tilde{t}}$. One can then conclude 
\begin{align*}
\|B_{\mathcal{Q},\text{lac}}^{\tilde{n},\tilde{m},0}(v_1,v_2)\|_{\tilde{t},\infty} \lesssim & C_1^{\alpha(1-\theta_1) + \beta(1-\theta_2)}2^{\tilde{n}\alpha(1-\theta_1)}2^{\tilde{m}\beta(1-\theta_2)}|V_1|^{\alpha(1-\theta_1) + \theta_1} |V_2|^{\beta(1-\theta_2)+ \theta_2}.
%\text{\ \ size}_{Q \in \mathcal{Q}: Q \cap S \neq \emptyset}((\langle v_1, \phi^1_Q \rangle)_{Q})^{1-\theta_1}\text{size}_{Q \in \mathcal{Q}:Q \cap S \neq \emptyset}((\langle v_2, \phi^2_Q \rangle)_{Q})^{1-\theta_2}  \nonumber \\
%& \cdot \text{\ \ energy} _{\mathcal{Q}}((\langle v_1, \phi^1_Q\rangle)_{Q})^{\theta_1}\text{energy} _{\mathcal{Q}}((\langle v_2, \phi^2_Q\rangle)_{Q})^{\theta_2} |S_2|^{\theta_3 - \frac{1}{p'}}
\end{align*} 
Since $\frac{1}{p} + \frac{1}{q} >1$, one can choose $0 \leq \alpha, \beta \leq 1$ and $0 \leq \theta_1,\theta_2 <1$ with $\theta_1 + \theta_2 = \frac{1}{\tilde{t}} \sim \frac{1}{t}$ such that
$$
\alpha(1-\theta_1) + \theta_1 = \frac{1}{p},
$$
$$
 \beta(1-\theta_2)+ \theta_2 = \frac{1}{q},
$$
the claim then follows.
\end{proof}
\end{enumerate}

\section{Proof of  Theorem \ref{thm_weak_mod} for $\Pi_{\text{flag}^{\#1} \otimes \text{flag}^{\#2}}$ - Haar Model} \label{section_thm_haar_fixed}
\begin{comment}
$$\Pi_1^{\#_1,\#_2}(f_1 \otimes g_1, f_2 \otimes g_2, h) := \displaystyle \sum_{I \times J \in \mathcal{I} \times \mathcal{J}} \frac{1}{|I|^{\frac{1}{2}} |J|^{\frac{1}{2}}} \langle B_I^{\#_1,H}(f_1,f_2),\vphi_I^{1,H} \rangle \langle \tilde{B}_J^{\#_2,H}(g_1, g_2), \vphi_J^1 \rangle  \langle h, \psi_I^{2,H} \otimes \psi_J^{2,H} \rangle \psi_I^{3,H} \otimes \psi_J^{3,H}$$ 
where 
$$B_I^{\#_1,H}(f_1,f_2)(x) := \displaystyle \sum_{K \in \mathcal{K}:|K| \sim 2^{\#_1} |I|} \frac{1}{|K|^{\frac{1}{2}}}\langle f_1, \phi_K^1 \rangle \langle f_2, \phi_K^2 \rangle \phi_K^3(x),$$
$$\tilde{B}_J^{\#_2,H}(g_1,g_2)(y) := \displaystyle \sum_{L \in \mathcal{L}:|L| \sim 2^{\#_2} |J|} \frac{1}{|L|^{\frac{1}{2}}}\langle g_1, \phi_L^1 \rangle \langle g_2, \phi_L^2 \rangle \phi_L^3(y).$$
\end{comment}

\begin{comment}
\begin{theorem} \label{fixed_scale_p}
For $ 1 < p_i, q_i < \infty$, $\displaystyle \sum_{i=1}^{3}\frac{1}{p_i} = \sum_{i=1}^{3}\frac{1}{q_i} = \frac{1}{r}$,
$$\Pi^l: L_x^{p_1}(L_y^{q_1}) \times L_x^{p_2}(L_y^{q_2}) \times L_x^{p_3}(L_y^{q_3}) \rightarrow L^r.$$
\end{theorem}

\begin{theorem} \label{fixed_scale_inf}
For $1<p,q \leq \infty$, $1<s<\infty$, $\displaystyle \frac{1}{p} + \frac{1}{s} = \frac{1}{r}$,
$$\Pi^l: L^{p}(\mathbb{R}^2) \times L^{q}(\mathbb{R}^2) \times L^{s}(\mathbb{R}^2) \rightarrow L^r.$$
$$\Pi^l: L_x^{p}(L_y^{\infty}) \times L_x^{\infty}(L_y^{p}) \times L^{s}(\mathbb{R}^2) \rightarrow L^r.$$
\end{theorem}
\end{comment}

%\begin{proof}

In this section, we will study the model operator $\Pi_{\text{flag}^{\#1} \otimes \text{flag}^{\#2}}$ in the Haar case. In particular, we will focus on
\begin{align} \label{Pi_fixed_haar}
\Pi^H_{\text{flag}^{\#1} \otimes \text{flag}^{\#2}}:= &  \displaystyle \sum_{I \times J \in \mathcal{R}} \frac{1}{|I|^{\frac{1}{2}} |J|^{\frac{1}{2}}} \langle B^{\#_1,H}_{\mathcal{K},I}(f_1,f_2),\vphi_I^{1,H} \rangle \langle  B^{\#_2,H}_{\mathcal{L},J}(g_1, g_2), \vphi_J^{1,H} \rangle  \langle h, \psi_I^{2} \otimes \psi_J^{2} \rangle \psi_I^{3,H} \otimes \psi_J^{3,H}.
\end{align}
We will first specify the localization for the operator $\Pi^H_{\text{flag}^{\#1} \otimes \text{flag}^{\#2}}$ (\ref{Pi_fixed_haar}), 
%in the Haar model, 
which can be viewed as a starting point for the stopping-time decompositions. Then we will introduce different stopping-time decompositions used in the estimates. 
%If multiple stopping-time algorithms are involved, how to hybrid them will be clarified. 
Finally, we will discuss how to apply information from the multiple stopping-time decompositions to obtain estimates. The organizations of Sections \ref{section_thm_haar}, \ref{section_thm_inf_fixed_haar} and \ref{section_thm_inf_haar} will follow the similar scheme.

\begin{notation}
Recall that $B_{\mathcal{K},I}$, $B_{\mathcal{L},J}$, $B_{\mathcal{K},I}^{\#_1}$ and $B^{\#_2}_{\mathcal{L},J}$ are bilinear operators (Definition \ref{B_definition}) that appear in the discrete model stated in Definition \ref{discrete_model_op}. To avoid being overloaded with heavy notation in further discussions, we introduce  simplified notation so that the dependence on the collections of dyadic intervals $\mathcal{K}$ and $\mathcal{L}$ are implicit. In particular, let
\begin{align}
&B_I(f_1,f_2):=  B_{\mathcal{K},I}(f_1,f_2), \ \ \tilde{B}_J(g_1,g_2):= B_{\mathcal{L},J}(g_1,g_2), \label{B_local_fourier_simple} \\
&B^{\#_1}_I(f_1,f_2):=  B_{\mathcal{K},I}^{\#_1}(f_1,f_2), \ \ \tilde{B}^{\#_2}_J(g_1,g_2):= B^{\#_2}_{\mathcal{L},J}(g_1,g_2).\label{B_fixed_local_fourier_simple} 
\end{align}
Similarly, for the other bilinear operators in Definition \ref{B_definition} and Definition \ref{B_definition_haar}, we adopt the following notation:
\begin{align}
& B(f_1,f_2):=  B_{\mathcal{K}}(f_1,f_2), \ \ \tilde{B}(g_1,g_2):= B_{\mathcal{L}}(g_1,g_2), \label{B_global_proof}\\
& B^H(f_1,f_2):=  B^H_{\mathcal{K}}(f_1,f_2), \ \ \tilde{B}^H(g_1,g_2):= B^H_{\mathcal{L}}(g_1,g_2), \label{B_global_haar_simplified} \\
& B_I^H(f_1,f_2):=  B^H_{\mathcal{K},I}(f_1,f_2), \ \ \tilde{B}^H_J(g_1,g_2):= B^H_{\mathcal{L},J}(g_1,g_2), \label{B_0_local_haar_simplified} \\
& B^{\#_1,H}_I(f_1,f_2):=  B_{\mathcal{K},I}^{\#_1,H}(f_1,f_2), \ \ \tilde{B}^{\#_2,H}_J(g_1,g_2):= B^{\#_2,H}_{\mathcal{L},J}(g_1,g_2). \label{B_fixed_local_haar_simplified} 
\end{align}
\end{notation}

\subsection{Localization} \label{section_thm_haar_fixed_localization}
The definition of the exceptional set (the set that would be taken away) settles the starting point for the stopping-time decompositions and thus is expected to be compatible with the stopping-time algorithms involved. There would be two types of stopping-time decompositions undertaken for the estimates of $\Pi^H_{\text{flag}^{\#1} \otimes \text{flag}^{\#2}}$ - one is the \textit{tensor-type stopping-time decomposition} and the other one the \textit{general two-dimensional level sets stopping-time decomposition}. While the second algorithm is related to a generic exceptional set (denoted by $\Omega^2$), the first algorithm aims to integrate information from two one-dimensional decompositions, which corresponds to the creation of a two-dimensional exceptional set (denoted by $\Omega^1$) as a Cartesian product of two one-dimensional exceptional sets.

One defines the exceptional set as follows. Let
$$\Omega := \Omega^1 \cup \Omega^2,$$ 
where
\begin{align*}
\displaystyle \Omega^1 := &\bigcup_{\mathfrak{n}_1 \in \mathbb{Z}}\{x: Mf_1(x) > C_1 2^{\mathfrak{n}_1}|F_1|\} \times \{y: Mg_1(y) > C_2 2^{-\mathfrak{n}_1}|G_1|\}\cup \nonumber \\
& \bigcup_{\mathfrak{n}_2 \in \mathbb{Z}}\{x:Mf_2(x) > C_1 2^{\mathfrak{n}_2}|F_2|\} \times \{y:Mg_2(y) > C_2 2^{-\mathfrak{n}_2}|G_2|\}\cup \nonumber \\
 &\bigcup_{\mathfrak{n}_3 \in \mathbb{Z}}\{x:Mf_1(x) > C_1 2^{\mathfrak{n}_3}|F_1|\} \times \{y:Mg_2(y) > C_2 2^{-\mathfrak{n}_3}|G_2|\}\cup \nonumber \\
& \bigcup_{\mathfrak{n}_4 \in \mathbb{Z}}\{x:Mf_2 (x)> C_1 2^{\mathfrak{n}_4 }|F_2|\} \times \{y:Mg_1(y) > C_2 2^{-\mathfrak{n}_4 }|G_1|\}, \nonumber \\
\Omega^2 := & \{(x,y) \in \mathbb{R}^2: SSh(x,y) > C_3 \|h\|_{L^s(\mathbb{R}^2)}\}.
\end{align*}
For technical reason that would become clear later on, we would need to get rid of the enlargement of $\Omega$ defined by
$$
Enl(\Omega) := \{(x,y)\in \mathbb{R}^2: MM\chi_{\Omega}(x,y) > \frac{1}{100}\}.
$$

\begin{remark} \label{subset}
Given by the boundedness of the Hardy-Littlewood maximal operator and the double square function operator, it is not difficult to check that if $C_1, C_2, C_3 \gg 1$,
then $
|Enl(\Omega)| \ll 1. 
$
For different model operators, we will define different exceptional sets based on different stopping-time decompositions to employ. Nevertheless, their measures can be controlled similarly using the boundedness of the operators enclosed in Theorem \ref{maximal-square}. 
%Such observation will be used throughout the proof for exceptional sets corresponding to other discrete models. 
\end{remark}
By scaling invariance, we will assume without loss of generality that $|E| = 1$ throughout the paper.
Let 
\begin{equation} \label{set_E'}
E' := E \setminus Enl(\Omega),
\end{equation}
then 
$
|E'| \sim |E|
$ and thus $|E'| \sim 1$. 
%\begin{remark}
%$E'$ will always be defined as (\ref{set_E'}) such that $|E'| \sim |E| =1$. 
Our goal is to show that (\ref{thm_weak_explicit}) holds with the corresponding subset (which will be different for each discrete model operator) $E' \subseteq E$. 
%\end{remark}
In the current setting, this is equivalent to proving that the multilinear form
\begin{equation} \label{form_haar_larger}
\Lambda^H_{\text{flag}^{\#1} \otimes \text{flag}^{\#2}}(f_1, f_2, g_1, g_2, h, \chi_{E'}) := \langle \Pi^H_{\text{flag}^{\#1} \otimes \text{flag}^{\#2}}(f_1, f_2, g_1, g_2, h), \chi_{E'} \rangle
\end{equation}
satisfies the following restricted weak-type estimate
\begin{equation} \label{form_haar_larger_goal}
|\Lambda^H_{\text{flag}^{\#1} \otimes \text{flag}^{\#2}}(f_1, f_2, g_1, g_2, h, \chi_{E'})| \lesssim |F_1|^{\frac{1}{p_1}} |G_1|^{\frac{1}{p_2}} |F_2|^{\frac{1}{q_1}} |G_2|^{\frac{1}{q_2}} \|h\|_{L^{s}(\mathbb{R}^2)}.
\end{equation}

\begin{remark}\label{localization_haar_fixed}
It is noteworthy that the discrete model operators are perfectly localized to $E'$ in the Haar model. More precisely,
\begin{align}\label{haar_local}
& \Lambda^H_{\text{flag}^{\#1} \otimes \text{flag}^{\#2}}(f_1, f_2, g_1, g_2, h, \chi_{E'})  \nonumber\\
= &  \displaystyle \sum_{I \times J \in \mathcal{R}} \frac{1}{|I|^{\frac{1}{2}} |J|^{\frac{1}{2}}} \langle B^{\#_1,H}_I(f_1,f_2),\vphi_I^{1,H} \rangle \langle \tilde{B}^{\#_2,H}_J(g_1, g_2), \vphi_J^{1,H} \rangle  \langle h, \psi_I^{2} \otimes \psi_J^{2} \rangle \langle \chi_{E'}, \psi_I^{3,H} \otimes \psi_J^{3,H} \rangle  \nonumber \\
= &  \displaystyle \sum_{\substack{I \times J \in \mathcal{R} \\ I \times J \cap Enl(\Omega)^c \neq \emptyset}} \frac{1}{|I|^{\frac{1}{2}} |J|^{\frac{1}{2}}} \langle B^{\#_1,H}_I(f_1,f_2),\vphi_I^{1,H} \rangle \langle \tilde{B}^{\#_2,H}_J(g_1, g_2), \vphi_J^{1,H} \rangle  \langle h, \psi_I^{2} \otimes \psi_J^{2} \rangle \langle \chi_{E'}, \psi_I^{3,H} \otimes \psi_J^{3,H} \rangle, 
\end{align}
because for $I \times J \cap Enl(\Omega)^c = \emptyset$, $I \times J \cap E' = \emptyset$ and thus
$
\langle \chi_{E'}, \psi_I^{3,H} \otimes \psi_J^{3,H} \rangle = 0,
$
which means that dyadic rectangles satisfying $I \times J \cap Enl(\Omega)^c = \emptyset$ do not contribute to the multilinear form. In the Haar model, we would heavily rely on the localization (\ref{haar_local}) and consider only the dyadic rectangles $I \times J \in \mathcal{R}$ such that $I \times J \cap Enl(\Omega)^c \neq \emptyset$.
\end{remark}

\subsection{Tensor-type stopping-time decomposition I - level sets} \label{section_thm_haar_fixed_tensor}
%Two-dimensional tensor-type stopping-time decompositions are stopping-time decomposition algorithms that
%\begin{enumerate} [Step 1:]
%\item
%perform one-dimensional stopping-time decompositions on $x$-variable and $y$-variable independently;
%\item
%combine two one-dimensional decompositions to generate two-dimensional .
%\end{enumerate}
The first tensor-type stopping time decomposition, denoted by the \textit{tensor-type stopping-time decomposition I}, will be performed to obtain estimates for $\Pi_{\text{flag}^{\#1} \otimes \text{flag}^{\#2}}^H$. It aims to recover intersections with two-dimensional level sets from intersections with one-dimensional level sets for each variable. Another tensor-type stopping-time decomposition, denoted by the \textit{tensor-type stopping-time decomposition II}, involves maximal intervals and plays an important role in the discussion for $\Pi_{\text{flag}^0 \otimes \text{flag}^0}$. We will focus on the \textit{tensor-type stopping-time decomposition I} in this section.

\subsubsection{One-dimensional stopping-time decompositions - level sets}
One can perform a one-dimensional stopping-time decomposition on $\mathcal{I} := \{I: I \times J \in \mathcal{R}\}$, or equivalently the collection of $x$-intervals of all dyadic rectangles in $\mathcal{R}$. Since $\mathcal{I}$ is a finite collection of dyadic intervals, there exists $N_1 \in \mathbb{Z}$ such that for any $I \in \mathcal{I}$,
\begin{equation*}
|I \cap \{x: Mf_1(x) \leq C_1 2^{N_1+1}|F_1| \}.
\end{equation*}
Now let
$$\Omega^{}_{N_1} :=  \{x: Mf_1(x) > C_1 2^{N_1}|F_1|\},$$ and 
$$\mathcal{I}_{N_1} := \{I \in \mathcal{I}: |I \cap \Omega^{}_{N_1}| > \frac{1}{10}|I| \}.$$
Define 
$$\Omega^{}_{N_1-1} :=  \{x: Mf_1(x) > C_1 2^{N_1-1}|F_1|\},$$ and 
$$\mathcal{I}_{N_1-1} := \{I \in \mathcal{I} \setminus \mathcal{I}_{N_1}: |I \cap \Omega^{}_{N_1-1}| > \frac{1}{10}|I| \}.$$
\ \ \ \ \ \ \ \ \ \ \ \ \ \ \ \ \ \ \ \ \ \ \ \ \ \ \ \ \ \ \ \ \ \ \ \ \ \ \ \ \ \ \ \ \ \ \ \ \ \ \ \ \vdots
\par 
The procedure generates the sets $(\Omega^{}_{n_1})_{n_1 \in \mathbb{Z}}$ and $(\mathcal{I}_{n_1})_{n_1 \in \mathbb{Z}}$.
Independently define
$$\Omega'^{}_{M_1} :=  \{x: Mf_2(x) > C_1 2^{M_1}|F_2|\},$$ for some $M_1 \in \mathbb{Z}$ and 
$$\mathcal{I}'_{M_1} := \{I \in \mathcal{I}: |I \cap \Omega'^{}_{M_1}| > \frac{1}{10}|I| \}.$$
Define
$$\Omega'^{}_{M_1-1} :=  \{x: Mf_2(x) > C_1 2^{M_1-1}|F_2|\},$$ and 
$$\mathcal{I}'_{M_1-1} := \{I \in \mathcal{I} \setminus \mathcal{I}'_{M_1}: |I \cap \Omega'^{}_{M_1-1}| > \frac{1}{10}|I| \}.$$
\ \ \ \ \ \ \ \ \ \ \ \ \ \ \ \ \ \ \ \ \ \ \ \ \ \ \ \ \ \ \ \ \ \ \ \ \ \ \ \ \ \ \ \ \ \ \ \ \ \ \ \ \vdots
\par 
The procedure generates the sets $(\Omega'^{}_{m_1})_{m_1 \in \mathbb{Z}}$ and $(\mathcal{I}'_{m_1})_{m_1 \in \mathbb{Z}}$. Now define 
$$\mathcal{I}_{n_1,m_1} := \mathcal{I}_{n_1} \cap \mathcal{I}'_{m_1}$$
 and one achieves the decomposition on $\displaystyle \mathcal{I} = \bigcup_{n_1,m_1 \in \mathbb{Z}}\mathcal{I}_{n_1,m_1}$.
\par
Same algorithm can be applied to $\mathcal{J}:= \{J: I \times J \in \mathcal{R}\}$ with respect to the level sets in terms of $Mg_1$ and $Mg_2$, 
%and $Mg_2$
which produces the sets 
\begin{enumerate} [(i)]
\item
$(\tilde{\Omega}^{}_{n_2})_{n_2 \in \mathbb{Z}}$ and $(\mathcal{J}_{n_2})_{n_2 \in \mathbb{Z}}$, where
$$\tilde{\Omega}^{}_{n_2 } :=  \{y: Mg_1(y) > C_2 2^{n_2}|G_1|\},$$ and 
$$\mathcal{J}_{n_2} := \{J \in \mathcal{J} \setminus \mathcal{J}_{n_2+1}: |J \cap \tilde{\Omega}^{}_{n_2}| > \frac{1}{10}|J| \}.$$
\item
$(\tilde{\Omega}'^{}_{m_2})_{m_2 \in \mathbb{Z}}$ and $(\mathcal{J}'_{m_2})_{m_2 \in \mathbb{Z}}$, where
$$\tilde{\Omega}'^{}_{m_2} :=  \{ y: Mg_2(y) > C_2 2^{m_2}|G_2|\},$$ and 
$$\mathcal{J}'_{m_2} := \{J \in \mathcal{J} \setminus \mathcal{J}'_{m_2+1}: |J \cap \tilde{\Omega}'^{}_{m_2}| > \frac{1}{10}|J| \}.$$
\end{enumerate}
One thus obtains the decomposition $\displaystyle \mathcal{J} = \bigcup_{n_2, m_2 \in \mathbb{Z}} \mathcal{J}_{n_2,m_2}$, where $\mathcal{J}_{n_2,m_2} := \mathcal{J}_{n_2} \cap \mathcal{J}'_{m_2}$. 
\begin{comment}
Let
$$\Omega^{y}_{N_2} :=  \{ Mg_1 > C_2 2^{N_2}|G_1|\}$$ for some $N_2 \in \mathbb{Z}$ and define
$$\mathcal{J}_{N_2} := \{J \in \mathcal{J}: |J \cap \Omega^{y}_{N_2}| > \frac{1}{10}|J| \}.$$
Iteratively define 
$$\Omega^{y}_{N_2-1} :=  \{ Mg_1 > C_2 2^{N_2-1}|G_1|\}$$ and 
$$\mathcal{J}_{N_2-1} := \{J \in \mathcal{J} \setminus \mathcal{J}^{N_2}: |J \cap \Omega^{y}_{N_2-1}| > \frac{1}{10}|J| \}.$$
\ \ \ \ \ \ \ \ \ \ \ \ \ \ \ \ \ \ \ \ \ \ \ \ \ \ \ \ \ \ \ \ \ \ \ \ \ \ \ \ \ \ \ \ \ \ \ \ \ \ \ \ \vdots
\par
The procedure produces the sets $(\Omega^{y}_{n_2})_{n_2}$ and $(\mathcal{J}_{n_2})_{n_2}$.  
\par 
Independently define
$$\Omega^{y}_{M_2} :=  \{ Mg_2 > C_2 2^{M_2}|G_2|\}$$ for some $M_2 \in \mathbb{Z}$ and define
$$\mathcal{J}_{M_2} := \{J \in \mathcal{J}: |J \cap \Omega^{y}_{M_2}| > \frac{1}{10}|J| \}.$$
Iteratively define 
$$\Omega^{y}_{M_2-1} :=  \{ Mg_2 > C_2 2^{M_2-1}|G_2|\}$$ and 
$$\mathcal{J}_{M_2-1} := \{J \in \mathcal{J} \setminus \mathcal{J}^{M_2}: |J \cap \Omega^{y}_{M_2-1}| > \frac{1}{10}|J| \}.$$
\ \ \ \ \ \ \ \ \ \ \ \ \ \ \ \ \ \ \ \ \ \ \ \ \ \ \ \ \ \ \ \ \ \ \ \ \ \ \ \ \ \ \ \ \ \ \ \ \ \ \ \ \vdots
\par
The procedure produces the sets $(\Omega^{y}_{m_2})_{m_2}$ and $(\mathcal{J}_{m_2})_{m_2}$. Therefore $\displaystyle \mathcal{J} = \bigcup_{n_2, m_2} \mathcal{J}_{n_2,m_2}$, where $\mathcal{J}_{n_2,m_2} := \mathcal{J}_{n_2} \cap \mathcal{J}_{m_2}$.  
\end{comment}
\newline
\subsubsection{Tensor product of two one-dimensional stopping-time decompositions - level sets} \label{section_thm_haar_fixed_tensor_1d_level}
If we assume that all dyadic rectangles satisfy $I \times J \cap Enl(\Omega)^{c} \neq \emptyset$ as in the Haar model, then we have the following observation.
\begin{obs} \label{obs_indice}
If $I \times J \in \mathcal{I}_{n_1,m_1} \times \mathcal{J}_{n_2,m_2}$, then $n_1 , m_1 ,n_2, m_2 \in \mathbb{Z}$ satisfies $n_1+n_2 < 0$ and $m_1 + m_2 < 0$. (Equivalently, $\forall I \times J  \cap Enl(\Omega)^{c} \neq \emptyset$, $I \times J \in \mathcal{I}_{-n-n_2,-m-m_2} \times \mathcal{J}_{n_2,m_2}$, for some $n_2, m_2 \in \mathbb{Z}$ and $n, m > 0$.)
\begin{remark}
The observation shows that how a rectangle $I \times J$ intersects with a two-dimensional level sets is closely related to  %$I \times J \in \mathcal{I}_{n_1} \times \mathcal{J}_{n_2}$ satisfying $n_1 + n_2 < 0$ gives information about 
how the corresponding intervals intersect with one-dimensional level sets (namely $I \in \mathcal{I}_{n_1.m_1}$ and $J \in \mathcal{J}_{n_2,m_2}$ with $n_1 + n_2 < 0$ and $m_1 + m_2 < 0$), as commented in the beginning of the section.
\end{remark}
\begin{proof}
Given $I \in \mathcal{I}_{n_1}$, one has 
$|I \cap \{x: Mf_1(x) > C_1 2^{n_1}|F_1|\}| > \frac{1}{10} |I|$; similarly, $J \in \mathcal{J}_{n_2}$ implies that $|J \cap \{y: Mg_1(y) > C_2 2^{n_2}|G_1|\}| > \frac{1}{10}|J|$. If $n_1 + n_2 \geq 0$ , then $\{x: Mf_1(x) > C_1 2^{n_1}|F_1|\} \times\{y: Mg_1(y) > C_2 2^{n_2}|G_1|\} \subseteq \Omega^1 \subseteq \Omega$. Then $|I \times J \cap \Omega| > \frac{1}{100}|I \times J|$, which implies that $I \times J \subseteq Enl(\Omega)$ and contradicts with the assumption. Same reasoning applies to the pairs $(m_1,m_2)$, $(n_1,m_2)$ and $(m_1,n_2)$.
\end{proof}
\end{obs}
\begin{remark}
Thanks to Obervation \ref{obs_indice}, we conclude that in the Haar model when $I \times J \cap Enl(\Omega)^{c} \neq \emptyset$,
\begin{equation*}
\mathcal{R} = \bigcup_{\substack{n_1,m_1,n_2,m_2 \in \mathbb{Z} \\ n_1 + n_2 < 0 \\ m_1 + m_2 < 0 \\ n_1 + m_2 < 0 \\ m_1 + n_2 < 0}}\mathcal{I}_{n_1,m_1} \times \mathcal{J}_{n_2,m_2}.
\end{equation*}
\end{remark}

\begin{comment}
Same procedure can be applied with respect to $Mf_2$ and $Mg_2$. The definition of the exceptional set then implies that $I \times J \in \mathcal{I}_{-m+m_1} \times \mathcal{J}_{-m_1}$ for some $m_1 \in \mathbb{Z}$ and $m \geq 0$.
\end{comment}

\subsection{General two-dimensional level sets stopping-time decomposition} \label{section_thm_haar_fixed_level}
With the assumption that $R \cap Enl(\Omega)^c \neq \emptyset$, one has that
$$
|R\cap \Omega^2| \leq \frac{1}{100}|R|,
$$
where
$$
\Omega^2 = \{(x,y)\in \mathbb{R}^2: SSh(x,y) >C_3 \|h\|_{L^s(\mathbb{R}^2)}\}.
$$
Then define
$$\Omega^2_{-1}:= \{(x,y)\in \mathbb{R}^2: SSh(x,y) > C_3 2^{-1}\|h\|_{L^s(\mathbb{R}^2)}\}$$
and
$$\mathcal{R}_{-1} := \{R \in \mathcal{R}: |R \cap \Omega^2_{-1}| > \frac{1}{100}|R|\}.$$
Successively define 
$$\Omega^2_{-2}:= \{(x,y)\in \mathbb{R}^2: SSh(x,y) > C_3 2^{-2}\|h\|_{L^s(\mathbb{R}^2)}\}
$$
and
$$\mathcal{R}_{-2} := \{R \in \mathcal{R} \setminus \mathcal{R}_{-1}: |R \cap \Omega^2_{-2}| > \frac{1}{100}|R|\}.$$
\ \ \ \ \ \ \ \ \ \ \ \ \ \ \ \ \ \ \ \ \ \ \ \ \ \ \ \ \ \ \ \ \ \ \ \ \ \ \ \ \ \ \ \ \ \ \ \ \ \ \ \ \ \ \ \ \ \ \ \ \ \ \        \vdots 
\newline
This two-dimensional stopping-time decomposition generates the sets $(\Omega^{2}_{k_1})_{k_1 \leq 0}$ and $(\mathcal{R}_{k_1})_{k_1 \leq 0}$.
\par

Independently one can apply the same algorithm involving $(SS)^H\chi_{E'}$ which generates $(\Omega'^{2}_{k_2})_{k_2 \leq K}$ and $(\mathcal{R}'_{k_2})_{k_2 \leq K}$ where $K$ can be arbitrarily large. The existence of $K$ is guaranteed by the finite cardinality of the collection of dyadic rectangles. Then define 
$$
\mathcal{R}_{k_1,k_2} := \mathcal{R}_{k_1} \cap \mathcal{R}'_{k_2}.
$$
and thus 
\begin{equation*}
\displaystyle \mathcal{R} = \bigcup_{k_1\leq 0, k_2 \leq K} \mathcal{R}_{k_1,k_2}.
\end{equation*}
%\begin{comment}
%$$\Omega^2_{k_2} := \{(SS)^H\chi_{E'} > C_3 2^{k_2}\}$$
%$$\mathcal{R}_{k_2-1} := \{R \in \mathcal{R} \setminus \mathcal{R}_{k_2}: |R \cap \Omega^2_{k_2}| > \frac{1}{100}|R|\}$$
%\end{comment}

\subsection{Sparsity condition} \label{section_thm_haar_fixed_sparsity}
One important property followed from the \textit{tensor-type stopping-time decomposition I - level sets} is the sparsity of dyadic intervals at different levels. Such geometric property plays an important role in the arguments for the main theorems.
\begin{proposition} \label{sparsity}
Suppose that $\displaystyle \mathcal{J} = \bigcup_{n_2 \in \mathbb{Z}} \mathcal{J}_{n_2}$ is a decomposition of dyadic intervals with respect to $Mg_1$ as specified in Section \ref{section_thm_haar_fixed_tensor}. For any fixed $n_2 \in \mathbb{Z}$, suppose that $J_0 \in \mathcal{J}_{n_2 - 10}$. Then
$$\displaystyle
\sum_{\substack{J \in \mathcal{J}_{n_2}\\J \cap J_0 \neq \emptyset}} |J| \leq \frac{1}{2}|J_0|.
$$
\end{proposition}
To prove the proposition, one would need the following claim about pointwise estimates for $Mg_1$ %(and similarly for $Mg_2$) 
on $J \in \mathcal{J}_{n_2}$:
\begin{claim} \label{ptwise}
Suppose that $\bigcup_{n_2}\mathcal{J}_{n_2}$ is a partition of dyadic intervals generated from the stopping-time decomposition described above. If $J \in \mathcal{J}_{n_2}$,
then for any $y \in J,$
$$
 %\Longrightarrow 
 Mg_1(y)> 2^{-7} \cdot C_2 2^{n_2}|G_1|. $$
\begin{comment}
Similarly, $J \in \mathcal{J}_{m_2} \Longrightarrow Mg_2 \gtrsim 2^{m_2}|G_2|$ on $J$.
\end{comment}
\end{claim}
\begin{proof}[Proof on Claim $\Longrightarrow$ Proposition \ref{sparsity}]
We will first explain why the proposition follows from the claim and then prove the claim.
One recalls that all the intervals are dyadic, which means if $J \cap J_0 \neq \emptyset$, then either $$J \subseteq J_0$$ or $$J_0 \subseteq J.$$ If $J_0 \subseteq J$, then the claim implies that 
$$J_0 \subseteq J \subseteq \{ Mg_1 > C_2 2^{n_2-7}|G_1|\}.$$ But $J_0 \in \mathcal{J}_{n_2-10}$ infers that 
$$ \big|J_0 \cap \{ Mg_1 > C_2 2^{n_2 - 7}|G_1|\}\big| < \frac{1}{10}|J_0|,$$ which is a contradiction. If $J \subseteq J_0$ and suppose that
$$
\displaystyle
\sum_{\substack{J \in \mathcal{J}_{n_2}\\J \subseteq J_0}} |J| > \frac{1}{2}|J_0|.
$$
Then one can derive from $J \in \mathcal{J}_{n_2}$ that  $$\big|J\cap \{Mg_1 > C_2 2^{n_2}|G_1| \} \big| > \frac{1}{10}|J|.$$ Therefore
$$\sum_{\substack{J \in \mathcal{J}_{n_2}\\J \subseteq J_0}} \big|J\cap \{Mg_1 > C_2 2^{n_2}|G_1| \} \big| > \frac{1}{10}\sum_{\substack{J \in \mathcal{J}_{n_2}\\J \subseteq J_0}}|J| > \frac{1}{20}|J_0|.$$
But by the disjointness of $(J)_{J \in \mathcal{J}_{n_2}}$,
$$\sum_{\substack{J \in \mathcal{J}_{n_2}\\J \subseteq J_0}} \big|J\cap \{Mg_1 > C_2 2^{n_2}|G_1| \} \big| \leq \big|J_0\cap \{Mg_1 > C_2 2^{n_2}|G_1| \} \big|.$$
Thus
$$
\big|J_0\cap \{Mg_1 > C_2 2^{n_2}|G_1| \} \big| > \frac{1}{20}|J_0|,
$$
Now the claim, with slight modifications, implies that $J_0 \subseteq \{Mg_1 > C_2 2^{n_2-8}|G_1| \}$. But $J_0 \in \mathcal{J}_{n_2-10}$, which gives the necessary condition that
$$
\big|J_0\cap \{Mg_1 > C_2 2^{n_2}|G_1| \} \big| \leq \frac{1}{10}|J_0|
$$
and reaches a contradiction.
\end{proof}

We will now prove the claim.
\begin{proof}[Proof of Claim]
Without loss of generality, we assume that $g_1$ is non-negative since if it is not, we can always replace it by $|g_1|$ where $Mg_1 = M(|g_1|)$.
We prove the claim case by case: \newline
Case (i): $\forall y \in \{Mg_1 > C_2 2^{n_2}|G_1|\}$, there exists $J_{y} \subseteq J$ such that $\text{ave}_{J_y}(g_1)\footnote{$\text{ave}_{J} (g_1) := \frac{1}{|J|}\int_J g(s) ds$.} > C_2 2^{n_2}|G_1|;$ \newline
Case (ii): There exists $y_0 \in \{Mg_1 > C_2 2^{n_2}|G_1|\}$ and $J_{y_0} \nsubseteq J$ such that $\text{ave}_{J_{y_0}}(g_1) > C_2 2^{n_2}|G_1|$ and \newline
\indent Case (iia): $\frac{1}{40}|J| \leq |J_{y_0} \cap J|$ and $|J_{y_0}| \leq |J|$; \newline
\indent Case (iib): $\frac{1}{40}|J| \leq |J_{y_0} \cap J|$ and $|J_{y_0}| > |J|$; \newline
\indent Case (iic): $|J_{y_0} \cap J| < \frac{1}{40}|J|$. \newline
\textit{Proof of (i):} In Case (i), one observes that $\{Mg_1 > C_2 2^{n_2}|G_1|\} \cap J$ can be rewritten as $\{M(g_1\cdot \chi_J) > C_2 2^{n_2}|G_1|\} \cap J$. Thus 
$$C_2 2^{n_2}|G_1||\{Mg_1 > C_2 2^{n_2}|G_1|\} \cap J| = C_2 2^{n_2}|G_1||\{M(g_1\chi_J) > C_2 2^{n_2}|G_1|\} \cap J| \leq \|g_1\chi_J\|_1.$$
One recalls that $|\{Mg_1 > C_2 2^{n_2}|G_1|\} \cap J| > \frac{1}{10}|J|$, which implies that 
$$C_2 2^{n_2}|G_1|\cdot \frac{1}{10}|J| \leq \|g_1\chi_J\|_1,$$
or equivalently,
$$\frac{\|g_1\chi_J\|_1}{|J|} \geq \frac{1}{10}C_2 2^{n_2}|G_1|.$$
Therefore $Mg_1 > 2^{-4} C_2 2^{n_2}|G_1|$. \newline
\textit{Proof of (ii)}: We will prove that if either (iia) or (iib) holds, then $Mg_1 > 2^{-7} C_2 2^{n_2}|G_1|$. If neither (iia) nor (iib) happens, then (iic) has to hold and in this case, $Mg_1 > 2^{-7} C_2 2^{n_2}|G_1|$. \par
If there exists $y_0 \in \{Mg_1 > C_2 2^{n_2}|G_1|\}$ such that (iia) holds, then 
$$\frac{\|g_1 \chi_{J_{y_0}} \|_1}{|J_{y_0}|} \leq \frac{\|g_1 \chi_{J_{y_0} \cup J} \|_1}{|J_{y_0}|} \leq \frac{\|g_1 \chi_{J_{y_0} \cup J} \|_1}{|J_{y_0}\cap J|} \leq \frac{\|g_1 \chi_{J_{y_0} \cup J} \|_1}{\frac{1}{40}|J|},$$
where the last inequality follows from $\frac{1}{40}|J| \leq |J_{y_0} \cap J|$. Moreover, $|J_{y_0}| \leq |J|$ and $y \in J_{y_0} \cap J \neq \emptyset$ infer that $|J_{y_0} \cup J| \leq 2|J|$. Thus 
$$\frac{\|g_1 \chi_{J_{y_0} \cup J} \|_1}{\frac{1}{20}|J|} \leq \frac{\|g_1 \chi_{J_{y_0} \cup J} \|_1}{\frac{1}{40}\frac{1}{2}|J_{y_0} \cup J|},$$
which implies 
$$\frac{\|g_1 \chi_{J_{y_0} \cup J} \|_1}{|J_{y_0} \cup J|} > \frac{1}{80}C_2 2^{n_2}|G_1|,$$
and as a result $Mg_1 > 2^{-7} C_2 2^{n_2}|G_1|$ on $J$. \par
If there exists $y \in \{Mg_1 > C_2 2^{n_2}|G_1|\}$ such that (iib) holds, then 
$$\frac{\|g_1 \chi_{J_{y_0}} \|_1}{|J_{y_0}|} \leq \frac{\|g_1 \chi_{J_{y_0} \cup J} \|_1}{|J_{y_0}|} = \frac{2\|g_1 \chi_{J_{y_0} \cup J} \|_1}{2|J_{y_0}|} \leq \frac{2\|g_1 \chi_{J_{y_0} \cup J} \|_1}{|J_{y_0} \cup J|},$$
where the last inequality follows from $|J_{y_0}| > |J|$. As a consequence, 
$$\frac{2\|g_1 \chi_{J_{y_0} \cup J} \|_1}{|J_{y_0} \cup J|} > C_2 2^{n_2}|G_1|,$$
and $Mg_1 > 2^{-1} C_22^{n_2}|G_1|$ on $J$. \par

If neither (i), (iia) nor (iib) happens, then for $\mathcal{S}_{(iic)} := \{y: Mg_1(y) > C_2 2^{n_2}|G_1| \text{\ \ and\ \ } (i) \text{\ \ does not hold}\}$, % one has that for any interval $\tilde{J}$ satisfying $$\tilde{J} \nsubseteq J \text{\ \ such that \ \ }y \in \tilde{J}, |\tilde{J} \cup J| \geq \frac{1}{40}|\tilde{J}| \Longrightarrow \text{ave}_{\tilde{J}}(g_1) \leq C_2 2^{n_2} |G_1|.$$
one direct geometric observation is that $|\mathcal{S}_{(iic)} \cap J| \leq \frac{1}{20}|J|$. In particular, suppose $y \in \mathcal{S}_{(iic)}$, then any $J_{y_0}$ with $\text{ave}_{J_{y_0}}(g_1) > C_2 2^{n_2}|G_1|$ has to contain the left endpoint or right endpoint of $J$, which we denote by $J_{\text{left}}$ and $J_{\text{right}}$. If $J_{\text{left}} \in J_{y_0}$, then the assumption that neither (iia) nor (iib) holds implies that 
$$|J_{y_0} \cap J| < \frac{1}{40} |J|,$$ 
and thus
$$|[J_{\text{left}}, y]| < \frac{1}{40}|J|.$$
Same implication holds true for $y \in \mathcal{S}_{(iic)}$ with $J_{\text{right}} \in J_{y_0}$. Therefore, for any $y \in \mathcal{S}_{(iic)}$, $|[J_{\text{left}}, y]| < \frac{1}{40}|J|$ or $|[y, J_{\text{right}}]| < \frac{1}{40}|J|$, which can be concluded as
$$\big|\mathcal{S}_{(iic)} \cap J\big|< \frac{1}{20}|J|.$$
Since $\big|\{Mg_1> C_2 2^{n_2}|G_1|\} \cap J\big| > \frac{1}{10}|J|$,
$$\bigg|\big(\{Mg_1> C_2 2^{n_2}|G_1|\} \setminus \mathcal{S}_{(iic)}\big) \cap J \bigg| > \frac{1}{20}|J|,$$
in which case one can apply the argument for (i) with $\{Mg_1> C_2 2^{n_2}|G_1|\}$ replaced by $\{Mg_1> C_2 2^{n_2}|G_1|\} \setminus \mathcal{S}_{(iic)}$ to conclude that $$Mg_1 > 2^{-5}C_2 2^{n_2} |G_1|.$$ This ends the proof for the claim.
\end{proof}
\begin{comment}
\begin{remark}
The non-negativity of the function is not only a sufficient condition as can be seen in the proof, but also a necessary condition. 
Let 
$$
g_1(y) = 
\begin{cases}
C_2 |G_2| \quad \quad \  \ \ \text{for}\ \ y \in [0,\frac{10}{11}]  \\
-10^5 C_2|G_2| \ \ \ \text{for}\ \ y \in (\frac{10}{11}, 1] \\ 
0 \quad \quad \quad \quad \quad \ \ \text{elsewhere}
\end{cases}
$$
Then clearly the dyadic interval
$$
[0,1] \in \mathcal{J}_{0}$$
However, 
$$
Mg_1(y) \leq C_22^{n_2}|G_1|
$$
for $y \in (\frac{10}{11}, 1]$. Such observation is important in the sense that the non-negativity is required for the application of the sparsity condition in Proposition \ref{sparsity} and its corollary - Proposition \ref{sp_2d} as specified below.
\end{remark}
\end{comment}

\begin{proposition}\label{sp_2d}
Given an arbitrary finite collection of dyadic rectangles $\mathcal{R}_0$. Define $\mathcal{J}:= \{J: I \times J \in \mathcal{R}_0 \}$. Suppose that $\displaystyle \mathcal{J} = \bigcup_{n_2 \in \mathbb{Z}} \mathcal{J}_{n_2}$ is a decomposition of dyadic intervals with respect to $Mg_1$ as specified in Section \ref{section_thm_haar_fixed_tensor} so that $\displaystyle \mathcal{R}_0 = \bigcup_{n_2 \in \mathbb{Z}} \bigcup_{\substack{R= I \times J \in \mathcal{R}_0 \\ J \in \mathcal{J}_{n_2} \\ }} R $ is a decomposition of dyadic rectangles in $\mathcal{R}_0$. Then
%Let $\displaystyle \mathcal{R} \subseteq \bigcup_{n}\bigcup_{n_2} \mathcal{I}^{-n-n_2} \times \mathcal{J}_{n_2}$ denote the ``tensor-type stopping-time decomposition I - level sets"  for $\mathcal{R}$ with respect to $Mf_1$ and $Mg_1$. Then for any fixed $n$,
$$
\sum_{n_2 \in \mathbb{Z}} \bigg|\bigcup_{\substack{R = I \times J \in \mathcal{R}_0 \\ J \in \mathcal{J}_{n_2}}}R\bigg| \lesssim \bigg|\bigcup_{R \in \mathcal{R}_0} R \bigg|.
%\bigg|\bigcup_{n_2 \in \mathbb{Z}}\bigcup_{\substack{R= I \times J \in \mahcal{R}_0 \\ J \in \mathcal{J}_{n_2}}} R \bigg|
$$
\end{proposition}
\begin{proof}[Proof of Proposition \ref{sp_2d}]
Proposition \ref{sparsity} gives a sparsity condition for intervals in the $y$-direction, which is sufficient to generate sparsity for dyadic rectangles in $\mathbb{R}^2$. In particular,
\begin{align*}
 \sum_{n_2 \in \mathbb{Z}} \bigg|\bigcup_{\substack{R= I \times J \in \mathcal{R}_0 \\ J \in \mathcal{J}_{n_2}}}R\bigg| =& \sum_{i = 0}^9 \sum_{n_2 \equiv i \ \ \text{mod} \ \ 10} \bigg|\bigcup_{\substack{R= I \times J \in \mathcal{R}_0 \\ J \in \mathcal{J}_{n_2}}}R\bigg| \nonumber \\
\lesssim & \sum_{i = 0}^9 \bigg|\bigcup_{n_2 \equiv i \ \ \text{mod} \ \ 10}\bigcup_{\substack{R= I \times J \in \mathcal{R}_0 \\ J \in \mathcal{J}_{n_2}}} R \bigg| \nonumber \\
\leq &10  \bigg|\bigcup_{n_2 \in \mathbb{Z}}\bigcup_{\substack{R= I \times J \in \mathcal{R}_0 \\ J \in \mathcal{J}_{n_2}}} R \bigg| \nonumber\\
= & 10 \big|\bigcup_{R \in \mathcal{R}_0} R \big|,
\end{align*}
where the second inequality follows from the sparsity condition in Proposition \ref{sparsity}. 
\end{proof}
\begin{remark}
The picture below illustrates from a geometric point of view why the two-dimensional sparsity condition (Proposition \ref{sp_2d}) follows naturally from the one-dimensional sparsity (Proposition \ref{sparsity}). In the figure, $A_1, A_2 \in \mathcal{I} \times \mathcal{J}_{n_2+20}$, $B \in \mathcal{I} \times \mathcal{J}_{n_2+10}$ and $C \in \mathcal{I} \times \mathcal{J}_{n_2}$ for some $n_2 \in \mathbb{Z}$.
\end{remark} %add pic!!!!!!!!!
\begin{comment}
\begin{center}
\begin{tikzpicture}
\draw[red] (0,0) -- (0,2) -- (2,2) -- (2,0) -- (0,0);
\draw[line width=0.4mm, dashed, blue] (1/2,-2) -- (1/2,4) -- (1,4) -- (1,-2) -- (1/2,-2);
%(1/2,-2) -- (1/2,4) -- (1,4) -- (1,-2) -- (1/2,-2);
%(-2,1/2) -- (4,1/2) -- (4,1) -- (-2,1) -- (-2,1/2);
\draw[line width=0.6mm]  (1/2,-2) -- (1,-2) -- (1,4) -- (1/2,4) -- (1/2,-2);
\draw[line width=0.6mm]  (5/4,-4) -- (3/2,-4) -- (3/2,4) -- (5/4,4) -- (5/4,-4);
\draw(0.75,3)node{$A_1$};
\draw(1.5,3)node{$A_2$};
\draw(1.75,1.5)node[red]{B};
\draw(3,0.7)node[blue]{C};
\end{tikzpicture}
\end{center}
\end{comment}

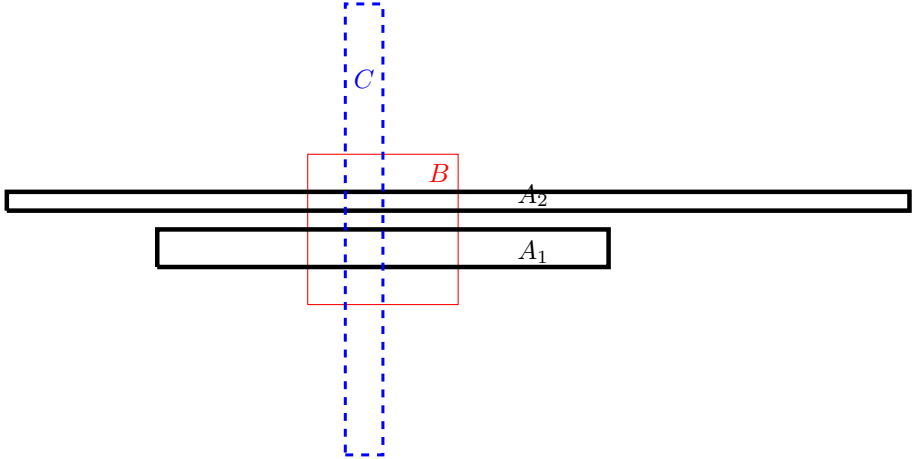
\begin{center}
\begin{tikzpicture}
\draw[red] (0,0) -- (2,0) -- (2,2) -- (0,2) -- (0,0);
\draw[line width=0.6mm]  (-2,1/2) -- (4,1/2) -- (4,1) -- (-2,1) -- (-2,1/2);
\draw[line width=0.6mm]  (-4,5/4) -- (8,5/4) -- (8,3/2) -- (-4,3/2) -- (-4,5/4);
\draw[line width=0.4mm, dashed, blue]  (1/2,-2) -- (1,-2) -- (1,4) -- (1/2,4) -- (1/2,-2);
%\draw[line width=0.6mm]  (5/4,-4) -- (3/2,-4) -- (3/2,4) -- (5/4,4) -- (5/4,-4);
\draw(0.75,3)node[blue]{$C$};
\draw(3,1.45)node{$A_2$};
\draw(1.75,1.75)node[red]{$B$};
\draw(3,0.7)node{$A_1$};
\end{tikzpicture}
\end{center}
\vskip .15in
\subsection{Summary of stopping-time decompositions} \label{section_thm_haar_fixed_summary}
%\begin{enumerate}[I.]
%\item
%Tensor-type stopping-time decomposition I
%\end{enumerate}
\ \ 
{\fontsize{10}{10}
\begin{center}
\begin{tabular}{ l l l }
I. Tensor-type stopping-time decomposition I on $\mathcal{I} \times \mathcal{J}$ & $\longrightarrow$ & $I \times J \in \mathcal{I}_{n_1,m_1} \times \mathcal{J}_{n_2,m_2}$ \\
 & & $(n_1 + n_2 < 0, m_1 + m_2 < 0, $ \\ 
  &  & $n_1 + m_2 < 0 , m_1 + n_2 < 0)$ \\  
II. General two-dimensional level sets stopping-time decomposition & $\longrightarrow$   & $I \times J \in  \mathcal{R}_{k_1,k_2} $  \\
\ \ \ \ on $\mathcal{I} \times \mathcal{J}$ & & $(k_1 <0, k_2 \leq K) $
\end{tabular}
\end{center}}
%Fix any $n_1, m_1, n_2, m_2, k_1$ and $k_2$, 
%$$
%I \times J \in \mathcal{}
%$$
\begin{comment}
\subsection{Hybrid of stopping-time decompositions}
\begin{center}
\begin{tabular}{ c c c }
Tensor-type stopping-time decomposition I & $\longrightarrow$ & $I \times J \in \mathcal{I}_{n_1,m_1} \times \mathcal{J}_{n_2,m_2}$ \\
on $\mathcal{I} \times \mathcal{J}$ & & $(n_1 + n_2 < 0, m_1 + m_2 < 0, $ \\ 
 $\Downarrow$ &  & $n_1 + m_2 < 0 , m_1 + n_2 < 0)$ \\  
 General two-dimensional level sets stopping-time decomposition& $\longrightarrow$ & $I \times J \in \mathcal{I}_{n_1,m_1} \times \mathcal{J}_{n_2,m_2} \cap \mathcal{R}_{k_1,k_2} $  \\
on $ \mathcal{I}_{n_1,m_1} \times \mathcal{J}_{n_2,m_2}$ & & $(n_1 + n_2 < 0, m_1 + m_2 < 0,$ \\
& & $n_1 + m_2 < 0 , m_1 + n_2 < 0, $ \\
& & $k_1 <0, k_2 \leq K) $
\end{tabular}
\end{center}
\end{comment}

\vskip .25in
\subsection{Application of stopping-time decompositions} \label{section_thm_haar_fixed_application_st}
With the stopping-time decompositions specified above, one can rewrite (\ref{haar_local}) as
{\fontsize{9.5}{9.5}\begin{align} \label{form_st_fixed}
%& |\Lambda^H_{\text{flag}^{\#1} \otimes \text{flag}^{\#2}}(f_1, f_2, g_1, g_2, h, \chi_{E'})| \nonumber \\
 %= 
 &\bigg|\displaystyle \sum_{\substack{n_1 + n_2 < 0 \\ m_1 + m_2 < 0 \\ n_1 + m_2 < 0 \\ m_1 + n_2 < 0 \\  k_1 < 0 \\ k_2 \leq K}} \sum_{\substack{I \times J \in \mathcal{I}_{n_1, m_1} \times \mathcal{J}_{n_2, m_2}  \\ \cap \mathcal{R}_{k_1,k_2} \\}} \frac{1}{|I|^{\frac{1}{2}} |J|^{\frac{1}{2}}} \langle B_I^{\#_1,H}(f_1,f_2),\vphi_I^{1,H} \rangle \langle \tilde{B}_J^{\#_2,H}(g_1,g_2),\vphi_J^{1,H} \rangle \langle h, \psi_I^{2} \otimes \psi_J^{2} \rangle \langle \chi_{E'},\psi_I^{3,H} \otimes \psi_J^{3,H} \rangle \bigg|  \nonumber \\
\leq &  \sum_{\substack{n_1 + n_2 < 0 \\ m_1 + m_2 < 0 \\ n_1 + m_2 < 0 \\ m_1 + n_2 < 0 \\  k_1 < 0 \\ k_2 \leq K}}  \sum_{\substack{I \times J \in \mathcal{I}_{n_1,m_1} \times \mathcal{J}_{n_2,m_2} \\ \cap \mathcal{R}_{k_1,k_2}\\ }}  \frac{|\langle B_I^{\#_1,H}(f_1,f_2),\vphi_I^{1,H} \rangle|}{|I|^{\frac{1}{2}}}   \frac{|\langle \tilde{B}_J^{\#_2,H}(g_1,g_2), \vphi_J^{1,H} \rangle|}{|J|^{\frac{1}{2}}} \cdot \frac{|\langle h, \psi_I^{2} \otimes \psi_J^{2} \rangle|}{|I|^{\frac{1}{2}}|J|^{\frac{1}{2}}} \frac{|\langle \chi_{E'},\psi_I^{3,H} \otimes \psi_J^{3,H} \rangle|}{|I|^{\frac{1}{2}}|J|^{\frac{1}{2}}} |I| |J|.\nonumber\\
\end{align}}

One recalls the \textit{general two-dimensional level sets stopping-time decomposition} that
$I \times J \in \mathcal{R}_{k_1,k_2}
$
only if 
$$ |I\times J \cap (\Omega^2_{k_1})^c |  \geq \frac{99}{100}|I\times J|$$
$$ |I\times J \cap (\Omega'^2_{k_2})^{c} |  \geq \frac{99}{100}|I \times J|$$
with $\Omega^2_{k_1} := \{SSh(x,y) > C_3 2^{k_1}\|h\|_{L^s(\mathbb{R}^2)}\}$, and $\Omega'^2_{k_2}:= \{(SS)^H\chi_{E'}(x,y) > C_3 2^{k_2}\}$.
As a result,
$$|I \times J| \sim |I \times J \cap (\Omega^2_{k_1})^c \cap (\Omega'^2_{k_2})^{c}|.$$ 
One can therefore majorize (\ref{form_st_fixed}) by
{\fontsize{10}{10}\begin{align}\label{form12}
%& |\Lambda^H_{\text{flag}^{\#1} \otimes \text{flag}^{\#2}}(f_1, f_2, g_1, g_2, h, \chi_{E'})|  \nonumber \\
%\lesssim 
&  \sum_{\substack{n_1 + n_2 < 0 \\ m_1 + m_2 < 0 \\ n_1 + m_2 < 0 \\ m_1 + n_2 < 0 \\  k_1 < 0 \\ k_2 \leq K}} \sum_{\substack{I \times J \in \mathcal{I}_{n_1,m_1} \times \mathcal{J}_{n_2,m_2} \cap\mathcal{R}_{k_1,k_2}}}  \frac{|\langle B_I^{\#_1,H}(f_1,f_2),\vphi_I^{1,H} \rangle|}{|I|^{\frac{1}{2}}}   \frac{|\langle \tilde{B}_J^{\#_2,H}(g_1,g_2), \vphi_J^{1,H} \rangle|}{|J|^{\frac{1}{2}}} \cdot \nonumber \\
&\quad \quad \quad \quad \quad \quad \quad \quad \quad \quad \quad \quad \quad  \quad \quad\frac{|\langle h, \psi_I^{2} \otimes \psi_J^{2} \rangle|}{|I|^{\frac{1}{2}}|J|^{\frac{1}{2}}} \frac{|\langle \chi_{E'},\psi_I^{3,H} \otimes \psi_J^{3,H} \rangle|}{|I|^{\frac{1}{2}}|J|^{\frac{1}{2}}} |I\times J \cap (\Omega^2_{k_1})^c \cap (\Omega'^2_{k_2})^c| \nonumber \\
\leq &   \sum_{\substack{n_1 + n_2 < 0 \\ m_1 + m_2 < 0 \\ n_1 + m_2 < 0 \\ m_1 + n_2 < 0 \\  k_1 < 0 \\ k_2 \leq K}} \displaystyle \sup_{I \in \mathcal{I}_{n_1,m_1}} \frac{|\langle B_I^{\#_1,H}(f_1,f_2),\vphi_I^{1,H} \rangle|}{|I|^{\frac{1}{2}}}  \sup_{J \in \mathcal{J}_{n_2,m_2}} \frac{|\langle \tilde{B}_J^{\#_2,H}(g_1,g_2),\vphi_J^{1,H} \rangle|}{|J|^{\frac{1}{2}}}\cdot \nonumber \\
&\quad \quad \quad \quad \quad \int_{(\Omega^2_{k_1})^c \cap (\Omega'^2_{k_2})^{c}} \sum_{\substack{I \times J \in \mathcal{I}_{n_1,m_1}  \times \mathcal{J}_{n_2,m_2} \cap \mathcal{R}_{k_1,k_2}}} \frac{|\langle h, \psi_I^{2} \otimes \psi_J^{2} \rangle|}{|I|^{\frac{1}{2}}|J|^{\frac{1}{2}}} \frac{|\langle \chi_{E'},\psi_I^{3,H} \otimes \psi_J^{3,H} \rangle|}{|I|^{\frac{1}{2}}|J|^{\frac{1}{2}}}\chi_{I}(x) \chi_{J}(y) dx dy. \nonumber\\
\end{align}}
We will now estimate each components in (\ref{form12}) separately for clarity. 
\subsubsection{Estimate for the $integral$} \label{section_thm_haar_fixed_est_integral}
One can apply the Cauchy-Schwarz inequality to the integrand and obtain
\begin{align} \label{integral12}
& \int_{(\Omega^2_{k_1})^c \cap (\Omega'^2_{k_2})^{c}} \sum_{\substack{I \times J \in \mathcal{I}_{n_1,m_1}  \times \mathcal{J}_{n_2,m_2} \cap \mathcal{R}_{k_1,k_2}}} \frac{|\langle h, \psi_I^{2} \otimes \psi_J^{2} \rangle|}{|I|^{\frac{1}{2}}|J|^{\frac{1}{2}}} \frac{|\langle \chi_{E'},\psi_I^{3,H} \otimes \psi_J^{3,H} \rangle|}{|I|^{\frac{1}{2}}|J|^{\frac{1}{2}}}\chi_{I}(x) \chi_{J}(y) dx dy \nonumber \\
\leq & \int_{(\Omega^2_{k_1})^c \cap (\Omega'^2_{k_2})^{c}} \bigg(\sum_{\substack{I \times J \in \mathcal{I}_{n_1,m_1}  \times \mathcal{J}_{n_2,m_2}\cap \mathcal{R}_{k_1,k_2}}}\frac{|\langle h, \psi_I^{2} \otimes \psi_J^{2} \rangle|^2}{|I||J|} \chi_I(x)\chi_J(y)\bigg)^{\frac{1}{2}} \nonumber \\
&\quad \quad \quad \quad \quad \quad \bigg(\sum_{\substack{I \times J \in \mathcal{I}_{n_1,m_1} \times \mathcal{J}_{n_2,m_2}\cap \mathcal{R}_{k_1,k_2}}}\frac{|\langle \chi_{E'},\psi_I^{3,H} \otimes \psi_J^{3,H} \rangle|^2}{|I||J|}\chi_{I}(x) \chi_{J}(y)\bigg)^{\frac{1}{2}} dxdy  \nonumber \\
\leq &\displaystyle \int_{(\Omega^2_{k_1})^c \cap (\Omega'^2_{k_2})^{c}} SSh(x,y) (SS)^H\chi_{E'}(x,y) \cdot \chi(\displaystyle\bigcup_{\substack{I \times J \in \mathcal{I}_{n_1,m_1}  \times \mathcal{J}_{n_2,m_2} \cap \mathcal{R}_{k_1,k_2}}}I \times J)(x,y) dxdy. 
\end{align}
Based on the \textit{general two-dimensional level sets stopping-time decomposition}, the hybrid functions have pointwise control on the domain for integration. In particular, for any $(x,y) \in (\Omega^2_{k_1})^c \cap (\Omega'^2_{k_2})^{c}$,
\begin{align*}
& SSh(x,y) \lesssim  C_3 2^{k_1} \|h\|_{L^s(\mathbb{R}^2)}, \nonumber \\
& (SS)^H\chi_{E'}(x,y) \lesssim  C_3 2^{k_2}.
\end{align*}
As a result, the integral can be estimated by
\begin{align}\label{h_integral}
 C_3^2 2^{k_1}\|h\|_{L^s(\mathbb{R}^2)} 2^{k_2} \bigg| \bigcup_{\substack{I \times J \in \mathcal{I}_{n_1,m_1}  \times \mathcal{J}_{n_2,m_2} \cap \mathcal{R}_{k_1,k_2}}}I \times J \bigg|.
\end{align}

\subsubsection{Estimate for $ \displaystyle \sup_{I \in \mathcal{I}_{n_1,m_1}} \frac{|\langle B_I^{\#_1,H}(f_1,f_2),\vphi_I^{1,H} \rangle|}{|I|^{\frac{1}{2}}} $ and $
\displaystyle \sup_{J \in \mathcal{J}_{n_2,m_2}} \frac{|\langle \tilde{B}_J^{\#_2,H}(g_1,g_2),\vphi_J^{1,H} \rangle|}{|J|^{\frac{1}{2}}}$}
One recalls the algorithm in the \textit{tensor type stopping-time decomposition I - level sets}, which incorporates the following information.
$$ I \in \mathcal{I}_{n_1, m_1}$$
implies that
$$|I \cap \{x: Mf_1(x) < C_1 2^{n_1}|F_1|\}| \geq \frac{9}{10}|I|,$$
$$|I \cap \{x: Mf_2(x) < C_1 2^{m_1}|F_2|\}| \geq \frac{9}{10}|I|,$$
which translates into
$$
I \cap \{x: Mf_1(x) < C_1 2^{n_1}|F_1|\} \cap \{x:Mf_2(x) < C_1 2^{m_1}|F_2|\} \neq \emptyset.
$$
Then one can recall Proposition \ref{size_cor} with 
\begin{equation*}
\mathcal{U}_{n_1,m_1}:= \{x: Mf_1 (x)< C_1 2^{n_1}|F_1|\} \cap  \{x:Mf_2 (x)< C_1 2^{m_1}|F_2|\}
\end{equation*} 
to estimate
\begin{comment}
apply Lemma \ref{B_size} with $S:  = \{ Mf_1 < C_1 2^{-n-n_2}|F_1|\} \cap  \{Mf_2 < C_1 2^{-m-m_2}|F_2|\}$ together with the estimates:
$$
\sup_{K \cap S \neq \emptyset} \frac{\langle f_1, \phi_K^1\rangle }{|K|^{\frac{1}{2}}} \leq C_1 2^{n_1}|F_1|
$$
$$
\sup_{K \cap S \neq \emptyset} \frac{\langle f_2, \phi_K^1\rangle }{|K|^{\frac{1}{2}}} \leq C_1 2^{m_1}|F_2|
$$
and the trivial estimates followed from the fact that $|f_i| \leq \chi_{F_i}$ for $ i = 1,2$:
$$
\sup_{K \cap S \neq \emptyset} \frac{\langle f_1, \phi_K^1\rangle }{|K|^{\frac{1}{2}}} \leq 1
$$
$$
\sup_{K \cap S \neq \emptyset} \frac{\langle f_2, \phi_K^1\rangle }{|K|^{\frac{1}{2}}} \leq 1
$$
One can derive that
\end{comment}
$$
\sup_{I \in \mathcal{I}_{n_1,m_1}} \frac{|\langle B_I^{\#_1,H}(f_1,f_2),\vphi_I^{1,H} \rangle|}{|I|^{\frac{1}{2}}} \lesssim C_1^2 (2^{n_1}|F_1|)^{\alpha_1} (2^{m_1}|F_2|)^{\beta_1},
$$
for any $0 \leq \alpha_1,\beta_1 \leq 1$.
Similarly, one can apply Proposition \ref{size_cor} with 
\begin{equation*}
\tilde{\mathcal{U}}_{n_2,m_2}:= \{y: Mg_1 (y)< C_2 2^{n_2}|G_1|\} \cap  \{y:Mg_2(y) < C_2 2^{m_2}|G_2|\}
\end{equation*}
to conclude that
$$
\sup_{J \in \mathcal{J}_{n_2,m_2}} \frac{|\langle \tilde{B}_J^{\#_2,H}(g_1,g_2),\vphi_J^{1,H} \rangle|}{|J|^{\frac{1}{2}}} \lesssim C_2^2 (2^{n_2}|G_1| )^{\alpha_2}(2^{m_2}|G_2|)^{\beta_2},
$$
for any $0 \leq \alpha_2,\beta_2 \leq 1$. 
By choosing 
$$\alpha_1 = \frac{1}{p_1},
\beta_1 = \frac{1}{q_1},
\alpha_2 = \frac{1}{p_2},
\beta_2 = \frac{1}{q_2},
$$
and applying (\ref{h_integral}), (\ref{form12}) can therefore be estimated by
\begin{align} \label{linear_form_fixed_scale}
%& |\Lambda^H_{\text{flag}^{\#1} \otimes \text{flag}^{\#2}}(f_1, f_2, g_1, g_2, h, \chi_{E'})| \nonumber \\
%\lesssim 
& C_1^2 C_2^2 C_3^2\sum_{\substack{n_1 + n_2  < 0 \\ m_1 + m_2 < 0 \\ n_1 + m_2 < 0 \\ m_1 + n_2 < 0 \\ k_1 < 0 \\ k_2 \leq K}} 2^{n_1 \frac{1}{p_1}}2^{m_1\frac{1}{q_1}}2^{n_2 \frac{1}{p_2}} 2^{m_2 \frac{1}{q_2}}|F_1|^{\frac{1}{p_1}}|F_2|^{\frac{1}{q_1}}|G_1|^{\frac{1}{p_2}}|G_2|^{\frac{1}{q_2}}\cdot  2^{k_1} \| h \|_{L^s} 2^{k_2}  \cdot \bigg|\bigcup_{\substack{ R \in \mathcal{I}_{n_1,m_1} \times \mathcal{J}_{n_2,m_2} \cap  \mathcal{R}_{k_1,k_2} }} R\bigg| . 
\end{align}
One recalls that 
$$
\frac{1}{p_1} + \frac{1}{q_1} = \frac{1}{p_2} + \frac{1}{q_2},
$$
then 
\begin{align} \label{exp_size}
2^{n_1 \frac{1}{p_1}}2^{m_1\frac{1}{q_1}}2^{n_2 \frac{1}{p_2}} 2^{m_2 \frac{1}{q_2}}  = & 2^{n_1\frac{1}{p_2}} 2^{n_1(\frac{1}{q_2} - \frac{1}{q_1})}2^{m_1 \frac{1}{q_1}} 2^{n_2\frac{1}{p_2}}2^{m_2(\frac{1}{q_2} - \frac{1}{q_1})} 2^{m_2\frac{1}{q_1}} \nonumber \\
= & (2^{n_1 + n_2})^{\frac{1}{p_2}} (2^{n_1+m_2})^{\frac{1}{q_2} - \frac{1}{q_1}}(2^{m_1+m_2})^{\frac{1}{q_1}}.
\end{align}
By the definition of exceptional sets, 
$
2^{n_1 + n_2} \lesssim 1, 2^{n_1 + m_2} \lesssim 1, 2^{m_1 + n_2} \lesssim 1, 2^{m_1 + m_2} \lesssim 1
$.
Then
$$
n := -(n_1 + n_2) \geq 0,
$$
$$
m := -(m_1 + m_2) \geq 0.
$$
Without loss of generality, one further assumes that $\frac{1}{q_2} \geq \frac{1}{q_1}$ (with $q_1$ and $q_2$ swapped in the opposite case), which implies that
$$
(2^{n_1+m_2})^{\frac{1}{q_2}- \frac{1}{q_1}} \lesssim 1.
$$
Now (\ref{linear_form_fixed_scale}) can be bounded by
\begin{align} \label{linear_almost}
%& |\Lambda^H_{\text{flag}^{\#1} \otimes \text{flag}^{\#2}}(f_1, f_2, g_1, g_2, h, \chi_{E'})| \nonumber \\
% \lesssim 
 & C_1^2 C_2^2 C_3^2\sum_{\substack{n > 0 \\ m > 0 \\  k_1 < 0 \\ k_2 \leq K}} \sum_{\substack{n_2 \in \mathbb{Z} \\ m_2 \in \mathbb{Z}}} 2^{-n \frac{1}{p_2}}2^{-m \frac{1}{q_1}}|F_1|^{\frac{1}{p_1}}|F_2|^{\frac{1}{q_1}}|G_1|^{\frac{1}{p_2}}|G_2|^{\frac{1}{q_2}}\cdot  2^{k_1} \| h \|_{L^s} 2^{k_2}  \cdot \bigg|\bigcup_{\substack{R \in \mathcal{I}_{-n-n_2,-m-m_2} \times \mathcal{J}_{n_2,m_2} \cap \mathcal{R}_{k_1,k_2} \\ }} R\bigg|. 
\end{align}
With $k_1, k_2, n, m$ fixed, one can apply the sparsity condition (Proposition \ref{sp_2d}) repeatedly and obtain the following bound for the expression
{\fontsize{8.5}{8.5}
\begin{align} \label{nested_area}
\sum_{\substack{n_2 \in \mathbb{Z} \\ m_2 \in \mathbb{Z}}} \bigg|\bigcup_{\substack{R\in \mathcal{R}_{k_1,k_2}\cap \mathcal{I}_{-n-n_2,-m-m_2} \times \mathcal{J}_{n_2,m_2}}} R \bigg| 
\lesssim  \sum_{m_2 \in \mathbb{Z}} \bigg| \bigcup_{\substack{R \in \mathcal{R}_{k_1,k_2} \cap \mathcal{I}_{-m-m_2} \times \mathcal{J}_{m_2}}} R \bigg| \lesssim& \bigg|\bigcup_{R\in \mathcal{R}_{k_1,k_2}} R\bigg| \leq  \min\left( \big|\bigcup_{R\in \mathcal{R}_{k_1}} R\big|, \big|\bigcup_{R\in \mathcal{R}_{k_2}} R\big|\right).
%\lesssim& \bigg|\bigcup_{R\in \mathcal{R}_{k_1,k_2}} R\bigg| \leq  \min\left( \big|\bigcup_{R\in \mathcal{R}_{k_1}} R\big|, \big|\bigcup_{R\in \mathcal{R}_{k_2}} R\big|\right).
\end{align}}
\begin{remark}
The arbitrariness of the collection of rectangles in Proposition \ref{sp_2d} provides the compatibility of different stopping-time decompositions. In the current setting, the notation $\mathcal{R}_0$ in Proposition \ref{sp_2d} is chosen to be $\mathcal{R}_{k_1, k_2}$. The sparsity condition allows one to combine the \textit{tensor-type stopping-time decomposition I} and \textit{general two-dimensional level sets stopping-time decomposition} and to obtain information from both stopping-time decompositions.
\end{remark}
\begin{remark}
The readers who are familiar with the proof of single-parameter paraproducts \cite{cw} or bi-parameter paraproducts \cite{cptt}, \cite{cw} might recall that (\ref{nested_area}) employs a different argument from the previous ones \cite{cptt}, \cite{cw}. In particular, by previous reasonings, one would fix $n_2, m_2 \in \mathbb{Z}$ and obtain 
\begin{equation} \label{old}
\bigg|\bigcup_{\substack{R\in \mathcal{R}_{k_1,k_2} \cap \mathcal{I}_{-n-n_2,-m-m_2} \times \mathcal{J}_{n_2,m_2}}} R \bigg| \lesssim \min \left( \big|\bigcup_{R\in \mathcal{R}_{k_1}} R\big|, \big|\bigcup_{R\in \mathcal{R}_{k_2}} R\big| \right).
\end{equation}
However, the expression on the right hand side of (\ref{old}) is independent of $n_2$ or $m_2$, which gives a divergent series when the sum is taken over all $n_2, m_2 \in \mathbb{Z}$. This explains the novelty and necessity of the sparsity condition (Proposition \ref{sp_2d}) for our argument.
\end{remark}
To estimate the right hand side of (\ref{nested_area}), one recalls from the \textit{general two-dimensional level sets stopping-time decomposition} that $R \in \mathcal{R}_{k_1}$ implies 
$$\big|R \cap \Omega^2_{k_1-1} \big| > \frac{1}{100}|R|,$$
or equivalently
 $$\displaystyle \bigcup_{R\in \mathcal{R}_{k_1}} R \subseteq \{(x,y)\in \mathbb{R}^2: MM (\chi_{\Omega^2_{k_1-1}})(x,y) > \frac{1}{100}\}.$$
As a result,
\begin{align} \label{rec_area_1}
\bigg|\bigcup_{R\in \mathcal{R}_{k_1}} R\bigg| \leq & \big|\{(x,y): MM (\chi_{\Omega^2_{k_1-1}})(x,y) > \frac{1}{100}\}\big| \lesssim |\Omega^2_{k_1-1}|=|\{(x,y): SSh(x,y) > C_3 2^{k_1} \|h\|_{L^s(\mathbb{R}^2)}\}| \lesssim C_3^{-s}2^{-k_1s},
\end{align}
where the second and last inequalities follow from the boundedness of the double maximal function and the double square function described in Theorem \ref{maximal-square}.
By a similar reasoning and the fact that $|E'| \sim 1$,
\begin{align} \label{rec_area_2}
\bigg|\bigcup_{R\in \mathcal{R}_{k_2}} R\bigg| 
%\leq & \big|\{(x,y): M (\chi_{\Omega'^2_{k_2-1}})(x,y) > \frac{1}{100}\}\big| 
\lesssim 
%|\Omega'^2_{k_2-1}|=
|\{(x,y): (SS)^H(\chi_{E'})(x,y) > C_3 2^{k_2}\}| \lesssim C_3^{-\gamma}2^{-k_2\gamma},
\end{align}
for any $\gamma >1$. Interpolation between (\ref{rec_area_1}) and (\ref{rec_area_2}) yields
\begin{equation} \label{int_area}
\bigg|\bigcup_{R\in \mathcal{R}_{k_1,k_2}} R\bigg| \lesssim 2^{-\frac{k_1s}{2}}2^{-\frac{k_2\gamma}{2}},
\end{equation}
and by plugging (\ref{int_area}) into (\ref{nested_area}), one has
\begin{equation} \label{rec_area_hybrid}
\sum_{\substack{n_2 \in \mathbb{Z} \\ m_2 \in \mathbb{Z}}} \bigg|\bigcup_{\substack{R\in \mathcal{R}_{k_1,k_2} \cap\mathcal{I}_{-n-n_2,-m-m_2} \times \mathcal{J}_{n_2,m_2}}} R \bigg| \lesssim 2^{-\frac{k_1s}{2}}2^{-\frac{k_2\gamma}{2}},
\end{equation}
for any $\gamma >1$.
One combines the estimates (\ref{rec_area_hybrid}) and (\ref{linear_almost}) to obtain
\begin{align*}
|\Lambda^H_{\text{flag}^{\#1} \otimes \text{flag}^{\#2}}(f_1, f_2, g_1, g_2, h, \chi_{E'})| \lesssim &C_1^2 C_2^2 C_3^2 \sum_{\substack{n > 0 \\ m > 0 \\ k_1 < 0 \\ k_2 \leq K}} 2^{-n \frac{1}{p_2}}2^{-m \frac{1}{q_1}}|F_1|^{\frac{1}{p_1}}|F_2|^{\frac{1}{p_2}}|G_1|^{\frac{1}{q_1}}|G_2|^{\frac{1}{q_2}}\cdot 2^{k_1(1-\frac{s}{2})} \| h \|_{L^s} 2^{k_2(1-\frac{\gamma}{2})}.  \nonumber 
\end{align*}
The geometric series $\displaystyle \sum_{k_1<0}2^{k_1(1-\frac{s}{2})}$ is convergent given that $s <2$. For $\displaystyle\sum_{k_2 \leq K}2^{k_2(1-\frac{\gamma}{2})}$, one can choose $\gamma >1$ to be sufficiently large for the range $0 \leq k_2 \leq K$ and $\gamma >1$ and close to $1$ for $k_2 <0$. One thus concludes that
\begin{align*}
|\Lambda^H_{\text{flag}^{\#1} \otimes \text{flag}^{\#2}}(f_1, f_2, g_1, g_2, h, \chi_{E'})| \lesssim &C_1^2 C_2^2 C_3^2 |F_1|^{\frac{1}{p_1}}|F_2|^{\frac{1}{p_2}}|G_1|^{\frac{1}{q_1}}|G_2|^{\frac{1}{q_2}}\| h \|_{L^s}.
\end{align*}

\begin{remark}
One important observation is that thanks to Lemma \ref{B_size}, the sizes
$$
\sup_{I \in \mathcal{I}_{n_1,m_1}} \frac{|\langle B_I^{\#_1,H}(f_1,f_2),\vphi_I^{1,H} \rangle|}{|I|^{\frac{1}{2}}}
$$
and 
$$
\sup_{J \in \mathcal{J}_{n_2,m_2}} \frac{|\langle \tilde{B}_J^{\#_2,H}(g_1,g_2),\vphi_J^{1,H} \rangle|}{|J|^{\frac{1}{2}}}
$$
can be estimated in the exactly same way as
$$
\text{size}_{\mathcal{I}_{n_1}}\big( (f_1,\phi_I)_I \big) \cdot \text{size}_{\mathcal{I}_{m_1}}\big( (f_2,\phi_I)_I \big)
$$
and 
$$
\text{size}_{\mathcal{J}_{n_2}}\big( (g_1,\phi_J)_J \big) \cdot \text{size}_{\mathcal{J}_{m_2}}\big( (g_2,\phi_J)_J \big)
$$
respectively.
Based on this observation, it is not difficult to verify that under the Haar assumption, the discrete models $\Pi_{\text{flag}^{\#1} \otimes \text{paraproduct}}$ and $\Pi_{\text{paraproduct}\otimes \text{paraproduct}}$
% in the Haar case 
can be estimated by an essentially same argument as $\Pi_{\text{flag}^{\#_1}\otimes \text{flag}^{\#_2}}$ while $\Pi_{\text{flag}^0 \otimes \text{flag}^{\#_2}}$ can be studied similarly as $\Pi_{\text{flag}^0 \otimes \text{paraproduct}}$.
\end{remark}
%\end{proof}
\vskip .15in
\section{Proof of Theorem \ref{thm_weak_mod} for $\Pi_{\text{flag}^0 \otimes \text{flag}^0}$ - Haar Model} \label{section_thm_haar}
\begin{comment}
$$\Pi (f_1 \otimes g_1, f_2 \otimes g_2, h) := \displaystyle \sum_{I \times J \in \mathcal{I} \times \mathcal{J}} \frac{1}{|I|^{\frac{1}{2}} |J|^{\frac{1}{2}}} \langle B_I^H(f_1,f_2),\vphi_I^1 \rangle \langle \tilde{B_J}(g_1, g_2), \vphi_J^1 \rangle  \langle h, \psi_I^{2} \otimes \psi_J^{2} \rangle \psi_I^{3,H} \otimes \psi_J^{3,H}$$ 
where 
$$B_I^H(f_1,f_2)(x) := \displaystyle \sum_{K \in \mathcal{K}:|K| \geq |I|} \frac{1}{|K|^{\frac{1}{2}}}\langle f_1, \vphi_K^1 \rangle \langle f_2, \psi_K^2 \rangle \psi_K^3(x),$$
$$\tilde{B}_J(g_1,g_2)(y) := \displaystyle \sum_{L \in \mathcal{L}:|L| \geq |J|} \frac{1}{|L|^{\frac{1}{2}}}\langle g_1, \vphi_L^1 \rangle \langle g_2, \psi_L^2 \rangle \psi_L^3(y).$$
\end{comment}

\begin{comment}
Then for $ 1 < p_i, q_i < \infty$, $\displaystyle \sum_{i=1}^{3}\frac{1}{p_i} = \sum_{i=1}^{3}\frac{1}{q_i} = \frac{1}{r}$,
$$\Pi: L_x^{p_1}(L_y^{q_1}) \times L_x^{p_2}(L_y^{q_2}) \times L_x^{p_3}(L_y^{q_3}) \rightarrow L^r.$$
\end{comment}
We would now derive estimates for the operator $\Pi_{\text{flag}^0 \otimes \text{flag}^0}$ in the Haar model, namely
\begin{equation} \label{Pi_larger_haar}
\Pi^H_{\text{flag}^0 \otimes \text{flag}^0}(f_1,f_2,g_1,g_2,h):= \sum_{I \times J \in \mathcal{R}} \frac{1}{|I|^{\frac{1}{2}} |J|^{\frac{1}{2}}} \langle B_I^H(f_1,f_2),\vphi_I^{1,H} \rangle \langle \tilde{B}_J^H(g_1,g_2),\vphi_J^{1,H} \rangle \langle h, \psi_I^{2} \otimes \psi_J^{2} \rangle \psi_I^{3,H} \otimes \psi_J^{3,H}.
\end{equation}
where $B_I^H$ and $B^H_J$ are defined in (\ref{B_0_local_haar_simplified}) and (\ref{B_local_definition_haar}). 
The argument in Section \ref{section_thm_haar_fixed} is not sufficient for $\Pi^H_{\text{flag}^0 \otimes \text{flag}^0}$ (\ref{Pi_larger_haar}) because the localized size
\begin{align*}
& \sup_{I \cap S \neq \emptyset} \frac{|\langle B_I^H(f_1,f_2), \vphi^{1,H}_I \rangle |}{|I|^{\frac{1}{2}}}, \nonumber \\
& \sup_{J \cap S' \neq \emptyset} \frac{|\langle \tilde{B}_J^H(g_1,g_2), \vphi^{1,H}_J \rangle }{|J|^{\frac{1}{2}}}
\end{align*}
cannot be controlled without information about corresponding level sets. In particular, one needs to impose the additional assumption that 
\begin{align*}
&I \cap \{x: MB^H(f_1,f_2)(x) \leq C_1 2^{l_1}\|B^H(f_1,f_2)\|_1\} \neq \emptyset, \nonumber \\
&J \cap \{y: M\tilde{B}^H(g_1,g_2)(y) \leq C_2 2^{l_2}\|\tilde{B}^H(g_1,g_2)\|_1\} \neq \emptyset,
\end{align*}
where $B^H$ and $\tilde{B}^H$ are defined in (\ref{B_global_haar_simplified}) and (\ref{B_global_haar}).
%\begin{align*}
%& B^H (f_1, f_2)(x):= \sum_{K \in \mathcal{K}} \frac{1}{|K|^{\frac{1}{2}}} \langle f_1, \phi_K^1\rangle \langle f_2, \phi_K^2\rangle \phi_{K}^{3,H}(x), \nonumber \\
%&\tilde{B}^H (g_1, g_2)(y):= \sum_{L \in \mathcal{L}} \frac{1}{|L|^{\frac{1}{2}}} \langle g_1, \phi_L^1\rangle \langle %g_2, \phi_L^2\rangle \phi_{L}^{3,H}(y).
%\end{align*}
%One notices that the above stopping-time algorithm differs from the one specified in Section 6 in the sense that a new type of stopping-time decomposition, namely tensor-type stopping-time decomposition II is involved.  Such stopping-time decomposition is necessary 
%or the more precise information attainable from the tensor-type stopping-time decomposition II. 
However, while the sizes of $B^H(f_1,f_2)$ and $\tilde{B}^H(g_1,g_2)$ can be controlled in this way, they lose the information from the localization (e.g. $K \cap \{x: Mf_1(x) \leq C_1 2^{n_1}|F_1|\} \neq \emptyset$ for some $n_1 \in \mathbb{Z}$) and are thus far away from satisfaction. It is indeed the energies which capture such local information and compensate for the loss from size estimates in this scenario.
\subsection{Localization} \label{section_thm_haar_localization}
As one would expect from the definition of the exceptional set, the \textit{tensor-type stopping-time decompositions} and the \textit{general two-dimensional level sets stopping-time decomposition} are involved in the argument.
We define the set 
$$\Omega := \Omega^1 \cup \Omega^2,$$ where
\begin{align*}
\displaystyle \Omega^1 := &\bigcup_{n_1 \in \mathbb{Z}}\{x:Mf_1(x) > C_1 2^{n_1}|F_1|\} \times \{y:Mg_1(y) > C_2 2^{-n_1}|G_1|\}\cup \nonumber \\
& \bigcup_{m_1 \in \mathbb{Z}}\{x: Mf_2(x) > C_1 2^{m_1}|F_2|\} \times \{y:Mg_2(y) > C_2 2^{-m_1}|G_2|\}\cup \nonumber \\
&\bigcup_{l_1 \in \mathbb{Z}} \{x: MB^H(f_1,f_2)(x) > C_1 2^{l_1}\| B^H(f_1,f_2)\|_1\} \times \{y: M\tilde{B}^H(g_1,g_2)(y) > C_2 2^{-l_1}\| \tilde{B}^H(g_1,g_2) \|_1\},\nonumber \\
%&\bigcup_{l_2 \in \mathbb{Z}} \{MB > C_1 2^{l_1}\| B\|_1\} \times \{Mg_1 > C_2 2^{-l_1}|G_1|\}\cup \nonumber \\
%&\bigcup_{l_3 \in \mathbb{Z}} \{MB > C_1 2^{l_2}\| B\|_1\} \times \{Mg_2 > C_2 2^{-l_2}|G_2|\};\nonumber \\
\Omega^2 := & \{(x,y) \in \mathbb{R}^2: SSh(x,y) > C_3 \|h\|_{L^s(\mathbb{R}^2)}\}, \nonumber \\
\end{align*}
and
$$Enl(\Omega) := \{(x,y)\in \mathbb{R}^2: MM\chi_{\Omega}(x,y) > \frac{1}{100}\}.$$
Let 
$$E' := E \setminus Enl(\Omega).$$ Then the argument in Remark \ref{subset} yields that $|E'| \sim |E|$ where $|E|$ can be assumed to be 1 by scaling invariance. We aim to prove that the multilinear form
\begin{equation}
\Lambda^H_{\text{flag}^{0} \otimes \text{flag}^{0}}(f_1, f_2, g_1, g_2, h, \chi_{E'}) := \langle \Pi^H_{\text{flag}^{0} \otimes \text{flag}^{0}}(f_1, f_2, g_1, g_2, h), \chi_{E'} \rangle
\end{equation}
satisfies the following restricted weak-type estimate
\begin{equation}
|\Lambda^H_{\text{flag}^{0} \otimes \text{flag}^{0}}(f_1, f_2, g_1, g_2, h, \chi_{E'})| \lesssim |F_1|^{\frac{1}{p_1}} |G_1|^{\frac{1}{p_2}} |F_2|^{\frac{1}{q_1}} |G_2|^{\frac{1}{q_2}} \|h\|_{L^{s}(\mathbb{R}^2)}.
\end{equation}
The localization argument in Remark \ref{localization_haar_fixed} can be applied so that 
\begin{align} \label{form_localized}
& |\Lambda^H_{\text{flag}^{0} \otimes \text{flag}^{0}}(f_1, f_2, g_1, g_2, h, \chi_{E'})| \nonumber\\
= &\displaystyle \sum_{\substack{I \times J \in \mathcal{R} \\ I \times J \cap Enl(\Omega)^c \neq \emptyset}} \frac{1}{|I|^{\frac{1}{2}} |J|^{\frac{1}{2}}} \langle B^{H}_I(f_1,f_2),\vphi_I^{1,H} \rangle \langle \tilde{B}^{H}_J(g_1, g_2), \vphi_J^{1,H} \rangle  \langle h, \psi_I^{2} \otimes \psi_J^{2} \rangle \langle \chi_{E'}, \psi_I^{3,H} \otimes \psi_J^{3,H} \rangle.
\end{align}
\vskip .15in
\subsection{Tensor-type stopping-time decomposition II - maximal intervals} \label{section_thm_haar_tensor}
\subsubsection{One-dimensional stopping-time decomposition - maximal intervals}
\begin{comment}
A one-dimensional stopping-time decomposition described in Section 5.4.1 can be performed to the sequence $(\langle B_I^H, \vphi_I\rangle)_{I}$. We denote the decomposition as
$$\displaystyle \mathcal{I} = \bigcup_{l_1}\bigcup_{T \in \mathbb{T}_{l_1}}T$$
where 
$$T := \{I \in \mathcal{I}: I \subseteq I_{T}\}$$ 
with $I_T$ being the corresponding tree-top. 
Similarly, 
\end{comment}
One applies the stopping-time decomposition described in Section \ref{section_size_energy_one_dim_st_maximal} to the sequences
$$
\big(\frac{|\langle B^H_{I}(f_1,f_2), \vphi^{1,H}_I \rangle|}{|I|^{\frac{1}{2}}}\big)_{I \in \mathcal{I}}
$$
and
$$
\big(\frac{|\langle \tilde {B}^H_{J}(g_1,g_2), \vphi^{1,H}_J \rangle|}{|J|^{\frac{1}{2}}}\big)_{J \in \mathcal{J}}
$$
%Let $\mathbb{T}_{l_1} := \{T  \} $
We will briefly recall the algorithm and introduce some necessary notations for the sake of clarity. Since $\mathcal{I}$ is finite, there exists some $L_1 \in \mathbb{Z}$ such that for any $I \in \mathcal{I}$, $\frac{|\langle B^H_{I}(f_1,f_2), \vphi^{1,H}_I \rangle|}{|I|^{\frac{1}{2}}} \leq C_1 2^{L_1} \|B^H(f_1,f_2)\|_1$. There exists a largest interval $I_{\text{max}}$ such that 
$$\frac{|\langle B^H_{I_{\text{max}}}(f_1,f_2), \vphi^{1,H}_{I_{\text{max}}} \rangle|}{|I_{\text{max}}|^{\frac{1}{2}}} \geq C_1 2^{L_1-1}\|B^H(f_1,f_2)\|_1.$$
Then we define a \textit{tree}
$$T := \{I \in \mathcal{I}: I \subseteq I_{\text{max}}\},$$
and the corresponding \textit{tree-top} 
$$I_T := I_{\text{max}}.$$  
Now we repeat the above step on $\mathcal{I} \setminus T$ to choose maximal intervals and collect their subintervals in their corresponding sets, which will end thanks to the finiteness of $\mathcal{I}$. Then collect all $T$'s in a set $\mathbb{T}_{L_1-1}$ and repeat the above algorithm to $\displaystyle \mathcal{I} \setminus \bigcup_{T \in \mathbb{T}_{L_1-1}} T$. Eventually the algorithm generates a decomposition 
$$\displaystyle \mathcal{I} = \bigcup_{l_1}\bigcup_{T \in \mathbb{T}_{l_1}}T.$$ 
One simple observation is that the above procedure can be applied to general sequences indexed by dyadic intervals. One can thus apply the same algorithm to $\mathcal{J} := \{J: I \times J \in \mathcal{R}\}$. We denote the decomposition as 
$$\displaystyle \mathcal{J} = \bigcup_{l_2}\bigcup_{S \in \mathbb{S}_{l_2}}S$$ 
with respect to the sequence 
$$\big(\frac{|\langle \tilde {B}^H_{J}(g_1,g_2), \vphi^{1,H}_J \rangle|}{|J|^{\frac{1}{2}}}\big)_{J \in \mathcal{J}},$$ where $S$ is a collection of dyadic intervals analogous to $T$ and is denoted by \textit{tree}. And $J_S$ represents the corresponding \textit{tree-top} analogous to $I_T$.

\vskip.15in
\subsubsection{Tensor product of two one-dimensional stopping-time decompositions - maximal intervals}\label{section_thm_haar_tensor_1d_maximal}
\begin{obs} \label{obs_st_B}
If $I \times J  \cap Enl(\Omega)^{c} \neq \emptyset$ and $I \times J \in T \times S$ with $T \in \mathbb{T}_{l_1}$ and $S \in \mathbb{S}_{l_2}$, then $l_1, l_2 \in \mathbb{Z}$ satisfies $l_1 + l_2 < 0$. Equivalently, $I \times J \in T \times S$ with $T \in \mathbb{T}_{-l - l_2}$ and $S \in \mathbb{S}_{l_2}$ for some $l_2 \in \mathbb{Z}$, $l> 0$.

\begin{proof}
$I \in T$ with $T \in \mathbb{T}_{l_1}$ means that $I \subseteq I_T$ where $\frac{|\langle B^H_{I_T}(f_1,f_2), \vphi^{1,H}_{I_T} \rangle|}{|I_T|^{\frac{1}{2}}} > C_12^{l_1} \|B^H(f_1,f_2)\|_1$. By the biest trick (Remark \ref{biest_trick_rmk}),
$$\frac{|\langle B^H_{I_T}(f_1,f_2), \vphi^{1,H}_{I_T} \rangle|}{|I_T|^{\frac{1}{2}}} = \frac{|\langle B^H(f_1,f_2), \vphi^{1,H}_{I_T} \rangle|}{|I_T|^{\frac{1}{2}}} \leq MB^H(f_1,f_2)(x)$$ 
for any $x \in I_T$. Thus 
$$I_T \subseteq \{x: MB^H(f_1,f_2)(x) > C_12^{l_1} \|B^H(f_1,f_2)\|_1\}.$$ By a similar reasoning, $J \in S$ with $S \in \mathbb{S}_{l_2}$ implies that 
$$J \subseteq J_S \subseteq  \{y: M\tilde{B}^H(g_1,g_2) (y)> C_22^{l_2} \|\tilde{B}^H(g_1,g_2)\|_1\}.$$
If $l_1 + l_2 \geq 0$, then 
$$\{x: MB^H(f_1,f_2)(x) > C_12^{l_1} \|B^H(f_1,f_2)\|_1\} \times \{y: M\tilde{B}^H(g_1,g_2)(y) > C_2 2^{l_2}\| \tilde{B}^H(g_1,g_2)\|_1\} \subseteq \Omega^1 \subseteq \Omega.$$ As a consequence, $I \times J \subseteq \Omega \subseteq Enl(\Omega)$, which is a contradiction. 
\end{proof}
\end{obs}
\vskip .15in

\subsection{Summary of stopping-time decompositions} \label{section_thm_haar_summary}
The notions of \textit{tensor-type stopping-time decomposition I} and \textit{general two-dimensional level sets stopping-time decomposition} introduced in Section \ref{section_thm_haar_fixed} will be applied without further specifications.
%\begin{center}
%\begin{tabular}{ c c c }
%One-dimensional stopping-time decomposition& $\longrightarrow$ & $K \in \mathcal{K}_{n_0}$ \\
%on $\mathcal{K}$ & & $(n_0 \in \mathbb{Z})$ \\
%\end{tabular}
%\end{center}
%\begin{center}
%\begin{tabular}{ c c c }
%One-dimensional stopping-time decomposition& $\longrightarrow$ & $L \in \mathcal{L}_{n'_0}$ \\
%on $\mathcal{L}$ & & $(n'_0 \in \mathbb{Z})$ \\
%\end{tabular}
%\end{center}
{\fontsize{9.5}{9.5}
\begin{center}
\begin{tabular}{ l l l }	
I. Tensor-type stopping-time decomposition I on $\mathcal{I} \times \mathcal{J}$& $\longrightarrow$ & $I \times J \in \mathcal{I}_{-n-n_2,-m-m_2} \times \mathcal{J}_{n_2,m_2}$ \\
& & $(n_2, m_2 \in \mathbb{Z}, n > 0, m>0)$\\
II. Tensor-type stopping-time decomposition II on $\mathcal{I} \times \mathcal{J}$ & $\longrightarrow$ & $I \times J \in T \times  S$,with $T \in \mathbb{T}_{-l-l_2}$, $S \in \mathbb{S}_{l_2}$\\
& & $(l_2 \in \mathbb{Z}, l> 0)$\\
III. General two-dimensional level sets stopping-time decomposition& $\longrightarrow$ & $I \times J \in \mathcal{R}_{k_1,k_2} $  \\
\ \ \ \ \ on $\mathcal{I} \times \mathcal{J}$& & $(k_1 <0, k_2 \leq K)$\\
\end{tabular}
\end{center}}

\begin{comment}
\subsection{Hybrid of stopping-time decompositions}
\begin{center}
\begin{tabular}{ c c c }
One-dimensional stopping-time decomposition& $\longrightarrow$ & $K \in \mathcal{K}_{n_0}$ \\
on $\mathcal{K}$ & & $(n_0 \in \mathbb{Z})$ \\
\end{tabular}
\end{center}

\begin{center}
\begin{tabular}{ c c c }
One-dimensional stopping-time decomposition& $\longrightarrow$ & $L \in \mathcal{L}_{n'_0}$ \\
on $\mathcal{L}$ & & $(n'_0 \in \mathbb{Z})$ \\
\end{tabular}
\end{center}

\begin{table}[h!]
\begin{tabular}{ c c c }	
Tensor-type stopping-time decomposition I& $\longrightarrow$ & $I \times J \in \mathcal{I}_{-n-n_2,-m-m_2} \times \mathcal{J}_{n_2,m_2}$ \\
on $\mathcal{I} \times \mathcal{J}$ && \\
& & $(n_2, m_2 \in \mathbb{Z}, n > 0)$\\
$\Downarrow$ & & \\
General two-dimensional level sets stopping-time decomposition& $\longrightarrow$ & $I \times J \in \mathcal{I}_{-n-n_2,-m-m_2} \times \mathcal{J}_{n_2,m_2} \cap \mathcal{R}_{k_1,k_2} $  \\
on $\mathcal{I}_{-n-n_2,-m-m_2} \times \mathcal{J}_{n_2,m_2}$ & & \\
& & $(n_2, m_2 \in \mathbb{Z}, n > 0,k_1 <0, k_2 \leq K)$\\
$\Downarrow$& & \\
Tensor-type stopping-time decomposition II  & $\longrightarrow$ & $I \times J \in \big(\mathcal{I}_{-n-n_2,-m-m_2} \cap T\big) \times \big(\mathcal{J}_{n_2,m_2} \cap S\big) \cap \mathcal{R}_{k_1,k_2} $\\
on $I \times J \in \mathcal{I}_{-n-n_2,-m-m_2} \times \mathcal{J}_{n_2,m_2} \cap \mathcal{R}_{k_1,k_2} $ & & with $T \in \mathbb{T}_{-l-l_2}$, $S \in \mathbb{S}_{l_2}$ \\
& & $(n_2, m_2, l_2 \in \mathbb{Z}, n, l > 0,k_1 <0, k_2 \leq K, )$\\
\end{tabular}
\end{table}
\end{comment}

\subsection{Application of stopping-time decompositions} \label{section_thm_haar_application_st}
One first rewrites (\ref{form_localized}) with the partition of dyadic rectangles specified in the stopping-time algorithm:
\begin{align} \label{form_decomposed}
%&|\Lambda^H_{\text{flag}^{0} \otimes \text{flag}^{0}}(f_1,f_2,g_1,g_2,h)| \\
%\lesssim 
&\displaystyle \sum_{\substack{n > 0 \\ m > 0 \\ l > 0 \\ k_1 < 0 \\ k_2 \leq K}} \sum_{\substack{n_2 \in \mathbb{Z} \\ m_2 \in \mathbb{Z} \\ l_2 \in \mathbb{Z}}}\sum_{\substack{T \in \mathbb{T}_{-l-l_2} \\ S \in \mathbb{S}_{l_2}}}\sum_{\substack{I \times J \in T \times S \\ I \times J \in \mathcal{I}_{-n - n_2, -m - m_2} \times \mathcal{J}_{n_2, m_2} \\I \times J \in \mathcal{R}_{k_1,k_2}}} \frac{1}{|I|^{\frac{1}{2}} |J|^{\frac{1}{2}}}| \langle B_I^H(f_1,f_2),\vphi_I^{1,H} \rangle| |\langle \tilde{B}_J^H(g_1,g_2),\vphi_J^{1,H} \rangle| \cdot \nonumber \\
&\ \ \ \ \ \ \ \ \ \ \ \ \ \ \ \ \ \ \ \ \ \ \ \ \ \ \ \ \ \ \ \ \ \ \ \ \ \ \ \ \ \ \ \ \ \ \ \ \ \ \ \ \ \ \ \  \ \ \ |\langle h, \psi_I^{2} \otimes \psi_J^{2} \rangle| |\langle \chi_{E'},\psi_I^{3,H} \otimes \psi_J^{3,H} \rangle|. 
%= & \displaystyle \sum_{\substack{n > 0 \\ m > 0 \\ l > 0 \\ k_1 < 0 \\ k_2 \leq K}} \sum_{\substack{n_2 \in \mathbb{Z} \\ m_2 \in \mathbb{Z} \\ l_2 \in \mathbb{Z}}}\sum_{\substack{T \in \mathbb{T}_{-l-l_2} \\ S \in \mathbb{S}_{l_2}}}\sum_{\substack{I \times J \in \mathcal{I}_{-n - n_2, -m - m_2} \cap{T} \times \mathcal{J}_{n_2, m_2} \cap S \\I \times J \in \mathcal{R}_{k_1,k_2}}} \frac{|\langle B_I^H(f_1,f_2),\vphi_I^1 \rangle|}{|I|^{\frac{1}{2}} } \frac{|\langle \tilde{B}_J(g_1,g_2),\vphi_J^1 \rangle| }{ |J|^{\frac{1}{2}}}\nonumber \\
%&\ \ \ \ \ \ \ \ \ \ \ \ \ \ \ \ \ \ \ \ \ \ \ \ \ \ \ \ \ \ \ \ \ \ \ \ \ \ \ \ \ \ \ \ \ \ \ \ \ \ \ \ \ \ \ \ \ \ \ \ \ \ \cdot \frac{|\langle h, \psi_I^{2} \otimes \psi_J^{2} \rangle|}{|I|^{\frac{1}{2}} |J|^{\frac{1}{2}}} \frac{|\langle \chi_{E'},\psi_I^{3,H} \otimes \psi_J^{3,H} \rangle|}{|I|^{\frac{1}{2}} |J|^{\frac{1}{2}}} |I| |J|  \nonumber \\
\end{align}
%\begin{comment}
%$$
%I \times J \in \mathcal{R}_{k_1,k_2}
%$$
%if and only if
%$$
%|I \times J \cap (\Omega^2_{k_1})^c| \geq \frac{99}{100}|I\times J|
%$$
%$$
%|I \times J \cap (\Omega^2_{k_2})^c| \geq \frac{99}{100}|I\times J|
%$$
%\end{comment}
One can now apply the exactly same argument in Section \ref{section_thm_haar_fixed_est_integral} to estimate (\ref{form_decomposed}) by
\begin{align} \label{form_00}
%&|\Lambda^H_{\text{flag}^{0} \otimes \text{flag}^{0}}| \nonumber\\
%\lesssim 
&\sum_{\substack{n > 0 \\ m > 0 \\ l > 0 \\ k_1 < 0 \\ k_2 \leq K}} \sum_{\substack{n_2 \in \mathbb{Z} \\ m_2 \in \mathbb{Z} \\ l_2 \in \mathbb{Z}}} \displaystyle & \sup_{I \in T} \frac{|\langle B_I^H(f_1,f_2),\vphi_I^{1,H} \rangle|}{|I|^{\frac{1}{2}}}  \sup_{J \in S} \frac{|\langle \tilde{B}_J^H(g_1,g_2),\vphi_J^{1,H} \rangle|}{|J|^{\frac{1}{2}}} \cdot 2^{k_1} \| h \|_{L^s} 2^{k_2} \sum_{\substack{T \in \mathbb{T}_{-l-l_2}\\S \in \mathbb{S}_{l_2}}} \bigg|\bigcup_{\substack{I \times J \in T \times S \\ I \times J \in \mathcal{I}_{-n - n_2, -m - m_2}  \times \mathcal{J}_{n_2,m_2} \\I \times J \in \mathcal{R}_{k_1,k_2}}} I \times J \bigg|. 
%& 2^{k_1} \| h \|_{L^s} 2^{k_2}  \cdot \sum_{\substack{T \in \mathbb{T}_{-l-l_2}\\S \in \mathbb{S}_{l_2}}} \bigg|\bigcup_{\substack{I \times J \in T \times S \\ I \times J \in \mathcal{I}_{-n - n_2, -m - m_2}  \times \mathcal{J}_{n_2,m_2} \\I \times J \in \mathcal{R}_{k_1,k_2}}} I \times J \bigg|.
\end{align}
Fix $-l-l_2$ and $T \in \mathbb{T}_{-l-l_2}$, one recalls the \textit{tensor-type stopping-time decomposition II} to conclude that
\begin{equation} \label{ave_1}
\sup_{I \in T } \frac{|\langle B_I^H(f_1,f_2),\vphi_I^{1H} \rangle|}{|I|^{\frac{1}{2}}} \lesssim C_1 2^{-l-l_2} \|B^H(f_1,f_2)\|_1.
\end{equation}
By the similar reasoning,
\begin{equation} \label{ave_2}
\sup_{J \in S } \frac{|\langle \tilde{B}^H_J(g_1,g_2),\vphi_J^{1,H} \rangle|}{|J|^{\frac{1}{2}}} \lesssim C_2 2^{l_2} \|\tilde{B}^H(g_1,g_2)\|_1.
\end{equation}
By applying the estimates (\ref{ave_1}) and (\ref{ave_2}) to (\ref{form_00}), one derives
\begin{align} \label{form00_set}
%&|\Lambda^H_{\text{flag}^{0} \otimes \text{flag}^{0}} | \nonumber\\
 %\lesssim 
 &C_1 C_2 C_3^2 \sum_{\substack{n > 0 \\ m > 0 \\ l > 0 \\ k_1 < 0 \\ k_2 \leq K}} 2^{-l}\|B^H(f_1,f_2)\|_1 \|\tilde{B}^H(g_1,g_2) \|_1\cdot 2^{k_1} \| h \|_{L^s} 2^{k_2}  \cdot \sum_{\substack{n_2 \in \mathbb{Z} \\ m_2 \in \mathbb{Z} \\ l_2 \in \mathbb{Z}}} \sum_{\substack{T \in \mathbb{T}_{-l-l_2}\\S \in \mathbb{S}_{l_2}}} \bigg|\bigcup_{\substack{I \times J \in T \times S \\ I \times J \in \mathcal{I}_{-n - n_2, -m - m_2}  \times \mathcal{J}_{n_2,m_2} \\I \times J \in \mathcal{R}_{k_1,k_2}}} I \times J \bigg|. 
%&\quad \quad \sum_{\substack{n_2 \in \mathbb{Z} \\ m_2 \in \mathbb{Z} \\ l_2 \in \mathbb{Z}}} \sum_{\substack{T \in \mathbb{T}_{-l-l_2}\\S \in \mathbb{S}_{l_2}}} \bigg|\bigcup_{\substack{I \times J \in T \times S \\ I \times J \in \mathcal{I}_{-n - n_2, -m - m_2}  \times \mathcal{J}_{n_2,m_2} \\I \times J \in \mathcal{R}_{k_1,k_2}}} I \times J \bigg|. 
\end{align}

%where by definition of the trees $T$, $S$ and their corresponding tree-tops $I_T$, $J_S$:
%\begin{align*}
%\bigg|\bigcup_{\substack{I \times J \in \mathcal{I}_{-n - n_2, -m - m_2} \cap{T} \times \mathcal{J}_{n_2,m_2} \cap S\\I \times J \in \mathcal{R}_{k_1,k_2}}} I \times J \bigg| = & \bigg|\big(\bigcup_{\substack{I \times J \in T \times S \\ I \times J \in \mathcal{I}_{-n - n_2, -m - m_2}\times \mathcal{J}_{n_2,m_2} \\I \times J \in \mathcal{R}_{k_1,k_2}}} I \times J\big) \cap \big(I_{T} \times J_{S}\big) \bigg| \nonumber \\
%\leq & \bigg|\big(\bigcup_{\substack{R \in \mathcal{R}_{k_1,k_2}\\R \in \mathcal{I}_{-n - n_2, -m - m_2} \times \mathcal{J}_{n_2,m_2} }} R \big) \cap \big(I_{T} \times J_{S}\big) \bigg|
%\end{align*}
\vskip 0.15in
\subsection{Estimate for nested sum of dyadic rectangles} \label{section_thm_haar_nestsum}
One can estimate the nested sum (\ref{ns}) in two approaches - one with the application of the sparsity condition and the other with a Fubini-type argument which will be introduced in Section \ref{section_thm_haar_ns_fubini}.
\begin{equation}\label{ns}
\sum_{\substack{n_2 \in \mathbb{Z} \\ m_2 \in \mathbb{Z} \\ l_2 \in \mathbb{Z}}}\sum_{\substack{T \in \mathbb{T}_{-l-l_2}\\S \in \mathbb{S}_{l_2}}} \bigg|\bigcup_{\substack{I \times J \in T\times S  \\ I \times J \in \mathcal{I}_{-n-n_2,-m-m_2} \times \mathcal{J}_{n_2,m_2} \\ I \times J \in \mathcal{R}_{k_1,k_2}}} I \times J \bigg|.
\end{equation}
 Both arguments aim to combine different stopping-time decompositions and to extract useful information from them. Generically, the sparsity condition argument employs the geometric property, namely Proposition \ref{sp_2d}, of the \textit{tensor-type stopping-time decomposition I} and applies the analytical implication from the \textit{general two-dimensional level sets stopping-time decomposition}. Meanwhile, the Fubini-type argument focuses on the hybrid of the \textit{tensor-type stopping time decomposition I - level sets} and the \textit{tensor-type stopping-time decomposition II - maximal intervals}. As implied by the name, the Fubini-type argument attempts to estimate the measure of a two dimensional set by the measures of its projected one-dimensional sets. The approaches to estimate projected one-dimensional sets are different depending on which tensor-type stopping decomposition is in consideration.
%\noindent 
\subsubsection{Sparsity condition.}
The first approach relies on the sparsity condition which mimics the argument in the Section \ref{section_thm_haar_fixed}. In particular, fix $n, m, l , k_1$ and $k_2$, one estimates (\ref{ns}) as follows.
\begin{align*}
& \sum_{l_2}\sum_{\substack{n_2 \in \mathbb{Z} \\ m_2 \in \mathbb{Z}}}\sum_{\substack{T \in \mathbb{T}_{-l-l_2}\\S \in \mathbb{S}_{l_2}}} \bigg|\bigcup_{\substack{I \times J \in T \times S  \\ I \times J \in \mathcal{I}_{-n-n_2,-m-m_2} \times \mathcal{J}_{n_2,m_2} \\  I \times J \in \mathcal{R}_{k_1,k_2}}} I \times J  \bigg| \nonumber \\
 \leq & \underbrace{ \sup_{l_2}\Bigg(\sum_{\substack{n_2 \in \mathbb{Z} \\ m_2 \in \mathbb{Z}}}\sum_{\substack{T \in \mathbb{T}_{-l-l_2}\\S \in \mathbb{S}_{l_2}}} \bigg|\bigcup_{\substack{I \times J \in T \times S  \\ I \times J \in \mathcal{I}_{-n-n_2,-m-m_2} \times \mathcal{J}_{n_2,m_2} \\  I \times J\in \mathcal{R}_{k_1,k_2}}} I \times J \bigg|\Bigg)^{\frac{1}{2}}}_{SC-I} \nonumber \\
 & \cdot \underbrace{\sum_{l_2}\Bigg(\sum_{\substack{n_2 \in \mathbb{Z} \\ m_2 \in \mathbb{Z}}}\sum_{\substack{T \in \mathbb{T}_{-l-l_2}\\S \in \mathbb{S}_{l_2}}}\bigg|\bigcup_{\substack{I \times J \in T \times S \\ I \times J \in \mathcal{I}_{-n-n_2,-m-m_2} \times \mathcal{J}_{n_2,m_2}\\ I \times J \in \mathcal{R}_{k_1,k_2}}} I \times J\bigg| \Bigg)^{\frac{1}{2}}}_{SC-II}.
\end{align*}

\noindent
\textbf{Estimate of $SC-I$.}
Fix $l, n, m, k_1, k_2$ and $l_2$. Then by the \textit{one-dimensional stopping-time decomposition - maximal intervals}, for any $I \in T$ and $I' \in T'$ such that $T, T' \in \mathbb{T}^{-l-l_2}$ and $T \neq T'$, one has $I \cap I' = \emptyset$. 
%$$
%\bigg(\bigcup_{\substack{I \in T  \\ I \in \mathcal{I}_{-n-n_2,-m-m_2} }} I \bigg) \cap \bigg(\bigcup_{\substack{I' \in T'  \\ I'  \in \mathcal{I}_{-n-n_2,-m-m_2} }} I' \bigg) = \emptyset $$
%which follows from Remark \ref{dis_sp} (1). Same reasoning applies to $J \in \mathcal{J}$. 
Hence for any fixed $n_2$ and $m_2$, one can rewrite
\begin{align} \label{SC-I}
\sum_{\substack{T \in \mathbb{T}_{-l-l_2}\\S \in \mathbb{S}_{l_2}}} \bigg|\bigcup_{\substack{I \times J \in T \times S \\ I \times J \in \mathcal{I}_{-n-n_2,-m-m_2} \times \mathcal{J}_{n_2,m_2}  \\  I \times J \in \mathcal{R}_{k_1,k_2}}} I \times J  \bigg|= & \bigg|\bigcup_{\substack{T \in \mathbb{T}_{-l-l_2}\\S \in \mathbb{S}_{l_2}}} \bigcup_{\substack{I \times J \in T \times S \\ I \times J \in \mathcal{I}_{-n-n_2,-m-m_2} \times \mathcal{J}_{n_2,m_2} \\  I \times J \in \mathcal{R} _{k_1,k_2} }} I \times J  \bigg|,  
\end{align}
where the right hand side of (\ref{SC-I}) can be trivially bounded by
$$
\bigg|\bigcup_{\substack{I \times J \in \mathcal{I}_{-n-n_2,-m-m_2} \times \mathcal{J}_{n_2,m_2} \\  I \times J \in \mathcal{R}_{k_1,k_2}}} I \times J  \bigg|.
$$
%The right hand side of (\ref{SC-I}) is indeed all the intervals in the collection $ \mathcal{I}_{-n-n_2,-m-m_2} \times \mathcal{J}_{n_2,m_2} \cap \mathcal{R}_{k_1,k_2}$, possibly collected from different trees. In particular, it can be rewritten as
%The estimate of $SC-II$ relies heavily on the second point of Remark \ref{dis_sp} and the sparsity condition described in Proposition \ref{sp_condition}. 
%With $n_2$ and $m_2$ varying, the collection of sets
%$$
%\bigg\{\bigcup_{\substack{T \in \mathbb{T}_{-l-l_2}\\S \in \mathbb{S}_{l_2}}} \bigcup_{\substack{I \times J \in T \times S \\  I \times J \in \mathcal{R}_{k_1,k_2} \\ I \times J \in \mathcal{I}_{-n-n_2,-m-m_2} \times \mathcal{J}_{n_2,m_2}}} I \times J: n_2, m_2 \in \mathbb{Z}\bigg\}
%$$
%can be viewed as a collection of arbitrary unions of dyadic rectangles. And for different elements in the collection, there are possibly a lot of overlaps. However, as illustrated in Proposition \ref{sp_2d}, the overlaps cannot be severe enough to ruin the nested sum. 
%$$
%SC-I = \sup_{l_2 \in \mathbb{Z}} \sum_{\substack{n_2 \in \mathbb{Z} \\ m_2 \in \mathbb{Z}}}\bigg| \bigcup_{\substack{I \times J \in \mathcal{I}_{-n-n_2,-m-m_2} \times \mathcal{J}_{n_2,m_2} \\ I \times J \in \mathcal{R}_{k_1,k_2} }} I \times J \bigg|
%$$ 
One can then recall the sparsity condition highlighted as Proposition \ref{sp_2d} and reduce the nested sum of measures of unions of rectangles to the measure of the corresponding union of rectangles. More precisely,
\begin{equation}\label{compare_nested_union}
\sum_{\substack{n_2 \in \mathbb{Z} \\ m_2 \in \mathbb{Z}}}\bigg| \bigcup_{\substack{I \times J \in \mathcal{I}_{-n-n_2,-m-m_2} \times \mathcal{J}_{n_2,m_2} \\ I \times J \in \mathcal{R}_{k_1,k_2} }} I \times J \bigg| \sim \bigg|\bigcup_{\substack{n_2 \in \mathbb{Z} \\ m_2 \in \mathbb{Z}}}\bigcup_{\substack{I \times J \in \mathcal{I}_{-n-n_2,-m-m_2} \times \mathcal{J}_{n_2,m_2} \\  I \times J \in \mathcal{R}_{k_1,k_2}}} I \times J\bigg|,
\end{equation}
where the right hand side of (\ref{compare_nested_union}) can be estimated by
$$
\bigg|\bigcup_{\substack{n_2 \in \mathbb{Z} \\ m_2 \in \mathbb{Z}}}\bigcup_{\substack{ I \times J \in \mathcal{I}_{-n-n_2,-m-m_2} \times \mathcal{J}_{n_2,m_2} \\  I \times J \in \mathcal{R}_{k_1,k_2}}} I \times J\bigg| \leq \bigg| \bigcup_{I \times J \in \mathcal{R}_{k_1,k_2}}I \times J\bigg| \lesssim \min(2^{-k_1s},2^{-k_2\gamma}),
$$
for any $\gamma >1$. The last inequality follows directly from (\ref{rec_area_hybrid}). Since the above estimates hold for any $l_2 \in \mathbb{Z}$, one can conclude that
\begin{equation} \label{SC-I-final}
SC-I \lesssim \min(2^{-\frac{k_1s}{2}},2^{-\frac{k_2\gamma}{2}}).
\end{equation}
\vskip .15in
\noindent
\textbf{Estimate of $SC-II$.}
One invokes (\ref{SC-I}) and Proposition \ref{sp_2d} to obtain
\begin{equation} \label{SC-II}
\sum_{\substack{n_2 \in \mathbb{Z} \\ m_2 \in \mathbb{Z}}}\sum_{\substack{T \in \mathbb{T}_{-l-l_2}\\S \in \mathbb{S}_{l_2}}}\bigg|\bigcup_{\substack{I \times J \in T \times S  \\ I \times J \in \mathcal{I}_{-n-n_2,-m-m_2} \times \mathcal{J}_{n_2,m_2} \\ I \times J \in \mathcal{R}_{k_1,k_2}}} I \times J\bigg| \sim \bigg|\bigcup_{\substack{n_2 \in \mathbb{Z} \\ m_2 \in \mathbb{Z}}}\bigcup_{\substack{T \in \mathbb{T}_{-l-l_2}\\S \in \mathbb{S}_{l_2}}} \bigcup_{\substack{I \times J \in T \times S  \\ I \times J \in \mathcal{I}_{-n-n_2,-m-m_2} \times \mathcal{J}_{n_2,m_2} \\  I \times J \in \mathcal{R}_{k_1,k_2}}} I \times J\bigg|.
\end{equation}
One can enlarge the collection of the rectangles by forgetting about the restriction that the rectangles lie in $\mathcal{R}_{k_1,k_2}$ and estimate the right hand side of (\ref{SC-II}) by 
\begin{equation} \label{SC-II2}
\bigg|\bigcup_{\substack{n_2 \in \mathbb{Z} \\ m_2 \in \mathbb{Z}}}\bigcup_{\substack{T \in \mathbb{T}_{-l-l_2}\\S \in \mathbb{S}_{l_2}}} \bigcup_{\substack{I \times J \in T \times S  \\ I \times J \in \mathcal{I}_{-n-n_2,-m-m_2} \times \mathcal{J}_{n_2,m_2} }} I \times J\bigg|,
\end{equation}
which is indeed the measure of the union of the rectangles collected in the \textit{tensor-type stopping-time decomposition II - maximal intervals} at a certain level. In other words,
$$
%\bigg|\bigcup_{\substack{n_2 \in \mathbb{Z} \\ m_2 \in \mathbb{Z}}}\bigcup_{\substack{T \in \mathbb{T}_{-l-l_2}\\S \in \mathbb{S}_{l_2}}} \bigcup_{\substack{I \times J \in T \times S  \\ I \times J \in \mathcal{I}_{-n-n_2,-m-m_2} \times \mathcal{J}_{n_2,m_2}}} I \times J\bigg| 
(\ref{SC-II2}) = \bigg| \bigcup_{T \times S \in \mathbb{T}_{-l-l_2} \times \mathbb{S}_{l_2}} I_T \times J_S\bigg|.
$$
\begin{comment}
One could then re-define the collection trees and tree-tops for this more universal stopping-time decomposition. More rigorously, for any two trees $T$ and $T'$, if $I_T \cap I_{T'} \neq \emptyset$, then define the new tree by
$$T:= T \cup T'.$$ Without loss of generality, suppose that $I_{T'} \subset I_{T}$, then define the new tree-top by 
$$
I_{T} := I_{T}
$$
The collection of trees at each level can be defined correspondingly. And the same mechanism can be applied to the dyadic intervals in $\mathcal{J}$. It follows from the new definition of the trees and tree-tops that for any fixed $-l-l_2$ and $l_2$,
$$
\bigcup_{T\times S \in \mathbb{T}_{-l-l_2} \times \mathbb{S}_{l_2}} I_T \times J_S = \bigcup_{T \times S \in \mathbb{T}_{-l-l_2} \times \mathbb{S}_{l_2}} I_{T} \times J_{S}
$$  
where 
\begin{align*}
&\{ I_{T}: T \in \mathbb{T}_{-l-l_2} \} \nonumber \\
& \{ J_{S}: S \in \mathbb{S}_{l_2} \}
\end{align*}
are disjoint collections of dyadic intervals.
\begin{remark}
The re-definition of trees and tree-tops creates disjointness highlighted above, which is essential for the future application of energy estimates as can be seen in the proof of Proposition \ref{energy_classical}.
\end{remark}
\end{comment}
Then
\begin{equation} \label{fb_simple}
SC-II \leq  \sum_{l_2 \in \mathbb{Z}}\bigg| \bigcup_{T \times S \in \mathbb{T}_{-l-l_2} \times \mathbb{S}_{l_2}} I_{T} \times J_{S}\bigg|^{\frac{1}{2}},
\end{equation}
whose estimate follows a Fubini-type argument that plays an important role in the proof. We will focus on the development of this Fubini-type argument in a separate section and discuss its applications in other useful estimates for the proof. 

\subsubsection{Fubini argument.} \label{section_thm_haar_ns_fubini}
Alternatively, one can apply a Fubini-type argument to estimate (\ref{ns}) in the sense that the measure of some two-dimensional set is estimated by the product of the measures of its projected one-dimensional sets. To introduce this argument, we will first look into (\ref{fb_simple}) which requires a simpler version of the argument.
\vskip .05in
\noindent
\textbf{Estimate of (\ref{fb_simple}) - Introduction of Fubini argument.}
As illustrated before, one first rewrites the measure of two dimensional-sets in terms of the measures of two one-dimensional sets as follows.
\begin{align} \label{fb_simple_2d}
%& \sum_{l_2 \in \mathbb{Z}}\bigg| \bigcup_{T \times S \in \mathbb{T}_{-l-l_2} \times \mathbb{S}_{l_2}} I_{T} \times J_{S}\bigg|^{\frac{1}{2}}  \nonumber \\
\text{Right hand side of } (\ref{fb_simple}) \leq & \bigg( \sum_{l_2 \in \mathbb{Z}}\big|\bigcup_{T\in \mathbb{T}_{-l-l_2}}  I_{T} \big|\bigg)^{\frac{1}{2}}\bigg( \sum_{l_2 \in \mathbb{Z}}\big|\bigcup_{S\in \mathbb{S}_{l_2}}  J_{S} \big|\bigg)^{\frac{1}{2}},
\end{align}
where the last step follows from the Cauchy-Schwarz inequality.
To estimate the measures of the one-dimensional sets appearing above, one can convert them to the form of ``global'' energies and apply the energy estimates specified in Proposition \ref{B_en_global}. In particular, (\ref{fb_simple_2d}) can be rewritten up to a constant as
\begin{align} \label{SC-II-en}
& \bigg( \sum_{l_2 \in \mathbb{Z}}(C_12^{-l-l_2} \|B^H(f_1,f_2)\|_1)^{1+\delta}\big|\bigcup_{T\in \mathbb{T}_{-l-l_2}}  I_{T} \big|\bigg)^{\frac{1}{2}}\bigg( \sum_{l_2 \in \mathbb{Z}}(C_22^{l_2} \|\tilde{B}^H(g_1,g_2)\|_1)^{1+\delta}\big|\bigcup_{S\in \mathbb{S}_{l_2}}  J_{S} \big|\bigg)^{\frac{1}{2}} \cdot \nonumber\\
& \ \ 2^{l\frac{(1+\delta)}{2}}\|B^H(f_1,f_2)\|_1^{-\frac{1+\delta}{2}}\|\tilde{B}^H(g_1,g_2)\|_{1}^{-\frac{1+\delta}{2}}, 
% \leq & \bigg(\sum_{l_2 \in \mathbb{Z}}(2^{-l-l_2}\|B\|_1)^{1+\epsilon} \big|\{MB > C_1 2^{-l-l_2}\|B\|_1\} \big|\bigg)^{\frac{1}{2}} \bigg(\sum_{l_2 \in \mathbb{Z}}(2^{l_2}\|\tilde{B}\|_1)^{1+\epsilon}
%\big|\{M\tilde{B} > C_2 2^{l_2}\|\tilde{B}\|_1\} \big|\bigg)^{\frac{1}{2}} \nonumber \\
%& \ \ \cdot 2^{l\frac{(1+\epsilon)}{2}}\|B\|_{1}^{-\frac{1+\epsilon}{2}}\|\tilde{B}\|_{1}^{-\frac{1+\epsilon}{2}}\nonumber \\
%\leq & \|MB\|_{1+\epsilon}^{\frac{1+\epsilon}{2}}\|M\tilde{B}\|_{1+\epsilon}^{\frac{1+\epsilon}{2}}\cdot 2^{l\frac{(1+\epsilon)}{2}}\|B\|_{1}^{-\frac{1+\epsilon}{2}}\|\tilde{B}\|_{1}^{-\frac{1+\epsilon}{2}} \nonumber \\
%\lesssim & \|B\|_{1+\epsilon}^{\frac{1+\epsilon}{2}}\|\tilde{B}\|_{1+\epsilon}^{\frac{1+\epsilon}{2}} \cdot 2^{l\frac{(1+\epsilon)}{2}}\|B\|_{1}^{-\frac{1+\epsilon}{2}}\|\tilde{B}\|_{1}^{-\frac{1+\epsilon}{2}}
\end{align}
for any $\delta >0$. One notices that for fixed $l$ and $l_2$, 
$$
\{I_T: T \in \mathbb{T}_{-l-l_2} \}
$$
is a disjoint collection of dyadic intervals according to the \textit{one-dimensional stopping-time decomposition - maximal interval}. Thus the first sum in (\ref{SC-II-en}) can be rewritten as 
\begin{equation} \label{en_global}
\sum_{l_2 \in \mathbb{Z}}(C_12^{-l-l_2} \|B^H(f_1,f_2)\|_1)^{1+\delta}\big|\bigcup_{T\in \mathbb{T}_{-l-l_2}}  I_{T} \big| = \sum_{l_2 \in \mathbb{Z}}(C_12^{-l-l_2} \|B^H(f_1,f_2)\|_1)^{1+\delta}\sum_{T\in \mathbb{T}_{-l-l_2}}|I_{T}|, 
\end{equation}
which is indeed a ``global'' $L^{1+\delta}$ energy, namely 
$$
\left(\text{energy}^{1+\delta}_{\mathcal{I}}((\langle B^H(f_1,f_2), \vphi_I^{1,H} \rangle)_{I \in \mathcal{I}})\right)^{1+\delta}
$$
 so that one can apply the energy estimates described in Proposition \ref{B_en_global} to obtain the bound
$$
\left(\text{energy}^{1+\delta}_{\mathcal{I}}((\langle B^H(f_1,f_2), \vphi_I^{1,H} \rangle)_{I \in \mathcal{I}})\right)^{1+\delta} \lesssim |F_1|^{\mu_1(1+\delta)}|F_2|^{\mu_2(1+\delta)},
$$
where $\delta, \mu_1, \mu_2 >0$ with $\mu_1 + \mu_2 = \frac{1}{1+\delta}$. Similarly, one can apply the same reasoning to the second sum in (\ref{SC-II-en}) to derive
\begin{equation} \label{SC-II-y}
\sum_{l_2 \in \mathbb{Z}}(C_22^{l_2} \|\tilde{B}^H(g_1,g_2)\|_1)^{1+\delta}\big|\bigcup_{S\in \mathbb{S}_{l_2}}  J_{S} \big| \lesssim |G_1|^{\nu_1(1+\delta)}|G_2|^{\nu_2(1+\delta)},
\end{equation}
for any $\nu_1, \nu_2 > 0$ with $\nu_1 + \nu_2  = \frac{1}{1+\delta}$. 
By applying (\ref{en_global}) and (\ref{SC-II-y}) to (\ref{SC-II-en}), one has that
\begin{equation}\label{SCiII-final}
% \sum_{l_2 \in \mathbb{Z}}\bigg| \bigcup_{T \times S \in \mathbb{T}_{-l-l_2} \times \mathbb{S}_{l_2}} I_{T} \times J_{S}\bigg|^{\frac{1}{2}} 
(\ref{SC-II-en}) \lesssim  2^{l \frac{(1+\delta)}{2}} |F_1|^{\frac{\mu_1(1+\delta)}{2}}|F_2|^{\frac{\mu_2(1+\delta)}{2}}|G_1|^{\frac{\nu_1(1+\delta)}{2}}|G_2|^{\frac{\nu_2(1+\delta)}{2}}\|B^H(f_1,f_2)\|_1^{-\frac{1+\delta}{2}}\|\tilde{B}^H(g_1,g_2)\|_1^{-\frac{1+\delta}{2}},
\end{equation}
for any $\delta,\mu_1,\mu_2,\nu_1,\nu_2 >0$ with $\mu_1+ \mu_2 = \nu_1+ \nu_2 = \frac{1}{1+\delta}$.
\begin{remark}
The reason for leaving the expressions $\|B^H(f_1,f_2)\|_1^{-\frac{1+\delta}{2}}$ or $\|\tilde{B}^H(g_1,g_2)\|_1^{-\frac{1+\delta}{2}}$ will become clear later. In short, $\|B^H(f_1,f_2)\|_1$ and $\|\tilde{B}^H(g_1,g_2)\|_1$ will appear in estimates for other parts. We will keep them as they are for the exponent-counting and then use the estimates for $\|B^H(f_1,f_2)\|_1$ and $\|\tilde{B}^H(g_1,g_2)\|_1$ at last.
\end{remark}
By combining the estimates for $SC-I$ (\ref{SC-I-final}) and $SC-II$ (\ref{SCiII-final}), one can conclude that (\ref{ns}) is majorized by
\begin{equation} \label{ns_sp}
 2^{-\frac{k_2\gamma}{2}}2^{l \frac{(1+\delta)}{2}} |F_1|^{\frac{\mu_1(1+\delta)}{2}}|F_2|^{\frac{\mu_2(1+\delta)}{2}}|G_1|^{\frac{\nu_1(1+\delta)}{2}}|G_2|^{\frac{\nu_2(1+\delta)}{2}}\|B^H(f_1,f_2)\|_1^{-\frac{1+\delta}{2}}\|\tilde{B}^H(g_1,g_2)\|_1^{-\frac{1+\delta}{2}},
\end{equation}
for $\gamma >1 $, $\delta,\mu_1,\mu_2,\nu_1,\nu_2 >0$ with $\mu_1+ \mu_2 = \nu_1+ \nu_2 = \frac{1}{1+\delta}$.

\begin{remark}
The framework for estimating the measure of a two-dimensional set by its corresponding one-dimensional sets, as illustrated by (\ref{fb_simple_2d}), is the so-called ``Fubini-type'' argument which we will heavily employ from now on.
\end{remark}

%In particular, one observes that for any fixed $l$ and $l_2$, $\{I_T\times J_S: T \in \mathbb{T}_{-l-l_2}, J \in \mathbb{S}_{l_2}\}$ is a disjoint collection of rectangles. One can therefore reduce the above expression to 
%$$
%\sum_{\substack{n_2 \in \mathbb{Z} \\ m_2 \in \mathbb{Z} \\ l_2 \in \mathbb{Z}}}\bigg|\bigcup_{\substack{R\in \mathcal{R}_{k_1,k_2} \\ R \in \mathcal{I}_{-n-n_2,-m-m_2} \times \mathcal{J}_{n_2,m_2}}} R \bigg| 
%$$
%which can be estimated by applying the sparsity condition (Proposition \ref{sp_2d}) twice:
%\begin{align*}
%\sum_{\substack{n_2 \in \mathbb{Z} \\ m_2 \in \mathbb{Z} \\ l_2 \in \mathbb{Z}}}\bigg|\bigcup_{\substack{R\in \mathcal{R}_{k_1,k_2} \\ R \in \mathcal{I}_{-n-n_2,-m-m_2} \times \mathcal{J}_{n_2,m_2}}} R \bigg| \lesssim & \sum_{m_2 \in \mathbb{Z}}\bigg|\bigcup_{\substack{R\in \mathcal{R}_{k_1,k_2} \\ R \in \mathcal{I}_{-m-m_2} \times \mathcal{J}_{m_2}}} R\bigg| \nonumber \\
%\lesssim & \bigg|\bigcup_{\substack{R\in \mathcal{R}_{k_1,k_2} }} R\bigg| \nonumber \\
% \lesssim & \min(2^{-k_1},2^{-k_2 \gamma})
%\end{align*}
%for any $\gamma >1$.
\vskip .15in
\noindent
\textbf{Estimate of (\ref{ns}) - Application of Fubini argument.}
It is not difficult to observe that (\ref{ns}) can also be estimated by 
\begin{equation} \label{set_00}
%\sum_{\substack{n_2 \in \mathbb{Z} \\ m_2 \in \mathbb{Z} \\ l_2 \in \mathbb{Z}}}\sum_{\substack{T \in \mathbb{T}_{-l-l_2}\\S \in \mathbb{S}_{l_2}}} \bigg|\bigcup_{\substack{I \times J \in T\times S  \\ I \times J \in \mathcal{I}_{-n-n_2,-m-m_2} \times \mathcal{J}_{n_2,m_2} }} I \times J \bigg| \leq
 \sum_{\substack{n_2 \in \mathbb{Z} \\ m_2 \in \mathbb{Z} \\ l_2 \in \mathbb{Z}}}\left(\sum_{\substack{T \in \mathbb{T}_{-l-l_2}}}\bigg|\bigcup_{\substack{I \in T \\ I  \in \mathcal{I}_{-n-n_2,-m-m_2}}} I \bigg|\right)\left(\sum_{\substack{S \in \mathbb{S}_{l_2}}}\bigg|\bigcup_{\substack{J \in  S  \\ J \in  \mathcal{J}_{n_2,m_2} }} J \bigg|\right).
\end{equation}
One now rewrites the above expression and separates it into two parts. Both parts can be estimated by the Fubini-type argument whereas the methodologies to estimate projected one-dimensional sets are different. More precisely,
\begin{align} \label{ns_AB}
(\ref{set_00}) = &\underbrace{\sup_{\substack{n_2 \in \mathbb{Z} \\ m_2 \in \mathbb{Z}}} \sum_{l_2 \in \mathbb{Z}}\bigg(\sum_{\substack{T \in \mathbb{T}_{-l-l_2}}}\bigg|\bigcup_{\substack{I \in T \\ I  \in \mathcal{I}_{-n-n_2,-m-m_2}}} I \bigg|\bigg)^{\frac{1}{2}}\bigg(\sum_{\substack{S \in \mathbb{S}_{l_2}}} \bigg|\bigcup_{\substack{J \in  S  \\ J \in  \mathcal{J}_{n_2,m_2} }} J \bigg|\bigg)^{\frac{1}{2}}}_{\mathcal{A}} \times \nonumber \\
& \underbrace{\sum_{\substack{n_2 \in \mathbb{Z} \\ m_2 \in \mathbb{Z}}}\sup_{l_2 \in \mathbb{Z}}\bigg(\sum_{\substack{T \in \mathbb{T}_{-l-l_2}}}\bigg|\bigcup_{\substack{I \in T \\ I  \in \mathcal{I}_{-n-n_2,-m-m_2}}} I \bigg|\bigg)^{\frac{1}{2}}\bigg(\sum_{\substack{S \in \mathbb{S}_{l_2}}}\bigg|\bigcup_{\substack{J \in  S  \\ J \in  \mathcal{J}_{n_2,m_2} }} J \bigg|\bigg)^{\frac{1}{2}}}_{\mathcal{B}}.
\end{align}
To estimate $\mathcal{A}$, one first notices that for for any fixed $n, m, n_2, m_2, l, l_2$ and a fixed tree $T \in \mathbb{T}^{-l-l_2}$, a dyadic interval $I \in T \cap \mathcal{I}_{-n-n_2.-m-m_2}$ means that 
\begin{enumerate}[(i)]
\item
$I \subseteq I_T$ where $I_T$ is the tree-top interval as implied by the \textit{one-dimensional stopping-time decomposition - maximal interval};
\item
$I \cap \mathcal{U}_{-n-n_2+1,-m-m_2+1} \neq \emptyset$,
where
$$
\mathcal{U}_{-n-n_2+1,-m-m_2+1} := \{x:Mf_1(x) \leq C_1 2^{-n-n_2+1}|F_1| \} \cap \{x:Mf_2(x) \leq C_1 2^{-m-m_2+1}|F_2| \}. $$
\end{enumerate}
By (i) and (ii), one can deduce that
$$
I_T \cap \mathcal{U}_{-n-n_2+1,-m-m_2+1} \neq \emptyset.
$$
As a consequence, 
\begin{equation} \label{a_x}
\sum_{\substack{T \in \mathbb{T}_{-l-l_2}}}\bigg|\bigcup_{\substack{I \in T \\ I  \in \mathcal{I}_{-n-n_2,-m-m_2}}} I \bigg| \leq \sum_{\substack{T \in \mathbb{T}_{-l-l_2} \\ I_T \cap \mathcal{U}_{-n-n_2+1,-m-m_2+1} \neq \emptyset}}|I_T|.
\end{equation}
A similar reasoning applies to the term involving intervals in the $y$-direction and generates
\begin{equation} \label{a_y}
\sum_{\substack{S \in \mathbb{S}_{l_2}}}\bigg|\bigcup_{\substack{J \in S \\ I  \in \mathcal{J}_{n_2,m_2}}} J \bigg| \leq \sum_{\substack{S \in \mathbb{S}_{l_2} \\J_S \cap \tilde{\mathcal{U}}_{n_2+1,m_2+1} \neq \emptyset}}|J_S|,
\end{equation}
where
$$
\tilde{\mathcal{U}}_{n_2+1,m_2+1} := \{y:Mg_1(y) \leq C_2 2^{n_2+1}|G_1| \} \cap \{y:Mg_2(y) \leq C_2 2^{m_2+1}|G_2| \}.
$$
By applying the Cauchy-Schwarz inequality together with (\ref{a_x}) and (\ref{a_y}), one obtains
\begin{align} \label{a_pre_en}
\mathcal{A} \leq & \sup_{\substack{n_2 \in \mathbb{Z} \\ m_2 \in \mathbb{Z}}} \bigg(\sum_{l_2 \in \mathbb{Z}}\sum_{\substack{T \in \mathbb{T}_{-l-l_2} \\ I_T \cap \mathcal{U}_{-n-n_2+1,-m-m_2+1} \neq \emptyset}}|I_T|\bigg)^{\frac{1}{2}}\cdot  \bigg(\sum_{l_2 \in \mathbb{Z}}\sum_{\substack{S \in \mathbb{S}_{l_2} \\J_S \cap \tilde{\mathcal{U}}_{n_2+1,m_2+1}\neq \emptyset}}|J_S|\bigg) ^{\frac{1}{2}}.
\end{align}
One then ``completes'' the expression (\ref{a_pre_en}) to produce localized energy-like terms as follows.
\begin{align} \label{A_energy}
 & \sup_{\substack{n_2 \in \mathbb{Z} \\ m_2 \in \mathbb{Z}}} \underbrace{\bigg[\sum_{l_2 \in \mathbb{Z}}(C_1 2^{-l-l_2}\|B^H(f_1,f_2)\|_1)^2\sum_{\substack{T \in \mathbb{T}_{-l-l_2}\\ I_T \cap \mathcal{U}_{-n-n_2+1,-m-m_2+1} \neq \emptyset}}|I_{T}|\bigg]^{\frac{1}{2}}}_{\mathcal{A}^1}\cdot \underbrace{\bigg[\sum_{l_2 \in \mathbb{Z}}(C_2 2^{l_2}\|\tilde{B}^H(g_1,g_2)\|_1)^{2} \sum_{\substack{S \in \mathbb{S}_{l_2}\\ J_S \cap \tilde{\mathcal{U}}_{n_2+1,m_2+1}\neq \emptyset}} |J_S |\bigg]^{\frac{1}{2}}}_{\mathcal{A}^2}  \nonumber \\
 &\cdot 2^{l}\|B^H(f_1,f_2)\|_1^{-1}\|\tilde{B}^H(g_1,g_2)\|_1^{-1}.\nonumber \\
 \end{align}
%loc_for_energy!!!!!
\begin{comment}
\begin{remark} \label{loc_for_energy}
\begin{enumerate}
\item
It is noteworthy that the estimates for $a^1$ and $a^2$ proceed under the condition that $-n-n_2, -m-m_2, n_2, m_2$ are fixed. This condition is essential because those indices specify the sub-collection of dyadic intervals, namely $\mathcal{I}_{-n-n_2,-m-m_2}$ and $\mathcal{J}_{n_2,m_2}$, that the tensor-type stopping-time decomposition is performed on. One can then apply the ``localization of energy'' described in Section 5.5  so that
\begin{equation}\label{alt_loc}
\text{energy}(\langle B_I, \vphi_I^1\rangle )_{I \in \mathcal{I}_{-n-n_2,-m-m_2}} \lesssim \text{energy}(\langle B_0^{-n-n_2,-m-m_2}, \vphi_I^1\rangle )_{I \in \mathcal{I}}.
\end{equation}

\item
One needs to keep in mind the hypothesis for the ``localization of energy'' used to obtain (\ref{alt_loc}). For non-lacunary family $(\phi_K^3)_K$, one assumes the compactness of supports of the bump functions; for lacunary family $(\phi_K^3)_K$, one has imposed the assumption (\ref{Haar_cond}).
\end{enumerate}
\end{remark}
\end{comment}
%!!!!!! to modify
%With the localization (\ref{alt_loc}), one is ready to apply the localized energy estimates specified in Proposition \ref{B_en} to obtain
It is not difficult to recognize that $\mathcal{A}^1$ and $\mathcal{A}^2$ are $L^2$ energies. Moreover, they follow stronger local energy estimates described in Proposition \ref{B_en}. %More precisely, the estimates for $\mathcal{A}^1$ and $\mathcal{A}^2$ proceed under the condition that $-n-n_2, -m-m_2, n_2, m_2$ are fixed. This condition is essential because those indices specify the sub-collection of dyadic intervals, namely $\mathcal{I}_{-n-n_2,-m-m_2}$ and $\mathcal{J}_{n_2,m_2}$, that the tensor-type stopping-time decomposition with respect to 
%$$
%\frac{\langle B^H_I, \vphi_I^{1,H} \rangle }{|I|^{\frac{1}{2}}}
%$$
%is performed on. 
$\mathcal{A}^1$ is indeed  an $L^2$ energy localized to $\mathcal{U}_{-n-n_2+1,-m-m_2+1}$. Then Proposition \ref{B_en} gives the estimate
\begin{align} \label{a_1}
\mathcal{A}^1 \leq & \text{energy}^2\left((\langle B_I^H(f_1,f_2), \vphi_I^{1,H} \rangle)_{I \cap \mathcal{U}_{-n-n_2+1,-m-m_2+1} \neq \emptyset }\right) \lesssim & (C_1 2^{-n-n_2})^{\frac{1}{p_1}-\theta_1} (C_1 2^{-m-m_2})^{\frac{1}{q_1} - \theta_2}|F_1|^{\frac{1}{p_1}}|F_2|^{\frac{1}{q_1}},
\end{align}
for any $0 \leq \theta_1, \theta_2 < 1$ satisfying $\theta_1 + \theta_2 = \frac{1}{2}$.
%Let $\theta_3 = \frac{1}{2}$ and equivalently $\theta_1 + \theta_2 = \frac{1}{2}$, one can simplify the expression as 
%$$
%a^1 \lesssim C_1^{\frac{3}{2}}2^{(-n-n_2)\alpha_1(1-\theta_1)}2^{(-m-m_2)\alpha_2(\frac{1}{2} + \theta_1)}|F_1|^{\alpha_1(1-\theta_1)+ \theta_1}|F_2|^{\alpha_2(\frac{1}{2}+\theta_1)_+\frac{1}{2}-\theta_1}
%$$
By the same reasoning,
\begin{equation} \label{a_2}
\mathcal{A}^2 \lesssim C_2^{2}2^{n_2(\frac{1}{p_2} - \zeta_1)}2^{m_2(\frac{1}{q_2}- \zeta_2)}|G_1|^{\frac{1}{p_2}}|G_2|^{\frac{1}{q_2}},
\end{equation}
where $0 \leq \zeta_1, \zeta_2 < 1$ and $\zeta_1 + \zeta_2 = \frac{1}{2}$.
One can now apply the estimates for $\mathcal{A}^1$ (\ref{a_1}) and $\mathcal{A}^2$ (\ref{a_2}) to (\ref{A_energy}) and derive 
\begin{align} \label{A_energy_almostfinal}
(\ref{A_energy}) \lesssim & C_1^{2}C_2^{2}\sup_{\substack{n_2 \in \mathbb{Z} \\ m_2 \in \mathbb{Z}}}2^{(-n-n_2)(\frac{1}{p_1}- \theta_1)}2^{(-m-m_2)(\frac{1}{q_1} - \theta_2)}2^{n_2(\frac{1}{p_2} - \zeta_1)}2^{m_2(\frac{1}{q_2}- \zeta_2)} \cdot \nonumber \\
& |F_1|^{\frac{1}{p_1}}|F_2|^{\frac{1}{q_1}}|G_1|^{\frac{1}{p_2}}|G_2|^{\frac{1}{q_2}}\cdot 2^{l}\|B^H(f_1,f_2)\|_1^{-1}\|\tilde{B}^H(g_1,g_2)\|_1^{-1}.
\end{align}
One observes that the following two conditions are equivalent:
\begin{equation} \label{exp_1}
\frac{1}{p_1} - \theta_1 = \frac{1}{p_2} - \zeta_1 \iff
 \frac{1}{q_1} - \theta_2 = \frac{1}{q_2} - \zeta_2.
\end{equation}
The equivalence is imposed by the fact that 
\begin{align} \label{exp_2}
&\frac{1}{p_1} + \frac{1}{q_1} = \frac{1}{p_2} + \frac{1}{q_2}, \nonumber \\
&\theta_1 + \theta_2 = \zeta_1 + \zeta_2 = \frac{1}{2} .
\end{align}
%and deduces
%\begin{equation} \label{nec_condition}
%\frac{1}{p_2} - \frac{1}{p_1} = \frac{1}{q_1} - \frac{1}{q_2} = \theta_2 - \theta_2' = \zeta_1 - \theta_1 
%\end{equation}
With the choice $ 0 \leq \theta_1, \zeta_1 < 1$ with $\theta_1- \zeta_1 = \frac{1}{p_1} - \frac{1}{p_2}$, one can simplify (\ref{A_energy_almostfinal}) and conclude
\begin{equation} \label{a_estimate}
\mathcal{A} \lesssim C_1^2 C_2^{2}2^{-n(\frac{1}{p_1} - \theta_1)}2^{-m(\frac{1}{q_1} - \theta_2)} |F_1|^{\frac{1}{p_1}}|F_2|^{\frac{1}{q_1}}|G_1|^{\frac{1}{p_2}}|G_2|^{\frac{1}{q_2}}\cdot 2^{l}\|B^H(f_1,f_2)\|_1^{-1}\|\tilde{B}^H(g_1,g_2)\|_1^{-1}.
\end{equation}

\begin{remark}
(\ref{exp_1}) and (\ref{exp_2}) together imposes a condition that 
\begin{equation} \label{pair_exp}
\left|\frac{1}{p_1} - \frac{1}{p_2}\right| = \left|\frac{1}{q_1} - \frac{1}{q_2}\right| < \frac{1}{2}.
\end{equation}
Without loss of generality, one can assume that $\frac{1}{p_1} \geq \frac{1}{p_2}$ and $\frac{1}{q_1} \leq \frac{1}{q_2}$. Then either (\ref{pair_exp}) holds or 
$$
\frac{1}{p_1} - \frac{1}{p_2} = \frac{1}{q_2} - \frac{1}{q_1} > \frac{1}{2},
$$
which implies 
$$
\left|\frac{1}{p_1} - \frac{1}{q_2}\right| = \left|\frac{1}{p_2} - \frac{1}{q_1}\right| < \frac{1}{2}.
$$
Then one can switch the role of $g_1$ and $g_2$ to ``pair`` the functions as $f_1$ with $g_2$ and $f_2$ with $g_1$. A parallel argument can be applied to obtain the desired estimates.
\end{remark}

One can apply another Fubini-type argument to estimate $\mathcal{B}$ with $l, n$ and $m$ fixed. Such argument again relies heavily on the localization. First of all, for any fixed $l_2 \in \mathbb{Z}$, 
$$
\{I: I \in T, T \in \mathbb{T}_{-l-l_2} \}
$$
is a disjoint collection of dyadic intervals. Thus
$$
\sum_{\substack{T \in \mathbb{T}_{-l-l_2}}}\bigg|\bigcup_{\substack{I \in T \\ I  \in \mathcal{I}_{-n-n_2,-m-m_2}}} I \bigg| \leq \bigg|\bigcup_{\substack{ I  \in \mathcal{I}_{-n-n_2,-m-m_2}}} I \bigg|. 
$$
One then recalls the pointwise estimate stated in Claim \ref{ptwise} to deduce 
$$
\bigcup_{\substack{ I  \in \mathcal{I}_{-n-n_2,-m-m_2}}} I \subseteq \{x: Mf_1(x) > C_1 2^{-n-n_2-10}|F_1|\} \cap \{x: Mf_2(x) > C_1 2^{-m-m_2-10}|F_2|\},
$$
and for arbitrary but fixed $l_2 \in \mathbb{Z}$,
\begin{equation} \label{b_x}
\sum_{\substack{T \in \mathbb{T}_{-l-l_2}}}\bigg|\bigcup_{\substack{I \in T \\ I  \in \mathcal{I}_{-n-n_2,-m-m_2}}} I \bigg|\leq \big|\{x: Mf_1(x) > C_1 2^{-n-n_2-10}|F_1|\} \cap \{x: Mf_2(x) > C_1 2^{-m-m_2-10}|F_2|\} \big|.
\end{equation}
%Moreover,  as pointed out in Remark \ref{dis_sp}, $\{I_T: T \in \mathbb{T}_{-l-l_2} \ \ \text{and} \ \ I_T \in \mathcal{I}_{-n-n_2,-m-m_2}\}$ is a disjoint collection of intervals for any fixed $-n-n_2, -m-m_2$ and $-l-l_2$. As a result,
%\begin{equation} \label{holder}
%\sum_{\substack{T \in \mathbb{T}_{-l-\tilde{l_2}}\\ I_T \in \mathcal{I}_{-n-n_2,-m-m_2}}}|I_{T}| \leq  |\{ Mf_1 > C_1 2^{-n-n_2-10}|F_1|\} \cap \{ Mf_2 > C_1 2^{-m-m_2-10}|F_2|\}|
%\end{equation}
A similar reasoning applies to the intervals in the $y$-direction and yields that for any fixed $l_2 \in \mathbb{Z}$, 
\begin{equation} \label{b_y}
\sum_{\substack{S \in \mathbb{S}_{l_2}}}\bigg|\bigcup_{\substack{J \in  S  \\ J \in  \mathcal{J}_{n_2,m_2} }} J \bigg| \leq
\big|\{y: Mg_1(y) > C_2 2^{n_2-10}|G_1|\} \cap \{y: Mg_2(y) > C_2 2^{m_2-10}|G_2|\} \big|.
\end{equation}
To apply the above estimates, one notices that the finite collection of dyadic rectangles guarantees the existence of some $\tilde{l}_2 \in \mathbb{Z}$ possibly depending $n, m, l, n_2, m_2$ such that 
\begin{align*}
\mathcal{B} =& \sum_{\substack{n_2 \in \mathbb{Z} \\ m_2 \in \mathbb{Z}}}\bigg(\sum_{T \in \mathbb{T}_{-l-\tilde{l}_2}}\bigg|\bigcup_{\substack{ I \in T \\ I  \in \mathcal{I}_{-n-n_2,-m-m_2}}} I \bigg| \bigg)^{\frac{1}{2}}\bigg(\sum_{S \in \mathbb{S}_{\tilde{l}_2}}\bigg|\bigcup_{\substack{ J \in S \\ J  \in \mathcal{I}_{n_2,m_2}}} J \bigg| \bigg)^{\frac{1}{2}}. \nonumber
\end{align*}
One can further``complete'' $\mathcal{B}$ in the following manner for an appropriate use of the Cauchy-Schwarz inequality.
\begin{align}
\mathcal{B}  = & \sum_{\substack{n_2 \in \mathbb{Z} \\ m_2 \in \mathbb{Z}}}\bigg((C_12^{-n-n_2}|F_1|)^{\mu(1+\epsilon)}(C_12^{-m-m_2}|F_2|)^{(1-\mu)(1+\epsilon)}\sum_{T \in \mathbb{T}_{-l-\tilde{l}_2}}\bigg|\bigcup_{\substack{ I \in T \\ I  \in \mathcal{I}_{-n-n_2,-m-m_2}}} I \bigg| \bigg)^{\frac{1}{2}} \cdot \nonumber \\
& \quad \quad \ \ \bigg((C_22^{n_2}|G_1|)^{\mu(1+\epsilon)}(C_2 2^{m_2}|G_2|)^{(1-\mu)(1+\epsilon)}\sum_{S \in \mathbb{S}_{\tilde{l}_2}}\bigg|\bigcup_{\substack{ J \in S \\ J  \in \mathcal{I}_{n_2,m_2}}} J \bigg|\bigg)^{\frac{1}{2}}\nonumber \\
&\quad \quad \ \ \cdot 2^{n\cdot\frac{1}{2}\mu(1+\epsilon)}2^{m\cdot \frac{1}{2}(1-\mu)(1+\epsilon)}|F_1|^{-\frac{1}{2}\mu(1+\epsilon)}|F_2|^{-\frac{1}{2}(1-\mu)(1+\epsilon)}|G_1|^{-\frac{1}{2}\mu(1+\epsilon)}|G_2|^{-\frac{1}{2}(1-\mu)(1+\epsilon)}  \nonumber  \\
 \leq &  \underbrace{\bigg[\sum_{\substack{n_2 \in \mathbb{Z} \\ m_2 \in \mathbb{Z}}}(C_12^{-n-n_2}|F_1|)^{\mu(1+\epsilon)}(C_12^{-m-m_2}|F_2|)^{(1-\mu)(1+\epsilon)}\sum_{T \in \mathbb{T}_{-l-\tilde{l}_2}}\bigg|\bigcup_{\substack{ I \in T \\ I  \in \mathcal{I}_{-n-n_2,-m-m_2}}} I \bigg|\bigg]^{\frac{1}{2}}}_{\mathcal{B}^1} \cdot \nonumber \\ 
&\underbrace{\bigg[\sum_{\substack{n_2 \in \mathbb{Z} \\ m_2 \in \mathbb{Z}}}(C_22^{n_2}|G_1|)^{\mu(1+\epsilon)}(C_2 2^{m_2}|G_2|)^{(1-\mu)(1+\epsilon)}\sum_{S \in \mathbb{S}_{\tilde{l}_2}}\bigg|\bigcup_{\substack{ J \in S \\ J  \in \mathcal{I}_{n_2,m_2}}} J \bigg|\bigg]^{\frac{1}{2}}}_{\mathcal{B}^2}\nonumber \\
&\ \ \cdot 2^{n\cdot\frac{1}{2}\mu(1+\epsilon)}2^{m\cdot \frac{1}{2}(1-\mu)(1+\epsilon)}|F_1|^{-\frac{1}{2}\mu(1+\epsilon)}|F_2|^{-\frac{1}{2}(1-\mu)(1+\epsilon)}|G_1|^{-\frac{1}{2}\mu(1+\epsilon)}|G_2|^{-\frac{1}{2}(1-\mu)(1+\epsilon)}, \label{B_completed}
%\leq & \bigg(\sum_{\substack{n_2 \in \mathbb{Z} \\ m_2 \in \mathbb{Z}}}\sum_{\substack{T \in \mathbb{T}_{-l-\tilde{l_2}}\\ I_T \in \mathc\mathcal{B}^2al{I}_{-n-n_2,-m-m_2}}}|I_{T}|\bigg)^{\frac{1}{2}}\bigg(\sum_{\substack{n_2 \in \mathbb{Z} \\ m_2 \in \mathbb{Z}}}\sum_{\substack{S \in \mathbb{S}_{\tilde{l_2}}\\ J_S \in \mathcal{J}_{n_2,m_2}}} |J_S |\bigg)^{\frac{1}{2}} 
\end{align} 
for any $\epsilon > 0$, $0 < \mu <1$, where the second inequality follows from the Cauchy-Schwarz inequality. 
%\begin{align*}
%b \leq & \underbrace{\bigg[\sum_{\substack{n_2 \in \mathbb{Z} \\ m_2 \in \mathbb{Z}}}(C_12^{-n-n_2}|F_1|)^{\mu(1+\epsilon)}(C_12^{-m-m_2}|F_2|)^{(1-\mu)(1+\epsilon)}\sum_{\substack{T \in \mathbb{T}_{-l-\tilde{l_2}}\\ I_T \in \mathcal{I}_{-n-n_2,-m-m_2}}}|I_{T}|\bigg]^{\frac{1}{2}}}_{\mathcal{B}^1} \nonumber \\
%&\underbrace{\bigg[\sum_{\substack{n_2 \in \mathbb{Z} \\ m_2 \in \mathbb{Z}}}(C_22^{n_2}|G_1|)^{\mu(1+\epsilon)}(C_2 2^{m_2}|G_2|)^{(1-\mu)(1+\epsilon)}\sum_{\substack{S \in \mathbb{S}_{\tilde{l_2}}\\ J_S \in \mathcal{J}_{n_2,m_2}}} |J_S |\bigg]^{\frac{1}{2}}}_{\mathcal{B}^2}\nonumber \\
%&\ \ \cdot 2^{n\cdot\frac{1}{2}\mu(1+\epsilon)}2^{m\cdot \frac{1}{2}(1-\mu)(1+\epsilon)}|F_1|^{-\frac{1}{2}\mu(1+\epsilon)}|F_2|^{-\frac{1}{2}(1-\mu)(1+\epsilon)}|G_1|^{-\frac{1}{2}\mu(1+\epsilon)}|G_2|^{-\frac{1}{2}(1-\mu)(1+\epsilon)}  \nonumber \\
%\end{align*}

To estimate $\mathcal{B}^1$, one recalls (\ref{b_x}) - which holds for any fixed $l_2 \in \mathbb{Z}$ - to obtain
\begin{align} \label{B_1}
 \mathcal{B}^1 \lesssim & \bigg[\sum_{\substack{n_2 \in \mathbb{Z} \\ m_2 \in \mathbb{Z}}}(C_1 2^{-n-n_2}|F_1|)^{\mu(1+\epsilon)}(C_1 2^{-m-m_2}|F_2|)^{(1-\mu)(1+\epsilon)}\big| \{ Mf_1(x) > C_1 2^{-n-n_2}|F_1|\} \cap \{ Mf_2(x) > C_1 2^{-m-m_2}|F_2|\} \big|\bigg]^{\frac{1}{2}} \nonumber \\
\leq & \bigg[\int (Mf_1(x))^{\mu(1+\epsilon)}(Mf_2(x))^{(1-\mu)(1+\epsilon)} dx\bigg]^{\frac{1}{2}} \nonumber \\
\leq & \Bigg[\bigg(\int (Mf_1(x))^{\mu(1+\epsilon) \frac{1}{\mu}} dx\bigg)^{\mu}\bigg(\int (Mf_2(x))^{(1-\mu)(1+\epsilon) \frac{1}{1-\mu}} dx\bigg)^{1-\mu}\Bigg]^{\frac{1}{2}},
\end{align}
where the second inequality holds by the definition of Lebesgue integration and the last step follows from H\"older's inequality. One can now use the mapping property for the Hardy-Littlewood maximal operator $M: L^{p} \rightarrow L^{p}$ for any $p >1$ and deduces that
\begin{align} 
&\bigg(\int (Mf_1(x))^{1+\epsilon} dx\bigg)^{\mu} \lesssim \|f_1\|_{1+\epsilon}^{(1+\epsilon)\mu} \leq |F_1|^{\mu},  \label{piece1} \\
&\bigg(\int (Mf_2(x))^{1+\epsilon} dx\bigg)^{1-\mu} \lesssim \|f_2\|_{1+\epsilon}^{(1+\epsilon)(1-\mu)} \leq |F_2|^{1-\mu}. \label{piece2}
\end{align}
%By the same reasoning, one has
%$$
%\bigg(\int (Mf_2(x))^{1+\epsilon} dx\bigg)^{1-\mu} \lesssim \|f_2\|_{1+\epsilon}^{(1+\epsilon)(1-\mu)} = |F_2|^{1-\mu}.
%$$
By plugging the estimate (\ref{piece1}) and (\ref{piece2}) into (\ref{B_1}),
\begin{equation} \label{B_1_final}
\mathcal{B}^1 \lesssim |F_1|^{\frac{\mu}{2}}|F_2|^{\frac{1-\mu}{2}}.
\end{equation}
By the same argument with $-n-n_2$ and $-m-m_2$ replaced by $n_2$ and $m_2$ correspondingly, one obtains
\begin{equation} \label{B_2_final}
\mathcal{B}^2 \lesssim  |G_1|^{\frac{\mu}{2}}|G_2|^{\frac{1-\mu}{2}}.
\end{equation}
Application of the estimates for $\mathcal{B}^1$ (\ref{B_1_final}) and $\mathcal{B}^2$ (\ref{B_2_final}) to (\ref{B_completed}) yields 
\begin{equation}\label{b_estimate}
\mathcal{B} \lesssim  |F_1|^{-\frac{\mu}{2}\epsilon}|F_2|^{-\frac{1-\mu}{2}\epsilon}|G_1|^{-\frac{\mu}{2}\epsilon}|G_2|^{-\frac{1-\mu}{2}\epsilon}2^{n\cdot\frac{1}{2}\mu(1+\epsilon)}2^{m\cdot \frac{1}{2}(1-\mu)(1+\epsilon)}. 
\end{equation}
By combining the results for both $\mathcal{A}$ (\ref{a_estimate}) and $\mathcal{B}$ (\ref{b_estimate}), one concludes with the following estimate for (\ref{ns_AB}).
\begin{align} \label{ns_fb}
%& \sum_{\substack{n_2 \in \mathbb{Z} \\ m_2 \in \mathbb{Z} \\ l_2 \in \mathbb{Z}}}\sum_{\substack{T \in \mathbb{T}_{-l-l_2}\\S \in \mathbb{S}_{l_2}}} \bigg|\bigcup_{\substack{I \times J \in T\times S  \\ I \times J \in \mathcal{I}_{-n-n_2,-m-m_2} \times \mathcal{J}_{n_2,m_2} \\ I \times J \in \mathcal{R}_{k_1,k_2}}} I \times J \bigg| \nonumber \\
%&\sum_{\substack{n_2 \in \mathbb{Z} \\ m_2 \in \mathbb{Z} \\ l_2 \in \mathbb{Z}}}\sum_{\substack{T \in \mathbb{T}_{-l-l_2}\\S \in \mathbb{S}_{l_2} \\ I_T \in \mathcal{I}_{-n-n_2,-m-m_2} \\ J_S \in \mathcal{J}_{n_2,m_2}}} \bigg|I_{T} \times J_S \bigg| \nonumber \\
%\lesssim & 
& C_1^{2} C_2^{2} 2^{-n(\frac{1}{p_1} - \theta_1-\frac{1}{2}\mu(1+\epsilon))}2^{-m(\frac{1}{q_1}- \theta_2-\frac{1}{2}(1-\mu)(1+\epsilon))} \cdot \nonumber\\
&  |F_1|^{\frac{1}{p_1}-\frac{\mu}{2}\epsilon}|F_2|^{\frac{1}{q_1}-\frac{1-\mu}{2}\epsilon}|G_1|^{\frac{1}{p_2}-\frac{\mu}{2}\epsilon}|G_2|^{\frac{1}{q_2}-\frac{1-\mu}{2}\epsilon}\cdot 2^{l}\|B^H(f_1,f_2)\|_1^{-1}\|\tilde{B}^H(g_1,g_2)\|_1^{-1},
\end{align}
for any $0 \leq \theta_1, \theta_2 < 1$ with $\theta_1 + \theta_2 = \frac{1}{2}$, $0 <\mu<1$ and $\epsilon > 0$. 

One can now interpolate between the estimates obtained with two different approaches, namely (\ref{ns_sp}) and (\ref{ns_fb}), to derive the following bound for (\ref{ns}).
\begin{align} \label{ns_sum_result}
&C_1^{2}C_2^{2} 2^{-\frac{k_2\gamma\lambda}{2}}2^{-n(\frac{1}{p_1}-\theta_1-\frac{1}{2}\mu(1+\epsilon))(1-\lambda)}2^{-m(\frac{1}{q_1} - \theta_2-\frac{1}{2}(1-\mu)(1+\epsilon))(1-\lambda)} \cdot \nonumber \\
& (2^{l})^{\lambda\frac{(1+\delta)}{2}+(1-\lambda)}\|B^H(f_1,f_2)\|_1^{-\lambda\frac{(1+\delta)}{2}-(1-\lambda)}\|\tilde{B}^H(g_1,g_2)\|_1^{-\lambda\frac{(1+\delta)}{2}-(1-\lambda)} \cdot \nonumber \\
& |F_1|^{\lambda \frac{\mu_1(1+\delta)}{2} + (1-\lambda)(\frac{1}{p_1}-\frac{\mu}{2}\epsilon)}|F_2|^{\lambda \frac{\mu_2(1+\delta)}{2} + (1-\lambda)(\frac{1}{q_1}-\frac{1-\mu}{2}\epsilon)}|G_1|^{\lambda \frac{\nu_1(1+\delta)}{2}+(1-\lambda)(\frac{1}{p_2}-\frac{\mu}{2}\epsilon)}|G_2|^{\lambda\frac{\nu_2(1+\delta)}{2}+(1-\lambda)(\frac{1}{q_2}-\frac{1-\mu}{2}\epsilon)}, \nonumber\\
\end{align}
for some $0 \leq \lambda \leq 1$.
By applying (\ref{ns_sum_result}) to (\ref{form00_set}), one has
\begin{align}
%& |\Lambda^H_{\text{flag}^{0} \otimes \text{flag}^{0}}| \nonumber \\
%\lesssim
& C_1^{3}C_2^{3} C_3^{3} \| h \|_{L^s}\sum_{\substack{n > 0 \\ m > 0 \\ l > 0 \\ k_1 < 0 \\ k_2 \leq K}}2^{-l\lambda(1-\frac{1+\delta}{2})}2^{k_1}2^{k_2(1-\frac{\lambda\gamma}{2})}2^{-n(\frac{1}{p_1}-\theta_1-\frac{1}{2}\mu(1+\epsilon))(1-\lambda)}2^{-m(\frac{1}{q_1}-\theta_2-\frac{1}{2}(1-\mu)(1+\epsilon))(1-\lambda)}   \nonumber \\
&\cdot  |F_1|^{\lambda \frac{\mu_1(1+\delta)}{2} + (1-\lambda)(\frac{1}{p_1}-\frac{\mu}{2}\epsilon)}|F_2|^{\lambda \frac{\mu_2(1+\delta)}{2} + (1-\lambda)(\frac{1}{q_1}-\frac{1-\mu}{2}\epsilon)}|G_1|^{\lambda \frac{\nu_1(1+\delta)}{2}+(1-\lambda)(\frac{1}{p_2}-\frac{\mu}{2}\epsilon)}|G_2|^{\lambda\frac{\nu_2(1+\delta)}{2}+(1-\lambda)(\frac{1}{q_2}-\frac{1-\mu}{2}\epsilon)} \nonumber \\
& \cdot \|B^H(f_1,f_2)\|_1^{\lambda(1-\frac{1+\delta}{2})}\|\tilde{B}^H(g_1,g_2)\|_1^{\lambda(1-\frac{1+\delta}{2})}. \label{form_nofixed_almost}
\end{align}
One notices that there exists $\epsilon > 0$, $0 < \mu < 1$ and $0 <\theta_1<\frac{1}{2}$ such that
\begin{align} \label{nec_condition}
&\frac{1}{p_1}-\theta_1-\frac{1}{2}\mu(1+\epsilon) >   0,  \nonumber \\
&\frac{1}{q_1} - \theta_2-\frac{1}{2}(1-\mu)(1+\epsilon) >   0.
\end{align}

% !!!!! to move
\begin{remark} \label{rmk_easyhard_exponent} 
One realizes that (\ref{nec_condition}) imposes a necessary condition on the range of exponents. In particular, 
\begin{equation} \label{>1/2}
\frac{1}{p_1} + \frac{1}{q_1} - (\theta_1 + \theta_2) > \frac{1}{2}\mu(1+ \epsilon) + \frac{1}{2}(1-\mu)(1+\epsilon).
\end{equation}
Using the fact that $\theta_1 + \theta_2 = \frac{1}{2}$, one can rewrite (\ref{>1/2}) as
$$
\frac{1}{p_1} + \frac{1}{q_1} > 1+ \frac{\epsilon}{2}.
$$ 
As a consequence, the case 
\begin{equation}\label{hard_exponent}
1< \frac{1}{p_1} + \frac{1}{q_1} = \frac{1}{p_2} + \frac{1}{q_2} < 2 
\end{equation}
 can be treated by the current argument. Meanwhile, the case $0 < \frac{1}{p_1} + \frac{1}{q_1} = \frac{1}{p_2} + \frac{1}{q_2} \leq 1$ follows a simpler argument which resembles the one for the estimates involving $L^{\infty}$ norms and will be postponed to Section \ref{section_thm_inf_haar}.
\end{remark}

Imposed by (\ref{nec_condition}), the geometric series involving $2^{-n}$ and $2^{-m}$ are convergent. The convergence of series involving $2^{k_1}$ is trivial. One also observes that for any $0 < \lambda < 1$ and $0 < \delta < 1$,
$$
\lambda(1-\frac{1+\delta}{2}) > 0,
$$
which implies that the series involving $2^{-l}$ is convergent.  One can separate the cases when $k_2 > 0 $ and $k_2 \leq 0$ and select $\gamma >1$ in each case to make the series about $2^{k_2}$ convergent. Therefore, one can estimate (\ref{form_nofixed_almost}) by 
\begin{align}
%& |\Lambda^H_{\text{flag}^{0} \otimes \text{flag}^{0}}| \nonumber \\
%\lesssim 
& C_1^3 C_2^3 C_3^{3}\| h \|_{L^s} \|B^H(f_1,f_2)\|_1^{\lambda(1-\frac{1+\delta}{2})}\|\tilde{B}^H(g_1,g_2)\|_1^{\lambda(1-\frac{1+\delta}{2})}  \nonumber \\
&\cdot  |F_1|^{\lambda \frac{\mu_1(1+\delta)}{2} + (1-\lambda)(\frac{1}{p_1}-\frac{\mu}{2}\epsilon)}|F_2|^{\lambda \frac{\mu_2(1+\delta)}{2} + (1-\lambda)(\frac{1}{q_1}-\frac{1-\mu}{2}\epsilon)}|G_1|^{\lambda \frac{\nu_1(1+\delta)}{2}+(1-\lambda)(\frac{1}{p_2}-\frac{\mu}{2}\epsilon)}|G_2|^{\lambda\frac{\nu_2(1+\delta)}{2}+(1-\lambda)(\frac{1}{q_2}-\frac{1-\mu}{2}\epsilon)},  \label{form_nofixed_exponent} 
\end{align}
where one can apply Lemma \ref{B_global_norm} to derive 
\begin{align}
& \|B^H(f_1,f_2)\|_1 \lesssim |F_1|^{\rho}|F_2|^{1-\rho}, \label{B_norm1} \\
& \|\tilde{B}^H(g_1,g_2)\|_1 \lesssim |G_1|^{\rho'}|G_2|^{1-\rho'}, \label{B_norm2}
\end{align}
with the corresponding exponent to be positive as guaranteed by the fact that $0 < \lambda,\delta < 1$. From , (\ref{form_nofixed_exponent}), (\ref{B_norm1}) and (\ref{B_norm2}), one thus obtains
\begin{align} \label{exp00}
 |\Lambda^H_{\text{flag}^{0} \otimes \text{flag}^{0}}|\lesssim&  C_1^3 C_2^3 C_3^{3}\|h\|_{L^s(\mathbb{R}^2)} |F_1|^{\lambda \frac{\mu_1(1+\delta)}{2} + (1-\lambda)(\frac{1}{p_1}-\frac{\mu}{2}\epsilon) + \rho\lambda(1-\frac{1+\delta}{2})}|F_2|^{\lambda \frac{\mu_2(1+\delta)}{2} + (1-\lambda)(\frac{1}{q_1}-\frac{1-\mu}{2}\epsilon) + (1-\rho)\lambda(1-\frac{1+\delta}{2})} \nonumber \\
 & \cdot |G_1|^{\lambda \frac{\nu_1(1+\delta)}{2}+(1-\lambda)(\frac{1}{p_2}-\frac{\mu}{2}\epsilon)+ \rho'\lambda(1-\frac{1+\delta}{2})}|G_2|^{\lambda\frac{\nu_2(1+\delta)}{2}+(1-\lambda)(\frac{1}{q_2}-\frac{1-\mu}{2}\epsilon)+(1-\rho')\lambda(1-\frac{1+\delta}{2})}. 
 %|F_1|^{(\frac{1}{p_1}-\frac{\mu}{2}\epsilon)(1-\lambda) + \rho \lambda }|F_2|^{(\frac{1}{q_1}-\frac{1-\mu}{2}\epsilon)(1-\lambda)+ (1-\rho)\lambda}|G_1|^{(\frac{1}{p_2}-\frac{\mu}{2}\epsilon)(1-\lambda) + \rho' \lambda}|G_2|^{(\frac{1}{q_2}-\frac{1-\mu}{2}\epsilon)(1-\lambda)+ (1-\rho')\lambda}
\end{align}
With a little abuse of notation, we use $\tilde{p_i}$ and $\tilde{q_i}$, $i = 1,2$ to represent $p_i$ and $q_i$ in the above argument. And from now on, $p_i$ and $q_i$ stand for the boundedness exponents specified in the main theorem. One has the freedom to choose $1 < \tilde{p_i}, \tilde{q_i}< \infty$, $0 < \mu,\lambda < 1$ and $\epsilon > 0 $ such that
\begin{align} \label{exp_tilde}
& \lambda \frac{\mu_1(1+\delta)}{2} + (1-\lambda)(\frac{1}{\tilde{p_1}}-\frac{\mu}{2}\epsilon) + \rho\lambda(1-\frac{1+\delta}{2}) =  \frac{1}{p_1} \nonumber \\
& \lambda \frac{\mu_2(1+\delta)}{2} + (1-\lambda)(\frac{1}{\tilde{q_1}}-\frac{1-\mu}{2}\epsilon) + (1-\rho)\lambda(1-\frac{1+\delta}{2}) = \frac{1}{q_1} \nonumber \\
& \lambda \frac{\nu_1(1+\delta)}{2}+(1-\lambda)(\frac{1}{\tilde{p_2}}-\frac{\mu}{2}\epsilon)+ \rho'\lambda(1-\frac{1+\delta}{2}) = \frac{1}{p_2} \nonumber \\
& \lambda\frac{\nu_2(1+\delta)}{2}+(1-\lambda)(\frac{1}{\tilde{q_2}}-\frac{1-\mu}{2}\epsilon)+(1-\rho')\lambda(1-\frac{1+\delta}{2}) = \frac{1}{q_2}.
\end{align}
\begin{remark}
To see that above equations can hold, one can view the parts without $\tilde{p_i}$ and $\tilde{q_i}$ as perturbations which can be controlled small. More precisely, when $0 < \delta < 1$ is close to $1$, 
$$
\lambda(1-\frac{1+\delta}{2}) \ll 1.
$$
When $0 < \lambda < 1$ is close to $0$, one has
$$
\lambda \frac{\mu_1(1+\delta)}{2},  \lambda \frac{\mu_2(1+\delta)}{2},  \lambda \frac{\nu_1(1+\delta)}{2}, \lambda\frac{\nu_2(1+\delta)}{2} \ll 1
$$
and 
\begin{align*}
&\frac{1}{p_i} - (1-\lambda)(\frac{1}{\tilde{p_i}}-\frac{\mu}{2}\epsilon) \ll 1,\nonumber \\
& \frac{1}{q_i} - (1-\lambda)(\frac{1}{\tilde{q_i}}-\frac{1-\mu}{2}\epsilon)\ll 1,
\end{align*}
for $ i = 1,2$.
\end{remark}
It is also necessary to check is that $\tilde{p_i}$ and $\tilde{q_i}$ satisfy the conditions which have been used to obtain (\ref{exp00}), namely
\begin{align}
\frac{1}{\tilde{p_1}} + \frac{1}{\tilde{q_1}} = &  \frac{1}{\tilde{p_2}} + \frac{1}{\tilde{q_2}} > 1.
\end{align}
One can easily verify the first equation and the second inequality by manipulating (\ref{exp_tilde}).
As a result, we have derived that
$$|\Lambda^H_{\text{flag}^{0} \otimes \text{flag}^{0}}(f_1,f_2,g_1,g_2,h, \chi_{E'})| \lesssim C_1^3 C_2^3 C_3^{3}|F_1|^{\frac{1}{p_1}}|F_2|^{\frac{1}{q_1}}|G_1|^{\frac{1}{p_2}}|G_2|^{\frac{1}{q_2}} \|h\|_{L^s(\mathbb{R}^2)}.
$$
\vskip .25in

%nonlac

\section{Proof of Theorem \ref{thm_weak_inf_mod} for $\Pi_{\text{flag}^{\#_1} \otimes \text{flag}^{\#_2}}$ - Haar Model} \label{section_thm_inf_fixed_haar}
%\begin{proof}[Proof of Theorem \ref{fixed_scale_inf}]
One can mimic the proof in Section \ref{section_thm_haar_fixed} with a change of perspectives on size estimates. More precisely, one applies some trivial size estimates for functions $f_2$ and $g_2$ lying in $L^{\infty}$ spaces while one needs to pay respect to the fact that $f_1$ and $g_1$ could lie in $L^p$ space for any $p >1$. Such perspective is demonstrated in the stopping-time decomposition and in the definition of the exceptional set. 
\subsection{Localization} \label{section_thm_inf_fixed_haar_localization}
One defines
$$\Omega := \Omega^1 \cup \Omega^2,$$ where
\begin{align*}
\Omega^1 := &\bigcup_{n_1 \in \mathbb{Z}}\{x:Mf_1(x) > C_1 2^{n_1}\|f_1\|_{p}\} \times \{y:Mg_1(y) > C_2 2^{-n_1}\|g_1\|_{p}\}, \nonumber \\
\Omega^2 := & \{(x,y) \in \mathbb{R}^2: SSh(x,y) > C_3 \|h\|_{L^s(\mathbb{R}^2)}\}, \nonumber \\
\end{align*}
and
$$Enl(\Omega) := \{(x,y) \in \mathbb{R}^2: MM\chi_{\Omega}(x,y) > \frac{1}{100}\}.$$
Let $$E' := E \setminus Enl(\Omega).$$
It is not difficult to check that given $C_1, C_2$ and $C_3$ are sufficiently large, $|E'| \sim |E|$ where $|E|$ can be assumed to be 1. It suffices to prove that the multilinear form defined in (\ref{form_haar_larger})
%\begin{equation}
%\Lambda_{\text{flag}^{\#1} \otimes \text{flag}^{\#2}}(f_1, f_2, g_1, g_2, h, \chi_{E'}) := \langle \Pi_{\text{flag}^{\#1} \otimes \text{flag}^{\#2}}(f_1, f_2, g_1, g_2, h), \chi_{E'} \rangle
%\end{equation}
satisfies %the restricted weak-type estimate
\begin{equation}
|\Lambda^H_{\text{flag}^{\#1} \otimes \text{flag}^{\#2}}(f_1,f_2,g_1,g_2,h, \chi_{E'})| \lesssim \|f_1\|_{L^p(\mathbb{R})}\|f_2\|_{L^{\infty}(\mathbb{R})} \|g_1\|_{L^p(\mathbb{R})}\|g_2\|_{L^{\infty}(\mathbb{R})} \|h\|_{L^{s}(\mathbb{R}^2)}.
\end{equation}
\subsection{Summary of stopping-time decompositions.} \label{section_thm_inf_fixed_haar_summary}
{\fontsize{10}{10}
\begin{center}
\begin{tabular}{ l l l }
I. Tensor-type stopping-time decomposition I on $\mathcal{I} \times \mathcal{J}$& $\longrightarrow$ & $I \times J \in \mathcal{I}_{-n-n_2}^{generic} \times \mathcal{J}_{n_2}^{generic}$ \\
 & & $(n_2 \in \mathbb{Z}, n > 0)$ \\ 
II. General two-dimensional level sets stopping-time decomposition & $\longrightarrow$ & $I \times J \in  \mathcal{R}_{k_1,k_2} $  \\
\ \ \ \ on $\mathcal{I} \times \mathcal{J}$& & $(k_1 <0, k_2 \leq K)$
\end{tabular}
\end{center}}
where
\begin{align*}
\mathcal{I}_{-n-n_2}^{generic} := & \{ I \in \mathcal{I} \setminus \mathcal{I}_{-n-n_2+1}: \left| I \cap \Omega^{generic}_{-n-n_2}\right| > \frac{1}{10}|I| \}, \nonumber \\
\mathcal{J}_{n_2}^{generic} := & \{ J \in \mathcal{J} \setminus \mathcal{J}_{n_2+1}: \left| I \cap \tilde{\Omega}^{generic}_{n_2}\right| > \frac{1}{10}|J| \},
\end{align*}
with
\begin{align*}
\Omega_{-n-n_2}^{generic} := &\{x: Mf_1(x)> C_1 2^{-n-n_2}\|f_1\|_p \}, \nonumber \\
\tilde{\Omega}_{n_2}^{generic}  := & \{y: Mg_1(y)> C_2 2^{n_2} \|g_1\|_p \}.
\end{align*}
\subsection{Application of stopping-time decompositions} \label{section_thm_inf_fixed_haar_application_st}
One can now apply the stopping-time decompositions and follow the same argument in Section \ref{section_thm_haar_fixed} to deduce that
{\fontsize{10}{10}\begin{align} \label{form11_inf}
& |\Lambda^H_{\text{flag}^{\#1} \otimes \text{flag}^{\#2}}(f_1,f_2,g_1,g_2,h, \chi_{E'})|  \nonumber\\
\lesssim &\bigg|\displaystyle \sum_{\substack{n> 0 \\ k_1 < 0 \\ k_2 \leq K}} \sum_{n_2 \in \mathbb{Z}}\sum_{\substack{I \times J \in \mathcal{I}^{generic}_{-n-n_2} \times \mathcal{J}^{generic}_{n_2}\\I \times J \in \mathcal{R}_{k_1,k_2}}} \frac{1}{|I|^{\frac{1}{2}} |J|^{\frac{1}{2}}} \langle B_I^{\#_1,H}(f_1,f_2),\vphi_I^{1,H} \rangle \langle \tilde{B}_J^{\#_2,H} (g_1,g_2),\vphi_J^{1,H} \rangle \langle h, \psi_I^{2} \otimes \psi_J^{2} \rangle \langle \chi_{E'},\psi_I^{3,H} \otimes \psi_J^{3,H} \rangle \bigg| \nonumber \\\nonumber \\
%\lesssim & \sum_{\substack{n> 0 \\ k_1 < 0 \\ k_2 \leq K}} \sum_{n_2 \in \mathbb{Z}}\sum_{\substack{I \times J \in \mathcal{I}'_{-n-n_2} \times \mathcal{J}'_{n_2}\\I \times J \in \mathcal{R}_{k_1,k_2}}} \frac{|\langle B^{\#_1}_I(f_1,f_2),\vphi_I^1 \rangle|}{|I|^{\frac{1}{2}}}   \frac{| \langle \tilde{B}^{\#_2}_J (g_1,g_2),\vphi_J^1 \rangle|}{|J|^{\frac{1}{2}}}\cdot \nonumber \\
%&\quad \quad \quad \quad \quad \quad \quad \quad \quad \quad \int_{(\Omega_{k_1})^c \cap (\Omega_{k_2})^c}  \frac{|\langle h, \psi_I^{2} \otimes \psi_J^{2} \rangle|}{|I|^{\frac{1}{2}}|J|^{\frac{1}{2}}} \frac{|\langle \chi_{E'},\psi_I^{3,H} \otimes \psi_J^{3,H} \rangle|}{|I|^{\frac{1}{2}}|J|^{\frac{1}{2}}}\chi_{I}(x) \chi_{J}(y) dx dy \nonumber \\
%\lesssim & \sum_{\substack{n> 0 \\ k_1 < 0 \\ k_2 \leq K}} \sum_{n_2 \in \mathbb{Z}}\sum_{\substack{I \times J \in \mathcal{I}'_{-n-n_2} \times \mathcal{J}'_{n_2}\\I \times J \in \mathcal{R}_{k_1,k_2}}} \frac{|\langle B^{\#_1}_I(f_1,f_2),\vphi_I^1 \rangle|}{|I|^{\frac{1}{2}}}   \frac{| \langle \tilde{B}^{\#_2}_J (g_1,g_2),\vphi_J^1 \rangle|}{|J|^{\frac{1}{2}}}\cdot \nonumber \\
%&\quad \quad \quad \quad \quad \quad \quad \quad \quad \quad \int_{(\Omega_{k_1})^c \cap (\Omega_{k_2})^c}  (SS)^Hh(x,y)) (SS)^H\chi_{E'}(x,y) dx dy \nonumber \\ \nonumber\\
\lesssim & \sum_{\substack{n> 0 \\ k_1 < 0 \\ k_2 \leq K}} \sum_{n_2 \in \mathbb{Z}}\sup_{I \in \mathcal{I}^{generic}_{-n-n_2}} \frac{|\langle B_I^{\#_1,H}(f_1,f_2),\vphi_I^{1,H} \rangle|}{|I|^{\frac{1}{2}}}   \cdot \sup_{J \in \mathcal{J}^{generic}_{n_2}}\frac{| \langle \tilde{B}_J^{\#_2,H} (g_1,g_2),\vphi_J^{1,H} \rangle|}{|J|^{\frac{1}{2}}}\cdot C_3 2^{k_1} \| h \|_{L^s} 2^{k_2}  \cdot  \nonumber \\
&\quad \quad \quad \quad \quad \quad \bigg|\big(\bigcup_{R\in \mathcal{R}_{k_1,k_2}} R\big) \cap \big(\bigcup_{I \in \mathcal{I}'_{-n-n_2}} I \times \bigcup_{J \in \mathcal{J}'_{n_2}}J\big)\bigg|.  
\end{align}}
%where for the second inequality one has again used the fact that 
%$$ |I\times J \cap (\Omega_{k_1})^c |  \geq \frac{99}{100}|I\times J|$$
%$$ |I \times J \cap (\Omega_{k_2})^c |  \geq \frac{99}{100}|I \times J|$$
%with $\Omega_{k_1} := \{ (SS)^Hh > C_3 2^{k_1}\|h\|_{L^s(\mathbb{R}^2)}\}$, and $\Omega_{k_2}:= \{ (SS)^H\chi_{E'} > C_3 2^{k_2}\}$.
To estimate $\displaystyle \sup_{I \in \mathcal{I}^{generic}_{-n-n_2}} \frac{|\langle B_I^{\#_1,H}(f_1,f_2),\vphi_I^{1,H} \rangle|}{|I|^{\frac{1}{2}}} $, one can now apply Lemma \ref{B_size} with $S:= \{Mf_1 \leq C_1 2^{-n-n_2+1}|F_1|^{\frac{1}{p}} \}$ and obtain
$$
\sup_{I \in \mathcal{I}^{generic}_{-n-n_2}}\frac{|\langle B_I^{\#_1,H}(f_1,f_2),\vphi_I^{1,H} \rangle|}{|I|^{\frac{1}{2}}} \lesssim \sup_{K \cap S \neq \emptyset}\frac{|\langle f_1, \vphi^1_K \rangle|}{|K|^{\frac{1}{2}}} \sup_{K \cap S \neq \emptyset} \frac{|\langle f_2, \phi_K^2 \rangle|}{|K|^{\frac{1}{2}}},
$$
where by the definition of $S$,
$$
\sup_{K \cap S \neq \emptyset}\frac{|\langle f_1, \vphi^1_K \rangle|}{|K|^{\frac{1}{2}}} \lesssim C_12^{-n-n_2}\|f_1\|_p, 
$$
and by the fact that $f_2 \in L^{\infty}$,
$$
\sup_{K \cap S \neq \emptyset} \frac{|\langle f_2, \phi_K^2 \rangle|}{|K|^{\frac{1}{2}}} \lesssim \|f_2\|_{\infty}.
$$
As a result, 
\begin{equation} \label{est_x}
\sup_{I \in \mathcal{I}^{generic}_{-n-n_2}}\frac{|\langle B_I^{\#_1,H}(f_1,f_2),\vphi_I^{1,H} \rangle|}{|I|^{\frac{1}{2}}} \lesssim C_12^{-n-n_2}\|f_1\|_p\|f_2\|_{\infty}.
\end{equation}
By a similar reasoning,
\begin{equation} \label{est_y}
\sup_{J \in \mathcal{J}^{generic}_{n_2}}\frac{|\langle \tilde{B}_J^{\#_2,H}(g_1,g_2),\vphi_J^{1,H} \rangle|}{|J|^{\frac{1}{2}}} \lesssim C_2 2^{n_2}\|g_1\|_p\|g_2\|_{\infty}.
\end{equation}
When applying the estimates (\ref{est_x}) and (\ref{est_y}) to (\ref{form11_inf}), one concludes that
\begin{align*}
|\Lambda^{H}_{\text{flag}^{\#1} \otimes \text{flag}^{\#2}}| \lesssim %& C_1 C_2 C_3^2 \sum_{\substack{n > 0 \\ k_1 < 0 \\ k_2 \leq K}} \sum_{ n_2 \in \mathbb{Z}}2^{-n-n_2}|F_1|^{\frac{1}{p}} 2^{n_2}|G_1|^{\frac{1}{p}} 2^{k_1} \| h \|_{L^s} 2^{k_2}  \cdot  \bigg|\big(\bigcup_{R\in \mathcal{R}_{k_1,k_2}} R\big) \cap \big(\bigcup_{I \in \mathcal{I}_{-n-n_2,-m-m_2}} I \times \bigcup_{J \in \mathcal{J}_{n_2,m_2}}J\big)\bigg|\nonumber \\ 
&C_1 C_2 C_3^2  \|f_1\|_{L^p(\mathbb{R})}\|f_2\|_{L^{\infty}(\mathbb{R})} \|g_1\|_{L^p(\mathbb{R})}\|g_2\|_{L^{\infty}(\mathbb{R})} \|h\|_{L^{s}(\mathbb{R}^2)} \cdot \\
& \sum_{\substack{n > 0 \\ k_1 < 0 \\ k_2 \leq K}} 2^{-n}2^{k_1} 2^{k_2}  \sum_{n_2 \in \mathbb{Z}}\bigg|\big(\bigcup_{R\in \mathcal{R}_{k_1,k_2}} R\big) \cap \big(\bigcup_{I \in \mathcal{I}^{generic}_{-n-n_2,-m-m_2}} I \times \bigcup_{J \in \mathcal{J}^{generic}_{n_2,m_2}}J\big)\bigg|\nonumber \\
\lesssim &C_1 C_2 C_3^2 \|f_1\|_{L^p(\mathbb{R})}\|f_2\|_{L^{\infty}(\mathbb{R})} \|g_1\|_{L^p(\mathbb{R})}\|g_2\|_{L^{\infty}(\mathbb{R})} \|h\|_{L^{s}(\mathbb{R}^2)} \sum_{\substack{n > 0 \\ k_1 < 0 \\ k_2 \leq K}} 2^{-n}2^{k_1(1-\frac{s}{2})} 2^{k_2(1-\frac{\gamma}{2})}, \nonumber 
\end{align*}
where the last inequality follows from the sparsity condition. With proper choice of $\gamma >1$, one obtains the desired estimate.
%\end{proof}

\section{Proof of Theorem \ref{thm_weak_inf_mod}  for $\Pi_{\text{flag}^0 \otimes \text{flag}^0}$ - Haar Model}\label{section_thm_inf_haar}
One interesting fact is that when 
\begin{equation} \label{easy_case}
\frac{1}{p_1} + \frac{1}{q_1} = \frac{1}{p_2} + \frac{1}{q_2} \leq 1,
\end{equation}
Theorem \ref{thm_weak_mod} can be proved by a simpler argument as commented in Remark \ref{rmk_easyhard_exponent} of Section \ref{section_thm_haar}. And Theorem \ref{thm_weak_inf_mod} for the model $\Pi_{\text{flag}^0 \otimes \text{flag}^0}$ can be viewed as a sub-case and proved by the same argument. The key idea is that in the case specified in (\ref{easy_case}), one no longer needs the localized operators $B^H_I$ and $\tilde{B}^H_J$ (defined in (\ref{B_0_local_haar_simplified}) and (\ref{B_local_definition_haar})) in the proof and the operators $B$ and $\tilde{B}$ (defined in (\ref{B_global_proof}) and (\ref{B_global_definition})) will be involved. In particular, we would consider the operator 
\begin{align} \label{Pi_nofixed_easy}
\Pi^{H,3}_{\text{flag}^0 \otimes \text{flag}^0}(f_1,f_2,g_1,g_2,h):= \sum_{I \times J \in \mathcal{R}} \frac{1}{|I|^{\frac{1}{2}} |J|^{\frac{1}{2}}} \langle B_I(f_1,f_2),\vphi_I^{1} \rangle \langle \tilde{B}_J(g_1,g_2),\vphi_J^{1} \rangle \langle h, \psi_I^{2} \otimes \psi_J^{2} \rangle \psi_I^{3,H} \otimes \psi_J^{3,H},
\end{align}
where the notation for the operator represents the Haar assumption on the third families of the functions $(\psi_I^{3,H})_{I \in \mathcal{I}}$ and $(\psi_J^{3,H})_{J \in \mathcal{J}}$.
%\begin{theorem}
%Let $\Pi$ be defined as in Theorem 5.1. For $1< p, s < \infty$, $\frac{1}{p} + \frac{1}{s} = \frac{1}{r}$, 
%$$
%\Pi: L^{p}(\mathbb{R}^2) \times L^{\infty}(\mathbb{R}^2) \times L^{s}(\mathbb{R}^2) \rightarrow L^{r}(\mathbb{R}^2)
%$$
%and the following mixed-norm estimates hold true as well:
%$$
%\Pi: L_x^{p}(L_y^{\infty}) \times L_x^{\infty}(L_y^p) \times L^{s}(\mathbb{R}^2) \rightarrow L^{r}(\mathbb{R}^2).
%$$
%\end{theorem}
%We will prove the first part and it is not hard to see that the mixed-norm estimates can be obtained using the same argument by changing the roles of $g_1$ and $g_2$. 
Let
$$
\frac{1}{t} := \frac{1}{p_1} + \frac{1}{q_1} = \frac{1}{p_2} + \frac{1}{q_2},
$$
and the condition of the exponents (\ref{easy_case}) translates to
$$
t \geq 1.
$$

\subsection{Localization.} \label{section_thm_inf_haar_localization}
One first defines 
$$\Omega := \Omega^1 \cup \Omega^2,$$ where
\begin{align*}
\displaystyle \Omega^1 := &\bigcup_{l_2 \in \mathbb{Z}} \{x:MB(f_1,f_2)(x) > C_1 2^{-l_2}\| B\|_t\} \times \{y: M\tilde{B}(g_1,g_2)(y) > C_2 2^{l_2}\|\tilde{B}\|_t\}, \nonumber \\
\Omega^2 := & \{(x,y) \in \mathbb{R}^2: SSh (x,y)> C_3 \|h\|_{L^s(\mathbb{R}^2)}\}, \nonumber \\
\end{align*}
and 
$$Enl(\Omega) := \{(x,y) \in \mathbb{R}^2: MM\chi_{\Omega}(x,y) > \frac{1}{100}\}.$$
Let 
$$
E' := E \setminus Enl(\Omega).
$$
\begin{remark}
We shall notice that $t \geq 1$ allows one to use the mapping property of the Hardy-Littlewood maximal operator, which plays an essential role in the estimate of $|\Omega|$. 
\end{remark}
A straightforward computation shows $|E'| \sim |E|$ given that $C_1, C_2$ and $C_3$ are sufficiently large. It suffices to assume that $|E'| \sim |E| = 1$ and to prove that the multilinear form
\begin{equation}
\Lambda^{H,3}_{\text{flag}^{0} \otimes \text{flag}^{0}}(f_1, f_2, g_1, g_2, h, \chi_{E'}) := \langle \Pi^{H,3}_{\text{flag}^{0} \otimes \text{flag}^{0}}(f_1, f_2, g_1, g_2, h), \chi_{E'} \rangle
\end{equation}
satisfies 
\begin{equation}
|\Lambda^{H,3}_{\text{flag}^{0} \otimes \text{flag}^{0}}(f_1,f_2,g_1,g_2,h, \chi_{E'})| \lesssim \|f_1\|_{L^p(\mathbb{R})}\|f_2\|_{L^{\infty}(\mathbb{R})} \|g_1\|_{L^p(\mathbb{R})}\|g_2\|_{L^{\infty}(\mathbb{R})} \|h\|_{L^{s}(\mathbb{R}^2)}. 
%|F_1|^{\frac{1}{p_1}} |G_1|^{\frac{1}{p_2}} |F_2|^{\frac{1}{q_1}} |G_2|^{\frac{1}{q_2}} \|h\|_{L^{s}(\mathbb{R}^2)}.
\end{equation}
\subsection{Summary of stopping-time decompositions.} \label{section_thm_inf_haar_summary}

\begin{center}
\begin{tabular}{ c c c }
General two-dimensional level sets stopping-time decomposition& $\longrightarrow$ & $I \times J \in \mathcal{R}_{k_1,k_2} $  \\
on $ \mathcal{I} \times \mathcal{J}$ & & $(k_1 <0, k_2 \leq K)$
\end{tabular}
\end{center}
\vskip .25in

\begin{comment}}
As before, one defines $E' := E \setminus \tilde{\Omega}$ for any $E \subseteq \mathbb{R}^2$ with $|E| < \infty$. It suffices to prove that
$$\Lambda(f_1\otimes g_1, f_2 \otimes g_2, h, \chi_{E'})  :=  \displaystyle \sum_{I \times J \in \mathcal{I} \times \mathcal{J}} \frac{1}{|I|^{\frac{1}{2}} |J|} \langle B_I(f_1,f_2),\vphi_I^1 \rangle \langle g_1, \vphi_J^1 \rangle \langle g_2, \vphi_J^2 \rangle  \langle h, \psi_I^{2} \otimes \psi_J^{2} \rangle \langle \chi_{E'},\psi_I^{3,H} \otimes \psi_J^{3,H} \rangle$$
satisfies 
$$|\Lambda(f_1\otimes g_1, f_2 \otimes g_2, h, \chi_{E'})| \lesssim \|f_1\|_{L_x^{p}} \|g_1\|_{L_y^{p}}\|h\|_{L^s(\mathbb{R}^2)}.$$
\end{comment}
\subsection{Application of the stopping-time decomposition} \label{section_thm_inf_haar_application_st}
One performs the \textit{general two-dimensional level sets stopping-time decomposition} with respect to the hybrid operators as specified in the definition of the exceptional set. It would be evident from the argument below that there is no stopping-time decomposition necessary for the maximal functions of $B$ and $\tilde{B}$ (defined in (\ref{B_global_definition}) and (\ref{B_global_proof})). One brief explanantion is that only ``averages'' for $B$ and $\tilde{B}$ are required while the measurement of the set where the averages are attained is not. As a consequence, the macro-control of the averages would be sufficient and the stopping-time decomposition, which can be seen as a more delicate ``slice-by-slice'' or ``level-by-level'' partition, is not compulsory. More precisely,
\begin{align} \label{form00_inf}
&|\Lambda^{H,3}_{\text{flag}^{0} \otimes \text{flag}^{0}}(f_1,f_2,g_1,g_2,h, \chi_{E'})|\nonumber  \\
= &\displaystyle \bigg|\sum_{\substack{ k_1 < 0 \\ k_2 \leq K}} \sum_{\substack{I \times J \in \mathcal{R}_{k_1,k_2}}} \frac{1}{|I|^{\frac{1}{2}} |J|^{\frac{1}{2}}} \langle B_I(f_1,f_2),\vphi_I^{1} \rangle \langle \tilde{B}_J(g_1, g_2), \vphi_J^{1} \rangle \langle h, \psi_I^{2} \otimes \psi_J^{2} \rangle \langle \chi_{E'},\psi_I^{3,H} \otimes \psi_J^{3,H} \rangle \bigg|\nonumber \\
\lesssim &  \sum_{\substack{k_1 < 0 \\ k_2 \leq K}}\displaystyle \sup_{I \times J \in \mathcal{I}\times \mathcal{J}} \bigg(\frac{|\langle B_I(f_1,f_2),\vphi_I^{1} \rangle|}{|I|^{\frac{1}{2}}} \frac{|\langle \tilde{B}_J(g_1, g_2), \vphi_J^{1} \rangle|}{|J|^{\frac{1}{2}}} \bigg) \cdot C_3^2 2^{k_1}\|h\|_{L^s(\mathbb{R}^2)} 2^{k_2} \bigg|\bigcup_{\substack{I \times J \in \mathcal{R}_{k_1,k_2}}}I \times J\bigg|.
\end{align}
By the same esimate applied in previous sections, namely (\ref{int_area}), one has
\begin{equation} \label{measure_ch9}
\bigg|\bigcup_{\substack{I \times J \in \mathcal{R}_{k_1,k_2}}}I \times J\bigg|  \lesssim \min(C_3^{-1}2^{-k_1s}, C_3^{-\gamma}2^{-k_2 \gamma}),
\end{equation}
for any $\gamma >1$. Meanwhile, an argument similar to the proof of Observation \ref{obs_st_B} in Section \ref{section_thm_haar_tensor_1d_maximal} implies that
\begin{equation} \label{B_global_ch9}
\frac{|\langle B_I(f_1,f_2),\vphi_I^{1,H} \rangle|}{|I|^{\frac{1}{2}}} \frac{|\langle \tilde{B}_J(g_1, g_2), \vphi_J^{1,H} \rangle|}{|J|^{\frac{1}{2}}} \lesssim C_1 C_2 \|B(f_1,f_2)\|_t \|\tilde{B}(g_1,g_2)\|_t,
\end{equation}
for any $I \times J \cap Enl(\Omega)^c \neq \emptyset$ as assumed in the Haar model.
By applying (\ref{measure_ch9}) and (\ref{B_global_ch9}) to (\ref{form00_inf}), one has
\begin{equation} \label{form00_inf_almost}
%|\Lambda^{H,3}_{\text{flag}^{0} \otimes \text{flag}^{0}}| \lesssim
 C_1 C_2 C_3^2  \| h \|_{L^s(\mathbb{R}^2)} \|B(f_1,f_2)\|_t \|\tilde{B}(g_1,g_2)\|_t \sum_{\substack{k_1 < 0 \\ k_2 \leq K}} C_3 2^{k_1(1-\frac{s}{2})}2^{k_2(1-\frac{\gamma}{2})} \lesssim  \| h \|_{L^s(\mathbb{R}^2)} \|B(f_1,f_2)\|_t \|\tilde{B}(g_1,g_2)\|_t,
\end{equation}
with appropriate choice of $\gamma>1$. One can now invoke Lemma \ref{B_global_norm} to conclude the estimate of (\ref{form00_inf_almost}). In particular, for $1 < p_i, q_j \leq \infty$ for $i, j = 1, 2$ satisfying (\ref{easy_case}),
\begin{align} \label{B_easy} 
\|B(f_1,f_2)\|_t \lesssim & \|f_1\|_{p_1} \|f_2\|_{q_1} \ \ \text{and} \ \  \|\tilde{B}(g_1,g_2)\|_t \lesssim  \|g_1\|_{p_2} \|g_2\|_{q_2},
\end{align}
%while the case described in Theorem \ref{thm_weak_inf_mod} is when $q_1 = q_2 = \infty$ and (\ref{B_easy}) can be rewritten as
%$$
%\|B\|_p \lesssim \|f_1\|_{p} \|f_2\|_{\infty},
%$$
%$$
%\|\tilde{B}\|_p \lesssim \|g_1\|_{p} \|g_2\|_{\infty}.
%$$
Therefore, 
\begin{equation*}
|\Lambda^{H,3}_{\text{flag}^{0} \otimes \text{flag}^{0}}(f_1,f_2,g_1,g_2,h, \chi_{E'})| \lesssim \|f_1\|_{p_1} \|f_2\|_{q_1}\|g_1\|_{p_2} \|g_2\|_{q_2} \| h \|_{L^s(\mathbb{R}^2)}.
\end{equation*}
\begin{remark} 
\begin{enumerate}
\noindent
\item
One notices that when 
\begin{align*}
& p_1 = p_2 = p \\
& q_1 = q_2 = \infty,
\end{align*}
Theorem \ref{thm_weak_inf_mod} in the Haar model is proved. Indeed Theorem \ref{thm_weak_inf_mod} is verified for generic functions in $L^p$ and $L^s$ spaces for $1< p < \infty,  1 < s < 2$.
\item
With the range of exponents described as (\ref{easy_case}), the above argument proves Theorem \ref{thm_weak_mod} for the operator $\Pi^{H,3}_{\text{flag}^0 \otimes \text{flag}^0}$ (\ref{Pi_nofixed_easy}) whereas for the exponents satisfying (\ref{hard_exponent}), we would need a more delicate and localized model operator $\Pi^{H}_{\text{flag}^0 \otimes \text{flag}^0}$ (\ref{Pi_larger_haar}) whose boundedness is proved in Section \ref{section_thm_haar}.
%We complete the proof of Theorem \ref{thm_weak} in Haar model. One 
\end{enumerate}
\end{remark}

\section{Generalization to Fourier Case} \label{section_fourier}
We will first highlight where we have used the assumption about the Haar model in the proof and then modify those partial arguments to prove the general case.

We have used the following implications specific to the Haar model.
\begin{enumerate}[H(I).]
\item
Let $\chi_{E'} := \chi_{E \setminus Enl(\Omega)}$. Then 
\begin{equation} \label{Haar_loc_bipara}
\langle \chi_{E'}, \phi^H_{I} \otimes \phi^H_J \rangle \neq 0
\iff
I \times J \cap Enl(\Omega)^c \neq \emptyset.
\end{equation}
As a result, what contributes to the multilinear forms in the Haar model are the dyadic rectangles $I \times J \in \mathcal{R}$ satisfying $I \times J \cap Enl(\Omega)^c \neq \emptyset$, which is a condition we heavily used in the proofs of the theorems in the Haar model.

\item
%\textbf{Localization of energy.}
For any dyadic intervals $K$ and $I$ with $|K| \geq |I|$, 
%For a non-lacunary family $(\phi^3_K)_K$,
$$
\langle \phi^{3,H}_K, \phi^H_I \rangle \neq 0
$$
if and only if
$$
K \supseteq I.
$$
Therefore, the non-degenerate case imposes the condition on the geometric containment we have employed for localizations of the operators $B^H_I$, $B^H_J$, $B^{\#_1,H}_I$ and $B^{\#_2,H}_J$ defined in (\ref{B_local_definition_haar}), (\ref{B_local_definition_haar_fix_scale}),  (\ref{B_0_local_haar_simplified}) and  (\ref{B_fixed_local_haar_simplified}).

\item
In the case $(\phi^{3,H}_K)_K$ is a family of Haar wavelets, the observation highlighted as (\ref{haar_biest_cond}) generates the biest trick (\ref{haar_biest}) which is essential in the energy estimates.
%\begin{remark}
%There is one exception that the perfect localization is not used in the discussion of the Haar model - energy estimates for $(\langle B_I, \vphi_I\rangle)_I$ with a lacunary family $(\psi^3_K)_K$. As emphasized in Section 5.5, one cannot expect to have both the perfect localization which refers to compact supports of spatial variables and the biest trick which relies on compact supports of frequency variables. In this case, (\ref{Haar_cond}) is imposed to localize $\text{energy}(\langle B_I, \vphi_I\rangle)_I$.
%\end{remark}
%For a lacunary family $(\psi^3_K)_K$

\end{enumerate}

%We will focus on how to tackle those two restrictions in the remaining of the section. One first notices that the main issue in the general case is that the bump functions are not compactly supported given that their Fourier supports are compact. As a consequence, Conditions (1) and (2) fail to hold.  
We will focus on how to generalize proofs of Theorem \ref{thm_weak_mod} for $\Pi^H_{\text{flag}^{\#_1} \otimes \text{flag}^{\#_2}} $ (\ref{Pi_fixed_haar}) and $\Pi^H_{\text{flag}^{0} \otimes \text{flag}^{0}}$ (\ref{Pi_larger_haar}) and discuss how to tackle restrictions listed as $H(I), H(II)$ and $H(III)$. The generalizations of arguments for other model operators and for Theorem \ref{thm_weak_inf_mod} follow from the same ideas.

\subsection{Generalized Proof of Theorem \ref{thm_weak_mod} for $\Pi_{\text{flag}^{\#_1} \otimes \text{flag}^{\#_2}} $} \label{section_fourier_fixed}
%The  argument for $\Pi_{\text{flag}^{\#_1} \otimes \text{flag}^{\#_2}} $ enclosed in Section 6 takes advantage of $H(I)$ and $H(II)$. 
\subsubsection{Localization and generalization of $H(I)$} 
The  argument for $\Pi_{\text{flag}^{\#_1} \otimes \text{flag}^{\#_2}} $ in Section \ref{section_thm_haar_fixed} takes advantages of the localization of spatial variables, as stated in $H(I)$. The following lemma allows one to decompose the original bump function into bump functions of compact supports so that a perfect localization in spatial variables can be achieved, which can be viewed as generalized $H(I)$ and whose proof is included in \cite{cptt_2} and Section 3 of \cite{cw}.
\begin{lemma} \label{decomp_compact}
%to add!!!!!!
Let $I \subseteq \mathbb{R}$ be an interval. Then any smooth bump function $\phi_I$ adapted to $I$ can be decomposed as
$$
\phi_I = \sum_{\tau \in \mathbb{N}} 2^{-100 \tau} \phi_I^{\tau}
$$
where for each $\tau \in \mathbb{N}$, $\phi_I^{\tau}$ is a smooth bump function adapted to $I$ and $\text{supp}(\phi_I^{\tau}) \subseteq 2^{\tau} I$.  If $\int \phi_I = 0$, then the functions $\phi_I^{\tau}$ can be chosen such that $\int \phi_I^{\tau} = 0$.
\end{lemma}
The multilinear form associated to $\Pi_{\text{flag}^{\#_1} \otimes \text{flag}^{\#_2}} $ in the general case can now be rewritten as 
\begin{align} \label{compact}
& \Lambda_{\text{flag}^{\#1} \otimes \text{flag}^{\#_2}}(f_1, f_2, g_1, g_2, h, \chi_{E'})  =  \displaystyle \sum_{\tau_1,\tau_2 \in \mathbb{N}}2^{-100(\tau_1+\tau_2)} \Lambda_{\text{flag}^{\#1} \otimes \text{flag}^{\#_2}}^{\tau_1, \tau_2} (f_1, f_2, g_1, g_2, h, \chi_{E'}),
%\sum_{I \times J \in \mathcal{R}} \frac{1}{|I|^{\frac{1}{2}} |J|^{\frac{1}{2}}} \langle B^{\#_1}_I(f_1,f_2),\phi_I^1 \rangle \langle \tilde{B}_J^{\#_2}(g_1,g_2), \phi_J^1 \rangle \langle h, \phi_{I}^2 \otimes \phi_{J}^2 \rangle \langle \chi_{E'},\phi_{I}^{3,\tau_1} \otimes \phi^{3, \tau_2}_{J} \rangle. \nonumber \\
\end{align} 
where for any fixed $\tau_1, \tau_2 \in \mathbb{N}$, 
{\fontsize{10}{10}\begin{align}
& \Lambda_{\text{flag}^{\#1} \otimes \text{flag}^{\#_2}}^{\tau_1, \tau_2} (f_1, f_2, g_1, g_2, h, \chi_{E'}) \nonumber\\
:= & \sum_{I \times J \in \mathcal{R}} \frac{1}{|I|^{\frac{1}{2}} |J|^{\frac{1}{2}}} \langle B^{\#_1}_I(f_1,f_2),\phi_I^1 \rangle \langle \tilde{B}^{\#_2}_J(g_1,g_2), \phi_J^1 \rangle \cdot \langle h, \phi_{I}^2 \otimes \phi_{J}^2 \rangle \langle \chi_{E'},\phi_{I}^{3,\tau_1} \otimes \phi^{3, \tau_2}_{J} \rangle.
\end{align}}
with $B^{\#_1}_I$ and $\tilde{B}^{\#_2}_J$ defined in (\ref{B_fixed_local_fourier_simple}) and (\ref{B_fixed_fourier}).
%One can now focus on $\Lambda_{\text{flag}^{\#1} \otimes \text{flag}^{\#_2}}^{\tau_1, \tau_2}$. 
To show that the multilinear form $\Lambda_{\text{flag}^{\#1} \otimes \text{flag}^{\#_2}} (f_1, f_2, g_1, g_2, h, \chi_{E'}) $ satisfies the same estimate as  \newline $\Lambda^H_{\text{flag}^{\#1} \otimes \text{flag}^{\#_2}}(f_1, f_2, g_1, g_2, h, \chi_{E'})$ (\ref{form_haar_larger_goal}), it suffices to prove that for any fixed $\tau_1,\tau_2 \in \mathbb{N}$,
\begin{equation} \label{linear_fix_fourier}
|\Lambda_{\text{flag}^{\#1} \otimes \text{flag}^{\#_2}}^{\tau_1, \tau_2}(f_1,f_2,g_1,g_2,h, \chi_{E'})| \lesssim (2^{\tau_1+ \tau_2})^{\Theta}|F_1|^{\frac{1}{p_1}}|F_2|^{\frac{1}{q_1}}|G_1|^{\frac{1}{p_2}}|G_2|^{\frac{1}{q_2}}\|h\|_{L^s(\mathbb{R}^2)}
\end{equation}
for some $0 < \Theta < 100$, thanks to the fast decay $2^{-100(\tau_1+\tau_2)}$ in the decomposition of the original multilinear form (\ref{compact}).

One first redefines the exceptional set with the replacement of $C_1$, $C_2$ and $C_3$ by $C_12^{10\tau_1}$, $C_2 2^{10\tau_2}$ and $C_3 2^{10\tau_1+10\tau_2}$ respectively. In particular, let 
\begin{align*}
& C_1^{\tau_1} := C_12^{10\tau_1}, \nonumber\\
& C_2^{\tau_2} := C_22^{10\tau_2}, \nonumber\\
& C_3^{\tau_1,\tau_2} := C_32^{10\tau_1+10\tau_2}. \nonumber
\end{align*}
Then define
\begin{align*}
\displaystyle \Omega_1^{\tau_1, \tau_2} := &\bigcup_{\mathfrak{n}_1 \in \mathbb{Z}}\{x: Mf_1(x) > C_1^{\tau_1} 2^{\mathfrak{n}_1}|F_1|\} \times \{y: Mg_1(y) > C_2^{\tau_2} 2^{-\mathfrak{n}_1}|G_1|\}\cup \nonumber \\
& \bigcup_{\mathfrak{n}_2 \in \mathbb{Z}}\{x:Mf_2(x) > C_1^{\tau_1} 2^{\mathfrak{n}_2}|F_2|\} \times \{y: Mg_2(y) > C_2^{\tau_2} 2^{-\mathfrak{n}_2}|G_2|\}\cup \nonumber \\
 &\bigcup_{\mathfrak{n}_3 \in \mathbb{Z}}\{x: Mf_1(x) > C_1^{\tau_1} 2^{\mathfrak{n}_3}|F_1|\} \times \{y: Mg_2(y) > C_2^{\tau_2} 2^{-\mathfrak{n}_3}|G_2|\}\cup \nonumber \\
& \bigcup_{\mathfrak{n}_4 \in \mathbb{Z}}\{x: Mf_2(x) > C_1^{\tau_1} 2^{\mathfrak{n}_4 }|F_2|\} \times \{y: Mg_1(y) > C_2^{\tau_2} 2^{-\mathfrak{n}_4 }|G_1|\},\nonumber \\
\Omega_2^{\tau_1,\tau_2} := & \{(x,y) \in \mathbb{R}^2: SSh(x,y) > C_3^{\tau_1, \tau_2} \|h\|_{L^s(\mathbb{R}^2)}\}. \nonumber \\
\end{align*}
%Let $\Omega_1^{\tau_1,\tau_2}$ and $ \Omega_2^{\tau_1,\tau_2}$ denote the variants corresponding to $\Omega^1$ and $\Omega^2$. 
One also defines 
\begin{align*}
& \Omega^{\tau_1,\tau_2} := \Omega_1^{\tau_1,\tau_2} \cup \Omega_2^{\tau_1,\tau_2}, \nonumber \\
& Enl(\Omega^{\tau_1,\tau_2}) := \{(x,y) \in \mathbb{R}^2: MM(\chi_{\Omega^{\tau_1,\tau_2}})(x,y)> \frac{1}{100} \}, \nonumber\\
& Enl_{\tau_1,\tau_2}(Enl(\Omega^{\tau_1,\tau_2})) := \{(x,y) \in \mathbb{R}^2: MM(\chi_{Enl(\Omega^{\tau_1,\tau_2})}(x,y)> \frac{1}{2^{2\tau_1+ 2\tau_2}} \},
\end{align*}
and finally
$$
\mathbf{\Omega}:= \bigcup_{\tau_1,\tau_2\in \mathbb{N}} Enl_{\tau_1,\tau_2}(Enl(\Omega^{\tau_1,\tau_2})) .
$$
\begin{remark}
It is not difficult to verify that $|\mathbf{\Omega}| \ll 1$ given that $C_1, C_2$ and $C_3$ are sufficiently large. One can then define $E' := E \setminus \mathbf{\Omega}$, where $|E'| \sim |E|$ as desired. For such $E'$, one has the following simple but crucial observation.
\end{remark}
\begin{obs} \label{start_point}
For any fixed $\tau_1, \tau_2 \in \mathbb{N}$ and any dyadic rectangle $I \times J$,
$$
\langle \chi_{E'},\phi_{I}^{3,\tau_1} \otimes \phi^{3, \tau_2}_{J}  \rangle \neq 0 
$$
implies that 
$$
I \times J \cap \left(Enl(\Omega^{\tau_1,\tau_2})\right)^c \neq \emptyset.
$$
%\iff 2^{\tau_1}I \times 2^{\tau_2}J \cap E'  \neq \emptyset
\end{obs}
\begin{proof}
We will prove the equivalent contrapositive statement. Suppose that $I \times J \cap \left(Enl(\Omega^{\tau_1,\tau_2})\right)^c = \emptyset$, or equivalently $I \times J \subseteq Enl(\Omega^{\tau_1,\tau_2})$, then 
$$|(2^{\tau_1}I \times 2^{\tau_2}J) \cap Enl(\Omega^{\tau_1,\tau_2})| > \frac{1}{2^{2\tau_1+2\tau_2}}|2^{\tau_1}I \times 2^{\tau_2}J|,
$$
which infers that
$$
(2^{\tau_1}I \times 2^{\tau_2}J) \subseteq Enl_{\tau_1,\tau_2}(Enl(\Omega^{\tau_1,\tau_2})) \subseteq \mathbf{\Omega}.
$$
Since $E' \cap \mathbf{\Omega} = \emptyset$, one can conclude that 
$$
\langle \chi_{E'},\phi_{I}^{3,\tau_1} \otimes \phi^{3, \tau_2}_{J}  \rangle = 0,
$$
which completes the proof of the observation.
\end{proof}

\begin{remark}\label{st_general}
Observation \ref{start_point} settles a starting point for the stopping-time decompositions with fixed parameters $\tau_1$ and $\tau_2$. More precisely, suppose that $\mathcal{R}$ is an arbitrary finite collection of dyadic rectangles.  Then with fixed $\tau_1, \tau_2 \in \mathbb{N}$, let $\displaystyle \mathcal{R} := \bigcup_{n_1, n_2 \in \mathbb{Z}}\mathcal{I}^{\tau_1}_{n_1} \times \mathcal{J}^{\tau_2}_{n_2}$ denote the \textit{tensor-type stopping-time decomposition I - level sets} introduced in Section \ref{section_thm_haar_fixed_tensor}. Now $\mathcal{I}^{\tau_1}_{n_1}$ and $\mathcal{J}^{\tau_2}_{n_2}$ are defined in the same way as $\mathcal{I}_{n_1}$ and $\mathcal{J}_{n_2}$ with $C_1$ and $C_2$ replaced by $C_1^{\tau_1}$ and $C_2^{\tau_2}$. By the argument for Observation \ref{obs_indice} in Section \ref{section_thm_haar_fixed_tensor_1d_level}, one can deduce the same conclusion that if for any $I \times J \in \mathcal{R}$, $I \times J \cap \mathbf{\Omega}^c \neq \emptyset$, then $n_1 + n_2 < 0$.
\end{remark}
Due to Remark \ref{st_general}, one can perform the stopping-time decompositions specified in Section \ref{section_thm_haar_fixed} with $C_1, C_2$ and $C_3$ replaced by $C_1^{\tau_1}$, $C_2^{\tau_2}$ and $C_3^{\tau_1,\tau_2}$ respectively and adopt the argument without issues. The difference that lies in the resulting estimate is the appearance of $O(2^{50\tau_1})$, $O(2^{50\tau_2})$ and $O(2^{50\tau_1+50\tau_2})$, which is not of concerns as illustrated in (\ref{linear_fix_fourier}). The only ``black-box'' used in Section \ref{section_thm_haar_fixed} is the local size estimates (Proposition \ref{size_cor}), which needs a more careful treatment and will be explored in the next subsection.
%One can see from the previous argument that $C_1, C_2$ and $C_3$ only appear in polynomial powers, which would translate to $O(2^{50\tau_1})$, $O(2^{50\tau_2})$ and $O(2^{50\tau_1+50\tau_2})$. 

%This would not be an issue because of the fast decay factor $2^{-100(\tau_1 + \tau_2)}$. As a consequence, the previous arguments based on (\ref{Haar_loc_bipara}) can be fixed with modifications. 

\subsubsection{Local size estimates and generalization of H(II)}
We will focus on the estimates of $\text{size}(\langle B^{\#_1}_I, \vphi_I\rangle)_I$, whose argument applies to $\text{size}(\langle \tilde{B}^{\#_2}_J, \vphi_J\rangle)_J$ as well. It suffices to prove Lemma \ref{B_size} in the Fourier case and the local size estimates described in Proposition \ref{size_cor} follow immediately. We will verify Lemma \ref{B_size} for the bilinear operator $B^{\#_1}_I = B^{\#_1}_{\mathcal{K},I} $ directly without going through a general formalization in terms of $\mathcal{Q}$, $P$ and $v_i$, $i = 1,2$ and then specifying $\mathcal{Q} := \mathcal{K}$, $P:= I$ and $v_i := f_i$ for $i= 1,2$. One first attempts to apply Lemma \ref{decomp_compact} to create a setting of compactly supported bump functions so that the same localization described in  Section \ref{section_size_energy} can be achieved. Suppose that for any $I \in \mathcal{I}'$, $I \cap S \neq \emptyset$. Then for some $I_0 \in \mathcal{I}'$ such that $I_0 \cap S \neq \emptyset$,
\begin{comment}
\begin{align} \label{loc_size}
\frac{|\langle B_{I}, \vphi^1_{I}\rangle|}{|I|^{\frac{1}{2}}} = & \frac{1}{|I|^{\frac{1}{2}}}\bigg|\sum_{\substack{K: |K| \sim 2^{\#_1}|I|}}\frac{1}{|K|^{\frac{1}{2}}} \langle f_1, \phi^1_K \rangle \langle f_2, \phi^2_K \rangle \langle \vphi^1_{I} ,\phi^3_K \rangle \bigg|  \nonumber \\
= & \frac{1}{|I|^{\frac{1}{2}}}\bigg|\sum_{\tau_3,\tau_4 \in \mathbb{N}}2^{-100\tau_3}2^{-100{\tau_4}}\sum_{K: |K| \sim 2^{\#_1} |I|} \frac{1}{|K|^{\frac{1}{2}}} \langle f_1, \phi^1_K \rangle \langle f_2, \phi^2_K \rangle  \langle \vphi^{1,\tau_3}_{I}, \phi^{3,\tau_4}_K\rangle\bigg|.
\end{align}
%By the disjointness of $(I)_{I \in \mathcal{I}_{-n-n_2,-m-m_2}}$,
%\begin{align*} 
%\mathcal{E}_{\alpha} \leq & 2^{-l-n_2}\Bigg(\sum_{\substack{T \in \mathbb{T}_{-l-n_2}\\I_T \in \mathcal{I}_{-n-n_2, m_1} }}|I_T|\Bigg)^{\alpha}\nonumber \\
%\end{align*}
One recalls the definition of 
\end{comment}
%$$
%\text{size}_{\mathcal{I}_{S}}((\langle B_I^{\#_1,H}, \vphi_I \rangle)_{I \in \mathcal{I}_{S}} = \frac{|\langle B^{\#_1}_{I_0}(f_1,f_2),\vphi_{I_0}^1 \rangle|}{|I_0|^{\frac{1}{2}}}
%$$
%for some $I_0 \in \mathcal{I}_{S}$ where $I \cap S \neq \emptyset$ for any $I \in \mathcal{I}'$, 
\begin{align*}
\text{size}_{\mathcal{I}'}((\langle B_I^{\#_1}(f_1,f_2), \vphi_I \rangle)_{I \in \mathcal{I}'} =&  \frac{|\langle B^{\#_1}_{I_0}(f_1,f_2),\vphi_{I_0}^1 \rangle|}{|I_0|^{\frac{1}{2}}} =\sum_{K: |K| \sim 2^{\#_1}|I_0|} \frac{1}{|K|^{\frac{1}{2}}} \langle f_1, \phi^1_K \rangle \langle f_2, \phi^2_K \rangle  \langle \vphi^{1}_{I_0}, \phi^{3}_K\rangle . \nonumber 
\end{align*}
One can now invoke Lemma \ref{decomp_compact} to decompose the bump functions $\vphi^1_{I_0}$  and $\phi^3_K$ to obtain
\begin{align} \label{form_f}
%& \frac{|\langle B^{\#_1}_{I_0}(f_1,f_2),\vphi_{I_0}^1 \rangle|}{|I_0|^{\frac{1}{2}}} 
%= 
 \frac{1}{|I_0|^{\frac{1}{2}}}\bigg|\sum_{\tau_3,\tau_4 \in \mathbb{N}}2^{-100 (\tau_3+\tau_4)}\sum_{K: |K| \sim 2^{\#_1}|I_0|} \frac{1}{|K|^{\frac{1}{2}}} \langle f_1, \phi^1_K \rangle \langle f_2, \phi^2_K \rangle  \langle \vphi^{1,\tau_3}_{I_0}, \phi^{3,\tau_4}_K\rangle\bigg|, 
\end{align}
where $\vphi^{1,\tau_3}_{I_0}$ is an $L^{2}$-normalized bump function adapted to $I_0$ with $\text{supp}(\vphi^{1,\tau_3}_{I_0}) \subseteq 2^{\tau_3}I_0$, and $\phi^{3,\tau_4}_K$ is an $L^2$-normalized bump function with $\text{supp}(\phi^{3,\tau_4}_K)\subseteq 2^{\tau_4}K$. With the property of being compactly supported, one has that if 
$$
\langle \vphi^{1,\tau_3}_{I_0}, \phi^{3, \tau_4}_K\rangle \neq 0,
$$
then 
$$2^{\tau_3} I_0 \cap 2^{\tau_4}K \neq \emptyset.$$
One also recalls that $I_0 \cap S \neq \emptyset$ and $|I_0| \leq |K|$, it follows that $\text{dist}(K,I)$\footnote{For any measurable sets $A,B \subseteq \mathbb{R}$, $\text{dist}(A,B) := \inf\{|a-b|: a \in A, b \in B \}$.} satisfies
\begin{equation} \label{geometry_fourier}
\frac{\text{dist}(K,S)}{|K|} \lesssim 2^{\tau_3 + \tau_4}.
\end{equation} 
Therefore, one can apply (\ref{geometry_fourier}) and rewrite (\ref{form_f}) as
%Since $|K| \sim 2^{\#_1}|I_0|$ implies that $|K| > |I_0|$, one can apply the geometric interpretation and obtain  
\begin{align} \label{size_B_f}
%& \frac{|\langle B^{\#_1}_{I_0}(f_1,f_2),\vphi_{I_0}^1 \rangle|}{|I_0|^{\frac{1}{2}}} \nonumber \\
%\leq & \sum_{\tau_3,\tau_4 \in \mathbb{N}} 2^{-100\tau_3}2^{-100\tau_4}\frac{1}{|I_0|} \sum_{K: |K| \sim 2^{\#_1}|I_0|} \frac{1}{|K|^{\frac{1}{2}}} |\langle f_1, \phi^1_K \rangle| |\langle f_2, \phi^2_K \rangle| | \langle |I_0|^{\frac{1}{2}}\vphi^{1,\tau_3}_{I_0}, \phi^{3,\tau_4}_K\rangle| \nonumber \\
%\leq 
&\sum_{\tau_3,\tau_4 \in \mathbb{N}} 2^{-100(\tau_3+\tau_4)} \frac{1}{|I_0|}  \sum_{K:|K|\sim 2^{\#_1}|I_0|}\frac{|\langle f_1, \phi_K^1 \rangle|}{|K|^{\frac{1}{2}}} \frac{|\langle f_2, \phi_K^2 \rangle|}{|K|^{\frac{1}{2}}} |\langle |I_0|^{\frac{1}{2}}\vphi^{1,\tau_3}_{I_0}, |K|^{\frac{1}{2}}\phi_K^3  \rangle| \nonumber \\
\leq & \sum_{\tau_3,\tau_4 \in \mathbb{N}} 2^{-100(\tau_3+\tau_4)}\frac{1}{|I_0|} \sup_{K:\frac{\text{dist}(K,S)}{|K|} \lesssim 2^{\tau_3 + \tau_4}}\frac{|\langle f_1, \phi_K^1 \rangle|}{|K|^{\frac{1}{2}}} \sup_{K:\frac{\text{dist}(K,S)}{|K|} \lesssim 2^{\tau_3 + \tau_4}}\frac{|\langle f_2, \phi_K^2 \rangle|}{|K|^{\frac{1}{2}}} \sum_{K:|K|\sim 2^{\#_1}|I_0|}|\langle |I_0|^{\frac{1}{2}}\vphi^{1,\tau_3}_{I_0}, |K|^{\frac{1}{2}}\phi_K^3  \rangle|.\nonumber \\
\end{align}
One notices that
\begin{align} \label{size_f_1}
& \sup_{K:\frac{\text{dist}(K,S)}{|K|} \lesssim 2^{\tau_3 + \tau_4}}\frac{|\langle f_1, \phi_K^1 \rangle|}{|K|^{\frac{1}{2}}} \lesssim 2^{\tau_3+\tau_4}\sup_{K:\frac{\text{dist}(K,S)}{|K|} \lesssim 2^{\tau_3 + \tau_4}}\frac{|\langle f_1, \phi_{2^{\tau_3+\tau_4}K}^1 \rangle|}{|2^{\tau_3+\tau_4}K|^{\frac{1}{2}}} \leq  2^{\tau_3+ \tau_4} \sup_{K' \cap  S \neq \emptyset} \frac{|\langle f_1, \phi_{K'}^1 \rangle|}{|K'|^{\frac{1}{2}}},
\end{align}
where $K' := 2^{\tau_3+\tau_4}K$ represents the interval with the same center as $K$ and the radius $2^{\tau_3 + \tau_4}|K|.$
Similarly,
\begin{equation}  \label{size_f_2}
 \sup_{K:\frac{\text{dist}(K,S)}{|K|} \lesssim 2^{\tau_3 + \tau_4}}\frac{|\langle f_2, \phi_K^2 \rangle|}{|K|^{\frac{1}{2}}} \lesssim 2^{\tau_3+ \tau_4} \sup_{K' \cap  S \neq \emptyset} \frac{|\langle f_2, \phi_{K'}^2 \rangle|}{|K'|^{\frac{1}{2}}}.
\end{equation}
Moreoever, 
\begin{align} \label{disjoint}
& \sum_{K:|K|\sim 2^{\#_1}|I_0|}|\langle |I_0|^{\frac{1}{2}}\vphi^{1,\tau_3}_{I_0}, |K|^{\frac{1}{2}}\phi_K^3  \rangle| \leq  \sum_{K:|K|\sim 2^{\#_1}|I_0|}\frac{1}{\left(1+\frac{\text{dist}(K,I_0)}{|K|}\right)^{100}} |I_0| \leq  |I_0| \sum_{k \in \mathbb{N}}k^{-100} \leq  |I_0|.
\end{align}
By combining (\ref{size_f_1}), (\ref{size_f_2}) and (\ref{disjoint}), one can estimate (\ref{size_B_f}) by
\begin{align*}
%& \frac{|\langle B^{\#_1}_{I_0}(f_1,f_2),\vphi_{I_0}^1 \rangle|}{|I_0|^{\frac{1}{2}}} \nonumber \\
%\lesssim 
& \frac{1}{|I_0|}\sum_{\tau_3,\tau_4 \in \mathbb{N}}2^{-100 \tau_3}2^{-100{\tau_4}} 2^{2(\tau_3+ \tau_4)}  \sup_{K' \cap  S \neq \emptyset} \frac{|\langle f_1, \phi_{K'}^1 \rangle|}{|K'|^{\frac{1}{2}}}  \sup_{K' \cap  S \neq \emptyset}\frac{|\langle f_2, \phi_{K'}^2 \rangle|}{|K'|^{\frac{1}{2}}} |I_0|\nonumber \\
\lesssim & \sup_{K' \cap  S \neq \emptyset} \frac{|\langle f_1, \phi_{K'}^1 \rangle|}{|K'|^{\frac{1}{2}}}  \sup_{K' \cap  S \neq \emptyset}\frac{|\langle f_2, \phi_{K'}^2 \rangle|}{|K'|^{\frac{1}{2}}},
\end{align*}
which is exactly the same estimate for the corresponding term in Lemma \ref{B_size}. %It is not difficult to check that the above estimate is sufficient to deduce the same conclusion for $\Pi_{\text{flag}^{\#_1} \otimes \text{flag}^{\#_2}}$ as in the Haar model. 
This completes the proof of Theorem \ref{thm_weak_mod} 
%and \ref{thm_weak_inf_mod} 
for $\Pi_{\text{flag}^{\#_1} \otimes \text{flag}^{\#_2}}$.
%The only nontrivial problem in this general case is that one can no longer easily localize $B$ as before. In particular, one would like to obtain same estimates for the size and energy terms in Lemma \ref{B_size} and Proposition \ref{B_en}.
 %$$\mathcal{E}_{\alpha} := 2^{-l-n_2}\bigg(\sum_{T \in \mathbb{T}_{-l-n_2}}\bigg|
%\bigcup_{\substack{I\in T \\ I \in \mathcal{I}_{-n-n_2,m_1}}}I\bigg|\bigg)^{\alpha}$$ for $0 < \alpha \leq 1$.
%Although the localization cannot be implemented as before, there are still some information one can extract if $I$ intersects some set nontrivially and if $|K| > |I|$. In particular, one considers
\subsection{Generalized Proof of Theorem \ref{thm_weak_mod} for $\Pi_{\text{flag}^0 \otimes \text{flag}^0}$} \label{section_fourier_0}
\subsubsection{Local energy estimates and generalization of H(III)}
%\subsection{Energy}
%\subsubsection{Generalized localization}
The delicacy of the argument for $\Pi_{\text{flag}^0 \otimes \text{flag}^0}$ with the lacunary family $(\phi_K^3)_K$ lies in the localization and the application of the biest trick for the energy estimates. It is worthy to note that Lemma \ref{decomp_compact} fails to generate the local energy estimates. In particular, one can decompose 
%It is not difficult to see that with the localization specified in (\ref{loc_size}), the biest trick employed for the energy estimates in the Haar setting fails for the same reason elaborated in Section 5.5. In particular,  one is not able to equate the terms
$$
\langle B_I(f_1,f_2), \vphi_I^1\rangle =
\sum_{\tau_3,\tau_4 \in \mathbb{N}}2^{-100 \tau_3}2^{-100{\tau_4}}\sum_{K: |K| \geq |I|} \frac{1}{|K|^{\frac{1}{2}}} \langle f_1, \phi^1_K \rangle \langle f_2, \phi^2_K \rangle  \langle \vphi^{1,\tau_3}_{I}, \psi^{3,\tau_4}_K\rangle,
$$
where $B_I$ is a bilinear operator defined in (\ref{B_local_fourier_simple}) and (\ref{B_local0_haar}). Without loss of generality, we further assume that $(\phi^1_K)_{K \in \mathcal{K}}$ is non-lacunary and $(\phi^2_K)_{K \in \mathcal{K}}$ is lacunary.
Then by the geometric observation (\ref{geometry_fourier}) implied by the non-degenerate condition $ \langle \vphi^{1,\tau_3}_{I}, \psi^{3,\tau_4}_K\rangle \neq 0$,
\begin{equation} \label{loc_attempt_fourier}
\langle B_I(f_1,f_2), \vphi_I^1\rangle =
\sum_{\tau_3,\tau_4 \in \mathbb{N}}2^{-100 \tau_3}2^{-100{\tau_4}}\sum_{\substack{K:|K| \geq |I| \\ K: \frac{\text{dist}(K,I)}{|K|} \lesssim 2^{\tau_3 + \tau_4}}} \frac{1}{|K|^{\frac{1}{2}}} \langle f_1, \vphi^1_K \rangle \langle f_2, \psi^2_K \rangle  \langle \vphi^{1,\tau_3}_{I}, \psi^{3,\tau_4}_K\rangle.
\end{equation}
The localization has been obtained in (\ref{loc_attempt_fourier}). Nonetheless, for each fixed $\tau_3$ and $\tau_4$, one cannot equate the terms
\begin{equation*}
\sum_{\substack{K:|K| \geq |I| \\ K: \frac{\text{dist}(K,I)}{|K|} \lesssim 2^{\tau_3 + \tau_4}}} \frac{1}{|K|^{\frac{1}{2}}} \langle f_1, \vphi^1_K \rangle \langle f_2, \psi^2_K \rangle  \langle \vphi^{1,\tau_3}_{I}, \psi^{3,\tau_4}_K\rangle \neq \sum_{\substack{ \\ K: \frac{\text{dist}(K,I)}{|K|} \lesssim 2^{\tau_3 + \tau_4}}} \frac{1}{|K|^{\frac{1}{2}}} \langle f_1, \vphi^1_K \rangle \langle f_2, \psi^2_K \rangle  \langle \vphi^{1,\tau_3}_{I}, \psi^{3,\tau_4}_K\rangle.
\end{equation*}
The reason is that $\vphi_I^{1,\tau_3}$ and $\psi_K^{3,\tau_4}$ are general $L^2$-normalized bump functions instead of Haar wavelets and $L^2$-normalized indicator functions. The ``if and only if'' condition 
\begin{equation} \label{haar_biest_tri}
\langle \vphi^{1,\tau_3}_{I}, \psi^{3,\tau_4}_K\rangle \neq 0 \iff |K| \geq |I|
\end{equation}
is no longer valid and is insufficient to derive the biest trick. As a consequence, one cannot simply apply Lemma \ref{decomp_compact} to localize the energy term and compare
$$
\text{energy}((\langle B^{\tau_4}_I(f_1,f_2), \vphi_I^{1,\tau_3} \rangle)_{I \cap S \neq \emptyset})
$$
with
$$
\text{energy}((\langle B^{S,\tau_3,\tau_4}_{\mathcal{K},\text{lac}}(f_2,f_2), \vphi_{I,\tau_3}^1 \rangle)_{I \cap S \neq \emptyset}),
$$
where
$$
B^{\tau_4}_I(f_1,f_2):= \sum_{\substack{K: K \in \mathcal{K} \\ |K| \geq |I| \\ }} \frac{1}{|K|^{\frac{1}{2}}} \langle f_1, \vphi^1_K \rangle \langle f_2, \psi^2_K \rangle \psi^{3,\tau_4}_K
$$
and 
$$
B^{S,\tau_3,\tau_4}_{\mathcal{K},\text{lac}}(f_2,f_2):= \sum_{\substack{K \in \mathcal{K} \\ K: \frac{\text{dist}(K,S)}{|K|} \lesssim 2^{\tau_3 + \tau_4}}} \frac{1}{|K|^{\frac{1}{2}}} \langle f_1, \vphi^1_K \rangle \langle f_2, \psi^2_K \rangle  \psi^{3,\tau_4}_K.
$$
The biest trick is crucial in the local energy estimates which shall be evident from the previous analysis. In order to use the biest trick in the Fourier case, one needs to exploit the compact Fourier supports instead of the compact supports for spatial variables in the Haar model. 
%As a consequence, one cannot simply apply Lemma \ref{decomp_compact} to localize the energy term involving $B$ as (\ref{loc_attempt_fourier}) since the bump functions $ \vphi^{1,\tau_3}_{I}, \psi^{3,\tau_4}_K$ are compactly supported in space and cannot be compactly supported in frequency due to the uncertainty principle.
%To achieve the biest trick, one needs to apply a generalized localization. 
One first recalls that the Littlewood-Paley decomposition imposes that $\text{supp}(\f{\vphi}_I^1) \subseteq [-\frac{1}{4}|I|^{-1}, \frac{1}{4}|I|^{-1}]$
%2^{-1}\omega_I$ 
and $\text{supp}(\f{\psi}_K^3) \subseteq [-4|K|^{-1}, -\frac{1}{4}|K|^{-1}] \cup [\frac{1}{4}|K|^{-1},4|K|^{-1}]$. The interaction between the two Fourier supports is shown in the following figure:  
%where $\omega_I$ and $\omega_K$ behave as follows in the frequency space:
\vskip.25 in
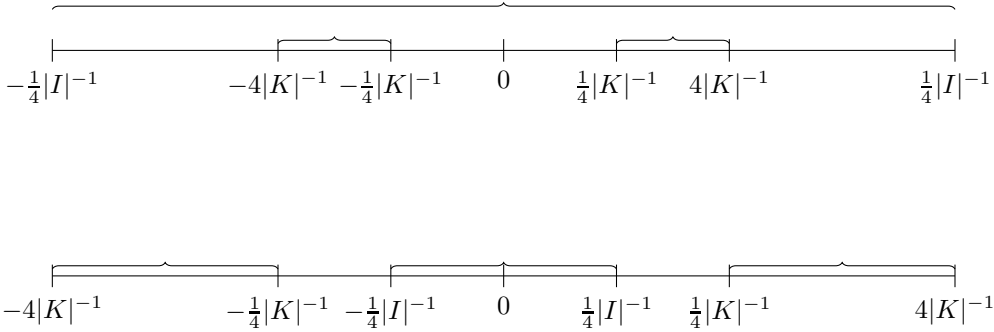
 \begin{tikzpicture}[scale=0.3, decoration=brace]
    \draw (-20,0)-- (20,0); %Axis
    \foreach \x/\xtext in {-20/$-\frac{1}{4}|I|^{-1}$, -10/$-4|K|^{-1}$, -5/$-\frac{1}{4}|K|^{-1}$, 0/0,5/$\frac{1}{4}|K|^{-1}$,10/$4|K|^{-1}$,20/$\frac{1}{4}|I|^{-1}$} {
        \draw (\x,0.5) -- (\x,-0.5) node[below] {\xtext};
    }
     \draw[decorate, yshift=2ex]  (-10,0) -- node[above=0.4ex] {}  (-5,0);
    \draw[decorate, yshift=2ex]  (5,0) -- node[above=0.4ex] {}  (10,0);
    \draw[decorate, yshift=12ex]  (-20,0) -- node[above=0.8ex] {}  (20,0);
    %\draw (-10,1) -- (5,1);
    %\draw[fill=white] (-10,1) circle (0.25);
    %\fill (5,1) circle (0.25);
   % \draw (1,2) -- (7,2);
   % \fill (1,2) circle (0.25);
    %\draw[fill=white] (7,2) circle (0.25);
     \draw (-20,-10)-- (20,-10); %Axis
    \foreach \x/\xtext in {-5/$-\frac{1}{4}|I|^{-1}$, -20/$-4|K|^{-1}$, -10/$-\frac{1}{4}|K|^{-1}$, 0/0,10/$\frac{1}{4}|K|^{-1}$,20/$4|K|^{-1}$,5/$\frac{1}{4}|I|^{-1}$} {
        \draw (\x,-9.5) -- (\x,-10.5) node[below] {\xtext};
    }
     \draw[decorate, yshift=2ex]  (-20,-10) -- node[above=0.4ex] {}  (-10,-10);
    \draw[decorate, yshift=2ex]  (10,-10) -- node[above=0.4ex] {}  (20,-10);
    \draw[decorate, yshift=2ex]  (-5,-10) -- node[above=0.4ex] {}  (5,-10);
 \end{tikzpicture}
 \newline
 \noindent
As one may notice, 
\begin{equation} \label{biest_fourier}
\langle \vphi^{1}_{I}, \psi^{3}_K \rangle \neq 0 \iff  |K| \geq |I|,
\end{equation}
which yields the biest trick as desired. Meanwhile, we would like to attain some generalized localization for the energy. In particular, fix any $n_1, m_1$, define the set
\begin{equation} \label{level_set_fourier}
\mathcal{U}_{n_1,m_1} := \{x: Mf_1(x) \lesssim C_12^{n_1}|F_1|\} \cap  \{x: Mf_2(x) \lesssim C_12^{m_1}|F_2|\}.
\end{equation}
Suppose that $\mathcal{I}'$ is a finite collection of dyadic intervals satisfying 
\begin{equation} \label{condition_collection}
I \cap \mathcal{U}_{n_1,m_1} \neq \emptyset, \ \ \text{for any} \ \ I \in \mathcal{I}'. 
\end{equation}
Then one would like to reduce 
$$\text{energy}_{\mathcal{I}'}(\langle B_I(f_1,f_2), \vphi_I\rangle)_{I \in \mathcal{I}'}$$ 
to 
$$\text{energy}_{\mathcal{I}'}(\langle B^{n_1,m_1,0}_{\mathcal{K}, \text{generic lac}}(f_1,f_2), \vphi_I\rangle)_{I \in \mathcal{I}'}$$ (or $\text{energy}(\langle B^{n_1,m_1,0}_{\mathcal{K}, \text{generic nonlac}}(f_1,f_2), \vphi_I\rangle)_{I}$ depending on whether the third family of bump functions defining $B_I$ is lacunary or not), where
\begin{align}
B^{n_1,m_1,0}_{\mathcal{K},\text{generic lac}}(f_1,f_2) := &  \sum_{\substack{K \in \mathcal{K}\\ K \cap \mathcal{U}_{n_1,m_1} \neq \emptyset}} \frac{1}{|K|^{\frac{1}{2}}} \langle f_1, \vphi_K^1\rangle \langle f_2, \psi_K^2 \rangle \psi_K^3, \\
B^{n_1,m_1,0}_{\mathcal{K},\text{generic nonlac}}(f_1,f_2) := &  \sum_{\substack{K \in \mathcal{K}\\ K \cap \mathcal{U}_{n_1,m_1} \neq \emptyset}} \frac{1}{|K|^{\frac{1}{2}}} \langle f_1, \psi_K^1\rangle \langle f_2, \psi_K^2 \rangle \vphi_K^3.
\end{align}
One observes that since $\psi_K^3$ and $\vphi_I^1$ are not compactly supported in $K$ and $I$ respectively, one cannot deduce that $K \cap \mathcal{U}_{n_1,m_1}  \neq \emptyset$ given that $I \cap \mathcal{U}_{n_1,m_1} \neq \emptyset$ and $|K| \geq |I|$. The localization in the Fourier case is attained in a more analytic fashion. One decomposes the sum 
\begin{align} \label{d_0}
&\frac{|\langle B_I(f_1,f_2), \vphi^1_I \rangle|}{|I|^{\frac{1}{2}}} \nonumber\\
= & \frac{1}{|I|^{\frac{1}{2}}}\bigg|\sum_{K: |K| \geq |I|} \frac{1}{|K|^{\frac{1}{2}}}\langle f_1, \vphi_K^1 \rangle \langle f_2, \psi_K^2 \rangle \langle \vphi^1_I, \psi_K^3 \rangle\bigg| \nonumber \\
= & \frac{1}{|I|^{\frac{1}{2}}}\bigg|\sum_{d >0} \sum_{\substack{|K| \geq |I| \\ K \in \mathcal{K}^{n_1,m_1}_{d}}} \frac{1}{|K|^{\frac{1}{2}}}\langle f_1, \vphi_K^1 \rangle \langle f_2, \psi_K^2 \rangle \langle \vphi^1_I, \psi_K^3 \rangle + \sum_{\substack{|K| \geq |I| \\ K \in \mathcal{K}^{n_1,m_1}_0}}\frac{1}{|K|^{\frac{1}{2}}}\langle f_1, \vphi_K^1 \rangle \langle f_2, \psi_K^2 \rangle \langle \vphi_I^1, \psi_K^3 \rangle \bigg|,
\end{align}
where for $d > 0$,
$$
\mathcal{K}_d^{n_1,m_1} := \{K: 1 + \frac{\text{dist}(K,\mathcal{U}_{n_1,m_1})}{|K|} \sim 2^{d} \},
$$
and 
$$
\mathcal{K}_0^{n_1,m_1} := \{K : K \cap \mathcal{U}_{n_1,m_1} \neq \emptyset\}.
$$
\noindent
Ideally, one would like to ''omit'' the former term, which is reasonable once 
\begin{equation} \label{energy_needed}
\bigg|\sum_{d >0} \sum_{\substack{|K| \geq |I| \\ K \in \mathcal{K}^{n_1,m_1}_{d}}} \frac{1}{|K|^{\frac{1}{2}}}\langle f_1, \vphi_K^1 \rangle \langle f_2, \psi_K^2 \rangle \langle \vphi^1_I, \psi_K^3 \rangle\bigg| \ll \bigg|\sum_{\substack{|K| \geq |I| \\ K \in \mathcal{K}^{n_1,m_1}_0}}\frac{1}{|K|^{\frac{1}{2}}}\langle f_1, \vphi_K^1 \rangle \langle f_2, \psi_K^2 \rangle \langle \vphi^1_I, \psi_K^3 \rangle \bigg|
\end{equation}
so that one can apply the previous argument discussed in Section \ref{section_thm_haar_fixed}. In the other case when
\begin{equation*}
\bigg|\sum_{d >0} \sum_{\substack{|K| \geq |I| \\ K \in \mathcal{K}^{n_1,m_1}_{d}}} \frac{1}{|K|^{\frac{1}{2}}}\langle f_1, \vphi_K^1 \rangle \langle f_2, \psi_K^2 \rangle \langle \vphi^1_I, \psi_K^3 \rangle \bigg| \gtrsim \bigg|\sum_{\substack{|K| \geq |I| \\ K \in \mathcal{K}^{n_1,m_1}_0}}\frac{1}{|K|^{\frac{1}{2}}}\langle f_1, \vphi_K^1 \rangle \langle f_2, \psi_K^2 \rangle \langle \vphi^1_I, \psi_K^3 \rangle \bigg|,
\end{equation*}
local energy estimates are not necessary to achieve the result. The following lemma generates estimates for the former term and provides a guideline about the separation of cases. The notation in the lemma is consistent with the previous discussion. 
\begin{lemma} \label{en_loc}
Suppose that $d >0$ and $I$ is a fixed dyadic interval such that $I \cap \mathcal{U}_{n_1,m_1} \neq \emptyset$ for $\mathcal{U}_{n_1,m_1}$ defined in (\ref{level_set_fourier}). Then
$$
\frac{1}{|I|^{\frac{1}{2}}}\bigg| \sum_{\substack{|K| \geq |I| \\ K \in \mathcal{K}_{d}^{n_1,m_1}}} \frac{1}{|K|^{\frac{1}{2}}}\langle f_1, \vphi_K^1 \rangle \langle f_2, \psi_K^2 \rangle \langle \vphi^1_I, \psi_K^3 \rangle \bigg| \lesssim 2^{-Nd}(C_12^{n_1}|F_1|)^{\alpha_1}  (C_12^{m_1}|F_2|)^{\beta_1}, 
$$
for any $0 \leq \alpha_1,\beta_1 \leq 1$ and some $N \gg 1$.
\end{lemma}
\begin{remark}
\begin{enumerate}
\noindent
\item
One simple but important fact is that for any fixed $d>0$, $n_1$ and $m_1$, $\mathcal{K}_d^{n_1,m_1}$ is a disjoint collection of dyadic intervals.
\item
Aware of the first comment, one can apply the exactly same argument in Section \ref{section_fourier_fixed} to prove the lemma.
\end{enumerate}
 \end{remark}
 Based on the estimates described in Lemma \ref{en_loc}, one has that 
\begin{align} \label{threshold}
\frac{1}{|I|^{\frac{1}{2}}}\bigg|\sum_{d >0} \sum_{\substack{|K| \geq |I| \\ K \in \mathcal{K}_{d}^{n_1,m_1}}} \frac{1}{|K|^{\frac{1}{2}}}\langle f_1, \vphi_K^1 \rangle \langle f_2, \psi_K^2 \rangle \langle \vphi^1_I, \psi_K^3 \rangle \bigg| & \lesssim \sum_{d>0} 2^{-Nd} (C_12^{n_1}|F_1|)^{\alpha_1}  (C_12^{m_1}|F_2|)^{\beta_1}  \nonumber \\
& \lesssim (C_12^{n_1}|F_1|)^{\alpha_1}  (C_12^{m_1}|F_2|)^{\beta_1}, 
\end{align}
for any $0 \leq \alpha_1, \beta_1 \leq 1$. One can then use the upper bound in (\ref{threshold}) to proceed the discussion case by case.
\vskip .15in
\noindent
\textbf{Case I: There exists $0 \leq \alpha_1, \beta_1 \leq 1$ such that $\frac{|\langle B_I, \vphi^1_I \rangle|}{|I|^{\frac{1}{2}}} \gg  (C_12^{n_1}|F_1|)^{\alpha_1}  (C_12^{m_1}|F_2|)^{\beta_1}.$}
\vskip .1in
%Therefore if 
%\begin{equation} \label{dom}
 %\frac{|\langle B_I, \vphi^1_I \rangle|}{|I|^{\frac{1}{2}}} \gg  (C_12^{n_1}|F_1|)^{\alpha_1}  (C_12^{m_1}|F_2|)^{\beta_1},
%\end{equation}
In Case I, (\ref{energy_needed}) holds and the dominant term in expression (\ref{d_0}) has to be
$$
\frac{1}{|I|^{\frac{1}{2}}}\sum_{\substack{|K| \geq |I| \\ K \in \mathcal{K}_0^{n_1,m_1}}}\frac{1}{|K|^{\frac{1}{2}}}\langle f_1, \vphi_K^1 \rangle \langle f_2, \psi_K^2 \rangle \langle \vphi^1_I, \psi_K^3 \rangle,
$$
which provides a localization for energy estimates. More precisely, in the current case,
\begin{align*}
 \text{energy}_{\mathcal{I}'}^{1,\infty}((\langle B_I(f_1,f_2), \vphi_I^1\rangle)_{I \in \mathcal{I}'} \lesssim &\text{energy}_{\mathcal{I}'}^{1,\infty}(\langle B^{n_1,m_1,0}_{\mathcal{K}, \text{generic lac}}(f_1,f_2), \vphi_I\rangle)_{I \in \mathcal{I}'}, \nonumber \\
 \text{energy}_{\mathcal{I}'}^t((\langle B_I(f_1,f_2), \vphi_I^1\rangle)_{I \in \mathcal{I}'}) \lesssim & \text{energy}_{\mathcal{I}'}^t(\langle B^{n_1,m_1,0}_{\mathcal{K},\text{generic lac}}(f_1,f_2), \vphi_I\rangle)_{I \in \mathcal{I}'}, \nonumber
\end{align*}
for any $t >1$, where one recalls that $\mathcal{I}'$ is a finite collection of dyadic intervals satisfying (\ref{condition_collection}). Furthermore, 
\begin{align*}
 \text{energy}_{\mathcal{I}'}^{1,\infty}(\langle B^{n_1,m_1,0}_{\mathcal{K},\text{generic lac}}(f_1,f_2), \vphi_I\rangle)_{I \in \mathcal{I}'} \lesssim& \|B^{n_1,m_1,0}_{\mathcal{K}, \text{generic lac}}(f_1,f_2)\|_1, \nonumber \\
 \text{energy}_{\mathcal{I}'}^t(\langle B^{n_1,m_1,0}_{\mathcal{K},\text{generic lac}}(f_1,f_2), \vphi_I\rangle)_{I \in \mathcal{I}'} \lesssim & \|B^{n_1,m_1, 0}_{\mathcal{K}, \text{generic lac}}(f_1,f_2)\|_t,
\end{align*}
for any $t > 1$. It is not difficult to verify that $\|B^{n_1,m_1,0}_{\mathcal{K}, \text{generic lac}}(f_1,f_2)\|_t$ for $t \geq 1$ follows from the same estimates for its Haar variant described in Lemma \ref{B_loc_norm}. We will explicitly state the local energy estimates in this case.

\begin{proposition}[Local Energy Estimates in Fourier Case in $x$-Direction] \label{localized_energy_fourier_x}
Suppose that $n_1, m_1 \in \mathbb{Z}$ are fixed and suppose that $\mathcal{I}'$ is a finite collection of dyadic intervals such that for any $I \in \mathcal{I}'$, $I $ satisfies  
\begin{enumerate}
\item
$I \cap \mathcal{U}_{n_1,m_1} \neq \emptyset$; 
\item
$I \in T $ with $T \in \mathbb{T}_{l_1}$ for some $l_1$ satisfying the condition that there exists some $ 0 \leq \alpha_1, \beta_1 \leq 1$ such that
\begin{equation} \label{loc_condition_x}
2^{l_1}\|B\|_1 \gg (C_12^{n_1}|F_1|)^{\alpha_1}  (C_12^{m_1}|F_2|)^{\beta_1}.
\end{equation}
\end{enumerate}
%Then
%\begin{align}
%\text{energy}_{I } ((\langle B_I, \vphi_I^1 \rangle)_{I \in \mathcal{I}'}) \lesssim  
%\end{align}

\begin{enumerate}[(i)]
\item
Then for any $0 < \theta_1,\theta_2 <1$ with $\theta_1 + \theta_2 = 1$, one has
\begin{align*}
&\text{energy}^{1,\infty}_{\mathcal{I}'}((\langle B_I, \vphi_I\rangle)_{I \in \mathcal{I}'}) \lesssim  C_1^{\frac{1}{p_1}+ \frac{1}{q_1} - \theta_1 - \theta_2} 2^{n_1(\frac{1}{p_1} - \theta_1)} 2^{m_1(\frac{1}{q_1} - \theta_2)} |F_1|^{\frac{1}{p_1}} |F_2|^{\frac{1}{q_1}}.\nonumber 
%& \text{energy}_{\mathcal{J}'}((\langle \tilde{B}_J, \vphi_J \rangle)_{J \in \mathcal{J}'}) \lesssim C_2^{\frac{1}{p_2}+ \frac{1}{q_2} - \theta_1' - \theta_2'} 2^{n_2(\frac{1}{p_2} - \theta_1')} 2^{m_2(\frac{1}{q_2} - \theta_2')} |G_1|^{\frac{1}{p_2}} |G_2|^{\frac{1}{q_2}}
%\text{size}_{K \in \mathcal{K}: K \cap S \neq \emptyset}((\langle f_1, \vphi_K \rangle)_{K})^{1-\theta_1}\text{size}_{K \in \mathcal{K}:K \cap S \neq \emptyset}((\langle f_2, \psi_K \rangle)_{K})^{1-\theta_2} \nonumber \\
%& \cdot |F_1|^{\theta_1}|F_2|^{\theta_2}
%\text{energy}_{\mathcal{I}'}((\langle B_I, \vphi_I \rangle)_{I \in \mathcal{I}'}) \lesssim & \text{size}_{K \in \mathcal{K}: K \cap S \neq \emptyset}((\langle f_1, \vphi_K \rangle)_{K})^{1-\theta_1}\text{size}_{K \in \mathcal{K}:K \cap S \neq \emptyset}((\langle f_2, \psi_K \rangle)_{K})^{1-\theta_2} \nonumber \\
%& \cdot |F_1|^{\theta_1}|F_2|^{\theta_2}
\end{align*}

%\item
%\begin{align*}
%\text{energy}^{p,\infty} _{\mathcal{I}'}((\langle B_I, \vphi_I \rangle)_{I \in \mathcal{I}'}) \lesssim & \text{size}_{K \in \mathcal{K}: K \cap S \neq \emptyset}((\langle f_1, \vphi_K \rangle)_{K})^{1-\theta_1}\text{size}_{K \in \mathcal{K}:K \cap S \neq \emptyset}((\langle f_2, \psi_K \rangle)_{K})^{1-\theta_2} \nonumber \\
%& \cdot |F_1|^{\theta_1}|F_2|^{\theta_2}|S_2|^{\theta_3-\frac{1}{p'}}
%\end{align*}

%\begin{align*}
%\text{energy}^{p,\infty} _{\mathcal{J}'}((\langle \tilde{B_J}, \vphi_J \rangle)_{J \in \mathcal{J}'}) \lesssim & \text{size}_{L \in \mathcal{L}: L \cap S' \neq \emptyset}((\langle g_1, \vphi_K \rangle)_{L})^{1-\theta_1}\text{size}_{L \in \mathcal{L}:L \cap S' \neq \emptyset}((\langle g_2, \psi_L \rangle)_{L})^{1-\theta_2} \nonumber \\
%& \cdot |G_1|^{\theta_1}|G_2|^{\theta_2}|S'_2|^{\theta_3-\frac{1}{p'}}
%\end{align*}
\item
Suppose that $t >1$. Then for any $0 \leq \theta_1, \theta_2 <1$ with $\theta_1 + \theta_2 = \frac{1}{t}$, one has
%for any $1 \leq t < \infty$, $1< p' \leq \infty $ and $\frac{1}{p}+ \frac{1}{p'}$ = 1,
\begin{align*}
%\sum_{l}2^{lp}\sup_{\mathbb{D}}\sum_{\substack{I \in \mathbb{D} \\ \frac{|\langle B_I, \vphi_I \rangle|}{|I|^{\frac{1}{2}}} > 2^{l} \\ I \in \mathcal{I}'}}|I|
& \text{energy}^{t} _{\mathcal{I}'}((\langle B_I, \vphi_I\rangle)_{I \in \mathcal{I}'}) \lesssim  C_1^{\frac{1}{p_1}+ \frac{1}{q_1} - \theta_1 - \theta_2}2^{n_1(\frac{1}{p_1} - \theta_1)}2^{m_1(\frac{1}{q_1} - \theta_2)}|F_1|^{\frac{1}{p_1}}|F_2|^{\frac{1}{q_1}}. \nonumber 
%& \text{energy}^{t} _{\mathcal{J}'}((\langle \tilde{B}_J, \vphi_J \rangle)_{J \in \mathcal{J}'}) \lesssim C_2^{\frac{1}{p_2}+ \frac{1}{q_2} - \theta_1' - \theta_2'}2^{n_2(\frac{1}{p_2} - \theta_1')}2^{m_2(\frac{1}{q_2} - \theta_2')}|G_1|^{\frac{1}{p_2}}|G_2|^{\frac{1}{q_2}}
%\text{size}_{K \in \mathcal{K}: K \cap S \neq \emptyset}((\langle f_1, \vphi_K \rangle)_{K})^{1-\theta_1}\text{size}_{K \in \mathcal{K}:K \cap S \neq \emptyset}((\langle f_2, \psi_K \rangle)_{K})^{1-\theta_2} \nonumber \\
%& \cdot |F_1|^{\theta_1}|F_2|^{\theta_2}|S_2|^{\theta_3-\frac{1}{p'}}
%\text{energy}^{2} _{\mathcal{J}'}((\langle \tilde{B_J}, \vphi_J \rangle)_{J \in \mathcal{J}'}) \lesssim & \text{size}_{L \in \mathcal{L}: L \cap S' \neq \emptyset}((\langle g_1, \vphi_K \rangle)_{L})^{1-\theta_1}\text{size}_{L \in \mathcal{L}:L \cap S' \neq \emptyset}((\langle g_2, \psi_L \rangle)_{L})^{1-\theta_2} \nonumber \\
%& \cdot |G_1|^{\theta_1}|G_2|^{\theta_2}|S'_2|^{\theta_3-\frac{1}{p'}}
\end{align*}
\end{enumerate}
\end{proposition}
A parallel statement holds for dyadic intervals in the $y$-direction, which will be stated for the convenience of reference later on.

\begin{proposition}[Local Energy Estimates in Fourier Case in $y$-Direction] \label{localized_energy_y}
Suppose that $ n_2, m_2 \in \mathbb{Z}$ are fixed and suppose that $\mathcal{J}'$ is a finite collection of dyadic intervals such that for any $J \in \mathcal{J} '$, $J $ satisfies  
\begin{enumerate}
\item
$J \cap \tilde{\mathcal{U}}_{n_2,m_2}$ where $\tilde{\mathcal{U}}_{n_2,m_2}:= \{ y: Mg_1(y) \lesssim C_2 2^{n_2}|G_1| \} \cap \{ y: Mg_2(y) \lesssim C_2 2^{m_2}|G_2|\} $.
\item
$J \in S $ with $S \in \mathbb{S}_{l_2}$ for some $l_2$ satisfying the condition that there exists some $0 \leq \alpha_2, \beta_2 \leq 1$ such that
\begin{equation} \label{loc_condition_y}
2^{l_2}\|\tilde{B}\|_1 \gg (C_22^{n_2}|G_1|)^{\alpha_2}  (C_22^{m_2}|G_2|)^{\beta_2}.
\end{equation}
\end{enumerate}
\begin{enumerate}[(i)]
\item
Then for any $0 < \zeta_1,\zeta_2 <1$ with $\zeta_1 + \zeta_2= 1$, one has
\begin{align*}
& \text{energy}^{1,\infty}_{\mathcal{J}'}((\langle \tilde{B}_J, \vphi_J \rangle)_{J \in \mathcal{J}'}) \lesssim C_2^{\frac{1}{p_2}+ \frac{1}{q_2} - \zeta_1 - \zeta_2} 2^{n_2(\frac{1}{p_2} - \zeta_1)} 2^{m_2(\frac{1}{q_2} - \zeta_2)} |G_1|^{\frac{1}{p_2}} |G_2|^{\frac{1}{q_2}}.
%\text{size}_{K \in \mathcal{K}: K \cap S \neq \emptyset}((\langle f_1, \vphi_K \rangle)_{K})^{1-\theta_1}\text{size}_{K \in \mathcal{K}:K \cap S \neq \emptyset}((\langle f_2, \psi_K \rangle)_{K})^{1-\theta_2} \nonumber \\
%& \cdot |F_1|^{\theta_1}|F_2|^{\theta_2}
%\text{energy}_{\mathcal{I}'}((\langle B_I, \vphi_I \rangle)_{I \in \mathcal{I}'}) \lesssim & \text{size}_{K \in \mathcal{K}: K \cap S \neq \emptyset}((\langle f_1, \vphi_K \rangle)_{K})^{1-\theta_1}\text{size}_{K \in \mathcal{K}:K \cap S \neq \emptyset}((\langle f_2, \psi_K \rangle)_{K})^{1-\theta_2} \nonumber \\
%& \cdot |F_1|^{\theta_1}|F_2|^{\theta_2}
\end{align*}

%\item
%\begin{align*}
%\text{energy}^{p,\infty} _{\mathcal{I}'}((\langle B_I, \vphi_I \rangle)_{I \in \mathcal{I}'}) \lesssim & \text{size}_{K \in \mathcal{K}: K \cap S \neq \emptyset}((\langle f_1, \vphi_K \rangle)_{K})^{1-\theta_1}\text{size}_{K \in \mathcal{K}:K \cap S \neq \emptyset}((\langle f_2, \psi_K \rangle)_{K})^{1-\theta_2} \nonumber \\
%& \cdot |F_1|^{\theta_1}|F_2|^{\theta_2}|S_2|^{\theta_3-\frac{1}{p'}}
%\end{align*}

%\begin{align*}
%\text{energy}^{p,\infty} _{\mathcal{J}'}((\langle \tilde{B_J}, \vphi_J \rangle)_{J \in \mathcal{J}'}) \lesssim & \text{size}_{L \in \mathcal{L}: L \cap S' \neq \emptyset}((\langle g_1, \vphi_K \rangle)_{L})^{1-\theta_1}\text{size}_{L \in \mathcal{L}:L \cap S' \neq \emptyset}((\langle g_2, \psi_L \rangle)_{L})^{1-\theta_2} \nonumber \\
%& \cdot |G_1|^{\theta_1}|G_2|^{\theta_2}|S'_2|^{\theta_3-\frac{1}{p'}}
%\end{align*}
\item
Suppose that $s >1$. Then for any $0 \leq \zeta_1, \zeta_2 <1$ with $\zeta_1 + \zeta_2= \frac{1}{s}$, one has
%for any $1 \leq t < \infty$, $1< p' \leq \infty $ and $\frac{1}{p}+ \frac{1}{p'}$ = 1,
\begin{align*}
%\sum_{l}2^{lp}\sup_{\mathbb{D}}\sum_{\substack{I \in \mathbb{D} \\ \frac{|\langle B_I, \vphi_I \rangle|}{|I|^{\frac{1}{2}}} > 2^{l} \\ I \in \mathcal{I}'}}|I|
& \text{energy}^{t} _{\mathcal{J}'}((\langle \tilde{B}_J, \vphi_J \rangle)_{J \in \mathcal{J}'}) \lesssim C_2^{\frac{1}{p_2}+ \frac{1}{q_2} - \zeta_1 - \zeta_2}2^{n_2(\frac{1}{p_2} - \zeta_1)}2^{m_2(\frac{1}{q_2} - \zeta_2)}|G_1|^{\frac{1}{p_2}}|G_2|^{\frac{1}{q_2}}.
%\text{size}_{K \in \mathcal{K}: K \cap S \neq \emptyset}((\langle f_1, \vphi_K \rangle)_{K})^{1-\theta_1}\text{size}_{K \in \mathcal{K}:K \cap S_1 \neq \emptyset}((\langle f_2, \psi_K \rangle)_{K})^{1-\theta_2} \nonumber \\
%& \cdot |F_1|^{\theta_1}|F_2|^{\theta_2}|S_2|^{\theta_3-\frac{1}{p'}}
%\text{energy}^{2} _{\mathcal{J}'}((\langle \tilde{B_J}, \vphi_J \rangle)_{J \in \mathcal{J}'}) \lesssim & \text{size}_{L \in \mathcal{L}: L \cap S' \neq \emptyset}((\langle g_1, \vphi_K \rangle)_{L})^{1-\theta_1}\text{size}_{L \in \mathcal{L}:L \cap S' \neq \emptyset}((\langle g_2, \psi_L \rangle)_{L})^{1-\theta_2} \nonumber \\
%& \cdot |G_1|^{\theta_1}|G_2|^{\theta_2}|S'_2|^{\theta_3-\frac{1}{p'}}
\end{align*}
\end{enumerate}
\end{proposition}
\begin{remark}
We would like to highlight that the localization of energies is attained under the additional conditions (\ref{loc_condition_x}) and (\ref{loc_condition_y}), in which case one obtains the local energy estimates stated in Proposition \ref{localized_energy_fourier_x} and \ref{localized_energy_y} that can be viewed as analogies of Proposition \ref{B_en}.
\end{remark}
\vskip .15in
\noindent
\textbf{Case II: For any $0 \leq \alpha_1, \beta_1 \leq 1$,  $ \frac{|\langle B_I, \vphi^1_I \rangle|}{|I|^{\frac{1}{2}}} \lesssim  (C_12^{n_1}|F_1|)^{\alpha_1}  (C_12^{m_1}|F_2|)^{\beta_1}.$}
\vskip .1in
%When (\ref{dom}) does not hold, or equivalently
%\begin{equation} \label{size_sufficient}
% \frac{|\langle B_I, \vphi^1_I \rangle|}{|I|^{\frac{1}{2}}} \lesssim  (C_12^{n_1}|F_1|)^{\alpha_1}  (C_12^{m_1}|F_2|)^{\beta_1},
%\end{equation}
In this alternative case, the size estimates are favorable and a simpler argument can be applied without invoking the local energy estimates.

\vskip .25in
\subsubsection{Proof Part 1 -  Localization} In this last section, we will explore how to implement the case-by-case analysis and develop a generalized proof of Theorem \ref{thm_weak_mod} for $\Pi_{\text{flag}^0 \otimes \text{flag}^0}$ when $(\phi^3_K)_K$ and $(\phi^3_L)_L$ are \textbf{lacunary} families and $\frac{1}{p_1} + \frac{1}{q_1} = \frac{1}{p_2} + \frac{1}{q_2} >1 $, which is the most tricky part to generalize from the argument in Section \ref{section_thm_haar}.
%More precisely, separation of cases and how to use size and energy estimates in each case. 
The generalized argument can be viewed as a combination of the discussions in Sections \ref{section_thm_haar_fixed} and \ref{section_thm_haar}. %Suppose that $(\phi_K^{3})$ and $(\phi_L^{3})_{L}$ are lacunary. 
One first defines the exceptional set $\Omega$ as follows.
For any $\tau_1,\tau_2 \in \mathbb{N}$, define
$$\Omega^{\tau_1,\tau_2} := \Omega_1^{\tau_1,\tau_2}\cup \Omega_2^{\tau_1,\tau_2} $$
with
\begin{align*}
\displaystyle \Omega_1^{\tau_1,\tau_2} := &\bigcup_{\mathfrak{n}_1 \in \mathbb{Z}}\{x: Mf_1(x) > C_1^{\tau_1} 2^{\mathfrak{n}_1}|F_1|\} \times \{y: Mg_1(y) > C_2^{\tau_2} 2^{-\mathfrak{n}_1}|G_1|\}\cup \nonumber \\
& \bigcup_{\mathfrak{n}_2 \in \mathbb{Z}}\{x: Mf_2(x) > C_1^{\tau_1} 2^{\mathfrak{n}_2}|F_2|\} \times \{y: Mg_2(y) > C_2^{\tau_2} 2^{-\mathfrak{n}_2}|G_2|\}\cup \nonumber \\
 &\bigcup_{\mathfrak{n}_3 \in \mathbb{Z}}\{x: Mf_1(x) > C_1^{\tau_1} 2^{\mathfrak{n}_3}|F_1|\} \times \{y: Mg_2(y) > C_2^{\tau_2} 2^{-\mathfrak{n}_3}|G_2|\}\cup \nonumber \\
& \bigcup_{\mathfrak{n}_4 \in \mathbb{Z}}\{x: Mf_2(x) > C_1^{\tau_1} 2^{\mathfrak{n}_4 }|F_2|\} \times \{y: Mg_1(y) > C_2^{\tau_2} 2^{-\mathfrak{n}_4 }|G_1|\}\cup \nonumber \\
& \bigcup_{l_2 \in \mathbb{Z}}\{x: MB(f_1,f_2)(x) > C_1^{\tau_1}2^{-l_2}\|B(f_1,f_2)\|_1\} \times \{y: M\tilde{B}(g_1,g_2)(y) > C_2^{\tau_2} 2^{l_2} \|\tilde{B}(g_1,g_2)\|_1\}, \nonumber \\
\Omega_2^{\tau_1,\tau_2} := & \{(x,y) \in \mathbb{R}^2: SSh(x,y) > C_3^{\tau_1,\tau_2} \|h\|_{L^s(\mathbb{R}^2)}\}, \nonumber \\
\end{align*}
and
\begin{align*}
& Enl(\Omega^{\tau_1,\tau_2}) := \{(x,y) \in \mathbb{R}^2: MM(\chi_{\Omega^{\tau_1,\tau_2}})(x,y) > \frac{1}{100}\}, \nonumber \\
& Enl_{\tau_1,\tau_2}(Enl(\Omega^{\tau_1,\tau_2})) := \{(x,y) \in \mathbb{R}^2: MM(\chi_{Enl(\Omega^{\tau_1,\tau_2})}(x,y)> \frac{1}{2^{2\tau_1+ 2\tau_2}} \},
\end{align*}
and lastly
$$\mathbf{\Omega} := \bigcup_{\tau_1,\tau_2 \in \mathbb{N}}Enl_{\tau_1,\tau_2}(Enl(\Omega^{\tau_1,\tau_2})).$$
Let 
$$E' := E \setminus \mathbf{\Omega},$$
where $|E'| \sim |E| =1$ given that $C_1, C_2$ and $C_3$ are sufficiently large constants. For any $\tau_1, \tau_2 \in \mathbb{N}$ fixed, let
\begin{equation} 
\Lambda_{\text{flag}^{0} \otimes \text{flag}^{0}}^{\tau_1, \tau_2} (f_1, f_2, g_1, g_2, h, \chi_{E'}):= \sum_{I \times J \in \mathcal{R}} \frac{1}{|I|^{\frac{1}{2}} |J|^{\frac{1}{2}}} \langle B_I(f_1,f_2),\phi_I^1 \rangle \langle \tilde{B}_J(g_1,g_2), \phi_J^1 \rangle \cdot \langle h, \phi_{I}^2 \otimes \phi_{J}^2 \rangle \langle \chi_{E'},\phi_{I}^{3,\tau_1} \otimes \phi^{3, \tau_2}_{J} \rangle,
\end{equation}
where $B_I$ and $\tilde{B}_J$ are bilinear operators defined in (\ref{B_local_fourier_simple}) and (\ref{B_local0_haar}).
Our goal is to prove that 
\begin{align} \label{compact}
\Lambda_{\text{flag}^{0} \otimes \text{flag}^{0}}(f_1, f_2, g_1, g_2, h, \chi_{E'})  =  \displaystyle \sum_{\tau_1,\tau_2 \in \mathbb{N}}2^{-100(\tau_1+\tau_2)}\Lambda_{\text{flag}^{0} \otimes \text{flag}^{0}}^{\tau_1, \tau_2} (f_1, f_2, g_1, g_2, h, \chi_{E'}) \nonumber
\end{align} 
satisfies the restricted weak-type estimates
\begin{equation} \label{final_linear}
|\Lambda_{\text{flag}^{0} \otimes \text{flag}^{0}}(f_1, f_2, g_1, g_2, h, \chi_{E'})| \lesssim |F_1|^{\frac{1}{p_1}}|F_2|^{\frac{1}{q_1}}|G_1|^{\frac{1}{p_2}}|G_2|^{\frac{1}{q_2}}\|h\|_{L^s(\mathbb{R}^2)},
\end{equation}
which  can be reduced to proving that for any fixed $\tau_1, \tau_2 \in \mathbb{N}$, %One can now focus on $\Lambda_{\text{flag}^{\#1} \otimes \text{flag}^{\#_2}}^{\tau_1, \tau_2}$. 
\begin{equation} \label{linear_fix_0_fourier}
|\Lambda_{\text{flag}^0 \otimes \text{flag}^{0}}^{\tau_1, \tau_2}(f_1, f_2, g_1, g_2, h, \chi_{E'})| \lesssim (2^{\tau_1+ \tau_2})^{\Theta}|F_1|^{\frac{1}{p_1}}|F_2|^{\frac{1}{q_1}}|G_1|^{\frac{1}{p_2}}|G_2|^{\frac{1}{q_2}}\|h\|_{L^s(\mathbb{R}^2)}
\end{equation}
for some $0 < \Theta < 100$.

\subsubsection{Proof Part 2 - Summary of stopping-time decompositions.} \label{fourier_summary_st}
For any fixed $\tau_1,\tau_2 \in \mathbb{N}$, one can carry out the exactly same stopping-time algorithms in Section \ref{section_thm_haar} with the replacement of $C_1, C_2$ and $C_3$ by $C_1^{\tau_1,\tau_2}$, $C_2^{\tau_1,\tau_2}$ and $C_3^{\tau_1,\tau_2}$ respectively. The resulting level sets, trees and collections of dyadic rectangles will follow the similar notation as before with extra indications of $\tau_1$ and $\tau_2$. 
\vskip .15in

%\begin{center}
%\begin{tabular}{ c c c }
%\hline
%One-dimensional stopping-time decomposition& $\longrightarrow$ & $K \in \mathcal{K}_{n_0}$ \\
%on $\mathcal{K}$ & & $(n_0 \in \mathbb{Z})$ \\
%\hline
%\end{tabular}
%\end{center}

%\begin{center}
%\begin{tabular}{ |c c c| }
%\hline
%One-dimensional stopping-time decomposition& $\longrightarrow$ & $L \in \mathcal{L}_{n'_0}$ \\
%on $\mathcal{L}$ & & $(n'_0 \in \mathbb{Z})$ \\
%\hline
%\end{tabular}
%\end{center}
{\fontsize{9.5}{9.5}
\begin{table}[h!]
\begin{tabular}{ l l l }	
%\hline
I. Tensor-type stopping-time decomposition I on $\mathcal{I} \times \mathcal{J}$& $\longrightarrow$ & $I \times J \in \mathcal{I}^{\tau_1}_{-n-n_2,-m-m_2} \times \mathcal{J}^{\tau_2}_{n_2,m_2}$ \\ 
& & $(n_2, m_2 \in \mathbb{Z}, n > 0, m>0)$\\
II. Tensor-type stopping-time decomposition II on $\mathcal{I} \times \mathcal{J}$ & $\longrightarrow$ & $I \times J \in T \times S $  with $T \in \mathbb{T}_{-l-l_2}^{\tau_1}$, $S \in \mathbb{S}_{l_2}^{\tau_2}$\\
& & $(l_2 \in \mathbb{Z},  l > 0)$\\
III. General two-dimensional level sets stopping-time & $\longrightarrow$ & $I \times J \in \mathcal{R}_{k_1,k_2}^{\tau_1,\tau_2} $  \\
\ \ \ \ \ decomposition on $\mathcal{I} \times \mathcal{J}$& & $(k_1 <0, k_2 \leq K)$\\
%& &\vline\ \ $l_2 \in \mathcal{EXP}_1$ \vline \ \ $l_2 \in \mathcal{EXP}_2$ \vline \ \ $l_2 \in \mathcal{EXP}_3$ \vline \ \  $l_2 \in \mathcal{EXP}_4$ \\
%\hline
\end{tabular}
\end{table}} %to modify!!!!
\noindent
\subsubsection{Proof Part 3 - Application of stopping-time decompositions.}
As one may recall, in Section \ref{section_thm_haar} the multilinear form is estimated based on the stopping-time decompositions, the sparsity condition and the Fubini-type argument. %In the general case, one first decomposes the multilinear form into (\ref{compact}). Then for any fixed $\tau_1, \tau_2 \in \mathbb{N}$, one uses the same argument as in the Haar setting to estimate the inner sum as follows:
Analogously, we first apply the stopping-time decompositions specified in Section \ref{fourier_summary_st} and denote the range of exponents 
\begin{equation}
\mathcal{W} := \{(l,n,m,k_1,k_2) \in \mathbb{Z}^5: l >0, n> 0, m> 0, k_1 < 0, k_2 \leq K \}
\end{equation}
so that the multilinear form can be estimated as 
\begin{align}
%=&\bigg| \sum_{I \times J \in \mathcal{I} \times \mathcal{J}}\frac{1}{|I|^{\frac{1}{2}} |J|^{\frac{1}{2}}} \langle B_I(f_1,f_2),\phi_I^1 \rangle \langle \tilde{B}(g_1,g_2), \phi_J^1 \rangle \langle h, \phi_I^2 \otimes \phi_J^2 \rangle \langle \chi_{E'}, \phi_I^{3,\tau_1} \otimes \phi_J^{3,\tau_2} \rangle\bigg| \nonumber \\
&|\Lambda_{\text{flag}^0 \otimes \text{flag}^{0}}^{\tau_1, \tau_2}(f_1,f_2,g_1,g_2,h, \chi_{E'})| \nonumber \\
= & \bigg| \sum_{(l,n,m,k_1,k_2) \in \mathcal{W}}\sum_{\substack{n_2 \in \mathbb{Z}\\ m_2 \in \mathbb{Z}\\ l_2\in \mathbb{Z}}} \sum_{\substack{T \in \mathbb{T}_{-l-l_2}^{\tau_1} \\ S \in \mathbb{S}_{l_2}^{\tau_2}}}\sum_{\substack{I \times J \in \mathcal{I}_{-n-n_2,-m-m_2}^{\tau_1} \times \mathcal{J}_{n_2,m_2}^{\tau_2} \\ I \times J \in T \times S \\ I \times J \in \mathcal{R}_{k_1,k_2}^{\tau_1,\tau_2}}}\frac{1}{|I|^{\frac{1}{2}} |J|^{\frac{1}{2}}} \langle B_I(f_1,f_2),\phi_I^1 \rangle \langle \tilde{B}_J(g_1,g_2), \phi_J^1 \rangle \cdot \nonumber\\
&  \quad \quad \quad \quad \quad \quad \quad \quad \quad \quad \quad \quad \quad \quad \quad \quad \quad \quad \quad \quad \quad \quad \quad \quad  \langle h, \phi_I^2 \otimes \phi_J^2 \rangle \langle \chi_{E'}, \phi_I^{3,\tau_1} \otimes \phi_J^{3,\tau_2} \rangle\bigg|  \nonumber \\
\lesssim &C_1^{\tau_1}C_2^{\tau_2}(C_3^{\tau_1,\tau_2})^2\sum_{(l,n,m,k_1,k_2) \in \mathcal{W}}2^{k_1} \|h\|_{L^s(\mathbb{R}^2)} 2^{k_2} \cdot \sum_{\substack{n_2 \in \mathbb{Z}\\ m_2 \in \mathbb{Z}\\ l_2\in \mathbb{Z}}} 2^{-l-l_2} \|B(f_1,f_2)\|_1 2^{l_2} \|\tilde{B}(g_1,g_2)\|_1 \cdot \nonumber\\
& \quad \quad \quad\quad \quad \quad\quad \quad \quad \quad \quad \quad\quad \quad \quad\quad \quad \quad  \quad \quad \quad\quad  \sum_{\substack{T \in \mathbb{T}_{-l-l_2}^{\tau_1} \\ S \in \mathbb{S}_{l_2}^{\tau_2}}} \bigg|\bigcup_{\substack{I \times J \in \mathcal{I}_{-n-n_2,-m-m_2}^{\tau_1} \cap T \times \mathcal{J}_{n_2,m_2}^{\tau_2} \cap S \\ I \times J \in \mathcal{R}_{k_1,k_2}^{\tau_1,\tau_2}}}I \times J\bigg|. 
\end{align}
%fourier-sparsity
The nested sum 
\begin{align} \label{ns_fourier}
& \sum_{\substack{n_2 \in \mathbb{Z}\\ m_2 \in \mathbb{Z}\\ l_2\in \mathbb{Z}}} 2^{-l-l_2} \|B(f_1,f_2)\|_1 2^{l_2}\|\tilde{B}(g_1,g_2)\|_1 \sum_{\substack{T \in \mathbb{T}_{-l-l_2}^{\tau_1} \\ S \in \mathbb{S}_{l_2}^{\tau_2}}} \bigg|\bigcup_{\substack{I \times J \in \mathcal{I}_{-n-n_2,-m-m_2}^{\tau_1} \cap T \times \mathcal{J}_{n_2,m_2}^{\tau_2} \cap S \\ I \times J \in \mathcal{R}_{k_1,k_2}^{\tau_1,\tau_2}}}I \times J \bigg| 
\end{align}
can be estimated using the sparsity condition (Proposition \ref{sp_2d}) and a modified Fubini argument as discussed in the following two subsections.
\subsubsection{Proof Part 4 - Sparsity condition}
One invokes the sparsity condition and the argument in Section \ref{section_thm_haar_fixed} to obtain the following estimate for (\ref{ns_fourier}) analogous to (\ref{ns_sp}) .
\begin{align} \label{ns_fourier_sp}
%& \sum_{\substack{n_2 \in \mathbb{Z}\\ m_2 \in \mathbb{Z}\\ l_2\in \mathbb{Z}}} 2^{-l-l_2} \|B(f_1,f_2)\|_1 2^{l_2}\|\tilde{B}(g_1,g_2)\|_1 \sum_{\substack{T \in \mathbb{T}_{-l-l_2}^{\tau_1} \\ S \in \mathbb{S}_{l_2}^{\tau_2}}} \bigg|\bigcup_{\substack{I \times J \in \mathcal{I}_{-n-n_2,-m-m_2}^{\tau_1} \cap T \times \mathcal{J}_{n_2,m_2}^{\tau_2} \cap S \\ I \times J \in \mathcal{R}_{k_1,k_2}^{\tau_1,\tau_2}}}I \times J \bigg| \nonumber \\
(\ref{ns_fourier}) \lesssim & 2^{-\frac{k_2\gamma}{2}}2^{-l(1- \frac{(1+\delta)}{2})} |F_1|^{\frac{\mu_1(1+\delta)}{2}}|F_2|^{\frac{\mu_2(1+\delta)}{2}}|G_1|^{\frac{\nu_1(1+\delta)}{2}}|G_2|^{\frac{\nu_2(1+\delta)}{2}}\|B(f_1,f_2)\|_1^{1-\frac{1+\delta}{2}}\|\tilde{B}(g_1,g_2)\|_1^{1-\frac{1+\delta}{2}}.
\end{align}
where $\gamma >1 $, $\delta,\mu_1,\mu_2,\nu_1,\nu_2 >0$ with $\mu_1+ \mu_2 = \nu_1+ \nu_2 = \frac{1}{1+\delta}$. For any $0 < \delta \ll1$, Lemma \ref{B_global_norm} implies that
\begin{align}
& \|B(f_1,f_2)\|_1^{1-\frac{1+\delta}{2}} \lesssim |F_1|^{\rho(1-\frac{1+\delta}{2})}|F_2|^{(1-\rho)(1-\frac{1+\delta}{2})}, \label{B_norm_fourier1} \\
& \|\tilde{B}(g_1,g_2)\|_1^{1-\frac{1+\delta}{2}} \lesssim |G_1|^{\rho'(1-\frac{1+\delta}{2})}|G_2|^{(1-\rho')(1-\frac{1+\delta}{2})}. \label{B_norm_fourier2}
\end{align}
%given that
%\begin{align*}
%& \|B\|_1 \lesssim |F_1|^{\rho}|F_2|^{1-\rho} \nonumber \\
%& \|\tilde{B}\|_1 \lesssim |G_1|^{\rho'}|G_2|^{1-\rho'}
%\end{align*}
By applying (\ref{B_norm_fourier1}) and (\ref{B_norm_fourier2}) to (\ref{ns_fourier_sp}), one derives the following bound:
\begin{equation} \label{ns_fourier_sp_final}
2^{-\frac{k_2\gamma}{2}}2^{-l(1- \frac{(1+\delta)}{2})} |F_1|^{\frac{\mu_1(1+\delta)}{2}+\rho(1-\frac{1+\delta}{2})}|F_2|^{\frac{\mu_2(1+\delta)}{2}+(1-\rho)(1-\frac{1+\delta}{2})}|G_1|^{\frac{\nu_1(1+\delta)}{2}+\rho'(1-\frac{1+\delta}{2})}|G_2|^{\frac{\nu_2(1+\delta)}{2}+(1-\rho')(1-\frac{1+\delta}{2})}.
\end{equation}
%\begin{align*}
%|\Lambda| \lesssim &  \sum_{\substack{n > 0 \\ m> 0 \\ l > 0 \\ k_1 < 0 \\ k_2 \leq K}} \sum_{\substack{n_2 \in \mathbb{Z}\\ m_2 \in \mathbb{Z}\\ l_2\in \mathbb{Z}}} \nonumber \\
%\lesssim &  \sum_{\substack{n > 0 \\ m> 0 \\ l > 0 \\ k_1 < 0 \\ k_2 \leq K}} \sum_{\substack{n_2 \in \mathbb{Z}\\ m_2 \in \mathbb{Z}\\ l_2\in \mathbb{Z}}} 2^{-l-l_2} \|B\|_1 2^{l_2} \|\tilde{B}\|_1 2^{k_1}\|h\|_{L^s(\mathbb{R}^2)} 2^{k_2} \sum_{\substack{S \in \mathbb{S}^{-l-l_2} \\ T \in \mathbb{T}^{l_2}}} \bigg|\bigcup_{\substack{I \times J \in \mathcal{I}_{-n-n_2,-m-m_2} \cap S \times \mathcal{J}_{n_2,m_2} \cap T \\ I \times J \in \mathcal{R}_{k_1,k_2}}}I \times J\bigg| \nonumber \\
%\end{align*}
\vskip .15 in
\subsubsection{Proof Part 5 - Fubini argument}
The separation of cases based on the levels of the stopping-time decompositions for $
\big(\frac{|\langle B_{I}(f_1,f_2), \vphi^{1}_I \rangle|}{|I|^{\frac{1}{2}}}\big)_{I \in \mathcal{I}}
$
and
$
\big(\frac{|\langle \tilde {B}_{J}(g_1,g_2), \vphi^{1}_J \rangle|}{|J|^{\frac{1}{2}}}\big)_{J \in \mathcal{J}}
$. In particular, the ranges of $l_2$ in the \textit{tensor-type stopping-time decomposition I}, plays an important role in the modified Fubini-type argument. With $l \in \mathbb{N}$ fixed, the ranges of exponents $l_2$ are defined as follows:
\begin{align*}
\mathcal{EXP}_1^{l,n,m,n_2,m_2} := & \{l_2 \in \mathbb{Z}: \text{for any}\ \ 0 \leq \alpha_1, \beta_1, \alpha_2, \beta_2 \leq 1, \nonumber \\
& \quad \quad \quad \quad 2^{-l-l_2} \| B(f_1,f_2)\|_1\lesssim (C_1^{\tau_1}2^{-n-n_2}|F_1|)^{\alpha_1} (C_1^{\tau_1}2^{-m-m_2}|F_2|)^{\beta_1} \ \ \text{and} \nonumber \\
& \quad \quad \quad \quad 2^{l_2} \|\tilde{B}(g_1,g_2)\|_1 \lesssim (C_2^{\tau_2}2^{n_2}|G_1|)^{\alpha_2} (C_2^{\tau_2}2^{m_2}|G_2|)^{\beta_2} \}, \nonumber \\
\mathcal{EXP}_2^{l,n,m,n_2,m_2} := & \{l_2 \in \mathbb{Z}: \text{there exists} \ \ 0 \leq \alpha_1, \beta_1 \leq 1 \ \ \text{such that}  \nonumber \\
& \quad \quad \quad \quad 2^{-l-l_2} \| B(f_1,f_2)\|_1 \gg (C_1^{\tau_1}2^{-n-n_2}|F_1|)^{\alpha_1} (C_1^{\tau_1}2^{-m-m_2}|F_2|)^{\beta_1} \ \ \text{and} \nonumber \\
& \quad \quad \quad \quad \text{for any} \ \ 0 \leq \alpha_2, \beta_2 \leq 1, \nonumber \\
& \quad \quad \quad \quad 2^{l_2} \|\tilde{B}(g_1,g_2)\|_1 \lesssim (C_2^{\tau_2}2^{n_2}|G_1|)^{\alpha_2} (C_2^{\tau_2}2^{m_2}|G_2|)^{\beta_2}\}, \nonumber \\
\mathcal{EXP}_3^{l,n,m,n_2,m_2} := & \{l_2 \in \mathbb{Z}: \text{for any} \ \ 0 \leq \alpha_1, \beta_1 \leq 1, \nonumber \\
 & \quad \quad \quad \quad 2^{-l-l_2} \| B(f_1,f_2)\|_1 \lesssim (C_1^{\tau_1}2^{-n-n_2}|F_1|)^{\alpha_1} (C_1^{\tau_1}2^{-m-m_2}|F_2|)^{\beta_1} \ \ \text{and} \nonumber \\
 & \quad \quad \quad \quad \text{there exists} \ \ 0 \leq \alpha_2, \beta_2 \leq 1 \ \ \text{such that} \nonumber \\
& \quad \quad \quad \quad 2^{l_2} \|\tilde{B}(g_1,g_2)\|_1 \gg (C_2^{\tau_2}2^{n_2}|G_1|)^{\alpha_2} (C_2^{\tau_2}2^{m_2}|G_2|)^{\beta_2}\}, \nonumber \\
\mathcal{EXP}_4^{l,n,m,n_2,m_2} := & \{l_2 \in \mathbb{Z}: \text{there exists} \ \ 0 \leq \alpha_1, \beta_1, \alpha_2, \beta_2 \leq 1 \ \ \text{such that} \nonumber \\ 
& \quad \quad \quad \quad  2^{-l-l_2} \| B(f_1,f_2)\|_1 \gg (C_1^{\tau_1}2^{-n-n_2}|F_1|)^{\alpha_1} (C_1^{\tau_1}2^{-m-m_2}|F_2|)^{\beta_1} \ \ \text{and} \nonumber \\
& \quad \quad \quad \quad 2^{l_2} \|\tilde{B}(g_1,g_2)\|_1 \gg (C_2^{\tau_2}2^{n_2}|G_1|)^{\alpha_2} (C_2^{\tau_2}2^{m_2}|G_2|)^{\beta_2}\}. \nonumber 
\end{align*}
%Due to the condition one needs to impose for the application of the ``pseudo-biest'' trick, we would decompose the sum 
With little abuse of notation, we will simplify $\mathcal{EXP}_i^{l,n,m,n_2,m_2}$ by $\mathcal{EXP}_i$, for $i = 1,2,3,4$. We will then decompose the sum into four parts based on the ranges specified above:
\begin{align*}
& \sum_{\substack{n_2 \in \mathbb{Z}\\ m_2 \in \mathbb{Z}\\ l_2\in \mathbb{Z}}} 2^{-l-l_2} \|B\|_1 2^{l_2} \|\tilde{B}\|_1\sum_{\substack{S \in \mathbb{S}^{-l-l_2} \\ T \in \mathbb{T}^{l_2}}} \bigg|\bigcup_{\substack{I \times J \in \mathcal{I}_{-n-n_2,-m-m_2} \cap S \times \mathcal{J}_{n_2,m_2} \cap T \\ I \times J \in \mathcal{R}_{k_1,k_2}}}I \times J\bigg|  \nonumber \\
= &\underbrace{ \sum_{\substack{n_2 \in \mathbb{Z}\\ m_2 \in \mathbb{Z}\\ }}  \sum_{l_2\in \mathcal{EXP}_1^{}}}_I+ \underbrace{\sum_{\substack{n_2 \in \mathbb{Z}\\ m_2 \in \mathbb{Z}\\ }} \sum_{l_2\in \mathcal{EXP}_2^{}}}_{II} +  \underbrace{\sum_{\substack{n_2 \in \mathbb{Z}\\ m_2 \in \mathbb{Z}\\ }} \sum_{_2\in \mathcal{EXP}_3^{}} }_{III}+ \underbrace{\sum_{\substack{n_2 \in \mathbb{Z}\\ m_2 \in \mathbb{Z}\\ }} \sum_{l_2\in \mathcal{EXP}_4^{}}}_{IV}.
\end{align*}
One denotes the four parts by $I$, $II$, $III$ and $IV$ and will derive estimates for each part separately. The multilinear form can thus be decomposed correspondingly as follows:
\begin{align*}
&|\Lambda_{\text{flag}^0 \otimes \text{flag}^{0}}^{\tau_1, \tau_2}| \nonumber\\
 \lesssim & C_1^{\tau_1}C_2^{\tau_2} (C_3^{\tau_1,\tau_2})^2 \bigg(\underbrace{\sum_{(l,n,m,k_1,k_2) \in \mathcal{W}} 2^{k_1}\|h\|_{L^s(\mathbb{R}^2)} 2^{k_2} \cdot I}_{\Lambda_{I}^{\tau_1, \tau_2}} + \underbrace{ \sum_{(l,n,m,k_1,k_2) \in \mathcal{W}} 2^{k_1}\|h\|_{L^s(\mathbb{R}^2)} 2^{k_2} \cdot II}_{\Lambda_{II}^{\tau_1, \tau_2}} +  \nonumber \\
&  \quad \quad \quad \quad \quad \quad \quad \quad \underbrace{ \sum_{(l,n,m,k_1,k_2) \in \mathcal{W}} 2^{k_1}\|h\|_{L^s(\mathbb{R}^2)} 2^{k_2} \cdot III}_{\Lambda_{III}^{\tau_1, \tau_2}} + \underbrace{\sum_{(l,n,m,k_1,k_2) \in \mathcal{W}} 2^{k_1}\|h\|_{L^s(\mathbb{R}^2)} 2^{k_2} \cdot IV}_{\Lambda_{IV}^{\tau_1, \tau_2}}\bigg).
\end{align*}
It would be sufficient to prove that each part satisfies the bound on the right hand side of (\ref{linear_fix_0_fourier}).
%$$
%(C_1^{\tau_1}C_2^{\tau_2} C_3 ^{\tau_1,\tau_2})^{\Theta}|F_1|^{\frac{1}{p_1}}|F_2|^{\frac{1}{q_1}}|G_1|^{\frac{1}{p_2}}|G_2|^{\frac{1}{q_2}}\|h\|_{L^s(\mathbb{R}^2)}
%$$
%for some constant $0<\Theta<100$.
\vskip .15in
\noindent
\textbf{Estimate of $\Lambda_I^{\tau_1,\tau_2}$.}
Though for $I$, the localization of energies cannot be applied at all, one observes that energy estimates are indeed not necessary. In particular,
\begin{align} \label{I}
&I \nonumber \\
\lesssim & \sum_{\substack{n_2 \in \mathbb{Z}\\ m_2 \in \mathbb{Z}}} \sum_{l_2\in \mathcal{EXP}_1}2^{-l-l_2}\|B(f_1,f_2)\|_1 2^{l_2}\|\tilde{B}(g_1,g_2)\|_1 \sum_{\substack{T \in \mathbb{T}_{-l-l_2}^{\tau_1} \\ S \in \mathbb{S}_{l_2}^{\tau_2}}} \bigg|\bigcup_{\substack{I \times J \in \mathcal{I}_{-n-n_2,-m-m_2}^{\tau_1} \cap T \times \mathcal{J}_{n_2,m_2}^{\tau_2} \cap S \\ I \times J \in \mathcal{R}_{k_1,k_2}^{\tau_1,\tau_2}}}I \times J\bigg| \nonumber \\
\leq & \sum_{\substack{n_2 \in \mathbb{Z}\\ m_2 \in \mathbb{Z}}} \bigg(\sup_{l_2\in \mathcal{EXP}_1} 2^{-l-l_2}\|B(f_1,f_2)\|_1\bigg)\bigg(\sup_{l_2\in \mathcal{EXP}_1} \sum_{\substack{T \in \mathbb{T}_{-l-l_2}^{\tau_1} \\ S \in \mathbb{S}_{l_2}^{\tau_2}}} \bigg|\bigcup_{\substack{I \times J \in \mathcal{I}_{-n-n_2,-m-m_2}^{\tau_1} \cap T \times \mathcal{J}_{n_2,m_2}^{\tau_2} \cap S \\ I \times J \in \mathcal{R}_{k_1,k_2}^{\tau_1,\tau_2}}}I \times J\bigg|\bigg)\cdot \nonumber\\ 
& \quad\quad\quad\quad\quad\quad\quad\quad\quad\quad\quad\quad\quad\quad\quad \quad \bigg(\sum_{l_2\in \mathcal{EXP}_1}2^{l_2}\|\tilde{B}(g_1,g_2)\|_1\bigg).
%& \sum_{\substack{n_2 \in \mathbb{Z}\\ m_2 \in \mathbb{Z}}} \sum_{l_2\in \mathcal{EXP}_1^{l,-n-n_2,n_2,\alpha_1,\beta_1,\alpha_2,\beta_2} } (2^{-n-n_2}|F_1|)^{\alpha_1} (2^{-m-m_2}|F_2|)^{\beta_1} (2^{n_2}|G_1|)^{\alpha_2} (2^{m_2}|G_2|)^{\beta_2}\cdot\nonumber\\
%& \quad \quad \quad \quad  \quad \quad \quad  \quad \quad  \quad \quad  \quad \ \sum_{\substack{S \in \mathbb{S}^{-l-l_2} \\ T \in \mathbb{T}^{l_2}}} \bigg|\bigcup_{\substack{I \times J \in \mathcal{I}_{-n-n_2,-m-m_2} \cap S \times \mathcal{J}_{n_2,m_2} \cap T \\ I \times J \in \mathcal{R}_{k_1,k_2}}}I \times J\bigg| \nonumber \\
 %\leq &   \sum_{\substack{n_2 \in \mathbb{Z}\\ m_2 \in \mathbb{Z}}} \sum_{l_2\in \mathbb{Z}} (2^{-n-n_2}|F_1|)^{\alpha_1} (2^{-m-m_2}|F_2|)^{\beta_1} (2^{n_2}|G_1|)^{\alpha_2} (2^{m_2}|G_2|)^{\beta_2}\cdot \sum_{\substack{S \in \mathbb{S}^{-l-l_2} \\ T \in \mathbb{T}^{l_2}}} \bigg|\bigcup_{\substack{I \times J \in \mathcal{I}_{-n-n_2,-m-m_2} \cap S \times \mathcal{J}_{n_2,m_2} \cap T \\ I \times J \in \mathcal{R}_{k_1,k_2}}}I \times J\bigg|\nonumber\\
%&  \quad  \quad \quad \quad \sum_{\substack{S \in \mathbb{S}^{-l-l_2} \\ T \in \mathbb{T}^{l_2}}} \bigg|\bigcup_{\substack{I \times J \in \mathcal{I}_{-n-n_2,-m-m_2} \cap S \times \mathcal{J}_{n_2,m_2} \cap T \\ I \times J \in \mathcal{R}_{k_1,k_2}}}I \times J\bigg| \nonumber \\
\end{align}
We will estimate the expressions in the parentheses separately. 
\begin{enumerate}[(i)]
\item
It is trivial from the definition of $\mathcal{EXP}_1$ that for any $0 \leq \alpha_1, \beta_1 \leq 1$,
\begin{equation} \label{I_i}
\sup_{l_2\in \mathcal{EXP}_1} 2^{-l-l_2}\|B(f_1,f_2)\|_1 \lesssim (C_1^{\tau_1}2^{-n-n_2}|F_1|)^{\alpha_1} (C_1^{\tau_1}2^{-m-m_2}|F_2|)^{\beta_1}.
\end{equation}
\item
The last expression is a geometric series with the largest term bounded by 
\begin{equation}\label{I_ii}
(C_2^{\tau_2}2^{n_2}|G_1|)^{\alpha_2} (C_2^{\tau_2}2^{m_2}|G_2|)^{\beta_2},
\end{equation}
for any $0 \leq \alpha_2, \beta_2 \leq 1$ according to the definition of $\mathcal{EXP}_1$. As a result, 
$$
\sum_{l_2 \in \mathcal{EXP}_1} 2^{l_2} \|\tilde{B}(g_1,g_2)\|_1 \lesssim (C_2^{\tau_2}2^{n_2}|G_1|)^{\alpha_2} (C_2^{\tau_2},2^{m_2}|G_2|)^{\beta_2},
$$
for any $0 \leq \alpha_2, \beta_2 \leq 1$.
\item
For any fixed $-n-n_2, -m-m_2, n_2,m_2, l_2, \tau_1,\tau_2$, 
\begin{align*} 
&\{I_T: I_T \in \mathcal{I}_{-n-n_2,-m-m_2}^{\tau_1} \ \ \text{and} \ \  T \in \mathbb{T}^{\tau_1}_{-l-l_2}  \}, \nonumber \\
& \{J_S: J_S \in \mathcal{I}_{n_2,-m_2}^{\tau_2} \ \ \text{and} \ \  S \in \mathbb{S}^{\tau_2}_{l_2}  \}
\end{align*}
are disjoint collections of dyadic intervals. Therefore, for some fixed $\tilde{l}_2 \in \mathbb{Z}$,
\begin{align} \label{I_iii}
\sup_{l_2\in \mathcal{EXP}_1} \sum_{\substack{T \in \mathbb{T}_{-l-l_2}^{\tau_1} \\ S \in \mathbb{S}_{l_2}^{\tau_2}}} \bigg|\bigcup_{\substack{I \times J \in \mathcal{I}_{-n-n_2,-m-m_2}^{\tau_1} \cap T \times \mathcal{J}_{n_2,m_2}^{\tau_2} \cap S \\ I \times J \in \mathcal{R}_{k_1,k_2}^{\tau_1,\tau_2}}}I \times J\bigg| = & \bigg|\bigcup_{\substack{T \in \mathbb{T}_{-l-\tilde{l}_2}^{\tau_1} \\ S \in \mathbb{S}_{\tilde{l}2}^{\tau_2}}}\bigcup_{\substack{I \times J \in \mathcal{I}_{-n-n_2,-m-m_2}^{\tau_1} \cap T \times \mathcal{J}_{n_2,m_2}^{\tau_2} \cap S \\ I \times J \in \mathcal{R}_{k_1,k_2}^{\tau_1,\tau_2}}}I \times J\bigg| \nonumber\\
\leq & \bigg|\bigcup_{\substack{I \times J \in \mathcal{I}_{-n-n_2,-m-m_2}^{\tau_1}  \times \mathcal{J}_{n_2,m_2}^{\tau_2} \\ I \times J \in \mathcal{R}_{k_1,k_2}^{\tau_1,\tau_2}}}I \times J \bigg|.
\end{align}
\end{enumerate}
One can now plug in the estimates (\ref{I_i}), (\ref{I_ii}) and (\ref{I_iii}) into (\ref{I}) and derive that for any $0 \leq \alpha_1, \beta_1, \alpha_2, \beta_2 \leq 1$,
\begin{align}
& I  \nonumber \\
\lesssim & (C_1^{\tau_1})^2(C_2^{\tau_2})^2\sum_{\substack{n_2 \in \mathbb{Z} \\ m_2 \in \mathbb{Z}}}(2^{-n-n_2}|F_1|)^{\alpha_1} (2^{-m-m_2}|F_2|)^{\beta_1}(2^{n_2}|G_1|)^{\alpha_2} (2^{m_2}|G_2|)^{\beta_2}\bigg|\bigcup_{\substack{I \times J \in \mathcal{I}_{-n-n_2,-m-m_2}^{\tau_1}  \times \mathcal{J}_{n_2,m_2}^{\tau_2} \\ I \times J \in \mathcal{R}_{k_1,k_2}^{\tau_1,\tau_2}}}I \times J \bigg| \nonumber \\ 
\leq &  (C_1^{\tau_1})^2(C_2^{\tau_2})^2 \sup_{\substack{n_2 \in \mathbb{Z} \\ m_2 \in \mathbb{Z}}}(2^{-n-n_2}|F_1|)^{\alpha_1} (2^{-m-m_2}|F_2|)^{\beta_1}(2^{n_2}|G_1|)^{\alpha_2} (2^{m_2}|G_2|)^{\beta_2}\sum_{\substack{n_2 \in \mathbb{Z} \\ m_2 \in \mathbb{Z}}}\bigg|\bigcup_{\substack{I \times J \in \mathcal{I}_{-n-n_2,-m-m_2}^{\tau_1}  \times \mathcal{J}_{n_2,m_2}^{\tau_2} \\ I \times J \in \mathcal{R}_{k_1,k_2}^{\tau_1,\tau_2}}}I \times J \bigg|. 
\end{align}
By letting $\alpha_1 = \frac{1}{p_1}$, $\beta_2 = \frac{1}{q_1}$, $\alpha_2 = \frac{1}{p_2}$ and  $\beta_2 = \frac{1}{q_2}$ and the argument for the choice of indices in Section \ref{section_thm_haar_fixed}, one has
$$
I \lesssim (C_1^{\tau_1})^2(C_2^{\tau_2})^2 2^{-n\frac{1}{p_2}}2^{-m\frac{1}{q_1}} |F_1|^{\frac{1}{p_1}} |F_2|^{\frac{1}{q_1}} |G_1|^{\frac{1}{p_2}} |G_2|^{\frac{1}{q_2}}\sum_{\substack{n_2 \in \mathbb{Z} \\ m_2 \in \mathbb{Z}}}\bigg|\bigcup_{\substack{I \times J \in \mathcal{I}_{-n-n_2,-m-m_2}^{\tau_1}  \times \mathcal{J}_{n_2,m_2}^{\tau_2} \\ I \times J \in \mathcal{R}_{k_1,k_2}^{\tau_1,\tau_2}}}I \times J \bigg| .
$$
where 
\begin{align} \label{I_measure}
& \sum_{\substack{n_2 \in \mathbb{Z}\\ m_2 \in \mathbb{Z}}}\bigg|\bigcup_{\substack{I \times J \in \mathcal{I}_{-n-n_2,-m-m_2}^{\tau_1}  \times \mathcal{J}_{n_2,m_2}^{\tau_2} \\ I \times J \in \mathcal{R}_{k_1,k_2}^{\tau_1,\tau_2}}}I \times J \bigg|  \lesssim \min(2^{-k_1s},2^{-k_2\gamma}),
\end{align}
for any $\gamma >1$. The estimate is a direct application of the sparsity condition described in Proposition \ref{sp_2d} that has been extensively used before (see (\ref{rec_area_hybrid})). One can now apply (\ref{I_measure}) to conclude that
\begin{align*}
 |\Lambda_I^{\tau_1,\tau_2}| = & \sum_{(l,n,m,k_1,k_2) \in \mathcal{W}} 2^{k_1}\|h\|_{L^s(\mathbb{R}^2)} 2^{k_2} \cdot I \nonumber \\
\lesssim & (C_1^{\tau_1}C_2^{\tau_2}C_3^{\tau_1,\tau_2})^{4}  |F_1|^{\frac{1}{p_1}} |F_2|^{\frac{1}{q_1}} |G_1|^{\frac{1}{p_2}} |G_2|^{\frac{1}{q_2}}\|h\|_{L^s(\mathbb{R}^2)} \sum_{(l,n,m,k_1,k_2) \in \mathcal{W}} 2^{k_1(1-\frac{s}{2})}2^{k_2(1-\frac{\gamma}{2})} 2^{-n\frac{1}{p_2}}2^{-m\frac{1}{q_1}},
\end{align*}
and achieve the desired bound with the appropriate choice of $\gamma>1$.
 
\vskip .15in
\noindent
\textbf{Estimate of $\Lambda_{II}^{\tau_1,\tau_2}$.} One first observes that the estimates for $\Lambda_{II}^{\tau_1,\tau_2}$ apply to $\Lambda_{III}^{\tau_1,\tau_2}$ due to symmetry. One shall notice that
\begin{align} \label{II}
II \leq & \sum_{\substack{n_2 \in \mathbb{Z} \\ m_2 \in \mathbb{Z}}}\bigg( \sum_{l_2 \in \mathcal{EXP}_2} 2^{l_2} \|\tilde{B}(g_1,g_2)\|_1\bigg) \bigg(\sup_{l_2 \in \mathcal{EXP}_2}2^{-l-l_2}\|B(f_1,f_2)\|_1\sum_{\substack{T \in \mathbb{T}_{-l-l_2}^{\tau_1} \\ I_T \in \mathcal{I}^{\tau_1}_{-n-n_2,-m-m_2}}}|I_T|\bigg)\bigg(\sup_{l_2}\sum_{\substack{S \in \mathbb{S}_{l_2}^{\tau_2} \\ J_S \in \mathcal{J}_{n_2,m_2}^{\tau_2}}}|J_S|\bigg).
\end{align}
\begin{enumerate}[(i)]
\item
The first expression is a geometric series which can be bounded by 
\begin{equation} \label{II_i}
(C_2^{\tau_2}2^{n_2}|G_1|)^{\alpha_2} (C_2^{\tau_2}2^{m_2}|G_2|)^{\beta_2},
\end{equation}
for any $0 \leq \alpha_2, \beta_2 \leq 1$ (up to some constant as discussed in the estimate of $I$).
\item
The second term in (\ref{II}) can be considered as a localized $L^{1,\infty}$ energy. In addition, given by the restriction that $l_2 \in \mathcal{EXP}_2$, one can apply the localization and the corresponding energy estimates described in Proposition \ref{localized_energy_fourier_x}. In particular, for any $0 < \theta_1, \theta_2 < 1$ with $\theta_1 + \theta_2 = 1$,
\begin{align} \label{II_ii}
& \sup_{l_2 \in \mathcal{EXP}_2}2^{-l-l_2}\|B(f_1,f_2)\|_1\sum_{\substack{T \in \mathbb{T}_{-l-l_2}^{\tau_1} \\ I_T \in \mathcal{I}_{-n-n_2,-m-m_2}}}|I_T|\lesssim & (C_1^{\tau_1}2^{-n-n_2})^{\frac{1}{p_1}-\theta_1}(C_1^{\tau_1} 2^{-m-m_2})^{\frac{1}{q_1}- \theta_2} |F_1|^{\frac{1}{p_1}}|F_2|^{\frac{1}{q_1}}.
\end{align}
\item
For any fixed $n_2,m_2,l_2$ and $\tau_2$, $\{J_S: J_S \in \mathcal{J}_{n_2,m_2}^{\tau_2} \ \ \text{and } \ \ S \in \mathbb{S}_{l_2}^{\tau_2}\}$ is a disjoint collection of dyadic intervals, which implies that
\begin{align} \label{II_iii}
\sup_{l_2}\sum_{\substack{S \in \mathbb{S}^{l_2} \\ J_S \in \mathcal{J}_{n_2,m_2}}}|J_S| & \leq \big| \bigcup_{\substack{J_S \in \mathcal{J}_{n_2,m_2}}}J_S\big| \lesssim  \big|\{ Mg_1 > C_2^{\tau_2} 2^{n_2-10}|G_1| \} \cap \{Mg_2 > C_2^{\tau_2} 2^{m_2-10}|G_2| \}\big|,
\end{align}
where the last inequality follows from the pointwise estimates indicated in Claim \ref{ptwise}.
\end{enumerate}
By combining (\ref{II_i}), (\ref{II_ii}) and (\ref{II_iii}),  one can majorize (\ref{II}) by
\begin{align}\label{II_final}
II \lesssim & \sum_{\substack{n_2 \in \mathbb{Z} \\ m_2 \in \mathbb{Z}}}(C_2^{\tau_2} 2^{n_2}|G_1|)^{\alpha_2} (C_2^{\tau_2} 2^{m_2}|G_2|)^{\beta_2}(C_1^{\tau_1}2^{-n-n_2})^{\frac{1}{p_1}- \theta_1}(C_1^{\tau_1} 2^{-m-m_2})^{\frac{1}{q_1} - \theta_2} |F_1|^{\frac{1}{p_1}}|F_2|^{\frac{1}{q_1}} \nonumber \\
& \quad \quad \cdot \big|\{ Mg_1 > C_2^{\tau_2} 2^{n_2-10}|G_1| \} \cap \{Mg_2 > C_2^{\tau_2} 2^{m_2-10}|G_2| \}\big| \nonumber \\
& \leq \sup_{\substack{n_2 \in \mathbb{Z} \\ m_2 \in \mathbb{Z}}} (C_1^{\tau_1}2^{-n-n_2})^{\frac{1}{p_1} - \theta_1}(C_1^{\tau_1} 2^{-m-m_2})^{\frac{1}{q_1} - \theta_2} |F_1|^{\frac{1}{p_1}}|F_2|^{\frac{1}{q_1}}\nonumber \\
& \quad \quad \quad \cdot (C_2^{\tau_2} 2^{n_2})^{\alpha_2 - (1+\epsilon)(1-\mu)}(C_2^{\tau_2 }2^{m_2})^{\beta_2-(1+\epsilon)\mu}|G_1|^{\alpha_2- (1+\epsilon)(1-\mu)} |G_2|^{\beta_2-(1+\epsilon)\mu} \cdot \nonumber \\
&\quad  \sum_{\substack{n_2 \in \mathbb{Z} \\ m_2 \in \mathbb{Z}}}(C_2^{\tau_2}2^{n_2}|G_1|)^{(1+\epsilon)(1-\mu)}(C_2^{\tau_2}2^{m_2}|G_2|)^{(1+\epsilon)\mu}\big|\{ Mg_1 > C_2^{\tau_2} 2^{n_2-10}|G_1| \} \cap \{Mg_2 > C_2^{\tau_2} 2^{m_2-10}|G_2| \}\big|.
\end{align}
By the H\"older-type argument introduced in Section \ref{section_thm_haar}, namely (\ref{B_2_final}), one can estimate the expression 
\begin{align} \label{II_fub}
& \sum_{\substack{n_2 \in \mathbb{Z} \\ m_2 \in \mathbb{Z}}}(C_2^{\tau_2}2^{n_2}|G_1|)^{(1+\epsilon)(1-\mu)}(C_2^{\tau_2}2^{m_2}|G_2|)^{(1+\epsilon)\mu}\big|\{ Mg_1 > C_2^{\tau_2} 2^{n_2-10}|G_1| \} \cap \{Mg_2 > C_2^{\tau_2} 2^{m_2-10}|G_2| \}\big| \nonumber \\
\lesssim & |G_1|^{1-\mu} |G_2|^{\mu}.
\end{align}
Therefore, by plugging in (\ref{II_fub}) and some simplifications, (\ref{II_final}) can be bounded by
\begin{align*}
%& II \nonumber \\
%\lesssim 
& (C_1^{\tau_1} C_2^{\tau_2})^2\sup_{\substack{n_2 \in \mathbb{Z} \\ m_2 \in \mathbb{Z}}} (2^{-n-n_2})^{\frac{1}{p_1} - \theta_1}( 2^{-m-m_2})^{\frac{1}{q_1} - \theta_2} |F_1|^{\frac{1}{p_1}}|F_2|^{\frac{1}{q_1}} (2^{n_2})^{\alpha_2 - (1+\epsilon)(1-\mu)}(2^{m_2})^{\beta_2-(1+\epsilon)\mu}|G_1|^{\alpha_2-\epsilon(1-\mu)} |G_2|^{\beta_2-\epsilon\mu}.
%& \quad \quad \quad  \quad \quad  \quad \ \ \cdot (2^{n_2})^{\alpha_2 - (1+\epsilon)(1-\mu)}(2^{m_2})^{\beta_2-(1+\epsilon)\mu}|G_1|^{\alpha_2-\epsilon(1-\mu)} |G_2|^{\beta_2-\epsilon\mu}.
\end{align*}
One would like to choose $0 \leq  \alpha_2, \beta_2 \leq 1, 0 < \mu < 1$ and $\epsilon>0$ such that
\begin{align} \label{exp_cond_fourier}
& \alpha_2-\epsilon(1-\mu) =  \frac{1}{p_2}, \nonumber \\
& \beta_2 - \epsilon\mu = \frac{1}{q_2}. 
\end{align}
Meanwhile, one can also achieve the equalities
\begin{align*}
& \frac{1}{p_1} - \theta_1 =  \alpha_2-(1+\epsilon)(1-\mu), \nonumber \\
& \frac{1}{q_1} - (1-\theta_1) =  \beta_2-(1+\epsilon)\mu, 
\end{align*}
which combined with (\ref{exp_cond_fourier}), yield
\begin{align*}
& \frac{1}{p_1} - \theta_1 =  \frac{1}{p_2} - (1-\mu), \nonumber \\
& \frac{1}{q_1} - (1-\theta_1) = \frac{1}{q_2} -\mu.
\end{align*}
Thanks to the condition that 
$$
\frac{1}{p_1} + \frac{1}{q_1} = \frac{1}{p_2} + \frac{1}{q_2},
$$
one only needs to choose $0 < \theta_1, \mu < 1$ such that
$$
\frac{1}{p_1} - \frac{1}{p_2} = \theta_1- (1-\mu).
$$
%As one may notice, the above equations are possible not only because of the extra degree of freedom for the choice of parameter, but also because of the condition that
%$$
%\frac{1}{p_1} + \frac{1}{q_1} = \frac{1}{p_2} + \frac{1}{q_2}.
%$$
To sum up, one has the following estimate for II:
\begin{equation} \label{II_ns}
II \lesssim (C_1^{\tau_1} C_2^{\tau_2})^2 2^{-n(\frac{1}{p_1}- \theta_1)}2^{-m(\frac{1}{q_1}- (1-\theta_1))}|F_1|^{\frac{1}{p_1}} |F_2|^{\frac{1}{q_1}}|G_1|^{\frac{1}{p_2}} |G_2|^{\frac{1}{q_2}}.
\end{equation}
Last but not least, one can interpolate between the estimates (\ref{II_ns}) and (\ref{ns_fourier_sp_final}) obtained from the sparsity condition to conclude that
\begin{align} \label{ns_fourier_fb_final}
|\Lambda_{II}^{\tau_1,\tau_2}| = &  \sum_{(l,n,m,k_1,k_2) \in \mathcal{W}} 2^{k_1}\|h\|_{L^s(\mathbb{R}^2)} 2^{k_2} \cdot II \nonumber \\
\lesssim & (C_1^{\tau_1} C_2^{\tau_2}C_3^{\tau_1,\tau_2})^4  \sum_{(l,n,m,k_1,k_2) \in \mathcal{W}} 2^{k_1}\|h\|_{L^s(\mathbb{R}^2)} 2^{k_2(1-\frac{\lambda\gamma}{2})} 2^{-l\lambda(1- \frac{(1+\delta)}{2})}2^{-n(1-\lambda)(\frac{1}{p_1}- \theta_1)}2^{-m(1-\lambda)(\frac{1}{q_1}- (1-\theta_1))} \nonumber \\
& \cdot |F_1|^{(1-\lambda)\frac{1}{p_1}+\lambda\frac{\mu_1(1+\delta)}{2}+\lambda\rho(1-\frac{1+\delta}{2})} |F_2|^{(1-\lambda)\frac{1}{q_1}+\lambda\frac{\mu_2(1+\delta)}{2}+\lambda(1-\rho)(1-\frac{1+\delta}{2})} \nonumber \\
&\cdot |G_1|^{(1-\lambda)\frac{1}{p_2}+\lambda\frac{\nu_1(1+\delta)}{2}+\lambda\rho'(1-\frac{1+\delta}{2})} |G_2|^{(1-\lambda)\frac{1}{q_2}+\lambda\frac{\nu_2(1+\delta)}{2}+\lambda(1-\rho')(1-\frac{1+\delta}{2})}. 
\end{align}
%fourier-ii
One has enough degree of freedom to choose the indices and obtain the desired estimate:
\begin{enumerate}[(i)]
\item
for any $0 < \lambda,\delta < 1$, the series $\displaystyle \sum_{l>0}2^{-l\lambda(1- \frac{(1+\delta)}{2})}$ is convergent;
\item 
one notices that for $0 < \theta_1 < 1$, $\displaystyle \sum_{n>0}2^{-n(1-\lambda)(\frac{1}{p_1}- \theta_1)}$ and $\displaystyle \sum_{m>0}2^{-m(1-\lambda)(\frac{1}{q_1}- (1-\theta_1))}$ converge if
\begin{align*}
&\frac{1}{p_1} - \theta_1>0, \nonumber \\
&\frac{1}{q_1} - (1-\theta_1)>0, 
\end{align*}
which implies that
$$
\frac{1}{p_1} + \frac{1}{q_1} > 1. 
$$
This would be the condition we impose on the exponents $p_1$ and $q_1$. The proof for the range $\frac{1}{p_1} + \frac{1}{q_1} \leq 1$ follows a simpler argument. 
\item
One can identify (\ref{ns_fourier_fb_final}) with (\ref{exp00}) and choose the indices to match the desired exponents for $|F_1|,|F_2|, |G_1|$ and $|G_2|$ in the exactly same fashion.
\end{enumerate}
\vskip .15 in
\noindent
\textbf{Estimate of $\Lambda_{IV}^{\tau_1,\tau_2}$.}
When $l_2 \in \mathcal{EXP}_4$, one has the localization that the main contribution of
$$
\sum_{|K| \geq |I|}\frac{1}{|K|^{\frac{1}{2}}}\langle f_1, \vphi_K^1 \rangle \langle f_2, \psi_K^2 \rangle \langle \chi_{E'}, \psi_K^3\rangle
$$
comes from
$$
\sum_{K \supseteq I}\frac{1}{|K|^{\frac{1}{2}}}\langle f_1, \vphi_K^1 \rangle \langle f_2, \psi_K^2 \rangle \langle \chi_{E'}, \psi_K^3\rangle
$$
as in the Haar model. As a consequence, it is not difficult to check that the argument in Section \ref{section_thm_haar} applies to the estimate of $IV$, where one employs the local energy estimates stated in Proposition \ref{localized_energy_fourier_x} and \ref{localized_energy_y} instead of Proposition \ref{B_en} and derive that
\begin{align} \label{ns_fourier_fb_iv}
 IV \lesssim  (C_1^{\tau_1} C_2^{\tau_2})^2 2^{-n(\frac{1}{p_1} - \theta_1-\frac{1}{2}\mu(1+\epsilon))}2^{-m(\frac{1}{q_1}- \theta_2-\frac{1}{2}(1-\mu)(1+\epsilon))}  |F_1|^{\frac{1}{p_1}-\frac{\mu}{2}\epsilon}|F_2|^{\frac{1}{q_1}-\frac{1-\mu}{2}\epsilon}|G_1|^{\frac{1}{p_2}-\frac{\mu}{2}\epsilon}|G_2|^{\frac{1}{q_2}-\frac{1-\mu}{2}\epsilon}. 
 \end{align}
 %fourier-fubini
By interpolating between (\ref{ns_fourier_fb_iv}) and (\ref{ns_fourier_sp}) which correspond to the estimates for the nested sum using the Fubini argument and the sparsity condition developed in Section \ref{section_thm_haar}, namely (\ref{ns_fb}) and (\ref{ns_sp}), one achieves the desired bound. 
\begin{remark}
When only one of the families $(\phi^3_K)_{K \in \mathcal{K}}$ and $(\phi^3_L)_{L \in \mathcal{L}}$ is lacunary, a simplified argument is sufficient. Without loss of generality, we assume that  $(\vphi^3_K)_{K \in \mathcal{K}}$ is a non-lacunary family while $(\psi^3_L)_{L \in \mathcal{L}}$ is a lacunary family. One can then split the argument into two parts depending on the range of the exponents $l_2$:
\begin{enumerate}[(i)]
\item
$l_2 \in \{l_2 \in \mathbb{Z}: 2^{l_2}\|\tilde{B}(g_1,g_2)\|_1 \lesssim (C_2^{\tau_2}2^{n_2}|G_1|)^{\alpha_2}(C_2^{\tau_2}2^{m_2}|G_2|)^{\beta_2} \}$;
\item
$l_2 \in \{l_2 \in \mathbb{Z}: 2^{l_2}\|\tilde{B}(g_1,g_2)\|_1 \gg (C_2^{\tau_2}2^{n_2}|G_1|)^{\alpha_2}(C_2^{\tau_2}2^{m_2}|G_2|)^{\beta_2} \}$;
\end{enumerate}
where Case (i) can be treated by the same argument for $II$ and Case (ii) by the reasoning for $IV$.
This completes the proof of Theorem \ref{thm_weak_mod} for $\Pi_{\text{flag}^0 \otimes \text{flag}^0}$ in the general case.
\end{remark}

As commented in the beginning of this section, the argument for Theorem \ref{thm_weak_mod} and \ref{thm_weak_inf_mod} developed in the Haar model can be generalized to the Fourier setting, which ends the proof of the main theorems. 

\section{Appendix I - Multilinear Interpolations}
This section is devoted to various multilinear interpolations that allow one to reduce Theorem \ref{main_theorem} to  \ref{thm_weak} (and Theorem \ref{main_thm_inf} to \ref{thm_weak_inf} correspondingly). We will start from the statement in Theorem \ref{thm_weak} and implement interpolations step by step to reach Theorem \ref{main_theorem}. Throughout this section, we will consider $T_{ab}$ as a trilinear operator with first two function spaces restricted to tensor product spaces.
\subsection{Interpolation of Multilinear Forms}
One may recall that Theorem \ref{thm_weak} covers all the restricted weak-type estimates except for the case $2 \leq s \leq \infty$. We will apply the interpolation of multilinear forms to fill in the gap. In particular, let $T^*_{ab}$ denote the adjoint operator of $T_{ab}$ such that
$$
 \langle T_{ab}(f_1 \otimes g_1, f_2 \otimes g_2, h),l\rangle = \langle T^*_{ab}(f_1 \otimes g_1, f_2 \otimes g_2, l), h \rangle
$$
Due to the symmetry between $T_{ab}$ and $T^*_{ab}$, one concludes that the multilinear form associated to $T^{*}_{ab}$ satisfies 
$$
|\Lambda(f_1 \otimes g_1, f_2 \otimes g_2, h, l)| \lesssim |F_1|^{\frac{1}{p_1}} |G_1|^{\frac{1}{p_2}} |F_2|^{\frac{1}{q_1}} |G_2|^{\frac{1}{q_2}} |H|^{\frac{1}{r'}} |L|^{\frac{1}{s}}
$$ 
for every measurable set $F_1, F_2 \subseteq \mathbb{R}_x$, $G_1, G_2 \subseteq \mathbb{R}_y$, $H, L \subseteq \mathbb{R}^2$ of positive and finite measure and every measurable function $|f_i| \leq \chi_{F_i}$, $|g_j| \leq \chi_{G_j}$, $|h| \leq \chi_{H}$ and $|l| \leq \chi_{L}$ for $i, j = 1, 2$. The notation and the range of exponents agree with the ones in Theorem \ref{thm_weak}.
One can now apply the interpolation of multilinear forms described in Lemma 9.6 of \cite{cw} to attain the restricted weak-type estimate with $1 < s \leq \infty $:
\begin{equation} \label{s=inf}
|\Lambda(f_1 \otimes g_1, f_2 \otimes g_2, h, l)| \lesssim |F_1|^{\frac{1}{p_1}} |G_1|^{\frac{1}{p_2}} |F_2|^{\frac{1}{q_1}} |G_2|^{\frac{1}{q_2}}|H|^{\frac{1}{s}} |L|^{\frac{1}{r'}}
\end{equation}
where $\frac{1}{s} = 0$ if $s= \infty$. For $1 \leq s < \infty$, one can fix $f_1, g_1, f_2, g_2 $ and apply linear Marcinkiewiecz interpolation theorem to prove the strong-type estimates for $h \in L^s(\mathbb{R}^2)$ with $1 < s < \infty$. 
The next step would be to validate the same result for $h \in L^{\infty}$. One first rewrites the multilinear form associated to $T_{ab}(f_1 \otimes g_1, f_2 \otimes g_2, h)$ as
\begin{align}\label{linear_form_interp}
\Lambda(f_1 \otimes g_1, f_2 \otimes g_2, h, \chi_{E'}) := & \langle T_{ab}(f_1 \otimes g_1, f_2 \otimes g_2, h), \chi_{E'}\rangle \nonumber \\
= & \langle T^*_{ab}(f_1 \otimes g_1, f_2 \otimes g_2, \chi_{E'}), h\rangle. 
\end{align} 
Let $Q_N := [ -N,N]^2$ denote the cube of length $2N$ centered at the origin in $\mathbb{R}^2$, then (\ref{linear_form_interp}) can be expressed as
\begin{align*}
& \displaystyle \lim_{N \rightarrow \infty} \int_{Q_N} T^*_{ab}(f_1 \otimes g_1, f_2 \otimes g_2, \chi_{E'})(x) h(x) dx \nonumber \\
= &  \lim_{N \rightarrow \infty}\int T^*_{ab}(f_1 \otimes g_1, f_2 \otimes g_2, \chi_{E'})(x) (h\cdot\chi_{Q_N})(x) dx \nonumber \\
= & \lim_{N \rightarrow \infty}\int T_{ab}(f_1 \otimes g_1, f_2 \otimes g_2, h\cdot\chi_{Q_N})(x) \chi_{E'}(x) dx \nonumber \\
= & \lim_{N \rightarrow \infty}\Lambda(f_1 \otimes g_1, f_2 \otimes g_2, h\cdot\chi_{Q_N}, \chi_{E'}).
\end{align*}
Let $\tilde{h}:= \frac{h \chi_{Q_N}}{\|h\|_{\infty}}$, where $|\tilde{h}| \leq \chi_{Q_N}$ with $|Q_N| \leq N^2$. One can thus invoke (\ref{s=inf}) to conclude that
\begin{align*}
|\Lambda(f_1 \otimes g_1, f_2 \otimes g_2, h\chi_{Q_N}, \chi_{E'})| =& \|h\|_{\infty} \cdot |\Lambda(f_1 \otimes g_1, f_2 \otimes g_2, \tilde{h}, \chi_{E'})| \nonumber \\
\lesssim & |F_1|^{\frac{1}{p_1}}|G_1|^{\frac{1}{p_2}}|F_2|^{\frac{1}{q_1}}|G_2|^{\frac{1}{q_2}}\|h\|_{\infty}|E|^{\frac{1}{r'}}.
\end{align*}
As the bound for the multilinear form is independent of $N$, passing to the limit when $N \rightarrow \infty$ yields that
$$
|\Lambda(f_1 \otimes g_1, f_2 \otimes g_2, h, \chi_{E'})| \lesssim |F_1|^{\frac{1}{p_1}}|G_1|^{\frac{1}{p_2}}|F_2|^{\frac{1}{q_1}}|G_2|^{\frac{1}{q_2}}\|h\|_{\infty}|E|^{\frac{1}{r'}}.
$$
Combined with the statement in Theorem \ref{thm_weak}, one has that for any $1 < p_1,p_2, q_1,q_2 < \infty$, $1<s \leq \infty$, $0 < r < \infty$, $\frac{1}{p_1} + \frac{1}{q_1} = \frac{1}{p_2} + \frac{1}{q_2} = \frac{1}{r} - \frac{1}{s}$,
\begin{equation} \label{restricted_weak}
\|T_{ab}(f_1 \otimes g_1, f_2 \otimes g_2, h)\|_{r,\infty} \lesssim |F_1|^{\frac{1}{p_1}}|G_1|^{\frac{1}{p_2}}|F_2|^{\frac{1}{q_1}}|G_2|^{\frac{1}{q_2}}\|h\|_{L^s(\mathbb{R}^2)}
\end{equation}
for every measurable set $F_1, F_2 \subseteq \mathbb{R}_x$, $G_1, G_2 \subseteq \mathbb{R}_y$ of positive and finite measure and every measurable function $|f_i| \leq \chi_{F_i}$, $|g_j| \leq \chi_{G_j}$ for $i, j = 1, 2$.
\subsection{Tensor-type Marcinkiewiecz Interpolation}
The next and final step would be to attain strong-type estimates for $T_{ab}$ from (\ref{restricted_weak}). We first fix $h \in L^{s}$ and define
$$T^{h}(f_1 \otimes g_1, f_2 \otimes g_2) := T_{ab}(f_1 \otimes g_1, f_2 \otimes g_2,h)$$
One can then apply the following tensor-type Marcinkiewiecz interpolation theorem to each $T^h$ so that Theorem \ref{main_theorem} follows.
\begin{theorem}\label{tensor_interpolation}
Let $1 < p_1,p_2, q_1, q_2< \infty$ and $ 0 < t < \infty$ such that $\frac{1}{p_1} + \frac{1}{q_1} = \frac{1}{p_2} + \frac{1}{q_2} = \frac{1}{t}$. Suppose a multilinear tensor-type operator $T(f_1 \otimes g_1, f_2 \otimes g_2)$ satisfies the restricted weak-type estimates for any $\tilde{p_1}, \tilde{p_2}, \tilde{q_1}, \tilde{q_2}$ in a neighborhood of $p_1, p_2, q_1, q_2$ respectively with $
\frac{1}{\tilde{p_1}} + \frac{1}{\tilde{q_1}} = \frac{1}{\tilde{p_2}} + \frac{1}{\tilde{q_2}} = \frac{1}{\tilde{t}}$, equivalently
$$
\|T(f_1 \otimes g_1, f_2 \otimes g_2) \|_{\tilde{t},\infty} \lesssim |F_1|^{\frac{1}{\tilde{p_1}}} |G_1|^{\frac{1}{\tilde{p_2}}} |F_2|^{\frac{1}{\tilde{q_1}}} |G_2|^{\frac{1}{\tilde{q_2}}}
$$
for every measurable set $F_1  \subseteq \mathbb{R}_{x} , F_2 \subseteq  \mathbb{R}_{x},  G_1\subseteq \mathbb{R}_y, G_2\subseteq \mathbb{R}_y$ of positive and finite measure and every measurable function $|f_1(x)| \leq \chi_{F_1}(x)$, $|f_2(x)| \leq \chi_{F_2}(x)$, $|g_1(y)| \leq \chi_{G_1}(y)$, $|g_2(y)| \leq \chi_{G_2}(y)$.
Then $T$ satisfies the strong-type estimate
$$
\|T(f_1 \otimes g_1, f_2 \otimes g_2) \|_{t} \lesssim \|f_1\|_{p_1} \|g_1\|_{p_2} \|f_2\|_{q_1} \|g_2\|_{q_2}
$$
for any $f_1 \in L^{p_1}(\mathbb{R}_x)$, $f_2 \in L^{q_1}(\mathbb{R}_x)$, $g_1 \in L^{p_2}(\mathbb{R}_y)$ and $g_2 \in L^{q_2}(\mathbb{R}_y)$.
\end{theorem}
\begin{remark}
The proof of the theorem resembles the argument for the multilinear Marcinkiewiecz interpolation(see \cite{bm}) with small modifications. 
\end{remark}

\section{Appendix II - Reduction to Model Operators}
\subsection{Littlewood-Paley Decomposition}
\subsubsection{Set up}
Let $\vphi \in \mathcal{S}(\mathbb{R})$ be a Schwartz function with $\text{supp} (\f{\vphi}) \subseteq [-2,2]$ and $\f{\vphi}(\xi) = 1$ on $[-1,1]$. Let 
$$ 
\f{\psi}(\xi) = \f{\vphi}(\xi) - \f{\vphi}(2\xi)
$$
so that $\text{supp} \f{\psi} \subseteq [-2,-\frac{1}{2}] \cup [-\frac{1}{2}, 2]$. Now for every $k \in \mathbb{Z}$, define 
$$
\f{\psi}_{k}: = \f{\psi}(2^{-k}\xi)
$$
One important observation is that
$$
\sum_{k \in \mathbb{Z}} \f{\psi}_k(\xi) = 1
$$
Let 
$$
\f{\vphi}_{k}(\xi) := \sum_{k' \leq k - 10} \f{\psi}_{k'}(\xi)
$$
\begin{notation}
We say that the family $({\psi}_k)_k$ is \textit{$\psi$ type} and the family $({\vphi}_k)_k$ is \textit{$\vphi$ type}.
\end{notation}

\subsubsection{Special Symbols}
We will first focus on a special case of the symbols and the general case will be studied as an extension afterwards.
Suppose that
$$
a(\xi_1,\eta_1,\xi_2,\eta_2) = a_1(\xi_1,\xi_2)a_2(\eta_1,\eta_2)
$$
$$
b(\xi_1,\eta_1,\xi_2,\eta_2,\xi_3,\eta_3) = b_1(\xi_1,\xi_2,\xi_3)  b_2(\eta_1,\eta_2,\eta_3)
$$
where
\begin{align}
a_1(\xi_1,\xi_2) = & \sum_{k_1} \f{\phi}^{1}_{k_1}(\xi_1) \f{\phi}^{2}_{k_1}(\xi_2) \label{special_symb_a_1} \\
b_1(\xi_1,\xi_2,\xi_3) = & \sum_{k_2} \f{\phi}^{1}_{k_2}(\xi_1) \f{\phi}^{2}_{k_2}(\xi_2) \f{\phi}^{3}_{k_2}(\xi_3) \label{special_symb_b_1} 
\end{align}
At least one of the families $({\phi}^1_{k_1})_{k_1}$ and $({\phi}^2_{k_1})_{k_1}$ is $\psi$ type and at least one of the families $({\phi}^1_{k_2})_{k_2}$, $({\phi}^2_{k_2})_{k_2}$ and $({\phi}^3_{k_2})_{k_2}$  is $\psi$ type. 
Similarly,
$$
a_2(\eta_1,\eta_2) =  \sum_{j_1} \f{{\phi}}^1_{j_1}(\eta_1) \f{{\phi}}^2_{j_1}(\eta_2) 
$$
$$
b_2(\eta_1,\eta_2,\eta_3) =  \sum_{j_2} \f{{\phi}}^1_{j_2}(\eta_1) \f{{\phi}}^2_{j_2}(\eta_2) \f{{\phi}}^3_{j_2}(\eta_3)
$$
where at least one of the families $({\phi}^1_{j_1})_{j_1}$ and $({\phi}^2_{j_1})_{j_1}$ is $\psi$ type and at least one of the families $({\phi}^1_{j_2})_{j_2}$, $({\phi}^2_{j_2})_{j_2}$ and $({\phi}^3_{j_2})_{j_2}$  is $\psi$ type. 

Then
\begin{align}
a_1(\xi_1,\xi_2) b_1(\xi_1,\xi_2,\xi_3) = & \sum_{k_1,k_2} \f{\phi}^1_{k_1}(\xi_1) \f{\phi}^2_{k_1}(\xi_2)  \f{\phi}^1_{k_2}(\xi_1) \f{\phi}^2_{k_2}(\xi_2) \f{\phi}^3_{k_2}(\xi_3) \nonumber \\
= & \underbrace{\sum_{k_1 \approx k_2}}_{I^1} + \underbrace{\sum_{k_1 \ll k_2}}_{II^1} + \underbrace{\sum_{k_1 \gg k_2}}_{III^1},
\end{align}
where $k_1 \approx k_2$ means $k_2-100 \leq k_1 \leq k_2 + 100$ and  $k_1 \ll k_2$ means $k_1 < k_2 -100$.

Case $I^1$ gives rise to the symbol of a paraproduct. More precisely,
\begin{equation}\label{I^1_step1}
I^1 = \sum_{k} \f{\tilde{\phi}^1_k}(\xi_1) \f{\tilde{\phi}^2_k}(\xi_2) \f{\phi}^3_k(\xi_3), 
\end{equation} 
where $\f{\tilde{\phi}^1_k}(\xi_1) := \f{\phi}^1_{k_1}(\xi_1) \f{\phi}^1_{k_2}(\xi_1)$ and  $\f{\tilde{\phi}^2_k}(\xi_2) := \f{\phi}^2_{k_1}(\xi_2) \f{\phi}^2_{k_2}(\xi_2)$ when $k := k_1 \approx k_2$. (\ref{I^1_step1}) can be completed as
\begin{equation} \label{I^1_completed}
I^1 = \sum_{k} \f{\tilde{\phi}^1_k}(\xi_1) \f{\tilde{\phi}^2_k}(\xi_2) \f{\phi}^3_{k}(\xi_3) \f{\tilde{\phi}^4_{k}}(\xi_1 + \xi_2 + \xi_3)
\end{equation}
and at least two of the families $ ({\tilde{\phi}}^1_{k})_{k}$, $({\tilde{\phi}}^2_{k})_{k}$, $ ({\phi}^3_{k})_{k}$, $(\tilde{{\phi}}^4_{k})_{k}$ are $\psi$ type.

Case $II^1$ and $III^1$ can be treated similarly. In Case $II^1$, the sum is non-degenerate when $(\phi_{k_2}^1)_{k_2}$ and $(\phi_{k_2}^2)_{k_2}$ are $\vphi$ type. In particular, one has
\begin{equation} \label{II^1_step1}
II ^1 = \sum_{k_1 \ll k_2} \f{\phi}^1_{k_1}(\xi_1) \f{\phi}^2_{k_1}(\xi_2)  \f{\vphi}^1_{k_2}(\xi_1) \f{\vphi}^2_{k_2}(\xi_2) \f{\psi}^3_{k_2}(\xi_3),
\end{equation}
where at least one of the families $(\phi^1_{k_1})_{k_1}$ and  $(\phi^2_{k_1})_{k_1}$ is $\psi$ type.
In the case when the symbols are assumed to take the special form (namely (\ref{special_symb_a_1}) and (\ref{special_symb_b_1})), (\ref{II^1_step1}) can be rewritten as
\begin{equation} \label{simp_special_symbol}
\sum_{k_1 \ll k_2} \f{\phi}^1_{k_1}(\xi_1) \f{\phi}^2_{k_1}(\xi_2) \f{\psi}^3_{k_2}(\xi_3),
\end{equation}
which can be ``completed" as 
\begin{equation} \label{completion}
II^1 = \sum_{k_1 \ll k_2} \f{\phi}^1_{k_1}(\xi_1) \f{\phi}^2_{k_1}(\xi_2)\f{\tilde{{\phi}}^3_{k_1}}(\xi_1+\xi_2) \f{\tilde{\vphi}^1_{k_2}}(\xi_1+\xi_2) \f{\tilde{\psi}^2_{k_2}}(\xi_3)\f{\tilde{\psi}^3_{k_2}}(\xi_1+\xi_2+\xi_3),
\end{equation}
where ${\tilde{\psi}^2_{k_2}}(\xi_3):= {\psi}^3_{k_2}(\xi_3)$ and at least two of the families $(\phi^1_{k_1})_{k_1}$, $(\phi^2_{k_1})_{k_1}$ and $(\tilde{\phi}^3_{k_1})_{k_1}$ are $\psi$ type.

The exact same argument can be applied to $a_2(\eta_1,\eta_2)b_2(\eta_1,\eta_2,\eta_3)$ so that the symbol can be decomposed as 
$$
\underbrace{\sum_{j_1 \approx j_2}}_{I^2} + \underbrace{\sum_{j_1 \ll j_2}}_{II^2} + \underbrace{\sum_{j_1 \gg j_2}}_{III^2}
$$
where
$$
I^2 =  \sum_{j} \f{\tilde{\phi}^1_j}(\eta_1) \f{\tilde{\phi}^2_j}(\xi_2) \f{\phi}^3_{j}(\eta_3) \f{\tilde{\phi}^4_{j}}(\eta_1+\eta_2+\eta_3)
$$
and at least two of the families $({\tilde{\phi}}^1_{j})_{j}$, $( {\tilde{\phi}}^2_{j})_{j}$, $({\phi}^3_{j})_{j}$ and $(\tilde{\phi}^4_{j})_j$ are $\psi$ type. Case $II^2$ and $III^2$ have similar expressions, where
$$
II ^2 = \sum_{j_1 \ll j_2} \f{\phi}^1_{j_1}(\eta_1) \f{\phi}^2_{j_1}(\eta_2)\f{\tilde{\phi}^3_{j_1}}(\eta_1+\eta_2) \f{\tilde{\vphi}^1_{j_2}}(\eta_1+\eta_2) \f{\tilde{\psi}^2_{j_2}}(\eta_3)\f{\tilde{\psi}^3_{j_2}}(\eta_1+\eta_2+\eta_3) 
$$
and at least two of the families $(\phi^1_{j_1})_{j_1}$, $(\phi^2_{j_1})_{j_1}$ and $(\tilde{\phi}^3_{j_1})_{j_1}$ are $
\psi$ type.

One can now combine the decompositions and analysis for $a_1,a_2,b_1$ and $b_2$ to study the original operator:
\begin{align*}
& T_{ab}(f_1 \otimes g_1,f_2 \otimes g_2, h) \nonumber\\
= & T_{ab}^{I^1I^2} + T_{ab}^{I^1 II^2} + T_{ab}^{I^1 III^2} + T_{ab}^{II^1 I^2} + T_{ab}^{II^1 II^2} + T_{ab}^ {II^1 III^2} + T_{ab}^{III^1 I^1} + T_{ab}^{III^2 II^2} + T_{ab}^{III^1 III^2},
\end{align*}
where each operator in the sum is defined to be the operator taking the form of $T_{ab}$ (\ref{bi_flag}) with the replacement of the symbol $a\cdot b$ by the symbol specified in the superscript. Because of the symmetry between frequency variables $(\xi_1,\xi_2,\xi_3)$ and $(\eta_1,\eta_2,\eta_3)$ and the symmetry between cases for frequency scales $k_1 \ll k_2$ and  $k_1 \gg k_2$, $j_1 \ll j_2$ and $j_1 \gg j_2$, it suffices to consider the following operators and others can be proved using the same argument.

\begin{enumerate}
\item
$T_{ab}^{I^1 I^2}$ is a bi-parameter paraproduct; 
\vskip 0.15in
\item
$ T_{ab}^{II^1 I^2}$ defined by
\begin{align*}
& T_{ab}^{II^1 I^2} \nonumber\\
:= & \displaystyle \sum_{\substack{k_1 \ll k_2 \\ j \in \mathbb{Z}}} \int \f{\phi}^1_{k_1}(\xi_1) \f{\phi}^2_{k_1}(\xi_2)\f{\tilde{\phi}^3_{k_1}}(\xi_1+\xi_2) \f{\tilde{\vphi}^1_{k_2}}(\xi_1+\xi_2) \f{\tilde{\psi}^2_{k_2}}(\xi_3)\f{\tilde{\psi}^3_{k_2}}(\xi_1+\xi_2+\xi_3)  \f{\tilde{\phi}^1_j}(\eta_1) \f{\tilde{\phi}^2_{j}}(\eta_2) \f{\phi}^3_{j}(\eta_3) \f{\tilde{\phi}^4_{j}}(\eta_1+\eta_2+\eta_3) \nonumber \\
& \quad \quad \quad \ \ \f{f_1}(\xi_1) \f{f_2}(\xi_2) \f{g_1}(\eta_1) \f{g_2}(\eta_2) \f{h}(\xi_3,\eta_3) e^{2\pi i x(\xi_1+\xi_2+\xi_3)} e^{2\pi i y(\eta_1+\eta_2+\eta_3)}d\xi_1 d\xi_2 d\xi_3  d\eta_1 d\eta_2 d\eta_3 \nonumber \\
= & \sum_{\substack{k_1 \ll k_2 \\ j \in \mathbb{Z}}}\left(\bigg(\big(( f_1 * \phi^1_{k_1}) (f_2 * \phi^2_{k_1}) * \tilde{\phi}^3_{k_1}\big) * \tilde{\vphi}^1_{k_2} \bigg) ( g_1 * \tilde{\phi}^1_{j}) (g_2 * \tilde{\phi}^2_{j})  (h * \tilde{\psi}_{k_2}^2\otimes \phi^3_{j})\right) * \tilde{\psi}^3_{k_2}\otimes \tilde{\phi}^4_{j}, \nonumber 
\end{align*}
where at least two of the families $(\phi^1_{k_1})_{k_1}$, $(\phi^2_{k_1})_{k_1}$ and $(\tilde{\phi}^3_{k_1})_{k_1}$ are $\psi$ type and at least two of the families $(\tilde{\phi}^1_{j})_{j}$, $(\tilde{\phi}^2_{j})_{j}$, $({\phi}^3_{j})_{j}$ and $(\tilde{\phi}^4_{j})_{j}$ are $\psi$ type.
\vskip .15 in
\item
$T_{ab}^{II^1 II^2}$ defined by
\begin{align*}
& T_{ab}^{II^1 II^2} \nonumber\\
:= & \displaystyle \sum_{\substack{k_1 \ll k_2 \\ j_1 \ll j_2}} \int \f{\phi}^1_{k_1}(\xi_1) \f{\phi}_{k_1}^2(\xi_2)\f{\tilde{\phi}^3_{k_1}}(\xi+\xi_2) \f{\tilde{\vphi}^1_{k_2}}(\xi_1+\xi_2) \f{\tilde{\psi}^2_{k_2}}(\xi_3)\f{\tilde{\psi}^3_{k_2}}(\xi_1+\xi_2+\xi_3) \cdot \nonumber \\
&  \quad \quad \quad \ \ \f{\phi}^1_{j_1}(\eta_1) \f{\phi}^2_{j_1}(\eta_2)\f{\tilde{\phi}^3_{j_1}}(\eta_1+\eta_2) \f{\tilde{\vphi}^1_{j_2}}(\eta_1+\eta_2) \f{\tilde{\psi}^2_{j_2}}(\eta_3)\f{\tilde{\psi}^3_{j_2}}(\eta_1+\eta_2+\eta_3) \cdot \nonumber \\
& \quad \quad \quad\ \  \f{f_1}(\xi_1) \f{f_2}(\xi_2) \f{g_1}(\eta_1) \f{g_2}(\eta_2) \f{h}(\xi_3,\eta_3) \cdot  e^{2\pi i x(\xi_1+\xi_2+\xi_3)} e^{2\pi i y(\eta_1+\eta_2+\eta_3)}d\xi_1 d\xi_2 d\xi_3  d\eta_1 d\eta_2 d\eta_3 \nonumber \\
= & \sum_{\substack{k_1 \ll k_2 \\ j_1 \ll j_2}} \left(\bigg(\big(( f_1 * \phi^1_{k_1}) (f_2 * \phi^2_{k_1}) * \tilde{\phi}^3_{k_1}\big) * \tilde{\vphi}^1_{k_2} \bigg) \bigg(\big(( g_1 * \phi^1_{j_1}) (g_2 * \phi^2_{j_1}) * \tilde{\phi}^3_{j_1}\big) * \tilde{\vphi}^1_{j_2} \bigg) \cdot (h * \tilde{\psi}^2_{k_2}\otimes \tilde{\psi}^2_{j_2}) \right)* \tilde{\psi}^3_{k_2}\otimes \tilde{\psi}^3_{j_2},\nonumber 
%& \quad \ \ \ \ \cdot (h * \psi_{k_2}\otimes \psi_{j_2}) * \psi_{k_2}\otimes \psi_{j_2},
\end{align*}
where at least two of the families $(\phi^1_{k_1})_{k_1}$, $(\phi^2_{k_1})_{k_1}$ and $(\tilde{\phi}^3_{k_1})_{k_1}$ are $\psi$ type and at least two of the families $(\phi^1_{j_1})_{j_1}$, $(\phi^2_{j_1})_{j_1}$ and $(\tilde{\phi}^3_{j_1})_{j_1}$ are $\psi$ type.
\end{enumerate}

\vskip .15in
%\subsubsection{Extension from Special Symbols to General Symbols}
%Suppose that $a((\xi_1,\eta_1),(\xi_2,\eta_2)) \in \mathcal{M}(\mathbb{R}^3)$ is a general symbol. We can decompose it as
%\begin{align*}
%a((\xi_1,\eta_1),(\xi_2,\eta_2)) = a((\xi_1,\eta_1),(\xi_2,\eta_2)) & \left(\sum_{k_1}\f{\vphi}_{k_1}(\xi_1) \f{\psi}_{k_1}(\xi_2) + \sum_{k_1} \f{\psi}_{k_1} (\xi_1) \f{\vphi}_{k_1}(\xi_2) + \sum_{k_1}\f{\psi}_{k_1}(\xi_1) \f{\psi}_{k_1}(\xi_2)\right) \cdot \nonumber \\
%& \left(\sum_{l_1}\f{\vphi}_{l_1}(\eta_1) \f{\psi}_{l_1}(\eta_2) + \sum_{l_1} \f{\psi}_{l_1} (\eta_1) \f{\vphi}_{l_1}(\eta_2) + \sum_{l_1}\f{\psi}_{l_1}(\eta_1) \f{\psi}_{l_1}(\eta_2)\right)
%\end{align*}
%For each summand, one can summarize it as
%$$
%\sum_{k_1,l_1}a((\xi_1,\eta_1),(\xi_2,\eta_2))\f{\phi}_{k_1}(\xi_1)\f{\phi}_{k_2}(\xi_2)\f{\phi}_{j_1}(\eta_1) \f{\phi}_{j_1}(\eta_2)
%$$
%with at least one of the families $(\f{\phi}_{k_1})_{k_1}$ being lacunary and at least one of the families $(\f{\phi}_{j_1})_{j_1}$ being lacunary.

\subsubsection{General Symbols}
The extension from special symbols to general symbols can be treated as specified in Section 2.13 of \cite{cw}. Due to the fact that generalized bump functions involved for general symbols do not necessarily equal to $1$ on their supports, simple manipulations in the previous section cannot be performed. In particular, for generic symbols 
\begin{align*}
& a_1(\xi_1,\xi_2) = \sum_{k \in \mathbb{Z}}\sum_{n_1,n_2 \in \mathbb{Z}} C^k_{n_1,n_2} \f{\phi}_{k,n_1}^1(\xi_1)\f{\phi}_{k,n_2}^2(\xi_2)  \\
& b(\xi_1, \xi_2,  \xi_3) = \sum_{j \in \mathbb{Z}}\sum_{m_1,m_2,m_3 \in \mathbb{Z}} C^j_{m_1,m_2,m_3} \f{\phi}_{j,m_1}^1(\xi_1)\f{\phi}_{j,m_2}^2(\xi_2) \f{\phi}_{j,m_3}^3(\xi_3) 
\end{align*} 
where
\begin{align}
C^k_{n_1,n_2} := & \frac{1}{2^{2k}}\int_{\mathbb{R}^2} a(\xi_1, \xi_2) \f{\phi}_k^1(\xi_1) \f{\phi}_k^2(\xi_2) e^{-2\pi i n_1 \frac{\xi_1}{2^k}} e^{-2\pi i n_2 \frac{\xi_2}{2^k}} d\xi_1 d\xi_2\nonumber \\
 C^j_{m_1,m_2,m_3} := & \frac{1}{2^{3j}}\int_{\mathbb{R}^3} b(\xi_1, \xi_2,\xi_3) \f{\phi}_j^1(\xi_1) \f{\phi}_j^2(\xi_2) \f{\phi}_j^3(\xi_3) e^{-2\pi i m_1 \frac{\xi_1}{2^j}} e^{-2\pi i m_2 \frac{\xi_2}{2^j}} e^{-2\pi i m_3 \frac{\xi_3}{2^j}} d\xi_1 d\xi_2 d\xi_3 \nonumber \\
 \f{\phi}^{\mathfrak{i}}_{k,n_{\mathfrak{i}}}(\xi_{\mathfrak{i}}) := &e^{2\pi i n_{\mathfrak{i}} \frac{\xi_{\mathfrak{i}}}{2^{k}}} \f{\phi}_k^{\mathfrak{i}}(\xi_{\mathfrak{i}}), \mathfrak{i} = 1,2, \\
\f{\phi}^{\mathfrak{l}}_{j,m_{\mathfrak{l}}}(\xi_{\mathfrak{l}}) := & e^{2\pi i m_{\mathfrak{l}} \frac{\xi_{\mathfrak{l}}}{2^{j}}} \f{\phi}_j^{\mathfrak{l}}(\xi_{\mathfrak{l}}), \mathfrak{l} = 1,2,3. 
%& \f{\phi}^1_{k,n_1}(\xi_1) := e^{2\pi i n_1 \frac{\xi_1}{2^{k}}} \f{\phi}_k^1(\xi_1), \f{\phi}^2_{k,n_2}(\xi_2) := e^{2\pi i n_2 \frac{\xi_2}{2^{k}}} \f{\phi}_k^2(\xi_2) 
\end{align}
We denote $ \f{\phi}^{\mathfrak{i}}_{k,n_{\mathfrak{i}}}$ and $\f{\phi}^{\mathfrak{l}}_{j,m_{\mathfrak{l}}}$ by generalized bump functions. As a consequence, one cannot equate the following two terms: 
\begin{equation*}
 \f{\phi}_{k,n_1}^1(\xi_1)\f{\phi}_{k,n_2}^2(\xi_2)\f{\vphi}_{j,m_1}^1(\xi_1)\f{\vphi}_{j,m_2}^2(\xi_2) \f{\psi}_{j,m_3}^3(\xi_3) \neq  \f{\phi}_{k,n_1}^1(\xi_1)\f{\phi}_{k,n_2}^2(\xi_2)\f{\psi}_{j,m_3}^3(\xi_3).
\end{equation*}
With abuse of notations, we will proceed the discussion as in the previous section with recognition of the presence of generalized bump functions. One notices that $I^1$ generates bi-parameter paraproduct as previously. In Case $II^1$, since $k_1 \ll k_2$, $\f{\vphi}_{k_2}(\xi_1)$ and $\f{\vphi}_{k_2}(\xi_2)$ behave like $\f{\vphi}_{k_2}(\xi_1 + \xi_2)$. One could obtain (\ref{completion}) as a result. To make the argument rigorous, one considers the Taylor expansions
$$
\f{\vphi}_{k_2}^1(\xi_1) = \f{\vphi}^1_{k_2}(\xi_1 + \xi_2) + \sum_{0 < l_1 < M_1} \frac{(\f{\vphi}^1_{k_2})^{(l_1)}(\xi_1+ \xi_2)}{{l_1}!}(-\xi_2)^{l_1} + R^1_{M_1}(\xi_1,\xi_2),
$$
$$
\f{\vphi}^2_{k_2}(\xi_2) = \f{\vphi}^2_{k_2}(\xi_1 + \xi_2) + \sum_{0 < l _2 <M_2} \frac{(\f{\vphi}^2_{k_2})^{(l_2)}(\xi_1+ \xi_2)}{{l_2}!}(-\xi_1)^{l_2} + R^2_{M_2}(\xi_1,\xi_2),
$$
%There are some abuse of notations in the sense that $\f{\vphi}_{k_2}(\xi_1+ \xi_2)$ in both equations do not represent for the same function - they correspond to $\f{\vphi}_{k_2}(\xi_1)$ and $\f{\vphi}_{k_2}(\xi_2)$ respectively, and share the common feature that 
where $({\vphi}_{k_2}^1)_{k_2}$ and $({\vphi}_{k_2}^2)_{k_2}$ are families of $\vphi$ type bump functions. Let 
\begin{equation*}
\f{\tilde{\vphi}^1_{k_2}}(\xi_1+\xi_2) :=  \f{\vphi}^1_{k_2}(\xi_1 + \xi_2) \cdot \f{\vphi}^2_{k_2}(\xi_1 + \xi_2) 
\end{equation*}
and one can rewrite $II^1$ as
\begin{align*}
&\underbrace{\sum_{k_1 \ll k_2} \f{\phi}^1_{k_1}(\xi_1) \f{\phi}_{k_1}^2(\xi_2) \f{\tilde{\vphi}^1_{k_2}}(\xi_1 + \xi_2)\f{\tilde{\psi}^2_{k_2}}(\xi_3)}_{II^1_0} + \nonumber \\
& \underbrace{\sum_{\substack{0 < l_1+l_2 \leq M}}\sum_{k_1 \ll k_2}\f{\phi}^1_{k_1}(\xi_1) \f{\phi}^2_{k_1}(\xi_2) \frac{(\f{\vphi}_{k_2}^1)^{(l_1)}(\xi_1 + \xi_2)}{{l_1}!} \frac{(\f{\vphi}^2_{k_2})^{(l_2)}(\xi_1 + \xi_2)}{{l_2}!} (-\xi_1)^{l_2}(-\xi_2)^{l_1} \f{\tilde{\psi}^2_{k_2}}(\xi_3)}_{II^1_1} + \underbrace{R_M(\xi_1,\xi_2)}_{II^1_{\text{rest}}}, \nonumber 
%\sum_{\substack{l_1 + l_2 > M }}\sum_{k_1 \ll k_2}\f{\phi}_{k_1}^1(\xi_1) \f{\phi}_{k_1}^2(\xi_2) \frac{(\f{\vphi}^1_{k_2})^{(l_1)}(\xi_1 + \xi_2)}{{l_1}!} \frac{(\f{\vphi}^2_{k_2})^{(l_2)}(\xi_1 + \xi_2)}{{l_2}!} (-\xi_1)^{l_2}(-\xi_2)^{l_1} \f{\tilde{\psi}^2_{k_2}}(\xi_3)
\end{align*}
where $M \gg |\alpha_1|$.

One observes that $II^1_0$ can be ``completed" to obtain (\ref{completion}) as desired. 

One can simplify $II^1_1$ as
\begin{align} \label{II^1_generic}
II^1_1 & = \sum_{\substack{0 < l_1 + l_2\leq M}} \sum_{\mu=100}^{\infty} \sum_{k_2 = k_1 + \mu} \f{\phi}^1_{k_1}(\xi_1) \f{\phi}^2_{k_1}(\xi_2) \frac{(\f{\vphi}^1_{k_2})^{(l_1)}(\xi_1 + \xi_2)}{{l_1}!} \frac{(\f{\vphi}^2_{k_2})^{(l_2)}(\xi_1 + \xi_2)}{{l_2}!} (-\xi_1)^{l_2}(-\xi_2)^{l_1} \f{\tilde{\psi}^2_{k_2}}(\xi_3) .
\end{align}
Let
\begin{align*}
\f{\tilde{\phi}^1_{k_1}}(\xi_1) :=& \frac{(-\xi_1)^{l_2}}{2^{k_1l_2}}\f{\phi}^1_{k_1}(\xi_1), \\
\f{\tilde{\phi}^2_{k_1}}(\xi_2) :=& \frac{(-\xi_2)^{l_1}}{2^{k_1l_1}}\f{\phi}^2_{k_1}(\xi_2), \\
\f{\tilde{\vphi}^1_{k_2,l_1,l_2}}(\xi_1 + \xi_2) := & 2^{k_2(l_1+l_2)} \frac{(\f{\vphi}^1_{k_2})^{(l_1)}(\xi_1 + \xi_2)}{{l_1}!} \frac{(\f{\vphi}^2_{k_2})^{(l_2)}(\xi_1 + \xi_2)}{{l_2}!}
\end{align*}
where ${\tilde{\vphi}}^1_{k_2,l_1,l_2} $ represents an $L^{\infty}$-normalized $\vphi$ type bump function with Fourier support at scale $2^{k_2}$.
Then (\ref{II^1_generic}) can be rewritten as
\begin{align*}
%& =  \sum_{\substack{0 < l_1 + l_2 \leq M }} \sum_{\mu=100}^{\infty} \sum_{k_2 = k_1 + \mu} \f{\phi}^1_{k_1}(\xi_1) \f{\phi}^2_{k_1}(\xi_2) 2^{-k_2 l_1}\f{\vphi}_{k_2,l_1}(\xi_1 + \xi_2) 2^{-k_2 l_2}\f{\vphi}_{k_2,l_2}(\xi_1 + \xi_2) (-\xi_1)^{l_2}(-\xi_2)^{l_1} \f{\psi}_{k_2}(\xi_3) \nonumber \\
%& \sim
& \sum_{\substack{0 < l_1+l_2 \leq M }} 2^{(k_1-k_2)(l_1+l_2)} \f{\tilde{\phi}^1_{k_1}}(\xi_1) \f{\tilde{\phi}^2_{k_1}}(\xi_2) \f{\tilde{\vphi}^1_{k_2,l_1,l_2}}(\xi_1 + \xi_2) \f{\tilde{\psi}^2_{k_2}}(\xi_3) \nonumber \\
& = \sum_{\substack{0 < l_1+l_2 \leq M }} \sum_{\mu=100}^{\infty} 2^{-\mu(l_1+\l_2)}\underbrace{\sum_{k_2 = k_1 + \mu} \f{\tilde{\phi}^1_{k_1}}(\xi_1) \f{\tilde{\phi}^2_{k_1}}(\xi_2)\f{\tilde{\vphi}^1_{k_2,l_1,l_2}}(\xi_1 + \xi_2) \f{\tilde{\psi}^2_{k_2}}(\xi_3)}_{II_{1,\mu}^{1}}, \nonumber 
%&= \sum_{\substack{0 < l_1+l_2 \leq M }}\sum_{\mu=100}^{\infty} 2^{-\mu(l_1+\l_2)} \underbrace{ \sum_{k_2 = k_1 + \mu} \f{\phi}^1_{k_1}(\xi_1) \f{\phi}^2_{k_1}(\xi_2)\f{\tilde{\vphi}}_{k_2,l_1,l_2}(\xi_1 + \xi_2) \f{\psi}_{k_2}(\xi_3)}_{II_{1,\mu}^{1}}, \nonumber \\
\end{align*}
 One notices that $II_{1,\mu}^{1}$ has a form similar to (\ref{completion}) and can be rewritten as
$$
\sum_{k_2 = k_1 + \mu} \f{\tilde{\phi}^1_{k_1}}(\xi_1) \f{\tilde{\phi}^2_{k_1}}(\xi_2)\f{\tilde{\phi}^3_{k_1}}(\xi_1+\xi_2) \f{\tilde{\vphi}^1_{k_2,l_1,l_2}}(\xi_1+\xi_2) \f{\tilde{\psi}^2_{k_2}}(\xi_3)\f{\tilde{\psi}^3_{k_2}}(\xi_1+\xi_2+\xi_3).
$$ 
Meanwhile, one can decompose $II^1_{\text{rest}}$ such that
\begin{align*}
R_M(\xi_1,\xi_2)= &\sum_{\mu = 100}^{\infty}2^{-\mu M} \cdot II^1_{\text{rest},\mu} %m^1_{\mu} 
%\sum_{k_2 = k_1 + \mu} \f{\phi}^1_{k_1}(\xi_1) \f{\phi}^2_{k_1}(\xi_2)\f{\tilde{\vphi}}_{k_2,l_1,l_2}(\xi_1 + \xi_2) \f{\psi}_{k_2}(\xi_3)\nonumber \\
%\leq &\sum_{\mu=100}^{\infty}  2^{-\mu M} \underbrace{\sum_{k_2 = k_1 + \mu} \sum_{\substack{l_1 +l_2 > M }}\f{\phi}^1_{k_1}(\xi_1) \f{\phi}^2_{k_1}(\xi_2)\f{\tilde{\vphi}}_{k_2,l_1,l_2}(\xi_1 + \xi_2) \f{\psi}_{k_2}(\xi_3)}_{II^{1}_{\text{rest},\mu}},\nonumber \\
\end{align*}
where $II^1_{\text{rest},\mu}$ is a Coifman-Meyer symbol satisfying
$$
\left|\partial^{\alpha_1} II^1_{\text{rest},\mu}(\xi_1,\xi_2,\xi_3)\right| \lesssim 2^{\mu |\alpha_1|}\frac{1}{|(\xi_1,\xi_2,\xi_3)|^{|\alpha_1|}}
$$
for sufficiently many multi-indices $\alpha_1$.
 
Same procedure can be applied to study $a_2(\eta_1,\eta_2)b_2(\eta_1 \eta_2,\eta_3)$. One can now combine all the arguments above to decompose and study
$$
T_{ab} = T_{ab}^{I^1I^2} + T_{ab}^{I^1 II^2} + T_{ab}^{I^1 III^2} + T_{ab}^{II^1 I^2} + T_{ab}^{II^1 II^2} + T_{ab}^ {II^1 III^2} + T_{ab}^{III^1 I^1} + T_{ab}^{III^2 II^2} + T_{ab}^{III^1 III^2}
$$
where each operator takes the form
$$
\displaystyle \int_{\mathbb{R}^6} \text{symbol} \cdot \f{f_1}(\xi_1) \f{f_2}(\xi_2) \f{g_1}(\eta_1) \f{g_2}(\eta_2) \f{h}(\xi_3,\eta_3)e^{2\pi i x(\xi_1+\xi_2+\xi_3)} e^{2\pi i y(\eta_1+\eta_2+\eta_3)}d\xi_1 d\xi_2 d\xi_3  d\eta_1 d\eta_2 d\eta_3
$$
with the symbol for each operator specified as follows.
\begin{enumerate}
\item
$T_{ab}^{I^1I^2}$ is a bi-parameter paraproduct as in the special case.
\vskip .15in
\item
$T_{ab}^{II^1 I^2}$: $(II^{1}_0 + II^{1}_{1} + II^{1}_{\text{rest}}) \otimes I^2$ \newline
where the operator associated with each symbol can be written as
\begin{enumerate}[(i)]
\item
$$
T^{II_0^1 I^2} := \sum_{\substack{k_1 \ll k_2 \\ j \in \mathbb{Z}}}\left(\bigg(\big(( f_1 * \phi^1_{k_1}) (f_2 * \phi^2_{k_1}) * \tilde{\phi}^3_{k_1}\big) * \tilde{\vphi}^1_{k_2} \bigg) ( g_1 * \tilde{\phi}^1_{j}) (g_2 * \tilde{\phi}^2_{j})  (h * \tilde{\psi}^2_{k_2}\otimes \phi^3_{j})\right) * \tilde{\psi}^3_{k_2}\otimes \tilde{\phi}^4_{j} 
$$
\item
$$
T^{II^1_1 I^2} := \sum_{\substack{0 < l_1+l_2  \leq M }}\sum_{\mu= 100}^{\infty} 2^{-\mu(l_1+\l_2)} T^{II^1_{1,\mu}I^2}
$$
with
\begin{align*}
& T^{II^1_{1,\mu}I^2}\\
:=& \sum_{\substack{k_2 = k_1 + \mu \\ j \in \mathbb{Z}}}\left(\bigg(\big(( f_1 * \tilde{\phi}^1_{k_1}) (f_2 * \tilde{\phi}^2_{k_1}) * \tilde{\phi}^3_{k_1}\big) * \tilde{\vphi}^1_{k_2,l_1,l_2} \bigg)( g_1 * \tilde{\phi}^1_{j}) (g_2 * \tilde{\phi}^2_{j})  (h * \tilde{\psi}^2_{k_2}\otimes \phi^3_{j}) \right)* 
\tilde{\psi}^3_{k_2}\otimes \tilde{\phi}^4_{j}
\end{align*}
\item
$$
T^{II^1_{\text{rest}}I^2} := \sum_{\mu= 100}^{\infty}2^{\mu M} T^{II^1_{\text{rest},\mu}I^2}
$$

%$$
%T^{II^1_{\text{rest},\mu}I^2} :=  \int m^{1}_{\mu} (\xi_1,\xi_2,\xi_3) \sum_{j \in \mathbb{Z}}\f{\tilde{\phi}}_{j}(\eta_1) \f{\tilde{\phi}}_{j}(\eta_2) \f{\phi}_{j}(\eta_3) \f{\phi}_{j}(\eta_1+\eta_2+\eta_3) \f{f_1}(\xi_1) \f{f_2}(\xi_2) \f{g_1}(\eta_1) \f{g_2}(\eta_2) \f{h}(\xi_3,\eta_3) e^{2\pi i x(\xi_1+\xi_2+\xi_3)} e^{2\pi i y(\eta_1+\eta_2+\eta_3)}d\xi_1 d\xi_2 d\xi_3  d\eta_1 d\eta_2 d\eta_3 
%$$
One notices that $II^1_{\text{rest},\mu}$ and $I^2$ are Coifman-Meyer symbols. $T^{II^1_{\text{rest},\mu}I^2}$ is therefore a bi-parameter paraproduct and one can apply the argument in \cite{cptt} to derive the bound of type $O(2^{|\alpha_1|\mu})$, which would suffice due to the decay factor $2^{-\mu M}$.
\end{enumerate}

\item 
$T^{II^1 II^2}$: $(II^1_0 + II^1_1 + II^1_{\text{rest}}) \otimes (II^2_0 + II^2_1 + II^2_{\text{rest}})$ \newline
where the operator associated with each symbol can be written as
\begin{enumerate}[(i)]
\item
\begin{align*}
& T^{II_0^1 II_0^2} \\
:=& \sum_{\substack{k_1 \ll k_2 \\ j_1 \ll j_2}}\left(\bigg(\big(( f_1 * \phi^1_{k_1}) (f_2 * \phi^2_{k_1}) * \tilde{\phi}^3_{k_1}\big) * \tilde{\vphi}^1_{k_2} \bigg) \bigg(\big(( g_1 * \phi^1_{j_1}) (g_2 * \phi^2_{j_1}) * \tilde{\phi}^3_{j_1}\big) * \tilde{\vphi}^1_{j_2} \bigg) (h * \tilde{\psi}^2_{k_2}\otimes \tilde{\psi}^2_{j_2}) \right)* \tilde{\psi}^3_{k_2}\otimes \tilde{\psi}^3_{j_2} 
\end{align*}

\item
$$
T^{II^1_1 II_0^2} := \sum_{\substack{0 < l_1+l_2  \leq M }}\sum_{\mu= 100}^{\infty} 2^{-\mu(l_1+\l_2)} T^{II^1_{1,\mu}II^2_0}
$$
with
\begin{align*}
& T^{II^1_{1,\mu}II^2_0} \\
:=& \sum_{\substack{k_2 = k_1 + \mu \\ j_1 \ll j_2}}\left(\bigg(\big(( f_1 * \tilde{\phi}^1_{k_1}) (f_2 * \tilde{\phi}^2_{k_1}) * \tilde{\phi}^3_{k_1}\big) * \tilde{\vphi}^1_{k_2,l_1,l_2} \bigg)\bigg(\big(( g_1 * \phi^1_{j_1}) (g_2 * \phi^2_{j_1}) *\tilde{\phi}^3_{j_1}\big) * \tilde{\vphi}^1_{j_2} \bigg) (h * \tilde{\psi}^2_{k_2}\otimes \tilde{\psi}^2_{j_2}) \right)* \tilde{\psi}^3_{k_2}\otimes \tilde{\psi}^3_{j_2} 
\end{align*}

\item
$$
T^{II^1_{\text{rest}}II^2_0}:= \sum_{\mu= 100}^{\infty}2^{\mu M} T^{II^1_{\text{rest},\mu}II^2_0}
$$
where $T^{II^1_{\text{rest},\mu}II^2_0}$ is a multiplier operator with the symbol
$$
II^1_{\text{rest},\mu}\otimes II^2_0
$$
which generates a model similar as $T^{I^1 II^2_0}$ or, by symmetry, $T^{II^1_0 I^2}$.
\item
$$
T^{II^1_1 II^2_1}:= \sum_{\substack{0 < l_1+l_2  \leq M \\ l_1' + l_2' \leq M'}}\sum_{\mu,\mu'= 100}^{\infty} 2^{-\mu(l_1+\l_2)} 2^{\mu'(l_1'+l_2')}T^{II^1_{1,\mu}II^2_{1,\mu'}}
$$
with
\begin{align*}
& T^{II^1_{1,\mu}II^2_{1,\mu'}} \\
:=&  \sum_{\substack{k_2 = k_1 + \mu \\ j_2 = j_1 + \mu'}}\left(\bigg(\big(( f_1 * \tilde{\phi}^1_{k_1}) (f_2 * \tilde{\phi}^2_{k_1}) * \tilde{\phi}^3_{k_1}\big) * \tilde{\vphi}^1_{k_2,l_1,l_2} \bigg)\bigg(\big(( g_1 * \tilde{\phi}^1_{j_1}) (g_2 * \tilde{\phi}^2_{j_1}) * \tilde{\phi}^3_{j_1}\big) * \tilde{\vphi}^1_{j_2,l_1',l_2'} \bigg) (h *\tilde{\psi}^2_{k_2}\otimes \tilde{\psi}^2_{j_2}) \right)* \tilde{\psi}^3_{k_2}\otimes \tilde{\psi}^3_{j_2} 
\end{align*}

\item
$$
T^{II^1_{\text{rest}} II^2_1} := \sum_{\mu= 100}^{\infty}2^{\mu M} T^{II^1_{\text{rest},\mu}II^2_1}
$$
where $T^{II^1_{\text{rest},\mu}II^2_1}$ has the symbol
$$
II^1_{\text{rest},\mu}\otimes II^2_1
$$
which generates a model similar as $T^{I^1 II^2_1}$ or $T^{II^1_1 I^2}$.

\item
$$
T^{II^1_{\text{rest}}II^2_{\text{rest}}} := \sum_{\mu,\mu'= 100}^{\infty}2^{\mu M}2^{\mu'M'} T^{II^1_{\text{rest},\mu}II^2_{\text{rest},\mu'}}
$$
where $ T^{II^1_{\text{rest},\mu}II^2_{\text{rest},\mu'}}$ is associated with the symbol
$$
II^1_{\text{rest},\mu}\otimes II^2_{\text{rest},\mu'}
$$
which generates a model similar as $T^{II^1_{\text{rest},\mu}I^2}$, $T^{I^1 II^2_{\text{rest},\mu'}}$ or $T^{I^1I^2}$.
%$T^{I^1 II^2_0}$ or $T^{II^1_0 I^2}$.

\end{enumerate}

\item 
$T^{III^1 II^2}$, $T^{III^1 I^2}$ and $T^{III^1 III^2}$ can be studied by the exact same reasoning for $T^{II^1II^2}$, $T^{II^1 I^2}$ and $T^{II^1 II^2}$ by the symmetry between symbols $II$ and $III$.
\end{enumerate}

\vskip .15in
\subsection{Discretization} With discretization procedure specified in Section 2.2 of \cite{cw}, one can reduce the above operators into the following discrete model operators listed in Theorem (\ref{thm_weak}):
\begin{center}
\begin{tabular}{ c c c }
$T^{II^1_0 I^2}$ & $ \longrightarrow$ & $\Pi_{\text{flag}^0 \otimes \text{paraproduct}}$ \\
$T^{II^1_{1,\mu} I^2}$ & $ \longrightarrow$ & $ \Pi_{\text{flag}^{\mu} \otimes \text{paraproduct}}$ \\
$T^{II^1_0 II^2_0}$ & $ \longrightarrow$ & $ \Pi_{\text{flag}^0 \otimes \text{flag}^0} $  \\
$T^{II^1_0 II^2_{1,\mu'}}$ & $ \longrightarrow$ & $\Pi_{\text{flag}^0 \otimes \text{flag}^{\mu'}} $ \\
$T^{II^1_{1,\mu} II^2_{1,\mu'}}$ & $ \longrightarrow$ & $\Pi_{\text{flag}^{\mu} \otimes \text{flag}^{\mu'}} $  \\
\end{tabular}
\end{center}

\end{document}